\PassOptionsToPackage{unicode}{hyperref}
\PassOptionsToPackage{hyphens}{url}
\PassOptionsToPackage{dvipsnames,svgnames,x11names}{xcolor}
\documentclass[
  12pt]{article}

\usepackage{amsmath,amssymb}

\usepackage{mathrsfs}
\usepackage{bm}
\usepackage{amssymb}
\usepackage{multirow} 
\usepackage{makecell}
\usepackage{algorithm}
\usepackage{algorithmicx}
\usepackage{algpseudocode}
\usepackage{color}
\usepackage{float}
\usepackage{setspace}
\usepackage{subfig}
\usepackage{amsthm} 
\theoremstyle{plain}

\newtheorem{theorem}{Theorem}[section]
\newtheorem{lemma}[theorem]{Lemma}
\newtheorem{remark}[theorem]{Remark}
\newtheorem{corollary}[theorem]{Corollary}

\newtheorem{assumption}[theorem]{Assumption}
\DeclareMathOperator*{\supp}{\text{supp}}

\DeclareMathOperator*{\argmin}{\textup{arg\,min}}

\usepackage{iftex}
\ifPDFTeX
  \usepackage[T1]{fontenc}
  \usepackage[utf8]{inputenc}
  \usepackage{textcomp} 
\else 
  \usepackage{unicode-math}
  \defaultfontfeatures{Scale=MatchLowercase}
  \defaultfontfeatures[\rmfamily]{Ligatures=TeX,Scale=1}
\fi
\usepackage{lmodern}
\ifPDFTeX\else  
\fi
\IfFileExists{upquote.sty}{\usepackage{upquote}}{}
\IfFileExists{microtype.sty}{
  \usepackage[]{microtype}
  \UseMicrotypeSet[protrusion]{basicmath} 
}{}
\makeatletter
\@ifundefined{KOMAClassName}{
  \IfFileExists{parskip.sty}{%
    \usepackage{parskip}
  }{
    \setlength{\parindent}{0pt}
    \setlength{\parskip}{6pt plus 2pt minus 1pt}}
}{
  \KOMAoptions{parskip=half}}
\makeatother
\usepackage{xcolor}
\makeatletter
\ifx\paragraph\undefined\else
  \let\oldparagraph\paragraph
  \renewcommand{\paragraph}{
    \@ifstar
      \xxxParagraphStar
      \xxxParagraphNoStar
  }
  \newcommand{\xxxParagraphStar}[1]{\oldparagraph*{#1}\mbox{}}
  \newcommand{\xxxParagraphNoStar}[1]{\oldparagraph{#1}\mbox{}}
\fi
\ifx\subparagraph\undefined\else
  \let\oldsubparagraph\subparagraph
  \renewcommand{\subparagraph}{
    \@ifstar
      \xxxSubParagraphStar
      \xxxSubParagraphNoStar
  }
  \newcommand{\xxxSubParagraphStar}[1]{\oldsubparagraph*{#1}\mbox{}}
  \newcommand{\xxxSubParagraphNoStar}[1]{\oldsubparagraph{#1}\mbox{}}
\fi
\makeatother

\usepackage{longtable,booktabs,array}
\usepackage{calc} 
\usepackage{etoolbox}
\makeatletter
\patchcmd\longtable{\par}{\if@noskipsec\mbox{}\fi\par}{}{}
\makeatother
\IfFileExists{footnotehyper.sty}{\usepackage{footnotehyper}}{\usepackage{footnote}}
\makesavenoteenv{longtable}
\usepackage{graphicx}
\makeatletter
\def\maxwidth{\ifdim\Gin@nat@width>\linewidth\linewidth\else\Gin@nat@width\fi}
\def\maxheight{\ifdim\Gin@nat@height>\textheight\textheight\else\Gin@nat@height\fi}
\makeatother
\setkeys{Gin}{width=\maxwidth,height=\maxheight,keepaspectratio}
\makeatletter
\def\fps@figure{htbp}
\makeatother

\makeatletter
\@ifpackageloaded{caption}{}{\usepackage{caption}}
\AtBeginDocument{%
\ifdefined\contentsname
  \renewcommand*\contentsname{Table of contents}
\else
  \newcommand\contentsname{Table of contents}
\fi
\ifdefined\listfigurename
  \renewcommand*\listfigurename{List of Figures}
\else
  \newcommand\listfigurename{List of Figures}
\fi
\ifdefined\listtablename
  \renewcommand*\listtablename{List of Tables}
\else
  \newcommand\listtablename{List of Tables}
\fi
\ifdefined\figurename
  \renewcommand*\figurename{Figure}
\else
  \newcommand\figurename{Figure}
\fi
\ifdefined\tablename
  \renewcommand*\tablename{Table}
\else
  \newcommand\tablename{Table}
\fi
}
\@ifpackageloaded{float}{}{\usepackage{float}}
\floatstyle{ruled}
\@ifundefined{c@chapter}{\newfloat{codelisting}{h}{lop}}{\newfloat{codelisting}{h}{lop}[chapter]}
\floatname{codelisting}{Listing}

\makeatother
\makeatletter
\@ifpackageloaded{caption}{}{\usepackage{caption}}
\@ifpackageloaded{subcaption}{}{\usepackage{subcaption}}
\makeatother

\ifLuaTeX
  \usepackage{selnolig}  
\fi
\usepackage[]{natbib}
\usepackage{bookmark}

\IfFileExists{xurl.sty}{\usepackage{xurl}}{} 
\hypersetup{
  pdftitle={Title},
  pdfauthor={Author 1; Author 2},
  pdfkeywords={3 to 6 keywords, that do not appear in the title},
  colorlinks=true,
  linkcolor={blue},
  filecolor={Maroon},
  citecolor={Blue},
  urlcolor={Blue},
  pdfcreator={LaTeX via pandoc}}

\newcommand{\anon}{1}

\begin{document}

\def\spacingset#1{\renewcommand{\baselinestretch}%
{#1}\small\normalsize} \spacingset{1}


\if1\anon
{
  \title{\bf Nonparametric Prior Learning in Differential \\
  	\vskip 0.3 cm
  	Equation Modeling}
  \author{
    Junxiong Jia \\
    School of Mathematics and Statistics, Xi'an Jiaotong University\\
    and \\
    Deyu Meng, Zongben Xu\\
    School of Mathematics and Statistics, Xi'an Jiaotong University\\
    Pazhou Laboratory (Huangpu)\\
    and \\
    Fang Yao\thanks{Junxiong Jia is the first author. Fang Yao is the corresponding author, E-mail: fyao@math.pku.edu.cn. This research is partially supported by the National Natural Science Foundation of Chain (No. 12322116, 12271428, 12326606, 12292981), the National Key Research and Development Program of China (No. 2022YFA1004100, 2022YFA1003800), the New Cornerstone Science Foundation through Xplorer Prize, the Major Key Project of PCL (No. PCL2024A06), the Tianyuan Fund for Mathematics of the National Natural Science Foundation of China (No. 12426105), the Major projects of the National Natural Science Foundation of Chain (No. 12090021, 12090020), and the Social Science Foundation of the Ministry of Education of China (No. 24YJC790233), Fundamental Research Funds for the Central Universities, Peking University, and LMEQF.} \\
    Department of Probability and Statistics, \\
    Center for Statistical Science, Peking University
    }
    \date{}
  \maketitle
} \fi

\if0\anon
{
  \bigskip
  \bigskip
  \bigskip
  \begin{center}
    {\LARGE\bf Nonparametric Prior Learning in Differential \\
  	\vskip 0.3 cm
  	Equation Modeling}
  \end{center}
  \medskip
} \fi

\bigskip
\begin{abstract}
This paper addresses Bayesian inference related to partial differential equations (PDEs), particularly nonparametric regression constrained by PDEs. To effectively encode prior information, we propose a novel framework that learns a prediction function of the prior distribution from historical training datasets. We introduce hyper-prior and hyper-posterior distributions and derive a generalization error estimate, which accommodates data-dependent priors by extending the concept of differential privacy. Some mild conditions are given to validate the error estimate, where various typical PDEs such as diffusion and Darcy flow equations can be integrated. We thus formulate an infinite-dimensional optimization problem to obtain the point estimate of the hyper-posterior. Numerical examples demonstrate the performance of our proposed method in learning the prediction function of priors.
\end{abstract}

\noindent%
\vfill

\newpage
\spacingset{1.2} 

\section{Introduction}

Estimating function parameters in partial differential equations (PDEs) from observed data is a fundamental problem in various scientific and engineering fields. Such problems can be framed as regression tasks constrained by PDEs. From a deterministic perspective, these inverse problems have been thoroughly studied in the mathematical literature \citep{Engl1996Book}. Recently, statistical approaches to efficient parameter estimation—particularly Bayesian inference—have gained significant attention \citep{Cotter2013SS, Oates2019JASA}. 

Unlike deterministic methods, Bayesian inference provides a general framework for addressing uncertainties in inverse problems \citep{Arridge2019ActaNum, Stuart2010AN}, originally developed in finite-dimensional spaces \citep{Kaipio2004Book}. However, finite-dimensional theories often fail to characterize the continuous nature of inverse problems of PDEs. Thus, this work focuses on recent advances in Bayesian inversion for infinite-dimensional (non-parametric) spaces, as detailed in \cite{Dashti2017} and \cite{Stuart2010AN}.  Bayesian inverse problems face two key challenges: constructing effective priors and developing computationally efficient algorithms. Effective priors are essential for algorithmic efficiency, acting as regularizers and encoding prior knowledge of model parameters. We now briefly summarize two research directions for constructing effective priors.  

In constructing structured priors, a common approach uses convergent series expansions with coefficients from finite-dimensional distributions, yielding uniform and Gaussian priors \citep{Dashti2017}. These have significantly contributed to the advancement of sequential Monte Carlo algorithms \citep{Beskos2015SC} and full waveform inversion \citep{Dunlop2021JUQ}. For further insights of Gaussian measures in Bayesian inversion, see \cite{Tan2013SISC}, \cite{Dashti2013IP}, and referenced works.  
The Besov prior, valued for edge preservation \citep{Lassas2009IPI}, often employs Haar wavelet bases, which can impose restrictive assumptions \citep{Arridge2019ActaNum}. As an alternative, \cite{Markkanen2019JIIP} proposed the Cauchy difference prior. While these priors characterize the smoothness of the function parameters, they may not fully reflect specific inverse problem structures.  

Optimal posterior contraction in infinite-dimensional settings presents unique challenges, as many priors fail to ensure posterior consistency \citep{Arridge2019ActaNum}. Thus, prior-induced errors may persist even with infinite amount of data. For linear forward operators with Gaussian priors and noise, consistency is well-established \citep{Knapik2016,Szabo2015AS}, achieving minimax contraction rates when prior regularity matches the true function.  In nonlinear cases, tail conditions on prior measures were proposed in \cite{Vollmer2013IP}, and recent studies applying measure concentration theory have derived minimax-optimal posterior contraction rates \citep{Nickl2017AoS,Nickl2020JEMS}. For detailed insights into posterior contraction in nonlinear inverse problems, see \cite{Monard2019AoS}, \cite{Nickl2022LectureNotes}, and cited works. As summarized in \cite{Arridge2019ActaNum}, the findings on posterior consistency and optimal contraction rates underscore the regularizing characteristics of the prior measure. However, research on selecting appropriate priors for fixed data and specific noise levels remains limited.  

The challenge in creating effective prior measures in infinite-dimensional spaces is highlighted due to their complexity, which is akin to manually designed features in machine learning. The paper draws a parallel with deep learning's ability to outperform traditional feature engineering by automatically extracting features from data \citep{Goodfellow2016Book}, suggesting a need for data-driven methods to learn prior measures. Despite numerous studies focusing on learning regularizers from data, there is a noted limitation in equating regularizers to probability measures, which hinders a direct connection to Bayesian inference \citep{Dashti2013IP}. A bilevel optimization framework has been proposed for learning regularizers, with a focus on parameterized regularizers that involve derivatives or filters \citep{Haber2003IP}. This approach has been extended to non-smooth regularizers, primarily addressing optimization issues such as the existence of optimal solutions, rather than learning theory \citep{Antil2020IP, Reyes2016JMIV}. Recent developments in this field incorporate deep learning techniques, including generative adversarial networks \citep{Lunz2018NIPS} and auto-encoder-based methods \citep{Li2020IP}. However, as noted by \cite{Afkham2021IP}, the efficacy of estimated regularizer parameters in being optimally averaged can be diminished when dealing with highly complex distributions of model parameters.

\subsection{Our contribution}

In this study, we reframe the inverse problem as a constrained regression problem, influenced by forward processes like PDEs. This perspective casts model parameter estimation as a machine learning problem, treating the parameters as components of an unknown regression operator. Consequently, learning prior measures emerges as a meta-learning problem, akin to methods like model-agnostic meta-learning \citep{Finn2018ICML}, albeit with distinct techniques and contexts. Recent extensions of the probably approximately correct (PAC)-Bayesian learning theory to meta-learning \citep{Rothfuss2021PMLR,Rezazadeh2022ArXiv,Riou2023ArXiv,Rothfuss2023JMLR} have primarily targeted data-independent priors, which capture shared hyperparameters across different learning tasks. 

{\color{black}Inspired by \cite{Shu2021JMLR}, we aim to develop a general method for learning prediction functions from historical inverse task datasets. The functions should generate prior measures tailored to a new single inverse problem. To this end, we construct a data-dependent prior learning method by rigorously extending PAC-Bayesian theory to infinite-dimensional settings. }Unlike traditional machine learning, Bayesian inversion for PDEs employs a fully nonparametric (infinite-dimensional) approach \citep{Stuart2010AN,Nickl2022LectureNotes}, with both unknown parameters and observed data in infinite-dimensional function spaces. Formulating suitable PAC-Bayesian theory here is challenging due to issues like absent classical density functions, PDE regularity, and singularities of Gaussian measures. Standard machine learning prior measures—such as the Gaussian $N(0,\sigma^2 I)$ with $\sigma^2 \in \mathbb{R}^{+}$—are ill-suited to this context. The principal contributions can be summarized as follows:  

\begin{enumerate}
	\item Inspired by the studies of empirical Bayesian approach \citep{Szabo2015AS} and investigations of meta-learning \citep{Shu2021JMLR}, we introduce the prediction function of the prior probability measures, i.e., a function mapping from the data space to the space of prior measures. By generalizing the classical concept of differential privacy, we prove a general PAC-Bayesian generalization estimate for learning data-dependent prior probability measures of the inverse problems of PDEs.
	\item We establish general conditions for forward PDE operators, encompassing both diffusion and Darcy flow equations. Utilizing estimations of PDEs, we derive specific generalization bounds under these conditions, thereby validating the sub-Gaussian assumption of loss functions, characterized by the negative logarithm of the likelihood function. To our knowledge, this represents a novel approach not previously explored in the realm of nonparametric Bayesian regression constrained by PDEs.
	\item Although the derived estimate may not be optimal, we formulate an infinite-dimensional optimization problem that avoids the bilevel structure prevalent in prior learning studies. This approach is thus more computationally tractable. The proposed learning algorithm is rigorously defined on infinite-dimensional spaces, consistent with the principles of infinite-dimensional sampling algorithms (see \cite{Cotter2013SS}). We apply these algorithms to both backward diffusion (linear) and Darcy flow (nonlinear) problems. The numerical results not only validate the theoretical insights but also demonstrate the practical efficacy of the proposed methods.
\end{enumerate}

The paper unfolds as follows: Section \ref{SectionLearningPriorMeasure} consists of three parts. Subsection \ref{SubsectionGeneralForm} introduces the general framework. Subsection \ref{GeneralTheory} develops a PAC-Bayesian generalization bound. Subsection \ref{SectionDataDependentPrior} presents refined bounds with data-dependent priors. Section \ref{SectionApplication} derives specific generalization bounds under general forward operator conditions. Section \ref{SectionLearningAlgorithms} constructs learning algorithms based on these bounds, applied to the Darcy flow problem in Section \ref{SectionNumerics}. Section \ref{SectionConcludion} discusses potential extensions for future research. Additional analyses, detailed proofs, and code implementations are provided in the online supplement materials.

\section{Learning Theory of Prior Measures}\label{SectionLearningPriorMeasure}
In this section, we develop PAC-Bayesian theory for the learning of data-dependent prior measures in a nonparametric Bayesian inversion approach.

\subsection{Bayesian formulation}\label{SubsectionGeneralForm}
This subsection establishes the fundamental framework of the study. {\color{black}Consider a forward operator $\mathcal{G}:\mathcal{U}\to\tilde{\mathcal{H}}$, where $\mathcal{U}$ and $\tilde{\mathcal{H}}$ are separable Banach spaces. Data follows the model:  
\begin{align}\label{problem1}
	y_j = \mathcal{L}_{x_j}(\mathcal{G}(u)) + \eta_j, \quad\forall \,j=1,\ldots,m,
\end{align}  
with $u$ as the model parameter, $x_j\in\mathcal{X}$ denoting measurement points or input functions, $y_j\in\mathcal{Y}$ as observed data, and $\eta_j$ representing independent identically distributed (i.i.d.) noise. Introduce another separable Banach space $\mathcal{H}\subset \mathcal{Y}$, where $\mathcal{X}$ and $\mathcal{Y}$ are separable Banach spaces and $\mathcal{L}_{x_j}:\tilde{\mathcal{H}}\to\mathcal{H}$ is a bounded linear operator dependent on $x_j$.} Define $\bm{y} = (y_1,\cdots,y_{m})^T$, $\bm{\eta} = (\eta_1,\cdots, \eta_{m})^T$, $\bm{x} = (x_1,\cdots,x_{m})^T$, so the model (\ref{problem1}) compacts to  
$\bm{y} = \mathcal{L}_{\bm{x}}\mathcal{G}(u) + \bm{\eta},$  
where $\mathcal{L}_{\bm{x}}\mathcal{G}(u) := (\mathcal{L}_{x_1}(\mathcal{G}(u)),\ldots,\mathcal{L}_{x_{m}}(\mathcal{G}(u)))^T$. For inverse problems, assume a true parameter $u^{\dagger}$, leading to the forward model  
$\bm{y} = \mathcal{L}_{\bm{x}}\mathcal{G}(u^{\dagger}) + \bm{\eta}.$  
Detailed examples (linear and nonlinear problems) are given in Section \ref{SectionApplication} and Subsection A.7 
of the supplement. In several subsequent sections, we instead focus on separable Hilbert spaces. 

Throughout this work, let $\mathcal{P}(\mathcal{Y})$ denote the set of all probability measures on a separable Banach space $\mathcal{Y}$. In Bayesian inverse theory, the unknown parameter $u$ is a random element under a prior measure $\mathbb{P}\in\mathcal{P}(\mathcal{U})$. The posterior measure $\mathbb{Q}\in\mathcal{P}(\mathcal{U})$ is defined by:  
\begin{align}\label{BayesFormula1}
	\frac{d\mathbb{Q}}{d\mathbb{P}}(u) = \frac{1}{Z_m}\exp\left( - \sum_{j=1}^{m}\Phi(u;y_j)\right),
\end{align}  
where $\Phi(u;y_j)$ is the potential function from the likelihood, and $Z_{m}$ is the normalizing constant. For a fixed $u$, data $\bm{y}$ depends on $\bm{x}$ and $\bm{\eta}$, so we also write the potential as $\Phi(u;x_j,\eta_j)$ for $j=1,\ldots,m$. Conditions on $\mathcal{G}$ for formula (\ref{BayesFormula1}) are reserved for Section~\ref{SectionApplication}, keeping fundamental concepts accessible to all readers..  

We consider a joint probability measure $\mathbb{D}(dx,dy) := \mathbb{D}_1(dx)\mathbb{D}_2(x, u^{\dagger}, dy)$, where $\mathbb{D}_1$ and $\mathbb{D}_2$ are probability measures on $\mathcal{X}$ and $\mathcal{Y}$, respectively. Define $\mathcal{Z} := \mathcal{X}\times\mathcal{Y}$, so the random element $z := (x,y)$ satisfies $z\sim\mathbb{D}\in\mathcal{P}(\mathcal{Z})$. The dataset $S = \{z_j = (x_j,y_j)\}_{j=1}^{m}$ consists of i.i.d. samples from $\mathbb{D}$ (denoted $S\sim\mathbb{D}^m$). The measure $\mathbb{D}_2$ depends on the true parameter $u^{\dagger}$, $\mathbb{D}_1$, and the noise distribution $\mathbb{D}_3$ (a probability measure on $\mathcal{Y}$), with $\{x_j\}_{j=1}^{m} \sim \mathbb{D}_1$. For the Darcy flow model detailed in Section~\ref{SectionApplication}, the set $\{x_j\}_{j=1}^{m}$ are sparse observed points, often assumed evenly distributed over the domain, see \cite{Dunlop2016SC}. Taking $\mathbb{D}_1$ as a uniform distribution ensures these points are statistically evenly spaced within the domain of the Darcy flow problem \citep{Monard2021CPAM}.

\subsection{General theory}\label{GeneralTheory}

We assume $n$ measurement datasets $\{S_1, \dots, S_n\}$, each generated by a distinct parameter $u_i^\dagger$, such that  
$y_{ij} = \mathcal{L}_{x_{ij}}\mathcal{G}(u_i^\dagger) + \eta_{ij}$  
for $i=1,\ldots,n$ and $j=1,\ldots,m_i$. Here, $\{u_i^\dagger\}_{i=1}^n$ are $n$ background true parameters, with dataset $S_i$ defined as  
$S_i = \{(x_{i1}, y_{i1}),\ldots,(x_{im_i}, y_{im_i})\}$ for all $i=1,\ldots,n$.  Further assume the background true parameters are random elements sampled from a probability measure $\mathscr{E} \in \mathcal{P}(\mathcal{U})$. This parallels the ``environment'' concept in meta-learning \citep{Maurer2005JMLR}.  

Observe that $\mathbb{D}(dx, dy) = \mathbb{D}_2(x, u, dy)\mathbb{D}_1(dx)$ reveals a correspondence between each $u \sim \mathscr{E}$ and a measure $\mathbb{D} \in \mathcal{P}(\mathcal{Z})$.
Based on these observations, we can introduce a probability measure $\mathcal{T}$ defined on $\mathcal{P}(\mathcal{Z})\times\mathbb{Z}^{+}$ such 
that $(\mathbb{D}, m)\sim\mathcal{T}$ corresponds to a parameter $u$ sampled from $\mathscr{E}$ and a positive integer $m$ sampled from some probability distribution defined on $\mathbb{Z}^+$. 
With these notations, the data $S_i$ can be recognized as a sample of $\mathbb{D}_i^{m_i}$ 
with i.i.d. samples $(\mathbb{D}_i, m_i)\sim\mathcal{T}$ for $i=1,\ldots,n$. 

{\color{black}Assume the prior measure $\mathbb{P}$ is parameterized by $v \in \mathcal{V}$ (a separable Banach space), so that $\mathbb{P} = \mathbb{P}(\cdot; v)$. Define a parameterized mapping $f(\cdot; m_i, \theta): \mathcal{Z}^{m_i} \to \mathcal{V}$ with $\theta \in \Theta$ (a separable Banach space). Since $m_i$ is irrelevant to the following discussions, we often omit it and just use $f(\cdot;\theta)$. Here, we define $S_i \mapsto f(S_i;\theta)$ as the \emph{prediction function} of the prior measure. For $\theta$, we consider two probability measures: the hyper-prior $\mathscr{P} \in \mathcal{P}(\Theta)$ and hyper-posterior $\mathscr{Q} \in \mathcal{P}(\Theta)$. }Our goal is to develop algorithms that integrate the information encoded in $\{S_1, \ldots, S_n\}$ with the hyper-prior $\mathscr{P}$ to produce the hyper-posterior $\mathscr{Q}$. This measure $\mathscr{Q}$ should be capable of yielding a meaningful prior measure $\mathbb{P}$.

\begin{remark}
	If $\mathcal{V} = \Theta$ and $f$ is a constant mapping that maps every $S \in \mathcal{Z}^{m_i}$ to $\theta$, then we can consider the prior measure $\mathbb{P} = \mathbb{P}(\cdot; \theta)$, where $\mathscr{P}$ and $\mathscr{Q}$ are the hyper-prior and hyper-posterior measures for the parameter $\theta$, respectively.
\end{remark}

{\color{black}
	\begin{remark} 
		Once the hyper-posterior is fixed, the data-dependent base prior $\mathbb{P}(\cdot;f(S;\theta))$ relies on a single dataset $S$ for a newly encountered inverse problem, consistent throughout this paper. After training, parameters $\theta\sim\mathscr{Q}$ depend on all historical datasets $\{S_i\}_{i=1}^{n}$. Conceptually, $\mathscr{Q}$ extracts information from the data distribution $\mathbb{D}$ or environment $\mathcal{T}$.  
	\end{remark}
	
	\begin{remark}
		For simplicity, we can model $\mathbb{P}$ as a Gaussian measure $\mathcal{N}(v, \mathcal{C}_0)$ defined on $\mathcal{U}$, parameterized by a learnable mean $v$ and fixed covariance $\mathcal{C}_0$. To ensure equivalence across means $v$, $\mathcal{V}$ is taken to be the Cameron-Martin space of $\mathcal{C}_0$. The mapping $f$ could be a neural network with parameters $\theta \in \Theta$. A sophisticated example is in Subsection \ref{SectionDataDependentPrior}.  
	\end{remark}
}

To validate our theoretical results, define a loss function $\ell: \mathcal{U}\times\mathcal{Z}\to\mathbb{R}$ and denote the expected error under $\mathbb{D}$ as $\mathcal{L}(u, \mathbb{D}) := \mathbb{E}_{z\sim\mathbb{D}}\ell(u, z)$. For inverse problems, since the true distributions of $u$ and $\bm{\eta}$ are unknown, we define the empirical error  
$\hat{\mathcal{L}}(u, S) := \frac{1}{m}\sum_{j=1}^{m}\ell(u, z_j).$  
The Gibbs error and its empirical counterpart are given by  
$\mathcal{L}(\mathbb{Q}, \mathbb{D}) := \mathbb{E}_{u\sim\mathbb{Q}}\mathcal{L}(u, \mathbb{D}) \text{ and } \hat{\mathcal{L}}(\mathbb{Q}, S) := \mathbb{E}_{u\sim\mathbb{Q}}\hat{\mathcal{L}}(u, S).$  
To distinguish hyper-prior/hyper-posterior from base prior/posterior, refer to (\ref{BayesFormula1}) as the base prior and base posterior. Since $\mathbb{Q}$ depends on $S$ and $\mathbb{P}$, denote it as $\mathbb{Q}(S,\mathbb{P})$ for clarity.  
Inspired by the investigations of meta-learning \citep{Shu2021JMLR,Rothfuss2021PMLR}, the performance of the hyper-posterior is measured via the \emph{transfer-error}:  
$\mathcal{L}(\mathscr{Q}, \mathcal{T}) := \mathbb{E}_{\theta\sim\mathscr{Q}}\left[ 
\mathbb{E}_{(\mathbb{D},m)\sim\mathcal{T}}\left[ 
\mathbb{E}_{S\sim\mathbb{D}^m}\mathcal{L}(\mathbb{Q}(S, \mathbb{P}_{S}^{\theta}), \mathbb{D})
\right]\right],$
where $\mathbb{P}_{S}^{\theta} := \mathbb{P}(\cdot; f(S;\theta))$. While the transfer error is unknown in practice, we can estimate it using the empirical \emph{multi-task error} defined by
$\hat{\mathcal{L}}(\mathscr{Q}, S_1,\ldots,S_n) := \mathbb{E}_{\theta\sim\mathscr{Q}}\left[ 
\frac{1}{n}\sum_{i=1}^n \hat{\mathcal{L}}(\mathbb{Q}(S_i, \mathbb{P}_{S_i}^{\theta}), S_i)
\right].$

Let us denote $D_{\text{KL}}(\cdot || \cdot)$ be the usual Kullback-Leibler (KL) divergence of two probability measures. Then, we provide an upper bound on the true transfer error $\mathcal{L}(\mathscr{Q}, \mathcal{T})$ in terms of the empirical multi-task error $\hat{\mathcal{L}}(\mathscr{Q}, S_1,\ldots,S_n)$ plus some general complexity terms. 

\begin{theorem}\label{PAC_Bounds_Theorem1}
	Given the data space $\mathcal{Z}=\mathcal{X}\times\mathcal{Y}$ where $\mathcal{X}$ and $\mathcal{Y}$ are separable 
	Banach spaces, parameter space $\mathcal{U}$ and $\Theta$ are separable Banach spaces. 
	Assume the loss function $\ell: \mathcal{U}\times\mathcal{Z}\rightarrow\mathbb{R}$ be a measurable function.
	For $i=1,\ldots,n$, let $\mathbb{Q}(S_i,\mathbb{P}_{S_i}^{\theta})$ be the $i$-th base 
	posterior measure, $\mathscr{P}\in\mathcal{P}(\Theta)$ be some fixed hyper-prior and $\lambda,\gamma > 0$. 
	Let us define $\mathbb{E}f := \mathbb{E}_{(\mathbb{D}_1,m_1)\sim\mathcal{T}}\cdots\mathbb{E}_{(\mathbb{D}_n,m_n)\sim\mathcal{T}}
	\mathbb{E}_{S_1\sim\mathbb{D}_1^{m_1}}\cdots\mathbb{E}_{S_n\sim\mathbb{D}_n^{m_n}}f$ for every measurable function $f$. 
	For any confidence level $\delta\in(0,1]$, we have 
	\begin{align}
			\mathbb{P}\Bigg(\forall \mathscr{Q}\in\mathcal{P}(\Theta), \,
			\mathcal{L}(\mathscr{Q}, & \mathcal{T}) \leq \, \hat{\mathcal{L}}(\mathscr{Q},S_1,\ldots,S_n) +  
			\left( \frac{1}{\lambda} + \frac{1}{\gamma} \right)D_{\text{KL}}(\mathscr{Q} || \mathscr{P})  \nonumber\\
			& + \frac{1}{\gamma}\sum_{i=1}^{n}\mathbb{E}_{\theta\sim\mathscr{Q}}D_{\text{KL}}
			\big(\mathbb{Q}(S_i,\mathbb{P}_{S_i}^{\theta}) || \mathbb{P}_{S_i}^{\theta}\big) \label{GeneralThmInequality1}\\
			& \, + \frac{1}{\sqrt{n}}\ln\mathbb{E}\big[\exp\big(\sqrt{n}\Pi^1(\gamma) + \sqrt{n}\Pi^2(\lambda)\big)\big]
			+ \frac{1}{\sqrt{n}}\ln\frac{1}{\delta}\Bigg) \geq 1-\delta,    \nonumber
	\end{align}
	where $\mathbb{P}$ is the probability with respect to the datasets $\{S_i\}_{i=1}^n$. 
	In estimate (\ref{GeneralThmInequality1}), the notations
	\begin{align*}
		& \Pi^{1}(\gamma) = \frac{1}{\gamma}\!\ln\mathbb{E}_{\theta\sim\mathscr{P}}\mathbb{E}_{u_1\sim\mathbb{P}_{S_1}^{\theta}}
		\cdots\mathbb{E}_{u_n\sim\mathbb{P}_{S_n}^{\theta}}\exp\left( 
		\frac{\gamma}{n} \sum_{i=1}^{n}\frac{1}{m_i}\sum_{j=1}^{m_i}\left[ 
		\mathbb{E}_{z\sim\mathbb{D}_{i}}\ell(u_i, z) -  \ell(u_i, z_{ij})
		\right]\!\right), \\
		& \Pi^2(\lambda) = \frac{1}{\lambda}  \ln\mathbb{E}_{\theta\sim\mathscr{P}}\exp\left(
		\frac{\lambda}{n} \sum_{i=1}^{n}\left[ 
		\mathbb{E}_{(\mathbb{D}, m)\sim\mathcal{T}}\mathbb{E}_{S\sim\mathbb{D}^{m}}\!\left[ \mathcal{L}(\mathbb{Q}(S, \mathbb{P}_{S}^{\theta}), \mathbb{D}) \right]  -  
		\mathcal{L}(\mathbb{Q}(S_i,\mathbb{P}_{S_i}^{\theta}),\mathbb{D}_i) 
		\right]	\right).
	\end{align*}
\end{theorem}

\begin{remark}\label{RemarkMeaningParam} {\color{black}
		Here we consider \( n \) i.i.d. historical datasets from \( n \) inverse tasks. For the \( n \) base inverse problems, the parameter \( \gamma \) is defined at the base level (not the task level), taking the form \( \gamma = n\beta \), where \( \beta \) is from Theorem A.1 
		(a classical PAC-Bayes bound) in the supplementary materials. Here, each task has a different number of measurement points. We define $\bar{m} = \left(\frac{1}{n}\sum_{i=1}^n\frac{1}{m_i}\right)^{-1}$ and take $\beta = \bar{m}$ or $\sqrt{\bar{m}}$ (see Corollary \ref{PACDependentThmSubGaussian}). In contrast, \( \lambda \) is characterized at the task level. Intuitively, each task-level inverse task can be viewed as a data pair. Analogous to classical settings, \( \lambda \) may be chosen as \( n \) or \( \sqrt{n} \). }
\end{remark}

In inequality (\ref{GeneralThmInequality1}), the term $\ln\mathbb{E}\big[\exp\big(\sqrt{n}\Pi^1(\gamma) + \sqrt{n}\Pi^2(\lambda)\big)\big]$ is a log moment generating function that quantifies how much the empirical multi-task error deviates from the transfer-error. It is crucial to derive an explicit estimate of this term to obtain a meaningful bound.

\subsection{Data-dependent prior}\label{SectionDataDependentPrior}

The conventional assumption posits that prior measures are independent of observed data. However, specifying prior parameters remains challenging in many practical inverse problems, often relying on the practitioner’s expertise. To mitigate this subjectivity, empirical Bayes methods adjust these parameters using dataset $S$ \citep{Szabo2015AS}. Building on this idea, we extend the prior measure $\mathbb{P}_{S}^{\theta}$ to be data-dependent in Subsection \ref{GeneralTheory}. Unlike empirical Bayes, our approach trains a probabilistic model to directly map datasets to hyper-parameters of the prior, analogous to non-Bayesian methods \citep{Shu2021JMLR}.

{\color{black}A PAC-Bayesian learning theory for meta-learning with data-dependent priors \citep{Liu2021Neurocomputing} has been developed for finite-dimensional settings. The data-dependency there arises from data separation, whereas our framework introduces a prediction function (not considered in previous work). Previous studies on data-dependent priors (not for meta-learning) leverage differential privacy \citep{Dziugaite2018NIPS}, but this approach is too restrictive for the infinite-dimensional Gaussian measures widely used in inverse problems. We address this limitation by proposing a new assumption inspired by differential privacy. The intuition of Assumption \ref{DataDependentAssumption} is that discrepancies between data-dependent prior measures can be controlled, enabling the use of techniques for data-independent priors. For further discussion on differential privacy, see \citep{Dimitrakakis2016JMLR,Dwork2015NIPS,Hall2013JMLR} and related works.}

In the following, for convenience, we define an induced measure $\mathbb{D}_{\mathcal{T}} \in\mathcal{P}(\mathcal{Z}^m)$ 
as in \cite{Maurer2005JMLR}, by $\mathbb{D}_{\mathcal{T}}(dS) = \mathbb{E}_{(\mathbb{D},m)\sim\mathcal{T}}\left[\mathbb{D}^m(dS)\right]$. The corresponding expectation for a measurable function $f$ on $\mathcal{Z}^m$ is then
$\mathbb{E}_{S\sim\mathbb{D}_{\mathcal{T}}}[f] = \mathbb{E}_{(\mathbb{D},m)\sim\mathcal{T}}\left[\mathbb{E}_{S\sim\mathbb{D}^m}[f]\right]$.

\begin{assumption}\label{DataDependentAssumption}
	Let $S\in\mathcal{Z}^m$, $S'\in\mathcal{Z}^{m'}$, $\theta\in\Theta$, and $u\in\mathcal{U}$. We assume that the probability measure 
	$\mathbb{P}_{S}^{\theta}$ is absolutely continuous with respect to the measure $\mathbb{P}_{S'}^{\theta}$ such that 
	$$\frac{d\mathbb{P}_{S}^{\theta}}{d\mathbb{P}_{S'}^{\theta}}(u) = \Psi(S,S',u,\theta),$$
	where the function $\Psi(S,S',u, \theta)$ satisfies 
	$\tilde{\Psi}_{E}:=\mathbb{E}_{S'\sim\mathbb{D}_{\mathcal{T}}}
	\mathbb{E}_{u\sim\mathbb{P}_{S'}^{\theta}}\mathbb{E}_{S\sim\mathbb{D}_{\mathcal{T}}}\Psi(S,S',u,\theta)^2 \leq \Psi_{E} < +\infty$
	with $\Psi_{E}$ being a positive constant independent of $m$, $m'$, and $\theta$.
\end{assumption}

{\color{black}
	To show that Assumption \ref{DataDependentAssumption} holds under certain conditions, consider a Gaussian prior $\mathbb{P} := \mathcal{N}(0, \mathcal{C}_0)$, where $\mathcal{C}_0$ is a positive-definite, symmetric trace-class operator. Let $\mathcal{U}$ be a separable Hilbert space with orthogonal basis $\{e_k\}_{k=1}^{\infty}$. Assume the prior covariance operator has spectral decomposition $\mathcal{C}_0 = \sum_{k=1}^{\infty} \lambda_k \, e_k \otimes e_k$ with $\lambda_1 \geq \lambda_2 \geq \cdots > 0.$  Define $\mathcal{U}^2$ as the Hilbert scale induced by $\mathcal{C}_0$ (see \cite{Engl1996Book} or Subsection \ref{SubsectionUnbounded}). Before proceeding, we state the following assumptions, which will be employed in the subsequent analysis and examples.  
	
	\begin{assumption}\label{HyperpriorAssumption}
		{\color{black}Consider the parameter space $\Theta = \Theta_1 \times \Theta_2$, with $\Theta_1$ a separable Hilbert space and $\Theta_2 = \mathbb{R}^{N_c}$. Assume the hyper-prior measure $\mathscr{P} = \mathscr{P}_1 \otimes \mathscr{P}_2$, where $\mathscr{P}_1$ is a probability measure on $\Theta_1$, and $\mathscr{P}_2$ is a probability measure on $\mathbb{R}^{N_c}$ with compact support $[\theta_{\text{min}},\theta_{\text{max}}]^{N_c}$. Here, the parameters $\theta_{\text{min}},\theta_{\text{max}}\in\mathbb{R}$ ($\theta_{\text{min}} < \theta_{\text{max}}$) are fixed real numbers. }
	\end{assumption} 
	
	We assume the above Assumption \ref{HyperpriorAssumption} holds true and introduce the mapping $f(S; \theta) := \big( f_m(S; \theta_1), \, \theta_2 \big) \in \mathcal{V} := \mathcal{U}^2 \times \mathbb{R}^{N_c}$. The data-dependent learned prior measure is specified as $\mathbb{P}_S^{\theta} := \mathcal{N} \Big( f_m(S; \theta_1), \, \mathcal{C}_0(\theta_2) \Big)$, where the covariance operator  
	\begin{align}\label{learnedCov}
		\mathcal{C}_0(\theta_2) = \sum_{k=1}^{N_c} e^{\theta_{2k}} \, e_k \otimes e_k + \sum_{k=N_c+1}^{\infty} \lambda_k \, e_k \otimes e_k
	\end{align}  
	with $N_c \in \mathbb{N}^+$ is a predefined constant. For the mean function $f_m(S; \theta_1)$, we constrain all admissible $f_m(S; \theta_1)$ to the ball $B_{\mathcal{U}^2}(R_u) \subset \mathcal{U}^2$ defined by $B_{\mathcal{U}^2}(R_u) := \Big\{ f \in \mathcal{U}^2 \, \big| \, \| \mathcal{C}_0^{-1} f \|_{\mathcal{U}} \leq R_u \Big\}$, where $R_u > 0$ being a sufficiently large prespecified constant. Under these assumptions, taking $0 < \epsilon < \min\left\{ \frac{1}{4}e^{-\theta_{\text{max}}}, \frac{1}{4\lambda_1} \right\}$ yields the finiteness of $\tilde{\Psi}_E$:
	\begin{align}\label{EstimatePsiE}
		\tilde{\Psi}_E \leq \Psi_{E} = \exp\left(\frac{(4+8\epsilon^2)R_u^2}{\epsilon}\right)\frac{N_c}{\sqrt{1-4\epsilon e^{\theta_{\text{max}}}}}\prod_{k=N_c+1}^{\infty}(1-4\epsilon \lambda_k)^{-\frac{1}{2}} < +\infty.
	\end{align}  
	Detailed explanations and the derivation of (\ref{EstimatePsiE}) are provided in Subsection A.4
	of the supplementary materials. To refine Theorem \ref{PAC_Bounds_Theorem1}, we introduce two additional assumptions.  
}

\begin{assumption}[sub-Gaussian assumption of $\Pi^1$]\label{UnboundedLossAssumptionSubGaussian1}
	For $i=1,\ldots,n$, we assume that the random variables $\mathcal{L}(u_i,\mathbb{D}_i) - \ell(u_i, z_i)$
	are sub-Gaussian with variance factor $s_{\text{I}}^2$. Specifically, for some $\tilde{\gamma} > 0$, the following inequality holds: 
	$$V_i^1 := \mathbb{E}_{(\mathbb{D}_i, m_i)\sim\mathcal{T}}\mathbb{E}_{z_i\sim\mathbb{D}_i}
	\mathbb{E}_{\theta\sim\mathscr{P}}\mathbb{E}_{u_i\sim\mathbb{P}_{S_i}^{\theta}}\exp\left( 
	\tilde{\gamma}\left[ 
	\mathcal{L}(u_i,\mathbb{D}_i) - \ell(u_i, z_i)
	\right]\right) 
	\leq \exp\left(\tilde{\gamma}^2s_{\text{I}}^2/2 \right).$$
\end{assumption}

\begin{assumption}[sub-Gaussian assumption of $\Pi^2$]\label{UnboundedLossAssumptionSubGaussian2}
	For $i=1,\ldots,n$, we assume that the random variables  $\mathbb{E}_{(\mathbb{D},m)\sim\mathcal{T}}\mathbb{E}_{S\sim\mathbb{D}^m}\mathcal{L}(\mathbb{Q}(S,\mathbb{P}^{\theta}),\mathbb{D})
	- \mathcal{L}(\mathbb{Q}(S_i,\mathbb{P}_{S_i}^{\theta}),\mathbb{D}_i)$ are sub-Gaussian with 
	variance factor $s_{\text{II}}^2$. Specifically, for some $\tilde{\lambda} > 0$, the following inequality holds:
	\begin{align*}
		V_i^2 &:= \mathbb{E}_{(\mathbb{D}_i,m_i)\sim\mathcal{T}}\mathbb{E}_{S_i\sim\mathbb{D}_i^{m_i}}\mathbb{E}_{\theta\sim\mathscr{P}}
		\exp\big( 
		\tilde{\lambda}[ 
		\mathbb{E}_{(\mathbb{D},m)\sim\mathcal{T}}\mathbb{E}_{S\sim\mathbb{D}^m}\mathcal{L}(\mathbb{Q}(S,\mathbb{P}_{S_i}^{\theta}),\mathbb{D}\big)
		- \mathcal{L}(\mathbb{Q}(S_i,\mathbb{P}_{S_i}^{\theta}),\mathbb{D}_i)]) \\
		& \,\, \leq \exp ( \tilde{\lambda}^2s_{\text{II}}^2/2 ).
	\end{align*}
\end{assumption}

Under Assumptions \ref{DataDependentAssumption}, \ref{UnboundedLossAssumptionSubGaussian1}, and \ref{UnboundedLossAssumptionSubGaussian2}, we resolve the challenge of interchanging the expectations $\mathbb{E}$ (from Theorem \ref{PAC_Bounds_Theorem1}) and $\mathbb{E}_{u\sim\mathbb{P}_{S}^{\theta}}$ (with $S\sim\mathbb{D}_{\mathcal{T}}$), yielding a refined estimate.  

\begin{corollary}\label{PACDependentThmSubGaussian}
	Assume that all of the assumptions in Theorem \ref{PAC_Bounds_Theorem1}, 
	Assumptions \ref{DataDependentAssumption}, \ref{UnboundedLossAssumptionSubGaussian1}, and \ref{UnboundedLossAssumptionSubGaussian2}
	hold true. Let us denote $\bar{m} = \left(\frac{1}{n}\sum_{i=1}^n\frac{1}{m_i}\right)^{-1}$. 
	For some $\gamma\geq2\sqrt{n}$, $\lambda\geq2\sqrt{n}$, and any $\delta\in(0,1]$, we have 
	\begin{align*}
		\mathbb{P}\Bigg( \forall \mathscr{Q}\in\mathcal{P}(\Theta), \,\,
		\mathcal{L}(\mathscr{Q}, \mathcal{T}) \leq & \, \hat{\mathcal{L}}(\mathscr{Q},S_1,\ldots,S_n) + 
		\left( \frac{1}{\lambda} + \frac{1}{\gamma} \right)D_{\text{KL}}(\mathscr{Q} || \mathscr{P})  \\
		& \, + \frac{1}{\gamma}\sum_{i=1}^{n}\mathbb{E}_{\theta\sim\mathscr{Q}}D_{\text{KL}}
		\big(\mathbb{Q}(S_i,\mathbb{P}_{S_i}^{\theta}) || \mathbb{P}_{S_i}^{\theta}\big)  \\
		& \, + \frac{n}{2\gamma}\ln\Psi_E  
		+ \frac{\gamma s_{\text{I}}^2}{n\bar{m}} + \frac{\lambda s_{\text{II}}^2}{2n}
		+ \frac{1}{\sqrt{n}}\ln\frac{1}{\delta}\Bigg) \geq 1-\delta.
	\end{align*}
\end{corollary}

\begin{remark}
	In Bayesian frameworks, the loss function is often associated with noise distribution, with Gaussian and Laplace noise as standard assumptions \citep{Stuart2010AN} yielding unbounded loss functions. We focus on sub-Gaussian assumptions for unbounded losses, contrasting with the bounded loss assumption in deterministic PDE inverse problems linked to the generalized Bayes' formula and PAC-Bayesian theory.  Here, the generalized Bayes' formula refers to Bayes' formula with tempered likelihoods; see \cite{Guedj2019ArXiv,Alquier2024Book} for details. Bounded loss functions are discussed in Subsection A.6
	of the supplement. Another standard assumption is the sub-Gamma condition (see Supplement, Subsection A.3).
\end{remark}

\begin{remark}
	The parameter $\gamma$ is typically chosen as $\gamma = n\sqrt{\bar{m}}$ or $n\bar{m}$ (where $\bar{m}$ is defined in Corollary \ref{PACDependentThmSubGaussian}). Accordingly, the external term from the data-dependent prior is $\frac{n}{2\gamma}\ln\Psi_E = \frac{1}{2\sqrt{\bar{m}}}\ln\Psi_E$ or $\frac{1}{2\bar{m}}\ln\Psi_E$, which asymptotically vanishes as $\bar{m}\to\infty$. 
\end{remark}

\begin{remark}
	An additional term arises when the prior measure depends on dataset $S$. It appears that the bound in Corollary \ref{PACDependentThmSubGaussian} is larger than that of the data-independent prior (Corollary A.6 
	in the supplement). However, the KL divergence term decreases when the prior depends on $S$, creating a trade-off between the KL term and the extra term $\frac{n}{2\gamma}\ln\Psi_E$. With $\gamma = n\sqrt{\bar{m}}$, see the estimate (\ref{EstimatePsiE}), the extra term scales as $\frac{R_u^2}{\sqrt{\bar{m}}}$. Since data-independent priors often lie far from the posterior, we can expect that introducing data-dependent priors is meaningful.  
\end{remark}

\section{Learning Priors for Inverse Problems}\label{SectionApplication}
This section presents our general framework for linear and nonlinear inverse problems, designed to unify the general theory with typical PDE problems. Here, we set $\gamma = n\beta$ throughout. For details on $\beta$, see Subsections A.2 and A.3 in the supplement. 

\subsection{General assumptions and formulas}\label{SubsectionUnbounded}
{\color{black}We outline standard prior and noise measure assumptions in nonparametric Bayesian inversion \citep{Stuart2010AN,Cotter2013SS}. The recent Gaussian process work \citep{Mora2025CMAME} may also apply; see Subsection A.5 
	of the supplementary materials.}
\begin{assumption}\label{ApplicationAssump1}
	Let us give some basic assumptions used in the rest of this section:
	\begin{enumerate}
		\item {\color{black}Assume $\mathcal{U}$ is an infinite-dimensional separable Hilbert space. The prior measure $\mathbb{P}_{S}^{\theta} := \mathcal{N}(f_m(S;\theta_1),\mathcal{C}_0(\theta_2))$ defined on $\mathcal{U}$ as in Subsection \ref{SectionDataDependentPrior}, with the mean function $f_m(S;\theta_1)$ lying in the Cameron-Martin space of $\mathcal{C}_0$ (see Subsection \ref{SectionDataDependentPrior}) for all $S\in\mathcal{Z}^m$ and $\theta \in\Theta$. Without loss of generality, let $\{\lambda_k, e_k\}_{k=1}^{\infty}$ be the eigensystem of $\mathcal{C}_0$, where $\{\lambda_k\}_{k=1}^{\infty}$ are non-increasing and $\{e_k\}_{k=1}^{\infty}$ are normalized eigenfunctions. Assume there exists $s_0 \in [0,1)$ such that $\mathcal{C}_0^s$ is trace class for all $s > s_0$.  
		}
		\item Assume that the space $\mathcal{H}$ introduced in Subsection \ref{SubsectionGeneralForm} is either $\mathbb{R}^{N_d}$, where $N_d$ is a positive integer, or $\mathcal{H}$ is an infinite-dimensional separable Hilbert space.
		\item Assume that the noise $\eta$ is distributed according to a Gaussian distribution $\mathbb{D}_{3}$ defined as $\mathbb{D}_{3} := \mathcal{N}(0, \Gamma)$, {\color{black}where $\Gamma:\mathcal{H}\rightarrow\mathcal{H}$ is a self-adjoint, positive-definite linear operator, but not necessarily trace class. Similar settings can be found in \citep{Agapiou2013SPA}.}
		\item If $\mathcal{H}$ is an infinite-dimensional Hilbert space, let us introduce $\mathcal{C}_1: \mathcal{H}\rightarrow\mathcal{H}$, which is a self-adjoint, positive-definite, trace class, linear operator. 
		Suppose $\mathcal{C}_1\Gamma = \Gamma\mathcal{C}_1$ and there exists $s_1\in[0,1)$ such that $\mathcal{C}_1^{s}$ 
		is trace class for all $s>s_1$.
	\end{enumerate}
\end{assumption}

Consider $\mathcal{H}$, a finite or infinite-dimensional space. For the infinite-dimensional case, we introduce the Hilbert scale \citep{Engl1996Book}, central to our analysis. Since $\mathcal{C}_1$ (Assumption \ref{ApplicationAssump1}) is injective and self-adjoint,  
$\mathcal{H} = \overline{\mathcal{R}(\mathcal{C}_1)}\oplus\mathcal{R}(\mathcal{C}_1)^{\perp} = \overline{\mathcal{R}(\mathcal{C}_1)},$  
making $\mathcal{C}_1^{-1}:\mathcal{R}(\mathcal{C}_1)\to\mathcal{H}$ a densely defined, unbounded, symmetric, positive-definite linear operator. It extends to a self-adjoint operator with domain $\mathcal{D}(\mathcal{C}_1^{-1}) := \{u\in\mathcal{H} \mid \mathcal{C}_1^{-1}u\in\mathcal{H}\}$. Thus, we define the Hilbert scale $(\mathcal{H}^t)_{t\in\mathbb{R}}$, where  
$\mathcal{H}^t := \overline{\mathcal{M}_1}^{\|\cdot\|_{\mathcal{H}^t}},\mathcal{M}_1 := \bigcap_{\ell=0}^{\infty}\mathcal{D}(\mathcal{C}_1^{-\ell}),$  
with inner product and norm  
$\langle u, v\rangle_{\mathcal{H}^t} = \left\langle \mathcal{C}_1^{-t/2}u, \mathcal{C}_1^{-t/2}v\right\rangle_{\mathcal{H}}, \|u\|_{\mathcal{H}^t} = \left\| \mathcal{C}_1^{-t/2}u \right\|_{\mathcal{H}}.$ Assumption \ref{ApplicationAssump1} states $\mathcal{C}_1\Gamma = \Gamma\mathcal{C}_1$ (infinite-dimensional case), a harmless assumption since $\mathcal{C}_1$ only defines the function space.  
Analogously, define the Hilbert scale $(\mathcal{U}^t)_{t\in\mathbb{R}}$ for $\mathcal{U}$ and prior covariance $\mathcal{C}_0$:  
$\mathcal{U}^t := \overline{\mathcal{M}_2}^{\|\cdot\|_{\mathcal{U}^t}}, \mathcal{M}_2 := \bigcap_{\ell=0}^{\infty}\mathcal{D}(\mathcal{C}_0^{-\ell}),$ with  
$\langle u, v\rangle_{\mathcal{U}^t} = \left\langle \mathcal{C}_0^{-t/2}u, \mathcal{C}_0^{-t/2}v\right\rangle_{\mathcal{U}}, \|u\|_{\mathcal{U}^t} = \left\| \mathcal{C}_0^{-t/2}u \right\|_{\mathcal{U}}.$
The Hilbert scale motivates the following assumptions.  

\begin{assumption}\label{ApplicationAssump2}
	Suppose constants $s\in(s_1, 1]$, $\tilde{s}\in(s_0,1]$, and $\alpha \geq 0$ such that 
	\begin{enumerate}
		\item If the space $\mathcal{H}$ is a separable Hilbert space, for all 
		$w\in\mathcal{H}^{\alpha-s}$, $s\in(s_1, 1]$ and $\rho\in[ \lceil \alpha-s_1-1 \rceil, \alpha-s_1 )$, 
		we assume $\left\|\mathcal{C}_1^{\frac{s}{2}}\Gamma^{-\frac{1}{2}}w\right\|_{\mathcal{H}} \leq C_1 
		\left\| \mathcal{C}_1^{\frac{s-\alpha}{2}}w \right\|_{\mathcal{H}}\text{ and }\left\| \mathcal{C}_1^{-\frac{\rho}{2}}\Gamma^{\frac{1}{2}}w \right\|_{\mathcal{H}} \leq 
		C_2 \left\| \mathcal{C}_1^{\frac{\alpha-\rho}{2}}w \right\|_{\mathcal{H}}.$
		\item If the space $\mathcal{H}$ is a separable Hilbert space, {\color{black}for any $x\!\in\!\mathcal{X}$ and $u,\! u_1,\! u_2\!\in\!\mathcal{U}^{1-\tilde{s}}$,} we have  
		\begin{align*}
			& \quad\quad\quad\quad\quad 
			\|\mathcal{C}_1^{-\frac{s}{2}}\Gamma^{-\frac{1}{2}}\mathcal{L}_{x}\mathcal{G}(u)\|_{\mathcal{H}} \leq  M_1(\|u\|_{\mathcal{U}^{1-\tilde{s}}}), \\
			& 	\|\mathcal{C}_1^{-\frac{s}{2}}\Gamma^{-\frac{1}{2}}(\mathcal{L}_{x}\mathcal{G}(u_1) -   \mathcal{L}_{x}\mathcal{G}(u_2))\|_{\mathcal{H}} 
			\leq M_2(\|u_1\|_{\mathcal{U}^{1-\tilde{s}}},\|u_2\|_{\mathcal{U}^{1-\tilde{s}}}) 
			\|u_1 - u_2\|_{\mathcal{U}^{1-\tilde{s}}}, \\
			& \quad\quad\quad
			\|\mathcal{C}_1^{-\frac{s}{2}}\Gamma^{-\frac{1}{2}}(\mathcal{L}_{x}\mathcal{G}(u) -   \mathcal{L}_{x'}\mathcal{G}(u))\|_{\mathcal{H}} \rightarrow 0, \quad\text{as }x\rightarrow x' \text{ in }\mathcal{X},
		\end{align*}
		where $M_1(\cdot)$ is a monotonic non-decreasing function and $M_2(\cdot,\cdot)$ is a function monotonic non-decreasing separately in each argument. If $\mathcal{H}$ is a finite-dimensional space, say $\mathcal{H} = \mathbb{R}^{N_d}$, we can omit the operator $\mathcal{C}_1$ from the aforementioned requirements.
	\end{enumerate}
\end{assumption}

The two inequalities given in (1) of Assumption \ref{ApplicationAssump2} can be intuitively interpreted as $\Gamma \simeq \mathcal{C}_1^{\alpha}$, where the notation $\simeq$ indicates that the two operators are equal in some sense of norm equivalence. Employing the second inequality in (1) of Assumption \ref{ApplicationAssump2}, we obtain $\mathbb{E}_{\eta\sim\mathbb{D}_3}\|\eta\|_{\mathcal{H}^{\alpha - s}}^2 \leq  C_2^2\mathbb{E}_{\eta\sim\mathbb{D}_3}\|\mathcal{C}_1^{\frac{s}{2}}\Gamma^{-\frac{1}{2}}\eta\|_{\mathcal{H}}^2
< +\infty$ by utilizing Lemma 3.3 of \cite{Agapiou2013SPA}. 
This implies that the noise $\eta$ drawn from $\mathbb{D}_3$ almost surely belongs to $\mathcal{Y}:=\mathcal{H}^{\alpha-s}$. 

When $\mathcal{H}$ is infinite-dimensional, we define the loss function as  
$\ell(u, z) := \Phi(u;x,y) = \frac{1}{2}\|\Gamma^{-\frac{1}{2}}\mathcal{L}_{x}\mathcal{G}(u)\|_{\mathcal{H}}^2 - 
\langle \Gamma^{-\frac{1}{2}}y, \Gamma^{-\frac{1}{2}}\mathcal{L}_{x}\mathcal{G}(u)\rangle_{\mathcal{H}}.$
Under the assumed Gaussian noise model, this function acts as the potential, or equivalently, the negative log-likelihood; see, e.g., \citep{Dashti2017,Stuart2010AN}. To establish generalization estimates, it is essential to verify Assumptions \ref{UnboundedLossAssumptionSubGaussian1} and \ref{UnboundedLossAssumptionSubGaussian2} in the current setting. This requires a rigorous formulation the generalized posterior measure $\mathbb{Q}(S,\mathbb{P}_{S}^{\theta})$ associated with the base inverse problem.

\begin{theorem}\label{BayesianTheorem}
	{\color{black}Suppose that Assumptions \ref{HyperpriorAssumption}, \ref{ApplicationAssump1}, and \ref{ApplicationAssump2} are satisfied.} For the function $M_1(\cdot)$ in Assumption \ref{ApplicationAssump2}, assume there exists a constant $\delta_0 > 0$ such that for all $\delta \in [0, \delta_0)$, the quantity  $\mathbb{E}_{u \sim \mathbb{P}_S^\theta} \exp\left( \delta M_1\left(\|u\|_{\mathcal{U}^{1 - \tilde{s}}}\right)^2 \right)$
	is finite. Let $\mathcal{H}$ be a separable Hilbert space. {\color{black}For any fixed $\theta \in \Theta$, suppose that $\|f_m(S;\theta_1)\|_{\mathcal{U}^{1+\tilde{s}}} \leq M(r)$ whenever $\|y_j\|_{\mathcal{H}^{\alpha-s}} \leq r$ for all $j = 1,\ldots,m$, where $M(\cdot)$ is a non-decreasing function. Additionally, assume that $\|f_m(S;\theta_1) - f(S';\theta_1)\|_{\mathcal{U}^{1+\tilde{s}}} \to 0, \text{ as } x_j \to x'_j \text{ in } \mathcal{X} \text{ and } y_j \to y'_j \text{ in } \mathcal{H}^{\alpha - s},$ for $j = 1,\ldots,m$, where $S' = \{(x'_j,y'_j)\}_{j=1}^m$.} Under these conditions, the generalized posterior measure $\mathbb{Q}(S,\mathbb{P}_S^{\theta})$ is absolutely continuous with respect to $\mathbb{P}_S^{\theta}$, and
	\begin{align}\label{ProofBayesFormula0}
		\frac{d\mathbb{Q}(S,\mathbb{P}_S^{\theta})}{d\mathbb{P}_{S}^{\theta}}(u) = \frac{1}{Z_m}\exp\left( 
		-\frac{\beta}{m}\sum_{j=1}^{m}\Phi(u;x_j,y_j)
		\right),
	\end{align}
	where $\beta > 0$ is a constant,
	$\Phi(u;x_j,y_j) := \frac{1}{2}\left\| \Gamma^{-\frac{1}{2}}\mathcal{L}_{x_j}\mathcal{G}(u) \right\|_{\mathcal{H}}^2 
	- \left\langle \Gamma^{-\frac{1}{2}}y_j, \Gamma^{-\frac{1}{2}}\mathcal{L}_{x_j}\mathcal{G}(u) \right\rangle_{\mathcal{H}},$ and $Z_m \in (0,\infty)$ is the normalizing constant. Moreover, the mapping $S \mapsto \mathbb{Q}(S,\mathbb{P}_S^{\theta})$ is continuous with respect to the Hellinger metric: $d_{\text{Hell}}\big(\mathbb{Q}(S,\mathbb{P}_S^{\theta}), \mathbb{Q}(S',\mathbb{P}_{S'}^{\theta})\big) \to 0, \text{ as } x_j \to x'_j \text{ in } \mathcal{X} \text{ and } y_j \to y'_j \text{ in } \mathcal{H}^{\alpha - s}.$ If $\mathcal{H} = \mathbb{R}^{N_d}$ is finite-dimensional, then all of the above results remain valid upon replacing $\|\cdot\|_{\mathcal{H}^{\alpha - s}}$ with $\|\cdot\|_{\mathcal{H}}$, which corresponds to the standard Euclidean norm on $\mathbb{R}^{N_d}$.
\end{theorem}


\begin{remark}\label{FiniteDimHRemark1}
	If the space $\mathcal{H}$ is finite-dimensional, we define the potential function as
	$\Phi(u;x_j,y_j) := \frac{1}{2}\left\| \Gamma^{-\frac{1}{2}}(\mathcal{L}_{x_j}\mathcal{G}(u) - y_j) \right\|_{\mathcal{H}}^2,$
	which is non-negative. Therefore, the condition
	$\mathbb{E}_{u\sim\mathbb{P}_{S}^{\theta}}\exp\left( \delta M_1\left(\|u\|_{\mathcal{U}^{1-\tilde{s}}}\right)^2 \right) < \infty$ is not required to ensure the lower bound in the proof.
\end{remark}

\begin{remark}\label{TruncaedGaussianRemark1}
	For nonlinear inverse problems, assume the true parameter $u^{\dagger}$ lies in the ball  
	$B_{\tilde{s}}(R_u) := \left\{ u\in\mathcal{U} \mid \|u\|_{\mathcal{U}^{1-\tilde{s}}} \leq R_u \right\}$  
	(see Subsection \ref{SubsectionLearnTheoryLinear}). Similar constraints are imposed in deterministic and statistical theories \citep{Engl1996Book,Nickl2022LectureNotes,Nickl2020JEMS}.  
	Define $\mathbb{P}_{R_u}^{S,\theta}$ as the truncated Gaussian measure:  
	$\mathbb{P}_{R_u}^{S,\theta}(du) := \frac{1}{\mathbb{P}_S^{\theta}(B_{\tilde{s}}(R_u))}
	1_{B_{\tilde{s}}(R_u)}(u)\mathbb{P}_{S}^{\theta}(du).$  
	When $u^{\dagger} \in B_{\tilde{s}}(R_u)$ is known a priori, using this truncated prior is more reasonable \citep{Beskos2018JUQ,Rudolf2018FCM}; see Subsections A.4 and A.6 
	in the supplement for additional discussions. Based on an analogous illustration of Theorem \ref{BayesianTheorem}, we find that 
	\begin{align}\label{TruncatedBayesFormula}
		\frac{d\mathbb{Q}(S_, \mathbb{P}_{R_u}^{S,\theta})}{d\mathbb{P}_{R_u}^{S,\theta}}(u) = 
		\frac{\mathbb{P}_{S}^{\theta}(B_{\tilde{s}}(R_u))}{Z_m^{R_u}}\exp\left( 
		-\frac{\beta}{m}\sum_{j=1}^{m}\Phi(u;x_j,y_j)
		\right)1_{B_{\tilde{s}}(R_u)}(u)
	\end{align} 
	is well-defined, where $Z_m^{R_u}$ is the normalization constant.  
\end{remark}

\subsection{Linear and nonlinear problems}\label{SubsectionLearnTheoryLinear}
In this subsection, assume the forward operator $\mathcal{G}$ is linear. Recall notations from Subsection \ref{SubsectionGeneralForm}: denote $\bm{y} = (y_1,\ldots,y_m)^T$ and $\bm{\eta} = (\eta_1,\dots,\eta_m)^T$. The forward problem with $m$ measurements compacts to  
$\bm{y} = \mathcal{L}_{\bm{x}}\mathcal{G}u + \bm{\eta},$   
where $\mathcal{L}_{\bm{x}}\mathcal{G} = (\mathcal{L}_{x_1}\mathcal{G}, \ldots, \mathcal{L}_{x_m}\mathcal{G})^{T}$. The noise vector $\bm{\eta} \sim \mathcal{N}(0,\Gamma_m)$ with $\Gamma_m = \text{diag}(\Gamma,\ldots,\Gamma)$. The generalized posterior measure in Theorem \ref{BayesianTheorem} thus takes the form:
\begin{align}\label{basePosteriorMeasure}
	\frac{d\mathbb{Q}(S,\mathbb{P}_{S}^{\theta})}{d\mathbb{P}_{S}^{\theta}}(u) \propto 
	\exp\left( -\frac{\beta}{2m}\|\mathcal{L}_{\bm{x}}\mathcal{G}u\|_{\Gamma_m}^2 + 
	\frac{\beta}{m}\langle \mathcal{L}_{\bm{x}}\mathcal{G}u, \bm{y} \rangle_{\Gamma_m} \right),
\end{align}
where 
$\|\mathcal{L}_{\bm{x}}\mathcal{G}u\|_{\Gamma_m}^2 = 
\sum_{j=1}^{m}\|\Gamma^{-\frac{1}{2}}\mathcal{L}_{x_j}\mathcal{G}u\|_{\mathcal{H}}^2$ and 
$\langle \mathcal{L}_{\bm{x}}\mathcal{G}u, \bm{y} \rangle_{\Gamma_m} =  
\sum_{j=1}^m \langle \Gamma^{-\frac{1}{2}}\mathcal{L}_{x_j}\mathcal{G}u, \Gamma^{-\frac{1}{2}} y_j \rangle_{\mathcal{H}}.$
Since the forward operator is assumed to be linear, the generalized posterior measure $\mathbb{Q}(S, \mathbb{P}_{S}^{\theta})$ is actually a Gaussian measure $\mathcal{N}(u_p, \mathcal{C}_p)$, Example 6.23 in \cite{Stuart2010AN} {\color{black}with 
	\begin{align}
		u_p = & \,f_m(S;\theta_1) + \mathcal{C}_0(\theta_2) (\mathcal{L}_{\bm{x}}\mathcal{G})^*\Big(\text{\small$\frac{m}{\beta}$}\Gamma_m + \mathcal{L}_{\bm{x}}\mathcal{G}\mathcal{C}_0(\theta_2)(\mathcal{L}_{\bm{x}}\mathcal{G})^*\Big)^{-1}(\bm{y} - \mathcal{L}_{\bm{x}}\mathcal{G}f_m(S;\theta_1)), 
		\label{PosteriorMeanVII1} \\
		\mathcal{C}_p = & \,\mathcal{C}_0(\theta_2) - \mathcal{C}_0(\theta_2)(\mathcal{L}_{\bm{x}}\mathcal{G})^*\Big(\text{\small$\frac{m}{\beta}$}\Gamma_m + \mathcal{L}_{\bm{x}}\mathcal{G}\mathcal{C}_0(\theta_2)(\mathcal{L}_{\bm{x}}\mathcal{G})^*\Big)^{-1}\mathcal{L}_{\bm{x}}\mathcal{G}\mathcal{C}_0(\theta_2).  \label{PosteriorVarVII1}
	\end{align}
	
}In Lemma \ref{LearningLemmaUnboundedExample2}, we show that Assumptions \ref{UnboundedLossAssumptionSubGaussian1} and \ref{UnboundedLossAssumptionSubGaussian2} hold under mild PDE-related conditions. This validates applying the general theory in Section \ref{SectionLearningPriorMeasure} to the linear case.

\begin{lemma}\label{LearningLemmaUnboundedExample2}
	{\color{black}Assume Assumptions \ref{HyperpriorAssumption}, \ref{ApplicationAssump1}, and \ref{ApplicationAssump2} hold.} For the linear forward operator, reformulate the conditions as  $\|\mathcal{C}_1^{-s/2}\Gamma^{-1/2}\mathcal{L}_{x}\mathcal{G}u\|_{\mathcal{H}} \leq M_1 \|u\|_{\mathcal{U}^{1-\tilde{s}}}$ and $\|\mathcal{C}_1^{-s/2}\Gamma^{-1/2}\mathcal{L}_{x}\mathcal{G}(u_1 - u_2)\|_{\mathcal{H}} \leq M_2 \|u_1 - u_2\|_{\mathcal{U}^{1-\tilde{s}}},$ with positive constants $M_1, M_2$. Let the probability measure $\mathscr{E}$ have compact support $\supp\mathscr{E} \subset B_{\tilde{s}}(R_u) := \{ u \in \mathcal{U} : \|u\|_{\mathcal{U}^{1-\tilde{s}}} \leq R_u \}$ for some $R_u \in \mathbb{R}^+$. {\color{black}Define $\tilde{C} := \|\mathcal{C}_1^{s/2}\|_{\mathcal{B}(\mathcal{H})}$ (the operator norm) and assume  
		$\max\left\{3\tilde{\gamma}\tilde{C}^2M_2^2, 3\tilde{\gamma}^2M_1^2\text{Tr}(\mathcal{C}_1^s)\right\} \leq \min\left\{e^{-\theta_{\text{max}}}, \lambda_1^{-1}\right\},$ where $\theta_{\text{max}}$ and $\lambda_1$ are as in Assumption \ref{ApplicationAssump1}. The base posterior measure follows from (\ref{basePosteriorMeasure}).  Consequently, $s_{\text{I}}^2$ and $s_{\text{II}}^2$ in Assumptions \ref{UnboundedLossAssumptionSubGaussian1} and \ref{UnboundedLossAssumptionSubGaussian2} are finite positive reals for $\tilde{\gamma} := \frac{\gamma}{n\bar{m}}$ and $\tilde{\lambda} := \frac{\lambda}{n}$, with $\gamma$ and $\lambda$ are proportional to $n\bar{m}$ and $n$ respectively. Under some boundedness assumptions (include the backward diffusion problem), $s_{\text{I}}^2$ and $s_{\text{II}}^2$ become independent of $\tilde{\gamma}$ and $\tilde{\lambda}$, allowing $\gamma := n\sqrt{\bar{m}}$ and $\lambda := 2\sqrt{n}$.}
\end{lemma}

For brevity, we omit the explicit formulas for $s_{\text{I}}$ and $s_{\text{II}}$ and the auxiliary boundedness assumptions, which are detailed in Subsection B.1 
of the online supplement. We now focus on nonlinear forward operators. Unlike linear cases, the generalized posterior measure lacks a tractable explicit form, rendering the exact expectation intractable—a key challenge in nonlinear settings. Leveraging Remark \ref{TruncaedGaussianRemark1}, we employ a truncated Gaussian prior, yielding a well-posed generalized posterior. Our proofs require the prior to be supported on a ball in some function space. Conceptually, the priors used in \cite{Nickl2020JEMS} are also applicable here. Specifically, we state a lemma that estimates $V_i^1$ and $V_i^2$ under Assumptions \ref{UnboundedLossAssumptionSubGaussian1} and \ref{UnboundedLossAssumptionSubGaussian2} for nonlinear problems, validating the general theory's application.  

\begin{lemma}\label{LearningLemmaUnboundedExample3}
	Suppose that the probability measure $\mathscr{E}$ and the loss
	function $\ell$ satisfy the conditions stated in Lemma \ref{LearningLemmaUnboundedExample2}.
	We further assume that {\color{black}Assumptions \ref{HyperpriorAssumption}, \ref{ApplicationAssump1}, and \ref{ApplicationAssump2} are valid,} with the prior measure $\mathbb{P}_{S}^{\theta}=\mathcal{N}(f_m(S;\theta_1),\mathcal{C}_0(\theta_2))$ being replaced by the truncated Gaussian assumption.
	Specifically, $\mathbb{P}_{S}^{\theta}$ is substituted with	$\mathbb{P}_{R_u}^{S,\theta}$ as defined in Remark \ref{TruncaedGaussianRemark1}. Additionally, we assume that the base posterior measure is obtained through the generalized Bayes' formula with truncated Gaussian prior as defined in (\ref{TruncatedBayesFormula}). {\color{black}Consequently, for $s_{\text{I}}^2$ and $s_{\text{II}}^2$ as defined in Assumptions \ref{UnboundedLossAssumptionSubGaussian1} and \ref{UnboundedLossAssumptionSubGaussian2}, we have 
		$s_{\text{I}}^2 = 4\tilde{C}^4M_2(R_u,R_u)^4R_u^4 + \frac{1}{4}\tilde{C}^4M_1(R_u)^4R_u^4 + M_1(R_u)^2\text{Tr}(\mathcal{C}_1^s)R_u^2 < +\infty$, and $s_{\text{II}}^2 = 25\|\mathcal{C}_1^{\frac{s}{2}}\|_{\mathcal{B}(\mathcal{H})}^4M_1(R_u)^4 < +\infty$, 
		where $\tilde{C}=\|\mathcal{C}_1^{\frac{s}{2}}\|_{\mathcal{B}(\mathcal{H})}$. }
\end{lemma}

For concrete linear and nonlinear inverse problems, the conditions in Assumption \ref{ApplicationAssump2} must be validated via PDE estimations in suitable function spaces. This allows explicit identification of constants $s_{\text{I}}$ and $s_{\text{II}}$, enabling derivation of an explicit PAC-Bayesian estimate by integrating results from Section~\ref{SectionLearningPriorMeasure}. Examples for backward diffusion (linear) and Darcy flow (nonlinear) problems are detailed in Subsection A.7 
of the online supplement.  

\section{Learning Algorithms}\label{SectionLearningAlgorithms}
In this section, we develop practical learning algorithms based on the derived PAC-Bayesian generalization bound. Following the finite-dimensional approach in \citep{Rothfuss2023JMLR}, we set $\gamma = n\beta$, where $\beta$ is the parameter in the generalized Bayes' formula of Theorem \ref{BayesianTheorem}. Next, we denote $\{Z_{m_i} = Z_{m_i}(S_i,\mathbb{P}_{S_i}^{\theta})\}_{i=1}^n$, noting that $Z_{m_i}$ depends on the dataset and prior measure. With these choices, we further refine the PAC-Bayesian estimate to obtain  
\begin{align}\label{AlgBound}
	\begin{split}
		\mathcal{L}(\mathscr{Q},\mathcal{T}) \leq & - \frac{1}{n\beta}\sum_{i=1}^{n}\mathbb{E}_{\theta\sim\mathscr{Q}}\ln 
		Z_{m_i}(S_i, \mathbb{P}_{S_i}^{\theta}) + \frac{\lambda + n\beta}{\lambda n\beta}D_{\text{KL}}(\mathscr{Q}||\mathscr{P}) \\
		& + \frac{1}{2\beta}\ln\Psi_{E} + \frac{\beta s_{\text{I}}^2}{\bar{m}} + \frac{\lambda s_{\text{II}}^2}{2n} + \frac{1}{\sqrt{n}}\ln\frac{1}{\delta}.
	\end{split}
\end{align}
To minimize the right-hand side of (\ref{AlgBound}), we address the following optimization problem:
\begin{align}\label{AlgOptimalProblem}
	\argmin_{\mathscr{Q}\in\mathcal{P}(\Theta)}\mathbb{E}_{\theta\sim\mathscr{Q}}\bigg[ 
	-\frac{\lambda}{\lambda + n\beta}\sum_{i=1}^{n}\ln Z_{m_i}(S_i, \mathbb{P}_{S_i}^{\theta}) 
	\bigg] + D_{\text{KL}}(\mathscr{Q}||\mathscr{P}).
\end{align} 
In light of this optimization problem, we aim to establish a theorem that provides an analytical solution for the optimal measure $\mathscr{Q}$, thereby avoiding the complex bilevel optimization.
\begin{theorem}\label{AlgSolOptimal}
	For both bounded and unbounded loss cases, assume Assumptions \ref{HyperpriorAssumption}, \ref{ApplicationAssump1}, \ref{ApplicationAssump2}, and the conditions in Theorem \ref{BayesianTheorem} hold. Further posit that for any dataset $S$ and $\theta_1 \in \Theta_1$, the mean function $f_m(S;\theta_1)$ satisfies $\|f_m(S;\theta_1)\|_{\mathcal{U}^{1+\tilde{s}}} \leq R_u$ for some positive constant $R_u$. 
	The optimization problem (\ref{AlgOptimalProblem}) then has a solution defined by  
	\begin{align}\label{AlgBayesHyperPosterior}
		\frac{d\mathscr{Q}}{d\mathscr{P}}(\theta) = \frac{1}{Z_p}\exp\left( \frac{\lambda}{\lambda + n\beta}\sum_{i=1}^{n}\ln Z_{m_i}(S_i,\mathbb{P}_{S_i}^{\theta}) \right),
	\end{align}  
	where $Z_p := \int_{\Theta} \exp\left( \frac{\lambda}{\lambda + n\beta}\sum_{i=1}^{n}\ln Z_{m_i}(S_i,\mathbb{P}_{S_i}^{\theta}) \right) \mathscr{P}(d\theta).$
\end{theorem}

Different from the parametric setting in \cite{Rothfuss2021PMLR}, the above Bayes' formula (\ref{AlgBayesHyperPosterior}) does not trivially hold true.
Equipped with the formula (\ref{AlgBayesHyperPosterior}), we can now extract information about the parameter $\theta$ from the hyper-posterior measure $\mathscr{Q}$.  
Note that the term inside the exponential of formula (\ref{AlgBayesHyperPosterior}) can be calculated approximately by:
\begin{align}\label{logsumexp3}
	\begin{split}
		\sum_{i=1}^{n}\ln Z_{m_i}(S_i,\mathbb{P}_{S_i}^{\theta}) & = \sum_{i=1}^{n}\ln\int_{\mathcal{U}}\exp\bigg(
		-\frac{\beta}{m_i}\sum_{j=1}^{m_i}\Phi(u;z_{ij})
		\bigg) \mathbb{P}_{S_i}^{\theta}(du) \\
		& \approx \sum_{i=1}^{n}\ln\left[\sum_{\ell=1}^{L}\exp\bigg(
		-\frac{\beta}{m_i}\sum_{j=1}^{m_i}\Phi(u_{\ell}^{i};z_{ij})
		\bigg)\right] - n\ln L,
	\end{split}
\end{align}
where $\{u_{\ell}^{i}\}_{\ell=1}^{L}$ are samples from the base prior $\mathbb{P}_{S_i}^{\theta}$ for $i=1,\ldots,n$. To evaluate the last term in (\ref{logsumexp3}), we must compute $nL$ forward PDEs, which is computationally intensive. Hence, we focus on the maximum a posteriori (MAP) estimate in this work.

{\color{black}Given the parameter $\theta$ in a general separable Hilbert space $\Theta$, we adopt the weak maximum a posteriori (MAP) estimator introduced in \cite{Agapiou2018IP}. Let the hyper-prior $\mathscr{P} = \mathscr{P}_1 \otimes \mathscr{P}_2$ (see Assumption \ref{HyperpriorAssumption}), where $\mathscr{P}_1$ denotes a Gaussian measure with Cameron–Martin space $E_{\Theta_1}$ and $p_2(\cdot)$ represents the density of $\mathscr{P}_2$. Under suitable conditions (see Subsection A.8 
	of the supplementary materials for a rigorous exposition), computing the MAP estimate reduces to solving:}
\begin{align}\label{MAPoptim4}
	\argmin_{\theta} - \frac{\lambda}{\lambda + n\beta}\sum_{i=1}^{n}\ln\left[\sum_{\ell=1}^{L}\exp\bigg(
	-\frac{\beta}{m_i}\sum_{j=1}^{m_i}\Phi(u_{\ell}^{i};z_{ij})
	\bigg)\right] + \frac{1}{2}\|\theta_1\|_{E_{\Theta_1}}^2 - \ln p_2(\theta_2).
\end{align}
The above minimization problem becomes computationally intensive for large $n$. We address this by substituting the first term with a mini-batch approximation during iteration:  
\begin{align}\label{MAPoptim5}
	- \frac{\lambda n}{(\lambda + n\beta)H}\sum_{h=1}^{H}\ln\left[\sum_{\ell=1}^{L}\exp\bigg(
	-\frac{\beta}{m_i}\sum_{j=1}^{m_i}\Phi(u_{\ell}^{h};z_{hj})
	\bigg)\right],
\end{align}
where the datasets $\{S_h\}_{h=1}^{H}$ are randomly selected from $\{S_i\}_{i=1}^{n}$. 
Algorithm \ref{alg A} summarizes the resulting generic prior distribution learning procedure. 
For further details and a discussion on potential sampling methods, refer to Subsection A.8 
of the supplementary materials. 

\begin{algorithm}
	\caption{Prior Learning Algorithm}
	\label{alg A}
	\begin{algorithmic}[1]
		\State \textbf{Input}: Specify the hyper-prior $\mathscr{P}$, datasets $\{ S_1,\ldots,S_n \}$, maximum iterative number $N$, 
		the mini-batch size $H$, and the learning rate schedule $\alpha_k$. Initialize the parameters of the prior $\mathbb{P}_{\cdot}^{\theta} = \mathcal{N}(f(\cdot;\theta_1),\mathcal{C}_0(\theta_2))$ with $\theta = \theta^0$ and set the counter $k$ to $0$; 
		\State \textbf{Repeat}
		\State \qquad Set $k=k+1$;
		\State \qquad Select $\{ S_{h}\}_{h=1}^H$ randomly from $\{ S_i \}_{i=1}^{n}$ with $S_i = \{ z_{i1},\ldots,z_{im_i} \}$;
		\State \qquad {\color{black}Calculate the parameters $\theta^{k+1} = \theta^{k} - \alpha_k\nabla_{\theta^k}L(\theta^k)$ with
			\begin{align*}
				L(\theta^k) = - \frac{\lambda n}{(\lambda + n\beta)H}\sum_{h=1}^{H}\ln\left[\sum_{\ell=1}^{L}\exp\bigg(
				-\frac{\beta}{m_i}\sum_{j=1}^{m_i}\Phi(u_{\ell}^{h}(\theta^k);z_{hj})
				\bigg)\right] + \frac{1}{2}\|\theta_1^k\|_{E_{\Theta_1}}^2 - \ln p_2(\theta_2^k).
			\end{align*}
			\qquad\, Variants of SGD can also be employed in this step. }
		\State \textbf{Until}: $k=K_{\text{max}}$.
		\State \textbf{Output}: The learned prior measure $\mathbb{P}_{\cdot}^{\theta^{K_{\text{max}}}} = \mathcal{N}(f(\cdot;\theta^{K_{\text{max}}}_1), \mathcal{C}_0(\theta^{K_{\text{max}}}_2))$.
	\end{algorithmic}
\end{algorithm}

\section{Numerical Examples}\label{SectionNumerics}
{\color{black}
	We focus on the inverse problem of the Darcy flow equation, a more challenging nonlinear inverse problem than the backward diffusion problem. Numerical results for the backward diffusion problem are deferred to Subsection A.9 
	of the supplementary materials. 
	
	Let $\Omega=(0,1)^2$, and define the Hilbert space $\mathcal{U}$ as $(L^2(\Omega), \langle \cdot, \cdot \rangle, \| \cdot \|)$. The Darcy flow equation takes the following form:
	\begin{align}\label{DarcyEqApp2}
		\begin{split}
			-\nabla\cdot(e^u\nabla w) = f \quad\text{in }\Omega, \qquad
			w = 0 \quad\text{on }\partial\Omega,
		\end{split}
	\end{align}
	where $f$ represents the sources, and $e^{u(x)}$ describes the permeability of the porous medium. For the inverse problem, the dataset consists of discrete measured values of $w$, obtained via the measurement operator $\mathcal{L}_{x_j}(w) = w(x_j)$ where $x_j \in \Omega$ for $j=1,\dots,m$. The function $u(x)$ is the parameter that needs to be inferred. As noted in Remark \ref{RemarkMeaningParam}, we assume $m_i = m$ for $i = 1, \dots, n$ and set $\beta = m$, $\gamma = nm$, and $\lambda = n$. For this problem, we have $\tilde{\mathcal{H}} = H^{1}(\Omega)$, $\mathcal{H} = \mathbb{R}$, and assume Gaussian noise $\eta \sim \mathcal{N}(0,\Gamma)$ with $\Gamma = \tau^2I$ ($\tau > 0$). In all numerical examples, $\tau$ is set to $0.01\max_j\{|w(x_j)|\}$. The equation is discretized on $\Omega$ using the finite element method with continuous Lagrange basis functions $\{\phi_k\}_{k=1}^N$. We define the finite-dimensional space $V_N = \text{span}\{\phi_1,\ldots,\phi_N\}$, where approximate functions $f = \sum_{k=1}^Nf_k\phi_k \in V_N\subset L^2(\Omega)$. For the parameters $H$ and $L$ in Algorithm \ref{alg A}, we set $H=20$ and $L = 10$. More details can be found in Subsection A.9 
	of the supplementary materials.  
	
	Recall the operator $\mathcal{C}_0$ introduced in Subsection \ref{SectionDataDependentPrior}, we take $\mathcal{C}_0 = (I - 0.1\Delta)^{-2}$, where $\Delta$ is the Laplace operator satisfying homogeneous Neumann boundary condition. For the definitions of $\theta := (\theta_1, \theta_2)$ and $\mathcal{C}_0(\theta_2)$, see Subsection \ref{SectionDataDependentPrior}. The dimension of $\theta_2$ is $N_c = 45$ in our implementation. Here, we consider three types of base prior measures as follows: 
	(1) Unlearned prior: $\mathbb{P} := \mathcal{N}(0, \mathcal{C}_0)$;
	(2) Learned data-independent prior: $\mathbb{P}^{\theta} := \mathcal{N}(f_m(\theta_1), \mathcal{C}_0(\theta_2))$;
	(3) Learned data-dependent prior: $\mathbb{P}_{S}^{\theta} := \mathcal{N}(f_m(S; \theta_1), \mathcal{C}_0(\theta_2))$.
	
	In the above, the data-independent learned mean function $f_m(\theta_1):=\sum_{k=1}^{N}\theta_{1k}\varphi_k$ with $\{\phi_k\}_{k=1}^{N}$ as the basis of $V_N$. To keep the discretization-invariant property of the infinite-dimensional formulation, the data-dependent learned mean function $f_m(S;\theta_1)$ is taken to be a Fourier neural operator illustrated in \cite{Kovachki2021JMLR}. Since the dataset \( S \) consists of discrete vectors, the Fourier neural network cannot be directly applied. Actually, we transform the dataset \( S \) into a function via the adjoint measurement operator \( \mathcal{L}_{\bm{x}}^*: \mathcal{H} \rightarrow \tilde{\mathcal{H}} \). For more implementation specifics, see Subsection A.9 
	in the supplementary materials.
	
	When the base prior is $\mathbb{P}^{\theta}$, we take the hyper-prior $\mathscr{P} := \mathcal{N}(0, A_h^{-2}) \otimes \mathcal{N}(\ln\bm{\lambda}, 10I_{N_c})$, where $A_h = 0.01 I - 0.1\Delta$, $\ln\bm{\lambda} := (\ln\lambda_1, \ldots, \ln\lambda_{N_c})$, and $I_{N_c}$ is the identity operator on $\mathbb{R}^{N_c}$.
	For the mean function $f_m(S;\theta_1)$, specifying a suitable hyper-prior is challenging due to the lack of intuition regarding the Fourier neural operator's parameter $\theta_1$. Instead, we leverage the a priori regularity of $f_m(S;\theta_1)$, randomly selecting $n_S$ datasets ($n_S=4$ in practice), and define $g(\theta)$ as the sum of $f_m(S_i;\theta_1)$ over these datasets. Consequently, we establish $\mathscr{P}$ as $\mathscr{P} := g_{\#}\mathcal{N}(0, A^{-2}) \otimes \mathcal{N}(\ln\bm{\lambda}, 10 I_{N_c})$ when the base prior is $\mathbb{P}_{S}^{\theta}$.
	
	\subsection{Simple environment}\label{NumericalSimple}
	We evaluate our method in both simple and complex environments. In the simple environment, the measure $\mathscr{E}$ generates ground-truth parameters via the random function:  
	\begin{align}\label{DarcyFlowSimpleUMain}
		u(x_1,x_2) = u_1(x_1,x_2) + u_2(x_1,x_2) + u_3(x_1,x_2),
	\end{align}
	where $u_i(x_1,x_2) = a_{i3}(1-x_1^2)^{a_{i1}}(1-x_2^2)^{a_{i2}}e^{-a_{i4}(x_1-a_{i5})^2 - a_{i6}(x_2-a_{i7})^2}$ with $i=1,2,3$. For the parameters in the above formula, we let 
	$a_{i1}\sim U(0.1,0.5)$, $a_{i2}\sim U(0.1,0.5)$, $a_{i3}\sim U(3,4)$, $a_{i4}\sim U(30, 35)$, $a_{i6}\sim U(30, 35)$, $a_{15}\sim U(0.15,0.25)$, $a_{17}\sim U(0.65,0.75)$, $a_{25}\sim U(0.45,0.55)$, $a_{27}\sim U(0.45,0.55)$, $a_{35}\sim U(0.65,0.75)$, and $a_{37}\sim U(0.15,0.25)$, where $i=1,2,3$. We generate 2000 and 50 random functions via formula (\ref{DarcyFlowSimpleUMain}) to construct the historical datasets $\{S_i\}_{i=1}^{2000}$ and testing datasets, respectively. For constructing these datasets, we select 100 statistically equally-spaced sparse measurement points ($m=100$).  
	
	\begin{figure}[ht!]
		\centering
		\includegraphics[width=1.0\textwidth]{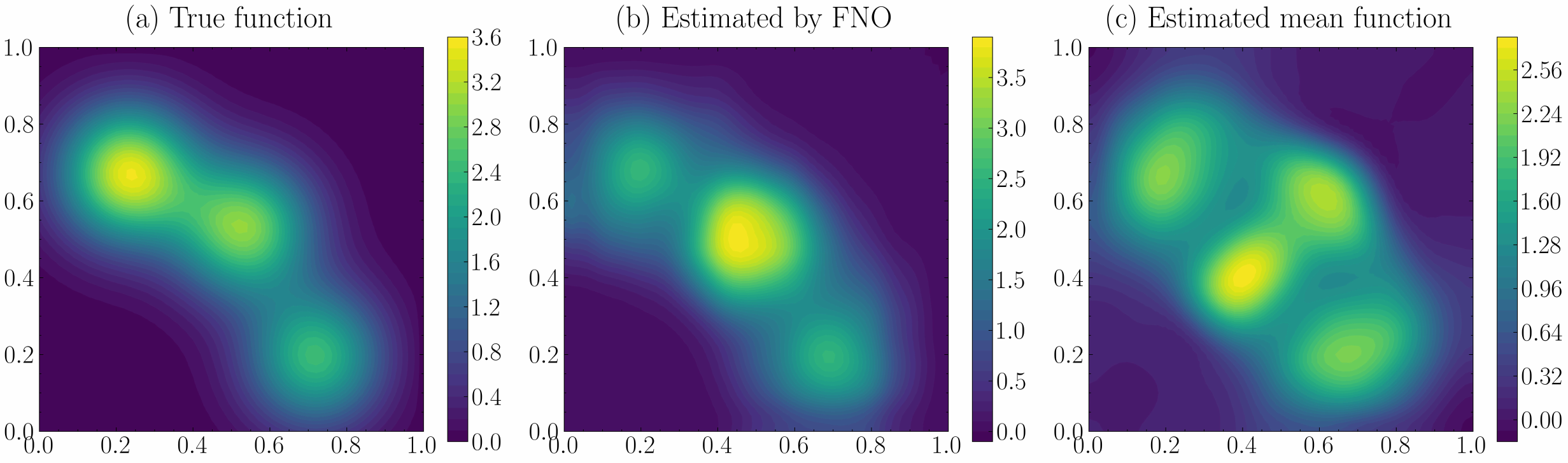}\\
		\caption{\emph{\small True function and estimated mean functions in the simple environment setting. 
				(a) One of the functions in the testing dataset;
				(b) The learned mean function $f_m(S;\theta_1)$ of $\mathbb{P}_{S}^{\theta}$; 
				(c) The learned mean function $f_m(\theta_1)$ of $\mathbb{P}^{\theta}$. }
		}\label{FigEx2_1}
	\end{figure}
	
	Before presenting quantitative comparisons, we provide a visual comparison in Fig. \ref{FigEx2_1}.
	In Fig. \ref{FigEx2_1} (a), we display one of the functions in the testing datasets, offering an intuitive understanding of the model parameters. In Fig. \ref{FigEx2_1} (b), we present the learned mean function $f_m(S;\theta_1)$ of $\mathbb{P}_{S}^{\theta}$.
	Lastly, in Fig. \ref{FigEx2_1} (c), we display the learned mean function $f_m(\theta_1)$ of $\mathbb{P}^{\theta}$. Notably, the learned mean function shown in (b) of Fig. \ref{FigEx2_1} closely resembles the true underlying function. In contrast, the mean function of the learned data-independent prior depicted in (c) of Fig. \ref{FigEx2_1} only represents the averaged information and visually deviates from this specific true underlying function. 
	
	The objective of learning prior measures is to enhance efficiency in solving novel inverse problems associated with historical inverse tasks. To show the proposed method's superiority, we evaluate base posteriors induced by three priors $\mathbb{P}$, $\mathbb{P}^{\theta}$, and $\mathbb{P}_{S}^{\theta}$ via maximum a posteriori (MAP) estimates, posterior means, and credible regions. The matrix-free Newton conjugate gradient algorithm \citep{Ghattas2021ActaNum} computes MAP estimates for testing datasets, while the mixture Gaussian sequential Monte Carlo (MGSMC) algorithm \citep{Lu2024arXiv} generates samples from the base posterior. Algorithmic details are provided in Subsection A.9 
	of the supplementary materials.  
	
	In Table \ref{TableDarcyFlowMain}, we report relative errors (in $L^2$-norm) of the computed MAP estimates against ground truth, indexed by the number of iterations. Regardless of the environment, the learned data-dependent Bayesian model yields the smallest relative error, requiring only 5 iterations. In the simple environment, the learned data-independent model converges faster than the unlearned Bayesian model but yields a similar relative error. 
	
	\begin{table}[ht!]
		\caption{For the 50 test datasets, we list the average relative errors of the MAP estimates from base posteriors under priors $\mathbb{P}$, $\mathbb{P}^{\theta}$, and $\mathbb{P}_{S}^{\theta}$, respectively. Here, ItN denotes the iteration count of the inexact matrix-free Newton-conjugate gradient method. }
		\begin{center}
			\begingroup
			\setlength{\tabcolsep}{1.0pt}
			\renewcommand{\arraystretch}{1.0}
			\begin{spacing}{1.0}
				\begin{tabular}{c|ccc|ccc}
					\Xhline{1.1pt}
					& \multicolumn{3}{c|}{Simple Environment} & \multicolumn{3}{c}{Complex Environment}  \\
					\hline
					$\quad$ItN $\quad$& $\qquad$$\mathbb{P}$$\qquad$ & $\qquad$$\mathbb{P}^{\theta}$$\qquad$ & $\qquad$$\mathbb{P}_S^{\theta}$$\qquad$ &  $\qquad$$\mathbb{P}$$\quad$ & $\qquad$$\mathbb{P}^{\theta}$$\qquad$ & $\qquad$$\mathbb{P}_S^{\theta}$$\qquad$ \\
					\hline
					$1$  & .7322 & .1932 & \textbf{.0973} & .7130 & .9120 & \textbf{.1078} \\
					$5$  & .2756 & .1674 & \textbf{.0911} & .3515 & .7056 & \textbf{.0895}	\\
					$10$ & .1614 & .1574 & \textbf{.0911} & .2798 & .6794 & \textbf{.0875} 	\\
					$15$ & .1460 & .1520 & \textbf{.0914} & .2714 & .6723 & \textbf{.0874}	\\                   
					$20$ & .1401 & .1457 & \textbf{.0914} & .2683 & .6685 & \textbf{.0869}	\\
					\Xhline{1.1pt}
				\end{tabular}
			\end{spacing}
			\endgroup
		\end{center}\label{TableDarcyFlowMain}
	\end{table}
	
	\begin{figure}[ht!]
		\centering
		\includegraphics[keepaspectratio=true, width=0.32\textwidth, clip=false]{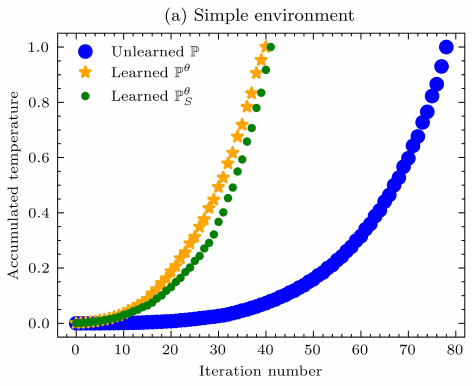}
		\hspace{0.0pt} 
		\includegraphics[keepaspectratio=true, width=0.32\textwidth, clip=false]{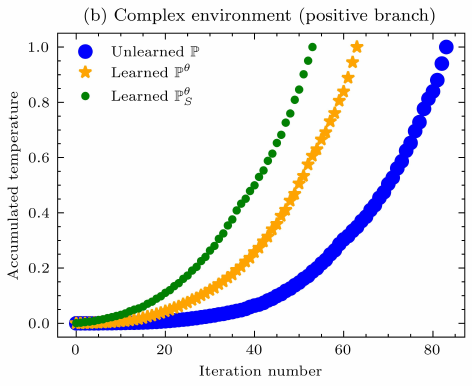} 
		\hspace{0.0pt} 
		\includegraphics[keepaspectratio=true, width=0.32\textwidth, clip=false]{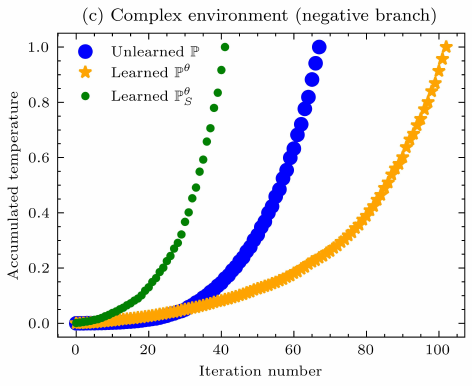}
		\vspace*{-1pt} 
		\caption{\emph{\small 
				Accumulated temperatures selected by effective sample size are shown in (a)-(c) for one test dataset from each case: the simple environment, the positive-valued branch, and the negative-valued branch of the complex environment. In all panels: blue dots represent unlearned Bayesian model temperatures, orange star dots denote learned data-independent Bayesian model temperatures, and green dots signify learned data-dependent Bayesian model temperatures.
		}}
		\label{FigSMCSpeed}
	\end{figure}
	
	In the Bayesian framework, we need to characterize the behavior of the base posterior. We use the MGSMC algorithm to generate 1000 samples from the posterior. Specifically, similar to other sequential Monte Carlo algorithms, the MGSMC algorithm decomposes the original problem into a series of simpler sampling problems. This decomposition relies on a sequence of positive numbers $\{h_k\}_{k=1}^{K_{\text{total}}}$ (termed temperatures) satisfying $\sum_{k=1}^{K_{\text{total}}}h_k = 1$. The temperatures $\{h_k\}_{k=1}^{K_{\text{total}}}$ are automatically determined based on the effective sample size (see supplementary materials for details). With other parameters fixed, $K_{\text{total}}$ governs the sampling speed of the MGSMC algorithm. The left panel of Fig. \ref{FigSMCSpeed} plots the accumulated temperatures $\sum_{k=1}^K h_k$ ($K=1,\dots,K_{\text{total}}$) under the simple environment setting. The learned data-dependent and learned data-independent models achieve a similar sampling speed, and the unlearned model is significantly slower than the other two methods. Thus, the Bayesian model with learned prior substantially accelerates posterior calculation. 
	
	\begin{figure}[ht!]
		\centering
		\includegraphics[width=1.0\textwidth]{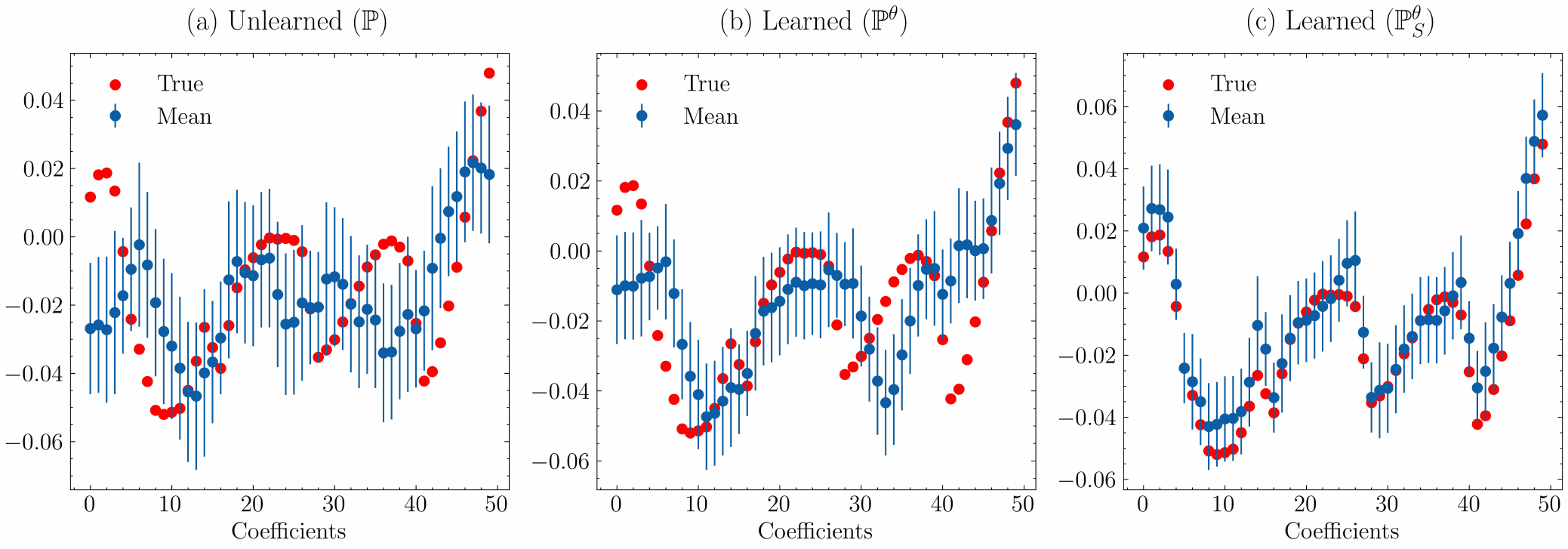}\\
		\vspace*{-2pt}
		\caption{\emph{\small In all panels, red and blue dots denote the true function's projected coefficients and the base posterior mean's projected coefficients, respectively. Blue vertical lines mark $95\%$ credible intervals. Panels (a), (b), and (c) display results of simple environment under the priors $\mathbb{P}$, $\mathbb{P}^{\theta}$, and $\mathbb{P}_S^{\theta}$, respectively.}
		}\label{FigSimpleStd}
	\end{figure}
	
	In addition to sampling speed, we compare uncertainty quantification results. Since parameter $u$ is a function, we project it onto the first 200 eigenfunctions $\{e_k\}_{k=1}^{200}$ of operator $\mathcal{C}_0$. We then compute the $95\%$ credible intervals for the projected variables $\{\langle u, e_k \rangle\}_{k=1}^{200}$. In Fig. \ref{FigSimpleStd}, red and blue dots denote the projected coefficients of the true function and the posterior mean of the base posterior, respectively. For clarity, we display only 50 projected points. Results from the unlearned, learned data-independent, and data-dependent models are shown in the left, middle, and right panels, respectively. Notably, all true values lie within the $95\%$ credible intervals obtained by the learned data-dependent model. By contrast, some true values fall outside the $95\%$ credible intervals of the other two models. Furthermore, the credible intervals of the unlearned and learned data-independent models are wider than those of the learned data-dependent model.
	
	\begin{figure}[ht!]
		\centering
		\includegraphics[width=1.0\textwidth]{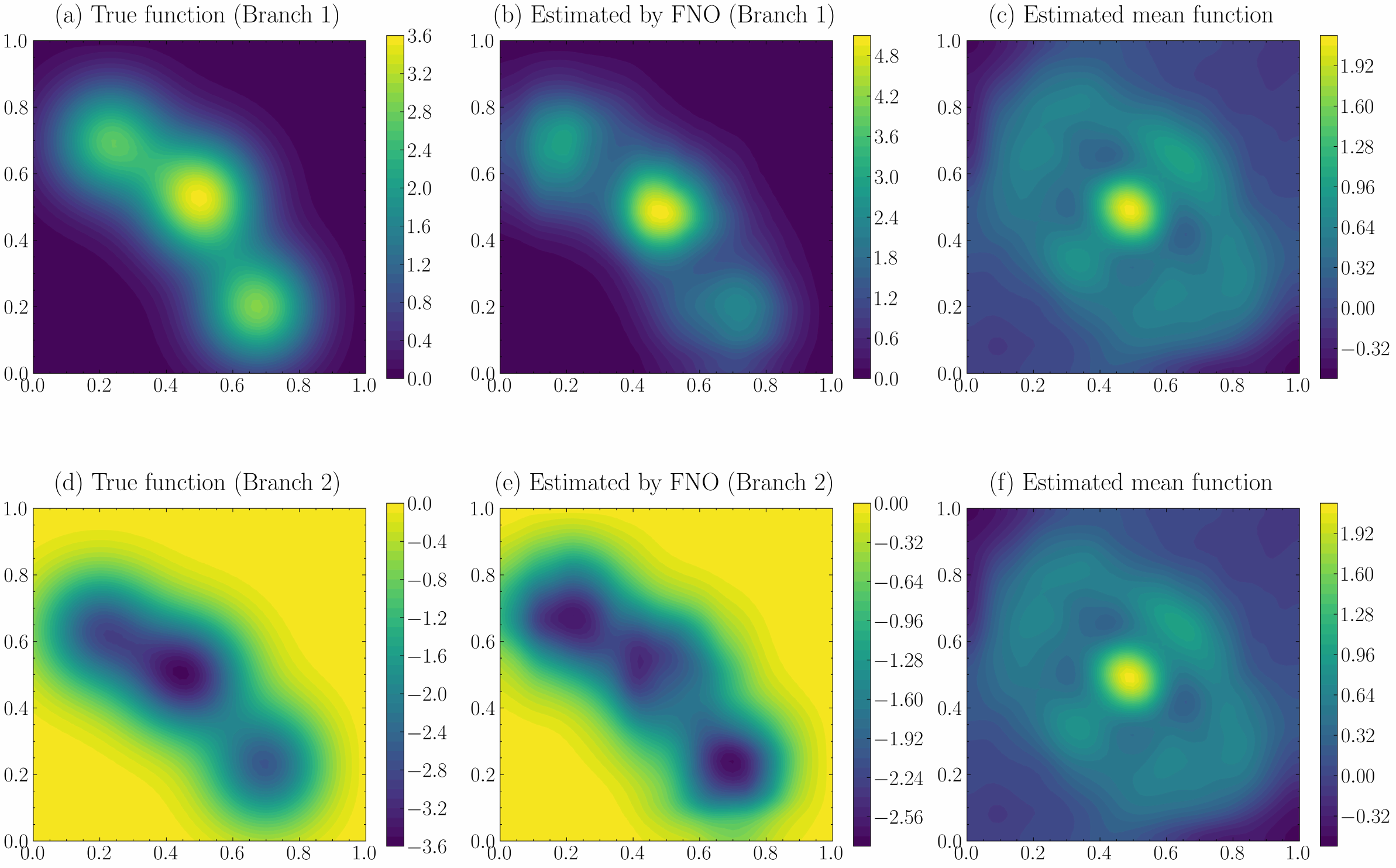}\\
		\caption{\emph{\small True and estimated functions for Branches 1 and 2: (a) shows a random function of Branch 1; (b) and (c) present the mean functions of priors $\mathbb{P}_{S}^{\theta}$ and $\mathbb{P}^{\theta}$ for Branch 1, respectively; (d) depicts a random function of Branch 2; (e) and (f) display the corresponding mean functions from both priors for Branch 2.}
		}\label{FigEx2_2_Main}
	\end{figure}
	
	\subsection{Complex environment}\label{NumericalComplex}
	
	Except for the simple environment, we design a more challenging complex environment, which has two branches compared to the simple environment. Specifically, the background true functions are generated based on the formula:  
	$u(x_1,x_2) = (2\alpha-1)(u_1(x_1,x_2) + u_2(x_1,x_2) + u_3(x_1,x_2)),$  
	where $\alpha\sim\text{Bern}(0.5)$ and settings are the same as in the simple environment. As in the simple environment, we generate 2000 training and 50 testing functions. In Figs. \ref{FigEx2_2_Main} (a) and (d), we show two functions randomly selected from the two branches of the test examples. Figs. \ref{FigEx2_2_Main} (b) and (e) present the learned mean functions under the data-dependent prior assumption. Notably, the Fourier neural operator distinguishes between the two branches and provides reasonable estimates. Figs. \ref{FigEx2_2_Main} (c) and (f) (same figure) display the learned mean functions under the data-independent prior assumption, demonstrating that this approach fails to capture the branches' characteristics.  
	
	\begin{figure}[ht!]
		\centering
		\includegraphics[width=1.0\textwidth]{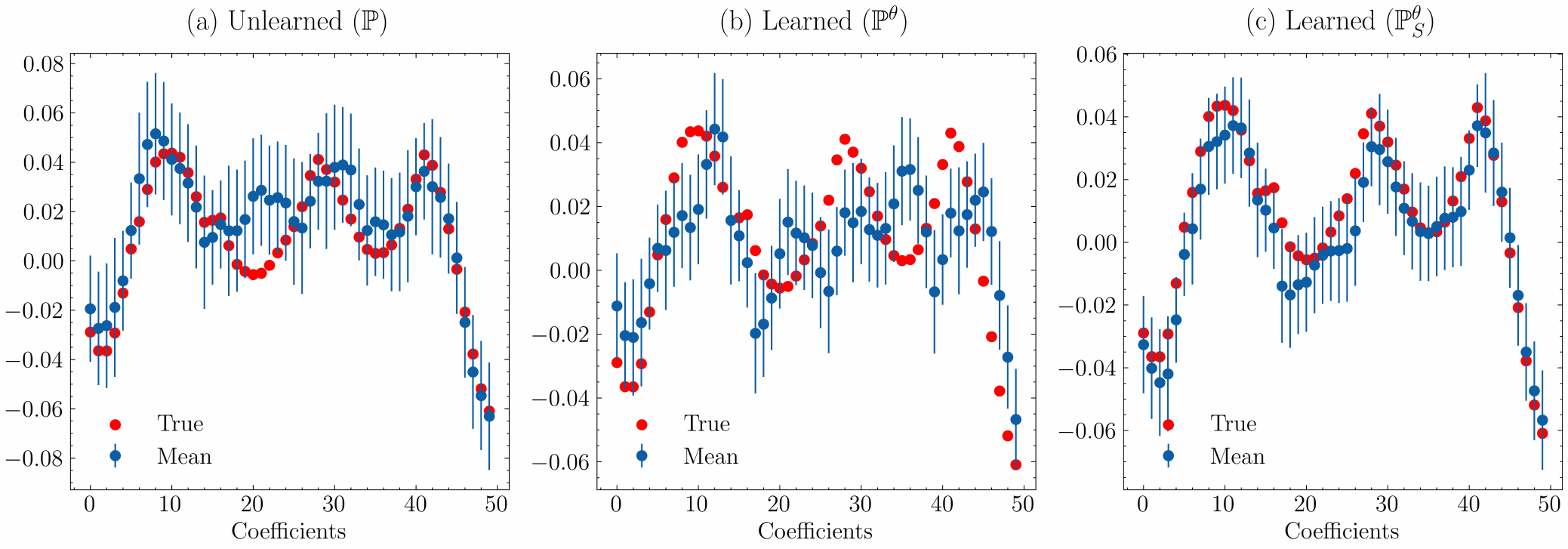}\\
		\vspace*{-3pt}
		\caption{\emph{\small In all panels, red and blue dots denote the true function's projected coefficients and the base posterior mean's projected coefficients, respectively. Blue vertical lines mark $95\%$ credible intervals. Panels (a), (b), and (c) show results for a negative-valued branch (complex environment) using priors $\mathbb{P}$, $\mathbb{P}^{\theta}$, and $\mathbb{P}_S^{\theta}$, respectively. }
		}\label{FigNegativeBranchStd}
	\end{figure}
	
	Table \ref{TableDarcyFlowMain} shows the average relative errors of MAP estimates in the complex environment. The learned data-dependent Bayesian model yields the lowest errors and requires only 5 iterations. The learned data-independent model performs significantly worse than the other two, with detailed analysis in Subsection A.9 
	of the supplement. Sampling speeds are shown in Fig. \ref{FigSMCSpeed} (b,c). For the positive-valued branch, all models exhibit similar speeds to the simple environment. For the negative-valued branch, however, the learned data-independent model requires $K_{\text{total}} > 100$—more than twice that of the data-dependent model — highlighting the efficiency of the data-dependent prior model in complex scenarios. In Fig. \ref{FigNegativeBranchStd}, we show the $95\%$ credible intervals of the three methods. The results clearly demonstrate that the data-dependent prior is capable of extracting valuable information from historical inverse tasks, utilizing only the historical measurement datasets without the true parameter.
	
	Finally, we analyze computational times. All programs ran on a system with an Intel(R) Xeon(R) Platinum 8180 CPU, 48 GB NVIDIA RTX A6000 GPU, and Ubuntu 22.04.5 LTS OS. We randomly selected a dataset generated from a negative-valued function in a complex environment and recorded the runtime of the MGSMC algorithm. The unlearned, learned data-independent, and learned data-dependent models consumed approximately 3074s, 6257s, and 1682s, respectively. Notably, the MGSMC algorithm used 1000 particles; increasing the particle number exacerbates the performance gap. By leveraging a mixture Gaussian approximation, MGSMC reduces computational cost significantly compared to the classical SMC algorithm \citep{Dashti2017}, though at the cost of reduced accuracy. For comparison, the classical SMC (500 particles) with preconditioned Crank–Nicolson sampling (200 steps) required approximately 57396s, 142843s, and 31362s for the respective models. This translates to a savings of nearly 7.2 hours per inverse problem for learned data-dependent models over unlearned models. Given computational constraints, we used 500 particles; extrapolating, 2000 particles would save nearly 28.8 hours.  Training $\mathbb{P}_S^{\theta} = \mathcal{N}(f_m(S;\theta_1), \mathcal{C}_0(\theta_2))$ with parameters $L=10$, $H=20$, and 7200 maximum iterations required 75750s (approximately 21 hours). While substantial, this is a one-time cost: once trained, the neural operator predictor can be employed to solve new inverse problems (given a new dataset $S$, no further training is needed).
}

\section{Concluding Remarks}\label{SectionConcludion}

This paper formulates infinite-dimensional inverse problems as regression tasks constrained by forward operators (e.g., heat or elliptic equations). Leveraging meta-learning, we develop a PAC-Bayesian framework for predicting prior probability measures in infinite-dimensional spaces. A key innovation is introducing data-dependent base priors via a condition inspired by the concept of differential privacy. By extending infinite-dimensional Bayesian theory, we derive general forward operator conditions compatible with both linear and nonlinear problems. Specific generalization bounds are established through careful log-moment generating function estimations. This theory yields an optimization problem avoiding the bilevel structure typical of prior learning. Finally, we propose a tractable learning algorithm that constructs effective prior-generating mappings.  

{\color{black}Designing neural operator architectures to process discrete datasets \( S \) and probability measures is critical for solving inverse problems. Architectural inspirations can be drawn from DeepSets \citep{ZaheerNIPS2017} and NIO frameworks \citep{MolinaroICML2023}, which offer structural paradigms for handling such input modalities. While only empirical PAC-Bayes bounds are obtained in this work, deriving oracle PAC-Bayes bounds and tighter empirical bounds \citep{Alquier2024Book,Riou2023ArXiv} remains an important future research direction.} 

\section{Disclosure statement}\label{disclosure-statement}

The authors report there are no competing interests to declare.

\section{Data Availability Statement}\label{data-availability-statement}

The generated data (6.84GB) have been made available at the following URL: \url{https://drive.google.com/drive/folders/1ekb_fWj4p7CZEBsx6BOMSkb5eCBMU_9A?usp=drive_link}.

\phantomsection\label{supplementary-material}
\bigskip

\begin{center}

{\large\bf SUPPLEMENTARY MATERIAL}

\end{center}

\begin{description}
\item[Learn\_Prior\_Supp:]
The additional theoretical results related to Sections \ref{SectionLearningPriorMeasure} and \ref{SectionApplication}, more numerical results, all proofs are deferred to the Supplementary Material. (.pdf file). 
\item[Code and Data:]
Python code to generate the synthetic datasets and reproduce the simulations and corresponding output.
\end{description}


\spacingset{1.1}
\bibliography{reference}

\end{document}


\def\spacingset#1{\renewcommand{\baselinestretch}%
{#1}\small\normalsize} \spacingset{1}

\bigskip
\bigskip
\bigskip
\begin{spacing}{1.5}
\begin{center}
	{\LARGE\bf Supplementary Materials of ``Nonparametric Prior Learning in Differential Equation Modeling"}
\end{center}
\end{spacing}


\tableofcontents 

\spacingset{1.2}

\begin{appendix}
\section{Additional Results}\label{AddReA}
\subsection{PAC-Bayesian for Inverse Problems}\label{SuppA:ShortIntroPACBayes}

The present work focuses on the infinite-dimensional Bayesian statistical approach for solving inverse problems associated with partial differential equations (PDEs). Bayesian statistics have a deep connection with the PAC-Bayesian framework, as illustrated in \cite{Germain2016NIPS}, specifically in the finite-dimensional setting. Although deep connections have been observed in the field of machine learning, to our knowledge, there seems to be little discussion on solving inverse problems of PDEs by linking the Bayesian inverse approach with PAC-Bayesian learning. For the reader's convenience, we intend to provide a brief introduction to the PAC-Bayesian learning theory from the perspective of infinite-dimensional inverse problems.

As in the main text, let us assume that the forward operator $\mathcal{G}:\mathcal{U}\times\mathcal{X}\rightarrow\tilde{\mathcal{H}}$ with $\mathcal{U}$, $\mathcal{X}$, and $\tilde{\mathcal{H}}$ being separable Banach spaces. We assume that the data are generated by:
\begin{align}\label{problem1A}
	y_j = \mathcal{G}(u,x_j) + \eta_j, \quad\forall \,j=1,\ldots,m,
\end{align}
where $u$ represents the model parameter, $x_j\in\mathcal{X}$ is the measurement point or an input function, $y_j\in\mathcal{Y}$ is the measured data, and $\eta_j$ is the independent identically distributed (i.i.d.) random noise.
Let us denote $\mathcal{H}$ be a separable Banach space. 
In the following, we assume that the forward operator can be written as $\mathcal{L}_{x_j}(\mathcal{G}(u))$, where $\mathcal{G}(u)$ belongs to some separable Hilbert space and $\mathcal{L}_{x_j}:\tilde{\mathcal{H}}\rightarrow\mathcal{H}$ is a bounded linear operator determined by $x_j\in\mathcal{X}$.
Denote the vector $\bm{x} = (x_1,\cdots,x_{m})^T$, the vector $\bm{y} = (y_1,\cdots,y_{m})^T$, and the vector $\bm{\eta} = (\eta_1,\cdots, \eta_{m})^T$, then model (\ref{problem1A}) can be rewritten in a more compact form
$$\bm{y} = \mathcal{L}_{\bm{x}}\mathcal{G}(u) + \bm{\eta},$$
where $\mathcal{L}_{\bm{x}}\mathcal{G}(u) := (\mathcal{L}_{x_1}\mathcal{G}(u),\ldots,\mathcal{L}_{x_{m}}\mathcal{G}(u))^T$.
For solving inverse problems, we assume that there is a background true parameter denoted by $u^{\dag}$.
Then, the forward model will be
$$\bm{y} = \mathcal{L}_{\bm{x}}\mathcal{G}(u^{\dag}) + \bm{\eta}.$$

Define $\mathcal{P}(\mathcal{Y})$ as the set of all probability measures defined on a separable Banach space $\mathcal{Y}$.
In the field of infinite-dimensional Bayesian inverse theory, we assume that the unknown parameter $u$ is a random function distributed according to a prior probability measure $\mathbb{P}$, which is an element of $\mathcal{P}(\mathcal{U})$. Then, we introduce a posterior probability measure $\mathbb{Q}$, also an element of $\mathcal{P}(\mathcal{U})$, such that
\begin{align}\label{BayesFormula1A}
	\frac{d\mathbb{Q}}{d\mathbb{P}}(u) = \frac{1}{Z_m} \exp\left( - \sum_{j=1}^{m}\Phi(u;y_j)\right),
\end{align}
where $\Phi(u;y_j)$ is the potential function derived from the likelihood function, and $Z_m$ is the normalization constant, defined as
\begin{align*}
	Z_{m} := \int_{\mathcal{U}} \exp\left( -\sum_{j=1}^{m}\Phi(u;y_j) \right) \mathbb{P}(du).
\end{align*}
For a fixed parameter $u$, the observed data vector $\bm{y}$ is determined by the input vector $\bm{x}$ and the noise vector $\bm{\eta}$. As in the main text, we also express the potential function $\Phi(u;y_j)$ as $\Phi(u;x_j,\eta_j)$ for $j=1,\ldots,m$, and assume that the formula (\ref{BayesFormula1A}) is well defined.

Assume that there is a probability measure $\mathbb{D}(dx, dy) := \mathbb{D}_1(dx) \mathbb{D}_2(x, u^{\dag}, dy)$, with $\mathbb{D}_1$ and $\mathbb{D}_2$ being two probability measures defined on $\mathcal{X}$ and $\mathcal{Y}$, respectively. Define $\mathcal{Z} := \mathcal{X} \times \mathcal{Y}$, then we have $z := (x, y) \sim \mathbb{D} \in \mathcal{P}(\mathcal{Z})$. For convenience, let us denote the dataset as $S = \{z_j = (x_j, y_j)\}_{j=1}^{m}$, consisting of i.i.d. samples from the distribution $\mathbb{D}$, that is, $S \sim \mathbb{D}^m$. Obviously, the measure $\mathbb{D}_2$ depends on the background true parameter $u^{\dag}$, the distribution $\mathbb{D}_1$, and the distribution of the noise, denoted by $\mathbb{D}_3$ (a probability measure defined on $\mathcal{Y}$). Here, we assume that $\{x_j\}_{j=1}^{m}$ are generated according to the probability measure $\mathbb{D}_1$. To provide a brief illustration of the PAC-Bayesian framework, we introduce a loss function $\ell: \mathcal{U} \times \mathcal{Z} \rightarrow \mathbb{R}$. Accordingly, we aim to minimize the expected error under the data distribution $\mathbb{D}$, that is,
$$\mathcal{L}(u, \mathbb{D}) := \mathbb{E}_{z \sim \mathbb{D}} \ell(u, z).$$
For solving inverse problems, the true distributions of $u$ and the noise $\bm{\eta}$ are unknown, i.e., the data distribution $\mathbb{D}$ is unknown. Hence, it is necessary to define the empirical error 
$$\hat{\mathcal{L}}(u, S) := \frac{1}{m} \sum_{j=1}^{m} \ell(u, z_j).$$

To give the key bound in the PAC-Bayesian framework, we define the so-called Gibbs error as $\mathcal{L}(\mathbb{Q}, \mathbb{D}) := \mathbb{E}_{u\sim\mathbb{Q}}\mathcal{L}(u, \mathbb{D})$ and its empirical counterpart as $\hat{\mathcal{L}}(\mathbb{Q}, S) := \mathbb{E}_{u\sim\mathbb{Q}}\hat{\mathcal{L}}(u, S)$.
Let us denote $D_{\text{KL}}(\cdot || \cdot)$ as the usual Kullback-Leibler (KL) divergence between two probability measures.
Now, we can state the following theorem, which provides an upper bound for the unknown generalization error based on its empirical estimate.

\begin{theorem}\label{PACBoundTheoremA} 
	Given data distribution $\mathbb{D}$, parameter space $\mathcal{U}$, loss function $\ell(u,z)$, 
	prior measure $\mathbb{P}\in\mathcal{P}(\mathcal{U})$, confidence level $\delta\in(0,1]$, and $\beta > 0$, with probability at least $1-\delta$
	over samples $S\sim\mathbb{D}^m$, we have that, for all $\mathbb{Q}\in\mathcal{P}(\mathcal{U})$:
	\begin{align}\label{PACBoundBaseA}
		\mathcal{L}(\mathbb{Q},\mathbb{D}) \leq \hat{\mathcal{L}}(\mathbb{Q},S) + \frac{1}{\beta}\left[
		D_{\text{KL}}(\mathbb{Q}||\mathbb{P}) + \ln\frac{1}{\delta} + \Psi(\beta, m)
		\right],
	\end{align}
	where $\Psi(\beta, m) := \ln\mathbb{E}_{u\sim\mathbb{P}}\mathbb{E}_{S\sim\mathbb{D}^m}\exp\left[
	\beta \left( \mathcal{L}(u,\mathbb{D}) - \hat{\mathcal{L}}(u,S) \right)
	\right]$ is a term that depends on the tail of the distribution of the loss function.
\end{theorem}
If $\beta = m$, given the loss function $\ell(u, z) := \Phi(u, x, \eta)$, and with a proper estimate of $\Psi(\beta, m)$, then the minimum value on the right-hand side of estimate (\ref{PACBoundBaseA}) corresponds to the posterior measure specified in equation (\ref{BayesFormula1A}). Generally, the optimal solution provided by the PAC-Bayesian bound may not correspond to the posterior measure given by the Bayes' formula, and the optimal solution typically corresponds to the posterior measure obtained from a generalized Bayes' formula. The generalization bound (\ref{PACBoundBaseA}) is selected from \cite{Germain2016NIPS} and  \cite{Alquier2016JMLR}, which generalizes the original PAC-Bayes bounds \citep{McAllester1998ACCLT}. For a comprehensive illustration, we refer to \cite{Alquier2024Book}.

\begin{remark}
	For the current work, we focus on the problem of learning data-dependent prior measures. Theorem \ref{PACBoundTheoremA} can be conceptually adapted to the statistical inverse problems of PDEs. Bounding the term $\Psi(\beta, m)$ and generalizing the results shown in \cite{Alquier2016JMLR} for inverse problems present interesting future research challenges. Similar to the adaptation of the nonparametric Bayesian approach to inverse problems \citep{Nickl2020JEMS, Nickl2017AoS, Monard2021CPAM}, such a generalization is not a trivial task, due to issues of measurability, appropriate estimations of PDEs, and the singularity of infinite-dimensional measures.
\end{remark}

\subsection{Data-Independent Prior}\label{SuppA:data-dependent prior}

In the main text, we provide illustrations of the data-dependent prior case. For completeness, we intend to give a detailed explanation of the data-independent prior case, i.e., the prior measure $\mathbb{P}_{S}^{\theta}$ is assumed to be independent of the dataset $S$. Hence, we denote this prior measure as $\mathbb{P}^{\theta}$.

\textbf{Case 1: Bounded loss.} 
For simplicity, let us first assume that the loss function $\ell(u, z)$ is bounded in $[a, b]$, where $-\infty < a < b < +\infty$. Under this assumption, we can apply Hoeffding's lemma to obtain the following corollary.
\begin{corollary}\label{PACBoundsBoundedLoss1}
	Assume that all of the assumptions in Theorem 2.4 
    hold true and the loss function 
	$\ell(u,\bm{z})$ is bounded in $[a,b]$ with $-\infty < a < b < +\infty$. 
	Let us denote 
	\begin{align}
		\bar{m} = \left( \frac{1}{n}\sum_{i=1}^{n}\frac{1}{m_i} \right)^{-1}.
	\end{align}
	For any confidence level $\delta\in(0,1]$, $\gamma\geq2\sqrt{n}$, and $\lambda\geq2\sqrt{n}$, we have  
	\begin{align*}
		\mathbb{P}\Bigg( \forall \mathscr{Q}\in\mathcal{P}(\Theta), \,\,
		\mathcal{L}(\mathscr{Q}, \mathcal{T}) \leq & \, \hat{\mathcal{L}}(\mathscr{Q},S_1,\ldots,S_n) + 
		\left( \frac{1}{\lambda} + \frac{1}{\gamma} \right)D_{\text{KL}}(\mathscr{Q} || \mathscr{P})  \\
		& \, + \frac{1}{\gamma}\sum_{i=1}^{n}\mathbb{E}_{\theta\sim\mathscr{Q}}D_{\text{KL}}
		\big(\mathbb{Q}(S_i,\mathbb{P}^{\theta}) || \mathbb{P}^{\theta}\big) \\
		& \,  
		+ \left( \frac{\gamma}{8n\bar{m}} + \frac{\lambda}{8n} \right)(b-a)^2 
		+ \frac{1}{\sqrt{n}}\ln\frac{1}{\delta}\Bigg) \geq 1-\delta.
	\end{align*}
\end{corollary}
\begin{proof}
	Remember that we assumed the loss function $\ell(u,z)$ is bounded in $[a,b]$ 
	with $-\infty < a < b < +\infty$.
	Under this assumption, we can apply Hoeffding's lemma (Lemma D.1 in \cite{Mohri2018FML}) to
	each factor of the term $\mathbb{E}\big[ \exp(\sqrt{n}\Pi^1(\gamma) + \sqrt{n}\Pi^2(\lambda)) \big]$, obtaining: 
	\begin{align}\label{DataIndependentBoundedLoss1}
		\begin{split}
			\mathbb{E}\left[ e^{(\sqrt{n}\Pi^1(\gamma) + \sqrt{n}\Pi^{2}(\lambda))} \right] \leq & 
			\bigg\{\prod_{i=1}^{n}\mathbb{E}_{u_i\sim\mathbb{P}^{\theta}}\bigg[
			e^{\frac{\gamma^2}{8n^2m_i}(b-a)^2}
			\bigg]\bigg\}^{\frac{\sqrt{n}}{\gamma}}
			\bigg\{  
			e^{\frac{\lambda^2}{8n}(b-a)^2}
			\bigg\}^{\frac{\sqrt{n}}{\lambda}}	\\
			\leq & \exp\left(\left( \frac{\gamma}{8\sqrt{n}\bar{m}} + \frac{\lambda}{8\sqrt{n}} \right)(b-a)^2\right),
		\end{split}
	\end{align}
	when $\gamma\geq2\sqrt{n}$ and $\lambda\geq2\sqrt{n}$. 
	Plugging this estimate into the general estimate shown in Theorem 2.4 
    of the main text, 
	we obtain an explicit bound on the moment generating function, which completes the proof.
\end{proof}

\textbf{Case 2: Unbounded loss.}
The inverse problems of PDEs can be seen as PDE-constrained regression problems, which are often modeled using unbounded loss functions, such as the squared loss function. To provide appropriate estimates of the log moment generating term, we consider sub-Gaussian type loss functions, which are employed in the study of regression problems with a fixed finite number of parameters \citep{Germain2016NIPS}.

\begin{assumption}[sub-Gaussian assumption related to $\Pi^1$]\label{UnboundedLossAssumptionSubGaussian1A}
	For $i=1,\ldots,n$, we assume that the random variables $\mathcal{L}(u_i,\mathbb{D}_i) - \ell(u_i, z_i)$
	are sub-Gaussian with variance factor $s_{\text{I}}^2$, i.e., for some $\tilde{\gamma} > 0$, we have 
	\begin{align*}
		V_i^1 := & \mathbb{E}_{(\mathbb{D}_i, m_i)\sim\mathcal{T}}\mathbb{E}_{z_i\sim\mathbb{D}_i}
		\mathbb{E}_{\theta\sim\mathscr{P}}\mathbb{E}_{u_i\sim\mathbb{P}^{\theta}}\exp\left( 
		\tilde{\gamma}\left[ 
		\mathcal{L}(u_i,\mathbb{D}_i) - \ell(u_i, z_i)
		\right]\right) 
		\leq \exp\left( \frac{\tilde{\gamma}^2s_{\text{I}}^2}{2} \right).
	\end{align*}
\end{assumption}

\begin{assumption}[sub-Gaussian assumption related to $\Pi^2$]\label{UnboundedLossAssumptionSubGaussian2A}
	For $i=1,\ldots,n$, we assume that the random variables $\mathbb{E}_{(\mathbb{D},m)\sim\mathcal{T}}\mathbb{E}_{S\sim\mathbb{D}^m}\mathcal{L}(\mathbb{Q}(S,\mathbb{P}^{\theta}),\mathbb{D})
	- \mathcal{L}(\mathbb{Q}(S_i,\mathbb{P}^{\theta}),\mathbb{D}_i)$ are sub-Gaussian with 
	variance factor $s_{\text{II}}^2$, i.e., for some $\tilde{\lambda} > 0$, we have
	\begin{align*}
		V_i^2 := &\mathbb{E}_{(\mathbb{D}_i,m_i)\sim\mathcal{T}}\mathbb{E}_{S_i\sim\mathbb{D}_i^{m_i}}\mathbb{E}_{\theta\sim\mathscr{P}}
		\exp\left( 
		\tilde{\lambda}\left[ 
		\mathbb{E}_{(\mathbb{D},m)\sim\mathcal{T}}\mathbb{E}_{S\sim\mathbb{D}^m}\mathcal{L}(\mathbb{Q}(S,\mathbb{P}^{\theta}),\mathbb{D})
		- \mathcal{L}(\mathbb{Q}(S_i,\mathbb{P}^{\theta}),\mathbb{D}_i)
		\right]\right) \\
		\leq & \exp\left( 
		\frac{\tilde{\lambda}^2s_{\text{II}}^2}{2}
		\right).
	\end{align*}
\end{assumption}

With these two assumptions, we can prove the following corollary.

\begin{corollary}\label{PACBoundsUnoundedLossThmSubGaussian}
	Assume that all of the assumptions in Theorem 2.4 
    and Assumptions \ref{UnboundedLossAssumptionSubGaussian1A} and \ref{UnboundedLossAssumptionSubGaussian2A} hold true.
	In addition, we assume $2\sqrt{n}\leq\lambda$ and $2\sqrt{n}\leq\gamma$.
	Let us denote $\bar{m} = \left( \frac{1}{n}\sum_{i=1}^{n}\frac{1}{m_i} \right)^{-1}.$
	For any confidence level $\delta\in(0,1]$, the inequality 
	\begin{align*}
			\mathbb{P}\Bigg( \forall \mathscr{Q}\in\mathcal{P}(\Theta), \,\,
			\mathcal{L}(\mathscr{Q}, \mathcal{T}) \leq & \, \hat{\mathcal{L}}(\mathscr{Q},S_1,\ldots,S_n) + 
			\left( \frac{1}{\lambda} + \frac{1}{\gamma} \right)D_{\text{KL}}(\mathscr{Q} || \mathscr{P})  \\
			& \, + \frac{1}{\gamma}\sum_{i=1}^{n}\mathbb{E}_{\theta\sim\mathscr{Q}}D_{\text{KL}}
			\big(\mathbb{Q}(S_i,\mathbb{P}^{\theta}) || \mathbb{P}^{\theta}\big) \\
			& \,
			+ \frac{\gamma s_{\text{I}}^2}{2n\bar{m}} + \frac{\lambda s_{\text{II}}^2}{2n}
			+ \frac{1}{\sqrt{n}}\ln\frac{1}{\delta}
			\Bigg) \geq 1-\delta. 
	\end{align*}
\end{corollary}
\begin{proof}
	In the following $\frac{\gamma}{nm_i}$ will be recognized as $\tilde{\gamma}$ and 
	we assume $\gamma \geq 2\sqrt{n}$ and $\lambda\geq2\sqrt{n}$.
	Employing the assumption of $V_i^1$ in Assumption \ref{UnboundedLossAssumptionSubGaussian1A}, 
	we have 
	\begin{align}\label{PACBoundsUnboundedLoss21}
		\begin{split}
			& \mathbb{E}\left\{ \mathbb{E}_{\theta\sim\mathscr{P}}\mathbb{E}_{u_1\sim\mathbb{P}^{\theta}}\cdots 
			\mathbb{E}_{u_n\sim\mathbb{P}^{\theta}}\exp\left( 
			\frac{\gamma}{n}\sum_{i=1}^{n}\frac{1}{m_i}\sum_{j=1}^{m_i}[\mathcal{L}(u_i,\mathbb{D}_i) - \ell(u_i,z_{ij})]
			\right) \right\}^{\frac{2\sqrt{n}}{\gamma}} \\
			& \quad
			\leq \left\{ \mathbb{E}\prod_{i=1}^n\mathbb{E}_{\theta\sim\mathscr{P}}\mathbb{E}_{u_i\sim\mathbb{P}^{\theta}}
			\prod_{j=1}^{m_i} \exp\left( \frac{\gamma}{nm_i}[\mathcal{L}(u_i,\mathbb{D}_i) - \ell(u_i, z_{ij})] \right) \right\}^{\frac{2\sqrt{n}}{\gamma}} \\
			& \quad
			\leq \exp\left( \frac{\gamma s_{\text{I}}^2}{\sqrt{n}\bar{m}} \right),
		\end{split}
	\end{align}
	which can be written more compactly as follow:
	\begin{align}\label{PACBoundsUnboundedLoss22}
		\Big(\mathbb{E}\exp(2\sqrt{n}\Pi^1(\gamma))\Big)^{1/2} \leq \exp\left( \frac{\gamma s_{\text{I}}^2}{2\sqrt{n}\bar{m}} \right).
	\end{align}
	Employing the assumption of $V_i^2$ in Assumption \ref{UnboundedLossAssumptionSubGaussian2A} of the main text, 
	we have 
	\begin{align*}
			& \mathbb{E}\left\{
			\mathbb{E}_{\theta\sim\mathscr{P}}\exp\left( 
			\frac{\lambda}{n}\sum_{i=1}^{n}
			\left[ 
			\mathbb{E}_{(\mathbb{D},m)\sim\mathcal{T}}\mathbb{E}_{S\sim\mathbb{D}^m}\mathcal{L}(\mathbb{Q}(S,\mathbb{P}^{\theta}),\mathbb{D})
			- \mathcal{L}(\mathbb{Q}(S_i,\mathbb{P}^{\theta}),\mathbb{D}_i)
			\right]
			\right)   
			\right\}^{\frac{2\sqrt{n}}{\lambda}} \\
			& 
			\leq \left\{ 
			\prod_{i=1}^{n}\mathbb{E}\mathbb{E}_{\theta\sim\mathscr{P}}\exp\left( 
			\frac{\lambda}{n}\left[ 
			\mathbb{E}_{(\mathbb{D},m)\sim\mathcal{T}}\mathbb{E}_{S\sim\mathbb{D}^m}\mathcal{L}(\mathbb{Q}(S,\mathbb{P}^{\theta}),\mathbb{D})
			- \mathcal{L}(\mathbb{Q}(S_i,\mathbb{P}^{\theta}),\mathbb{D}_i)
			\right]
			\right)\right\}^{2\frac{\sqrt{n}}{\lambda}}  \\
			& \leq \left\{ 
			\prod_{i=1}^{n}\exp\left( 
			\frac{\lambda^2 s_{\text{\text{II}}}^2}{2n^2}
			\right)\right\}^{\frac{2\sqrt{n}}{\lambda}} = \exp\left( 
			\frac{\lambda s_{\text{II}}^2}{\sqrt{n}}
			\right),
	\end{align*}
	which yields
	\begin{align}\label{PACBoundsUnboundedLoss23}
		\Big(\mathbb{E}\exp(2\sqrt{n}\Pi^2(\lambda))\Big)^{1/2} \leq \exp\left( 
		\frac{\lambda s_{\text{II}}^2}{2\sqrt{n}}
		\right).
	\end{align}
	Considering estimates (\ref{PACBoundsUnboundedLoss22}) and (\ref{PACBoundsUnboundedLoss23}), we obtain the desired result. 
\end{proof}

\begin{remark}
	In Assumptions \ref{UnboundedLossAssumptionSubGaussian1A} and \ref{UnboundedLossAssumptionSubGaussian2A}, 
	we introduced the parameters $\tilde{\gamma}$ and $\tilde{\lambda}$, which correspond to the parameters $\gamma/nm_i$ and $\lambda/n$ in Corollary \ref{PACBoundsUnoundedLossThmSubGaussian}.
	Typically, we set $\gamma=n\beta$ with $\beta=\sqrt{\bar{m}} \text{ or }\beta=\bar{m}$ when $2\sqrt{n}\leq n\beta$, 
	and $\lambda=2\sqrt{n} \text{ or }n$ when $2\sqrt{n} \leq n$. 
	For deriving Corollary \ref{PACBoundsUnoundedLossThmSubGaussian}, we require 
	$\lambda \geq 2\sqrt{n}$ and $\gamma\geq 2\sqrt{n}$, conditions not mentioned in the work of meta-learning \citep{Rothfuss2021PMLR} for the finite-dimensional case. Since the random variables $\Pi^1(\gamma)$ and $\Pi^2(\lambda)$ are not independent, we cannot obtain 
	\begin{align*}
		\mathbb{E}\left[ e^{(\sqrt{n}\Pi^1(\gamma) + \sqrt{n}\Pi^{2}(\lambda))} \right] = 
		\mathbb{E}\left[ e^{\sqrt{n}\Pi^1(\gamma)} \right]
		\mathbb{E}\left[ e^{\sqrt{n}\Pi^{2}(\lambda)} \right]. 
	\end{align*}
	Hence, it may be necessary for us to employ
	\begin{align*}
		\mathbb{E}\left[ e^{(\sqrt{n}\Pi^1(\gamma) + \sqrt{n}\Pi^{2}(\lambda))} \right] & \leq 
		\left(\mathbb{E}\left[ e^{2\sqrt{n}\Pi^1(\gamma)} \right]\right)^{1/2}
		\left(\mathbb{E}\left[ e^{2\sqrt{n}\Pi^{2}(\lambda)} \right]\right)^{1/2} \\
		& \leq \left(\mathbb{E}\left[ e^{\gamma\Pi^1(\gamma)} \right]\right)^{\sqrt{n}/\gamma}
		\left(\mathbb{E}\left[ e^{\lambda\Pi^{2}(\lambda)} \right]\right)^{\sqrt{n}/\lambda},
	\end{align*}
	which requires that $\lambda \geq 2\sqrt{n}$ and $\gamma\geq 2\sqrt{n}$. 
\end{remark}

\subsection{Results Under Sub-Gamma Assumptions}\label{AppA1}
In the main text, we provide the sub-Gaussian assumptions for the terms $\Pi^1$ and $\Pi^2$ related to the loss function.
In the following, we list the sub-Gamma assumptions, which often appear in the studies of the PAC-Bayesian theory \citep{Germain2016NIPS,Rothfuss2021PMLR}.

\begin{assumption}[sub-Gamma assumption related to $\Pi^1$]\label{UnboundedLossAssumptionSubGamma1}
	For $i=1,\ldots,n$, we assume that the random variables $$\mathcal{L}(u_i,\mathbb{D}_i) - \ell(u_i, z_i)$$
	are sub-Gamma with variance factor $s_{\text{I}}^2$ and scale parameter $c_{\text{I}} < 1$, i.e., 
	\begin{align*}
		V_i^1 := & \mathbb{E}_{(\mathbb{D}_i,m_i)\sim\mathcal{T}}\mathbb{E}_{z_i\sim\mathbb{D}_i}
		\mathbb{E}_{\theta\sim\mathscr{P}}\mathbb{E}_{u_i\sim\mathbb{P}^{\theta}}\exp\left( 
		\tilde{\gamma}\left[ 
		\mathcal{L}(u_i,\mathbb{D}_i) - \ell(u_i, z_i)
		\right]\right) \\
		\leq & \exp\left( \frac{\tilde{\gamma}^2s_{\text{I}}^2}{2(1-\tilde{\gamma}c_{\text{I}})} \right)
	\end{align*}
	for all $\tilde{\gamma} \in (0,1/c_{\text{I}})$. 
\end{assumption}

\begin{assumption}[sub-Gamma assumption related to $\Pi^2$]\label{UnboundedLossAssumptionSubGamma2}
	For $i=1,\ldots,n$, we assume that the random variables $$\mathbb{E}_{(\mathbb{D},m)\sim\mathcal{T}}\mathbb{E}_{S\sim\mathbb{D}^m}\mathcal{L}(\mathbb{Q}(S,\mathbb{P}^{\theta}),\mathbb{D})
	- \mathcal{L}(\mathbb{Q}(S_i,\mathbb{P}^{\theta}),\mathbb{D}_i)$$ are sub-Gamma with 
	variance factor $s_{\text{II}}^2$ and scale parameter $c_{\text{II}} < 1$, i.e., 
	\begin{align*}
		V_i^2 := &\mathbb{E}_{(\mathbb{D}_i,m_i)\sim\mathcal{T}}\mathbb{E}_{S_i\sim\mathbb{D}_i^{m_i}}\mathbb{E}_{\theta\sim\mathscr{P}}
		\exp\left( 
		\tilde{\lambda}\left[ 
		\mathbb{E}_{(\mathbb{D},m)\sim\mathcal{T}}\mathbb{E}_{S\sim\mathbb{D}^m}\mathcal{L}(\mathbb{Q}(S,\mathbb{P}^{\theta}),\mathbb{D})
		- \mathcal{L}(\mathbb{Q}(S_i,\mathbb{P}^{\theta}),\mathbb{D}_i)
		\right]\right) \\
		\leq & \exp\left( 
		\frac{\tilde{\lambda}^2s_{\text{II}}^2}{2(1-\tilde{\lambda}c_{\text{II}})}
		\right)
	\end{align*}
	for all $\tilde{\lambda} \in (0,1/c_{\text{II}})$. 
\end{assumption}

With these two assumptions, we can derive the following two corollaries of Theorem 2.4 
in the main text, when we assume that the prior measures are either data-independent or data-dependent, respectively.

\begin{corollary}\label{PACBoundsUnoundedLossThmSubGamma}
	Assume that all of the assumptions in Theorem 2.4 
	and Assumptions \ref{UnboundedLossAssumptionSubGamma1} and \ref{UnboundedLossAssumptionSubGamma2} hold true.
	In addition, we assume $2\sqrt{n}\leq\lambda \leq n$ and $2\sqrt{n}\leq\gamma \leq n\min_{1\leq i\leq n}m_i$.
	For any confidence level $\delta\in(0,1]$, we have  
	\begin{align*}
			\mathbb{P}\Bigg( \forall \mathscr{Q}\in\mathcal{P}(\Theta), \,\,
			\mathcal{L}& (\mathscr{Q}, \mathcal{T}) \leq  \hat{\mathcal{L}}(\mathscr{Q},S_1,\ldots,S_n) + 
			\left( \frac{1}{\lambda} + \frac{1}{\gamma} \right)D_{\text{KL}}(\mathscr{Q} || \mathscr{P})  \\
			& \, + \frac{1}{\gamma}\sum_{i=1}^{n}\mathbb{E}_{\theta\sim\mathscr{Q}}D_{\text{KL}}
			\big(\mathbb{Q}(S_i,\mathbb{P}^{\theta}) || \mathbb{P}^{\theta}\big) \\
			& \, + \frac{\gamma s_{\text{I}}^2}{2n}\left( \frac{1}{n}\sum_{i=1}^{n}\frac{1}{m_i\left( 1-c_I\frac{\gamma}{nm_i} \right)} \right) + \frac{\lambda s_{\text{II}}^2}{2n\left(1-\frac{\lambda}{n} c_{\text{II}}\right)}
			+ \frac{1}{\sqrt{n}}\ln\frac{1}{\delta} \Bigg)\geq 1-\delta. 
	\end{align*}
\end{corollary}
\begin{proof}
	In the following $\frac{\gamma}{nm_i}$ will be recognized as $\tilde{\gamma}$ and we assume $\gamma \geq 2\sqrt{n}$
	and $\lambda\geq2\sqrt{n}$.
	Employing the assumption of $V_i^1$ in Assumption \ref{UnboundedLossAssumptionSubGamma1}, we have 
	\begin{align}\label{PACBoundsUnboundedLoss1}
		\begin{split}
			& \mathbb{E}\left\{ \mathbb{E}_{\theta\sim\mathscr{P}}\mathbb{E}_{u_1\sim\mathbb{P}^{\theta}}\cdots 
			\mathbb{E}_{u_n\sim\mathbb{P}^{\theta}}\exp\left( 
			\frac{\gamma}{n}\sum_{i=1}^{n}\frac{1}{m_i}\sum_{j=1}^{m_i}[\mathcal{L}(u_i,\mathbb{D}_i) - \ell(u_i,z_{ij})]
			\right) \right\}^{\frac{2\sqrt{n}}{\gamma}} \\
			& \quad
			\leq \left\{ \mathbb{E}\prod_{i=1}^n\mathbb{E}_{\theta\sim\mathscr{P}}\mathbb{E}_{u_i\sim\mathbb{P}^{\theta}}
			\prod_{j=1}^{m_i} \exp\left( \frac{\gamma}{nm_i}[\mathcal{L}(u_i,\mathbb{D}_i) - \ell(u_i, z_{ij})] \right) \right\}^{\frac{2\sqrt{n}}{\gamma}} \\
			& \quad
			\leq \exp\left( \frac{\gamma s_{\text{I}}^2}{\sqrt{n}} \left( \frac{1}{n}\sum_{i=1}^{n}\frac{1}{m_i\left( 1-c_I\frac{\gamma}{nm_i} \right)} \right) \right),
		\end{split}
	\end{align}
	which can be written more compactly as follow:
	\begin{align}\label{PACBoundsUnboundedLoss2}
		\Big(\mathbb{E}\exp(2\sqrt{n}\Pi^1(\gamma))\Big)^{1/2} \leq \exp\left( \frac{\gamma s_{\text{I}}^2}{2\sqrt{n}} \left( \frac{1}{n}\sum_{i=1}^{n}\frac{1}{m_i\left( 1-c_I\frac{\gamma}{nm_i} \right)} \right) \right).
	\end{align}
	Employing the assumption of $V_i^2$ in Assumption \ref{UnboundedLossAssumptionSubGamma2}, we have 
	\begin{align*}
			& \mathbb{E}\left\{
			\mathbb{E}_{\theta\sim\mathscr{P}}\exp\left( 
			\frac{\lambda}{n}\sum_{i=1}^{n}
			\left[ 
			\mathbb{E}_{(\mathbb{D},m)\sim\mathcal{T}}\mathbb{E}_{S\sim\mathbb{D}^m}\mathcal{L}(\mathbb{Q}(S,\mathbb{P}^{\theta}),\mathbb{D})
			- \mathcal{L}(\mathbb{Q}(S_i,\mathbb{P}^{\theta}),\mathbb{D}_i)
			\right]
			\right)   
			\right\}^{\frac{2\sqrt{n}}{\lambda}} \\
			& 
			\leq \left\{ 
			\prod_{i=1}^{n}\mathbb{E}\mathbb{E}_{\theta\sim\mathscr{P}}\exp\left( 
			\frac{\lambda}{n}\left[ 
			\mathbb{E}_{(\mathbb{D},m)\sim\mathcal{T}}\mathbb{E}_{S\sim\mathbb{D}^m}\mathcal{L}(\mathbb{Q}(S,\mathbb{P}^{\theta}),\mathbb{D})
			- \mathcal{L}(\mathbb{Q}(S_i,\mathbb{P}^{\theta}),\mathbb{D}_i)
			\right]
			\right)\right\}^{2\frac{\sqrt{n}}{\lambda}}  \\
			& \leq \left\{ 
			\prod_{i=1}^{n}\exp\left( 
			\frac{\lambda^2 s_{\text{II}}^2}{2n^2\left(1-\frac{\lambda}{n} c_{\text{II}}\right)}
			\right)\right\}^{2\frac{\sqrt{n}}{\lambda}} = \exp\left( 
			\frac{\lambda s_{\text{II}}^2}{\sqrt{n}\left(1-\frac{\lambda}{n} c_{\text{II}}\right)}
			\right),
	\end{align*}
	which yields
	\begin{align}\label{PACBoundsUnboundedLoss3}
		\Big(\mathbb{E}\exp(2\sqrt{n}\Pi^2(\lambda))\Big)^{1/2} \leq \exp\left( 
		\frac{\lambda s_{\text{II}}^2}{2\sqrt{n}\left(1-\frac{\lambda}{n} c_{\text{II}}\right)}
		\right).
	\end{align}
	Considering estimates (\ref{PACBoundsUnboundedLoss2}) and (\ref{PACBoundsUnboundedLoss3}), we obtain the desired result. 
\end{proof}

\begin{corollary}\label{PACDependentThm}
	Assume that all of the assumptions in Theorem 2.4 
	Assumptions \ref{UnboundedLossAssumptionSubGamma1}, \ref{UnboundedLossAssumptionSubGamma2}, and 2.6 
    (in the main text) 
	hold true. In addition, we assume $2\sqrt{n}\leq\lambda\leq n$ and $4\sqrt{n}\leq2\gamma\leq n\min_{1\leq i\leq n}m_i$.
	For any confidence level $\delta\in(0,1]$, we have  
	\begin{align*}
			\mathbb{P}\Bigg( \forall \mathscr{Q}\in\mathcal{P}(\Theta), \,\, 
			\mathcal{L} & (\mathscr{Q}, \mathcal{T}) \leq \, \hat{\mathcal{L}}(\mathscr{Q},S_1,\ldots,S_n) + 
			\left( \frac{1}{\lambda} + \frac{1}{\gamma} \right)D_{\text{KL}}(\mathscr{Q} || \mathscr{P})  \\
			& \, + \frac{1}{\gamma}\sum_{i=1}^{n}\mathbb{E}_{\theta\sim\mathscr{Q}}D_{\text{KL}}
			\big(\mathbb{Q}(S_i,\mathbb{P}_{S_i}^{\theta}) || \mathbb{P}_{S_i}^{\theta}\big) 
			+ \frac{n}{2\gamma}\ln\Psi_E \\
			& \, 
			+ \frac{\gamma s_{\text{I}}^2}{n} \left( \frac{1}{n}\sum_{i=1}^{n}\frac{1}{m_i\left( 1-c_I\frac{\gamma}{nm_i} \right)} \right)
			+ \frac{\lambda s_{\text{II}}^2}{2n\left(1-\frac{\lambda}{n} c_{\text{II}}\right)}
			+ \frac{1}{\sqrt{n}}\ln\frac{1}{\delta}\Bigg) \geq 1-\delta. 
	\end{align*}
\end{corollary}
\begin{proof}
	For the term $V_i^1$, we have 
	\begin{align}\label{PACDependentUnboundedLoss1}
		\begin{split}
			V_i^1 = & \mathbb{E}_{S_i\sim\mathbb{D}_{\mathcal{T}}}
			\mathbb{E}_{\theta\sim\mathscr{P}}\mathbb{E}_{u_i\sim\mathbb{P}_{S_i}^{\theta}}\exp\left( 
			\tilde{\gamma}\left[ 
			\mathcal{L}(u_i,\mathbb{D}_i) - \ell(u_i, z_{i})
			\right]\right) \\
			\leq & \mathbb{E}_{S_i^{'}\sim\mathbb{D}_{\mathcal{T}}}\mathbb{E}_{S_i\sim\mathbb{D}_{\mathcal{T}}}
			\mathbb{E}_{\theta\sim\mathscr{P}}\mathbb{E}_{u_i\sim\mathbb{P}_{S_i^{'}}^{\theta}} 
			\Psi(S_i,S_i^{'},u_i,\theta)\exp\left( \tilde{\gamma}[\mathcal{L}(u_i,\mathbb{D}_i) - \ell(u_i,z_i)] \right) \\
			\leq & \Psi_E^{1/2}
			\bigg(\mathbb{E}_{S_i^{'}\sim\mathbb{D}_{\mathcal{T}}}\mathbb{E}_{S_i\sim\mathbb{D}_{\mathcal{T}}}
			\mathbb{E}_{\theta\sim\mathscr{P}}\mathbb{E}_{u_i\sim\mathbb{P}_{S_i^{'}}^{\theta}} 
			\exp\left( 2\tilde{\gamma}[\mathcal{L}(u_i,\mathbb{D}_i) - \ell(u_i,z_i)] \right)\bigg)^{1/2} \\
			\leq & \Psi_E^{1/2}
			\exp\left( \frac{\tilde{\gamma}^2s_{\text{I}}^2}{1-\tilde{\gamma}c_{\text{I}}} \right),
		\end{split}
	\end{align}
	where we need $0<2\tilde{\gamma}<1/c_{\text{I}}$ for deriving the last inequality. 
	Through some simple calculations as for the bounded loss case shown in Corollary \ref{PACBoundsBoundedLoss2}, we find that 
	\begin{align}\label{GeneralSubGammaLossS2}
		\begin{split}
			\left(\mathbb{E}e^{2\sqrt{n}\Pi^1(\gamma)}\right)^{1/2} \leq & \Psi_E^{\frac{n\sqrt{n}}{2\gamma}}
			\exp\left( \frac{\gamma s_{\text{I}}^2}{\sqrt{n}} \left( \frac{1}{n}\sum_{i=1}^{n}\frac{1}{m_i\left( 1-c_I\frac{\gamma}{nm_i} \right)} \right) \right).
		\end{split}
	\end{align}
	Combining estimates on the term $\mathbb{E}e^{\sqrt{n}\Pi^2(\lambda)}$ (not given here since it is 
	similar to the data-independent case), we obtain 
	\begin{align}\label{GeneralSubGammaLossS3}
		\begin{split}
			\mathbb{E}\left[ e^{\sqrt{n}\Pi^1(\gamma) + \sqrt{n}\Pi^2(\lambda)} \right] \! \leq \! \Psi_E^{\frac{n\sqrt{n}}{2\gamma}}\!
			\exp\left(\! \frac{\gamma s_{\text{I}}^2}{\sqrt{n}}\left( \frac{1}{n}\sum_{i=1}^{n}\frac{1}{m_i\left( 1-c_I\frac{\gamma}{nm_i} \right)} \right) 
			\! + \! \frac{\lambda s_{\text{II}}^2}{2\sqrt{n}\left(1-\frac{\lambda}{n} c_{\text{II}}\right)} \! \right).
		\end{split}
	\end{align}
	Combining estimate (\ref{GeneralSubGammaLossS3}) with the results given in Theorem 2.4, 
	we obtain the desired result.
\end{proof}

\subsection{Discussion of Assumption 2.6}\label{AppA2}
{\color{black}
In this subsection, we revisit the example provided in Subsection 2.3, 
which demonstrates that Assumption 2.6 
from the main text could hold true under specific conditions. In the main text, we consider a Gaussian prior $\mathbb{P} := \mathcal{N}(0, \mathcal{C}_0)$, where $\mathcal{C}_0$ denotes a positive-definite, symmetric, and trace-class operator. Let $\mathcal{U}$ be a separable Hilbert space equipped with an orthogonal basis $\{e_k\}_{k=1}^{\infty}$. Assume the prior covariance operator admits the spectral decomposition 
$$\mathcal{C}_0 = \sum_{k=1}^{\infty} \lambda_k \, e_k \otimes e_k$$ 
with $\lambda_1 \geq \lambda_2 \geq \cdots > 0.$ Define $\mathcal{U}^2$ as the Hilbert scale induced by $\mathcal{C}_0$ (see \cite{Engl1996Book} or Subsection 3.1 
for details). Introduce the mapping 
$$f(S; \theta) := \big( f_m(S; \theta_1), \, \theta_2 \big) \in \mathcal{V} := \mathcal{U}^2 \times \mathbb{R}^{N_c}.$$  
Under Assumption 2.7 
in the main text, we know that the parameter $\theta_2 \in \big[ \theta_{\text{min}}, \theta_{\text{max}} \big]^{N_c}$ almost surely, where $\theta_{\text{min}}, \theta_{\text{max}} \in \mathbb{R}$ ($\theta_{\text{min}} < \theta_{\text{max}}$) are fixed constants. The data-dependent learned prior measure is specified as 
$$\mathbb{P}_S^{\theta} := \mathcal{N} \Big( f_m(S; \theta_1), \, \mathcal{C}_0(\theta_2) \Big),$$ 
where the covariance operator  
$$\mathcal{C}_0(\theta_2) = \sum_{k=1}^{N_c} e^{\theta_{2k}} \, e_k \otimes e_k + \sum_{k=N_c+1}^{\infty} \lambda_k \, e_k \otimes e_k$$
with $N_c \in \mathbb{N}^+$ is a predefined constant. For the mean function $f_m(S; \theta_1)$, we constrain all admissible $f(S; \theta_1)$ to the ball $B_{\mathcal{U}^2}(R_u) \subset \mathcal{U}^2$ defined by 
\begin{align}\label{BallConditionA}
	B_{\mathcal{U}^2}(R_u) := \{ f\in\mathcal{U}^2:\, \|\mathcal{C}_0^{-1}f\|_{\mathcal{U}} \leq R_u \},
\end{align}
where $R_u > 0$ being a sufficiently large prespecified constant.}
For the above condition (\ref{BallConditionA}), similar requirements are employed in classical investigations of inverse problems of PDEs. When studying inverse problems of PDEs using regularization methods, it is important for us to derive conditional stability estimates, which require the introduction of an admissible set \citep{Engl1996Book}. The admissible set is usually defined as a ball in some appropriate Banach space. For example, in the investigation of the backward diffusion problem, the admissible set is defined as a ball in the Sobolev space $H^{2\epsilon}$ (for any fixed $\epsilon > 0$) with a fixed finite radius \citep{Jingzhi2009CPAA}. In studies of the inverse coefficient for steady-state Darcy flow equations, the admissible set is defined as a ball in the $C^1$ space with a finite radius \citep{Richter1981IP,Vollmer2013IP}. In some recent investigations on statistical inverse problems, similar admissible sets have also been introduced; see Section 2 of \cite{Nickl2022LectureNotes}, for example. 

{\color{black}From a practical point of view, it is reasonable to restrict the parameter of interest within a ball of some Banach space, since such a priori information can often be obtained in practical problems. In some practical applications of PDE inverse problems (such as seismic imaging and medical imaging), the function parameter $u$ represents the sound speed of a medium (e.g., rocks or biological tissues). The sound speeds of rocks and biological tissues are typically determined through laboratory experiments. Furthermore, the types of rocks and tissues are finite and may be known to experts prior to inversion. In these applications, we aim to determine the sound speed distribution, which then allows us to identify the medium. Therefore, for many PDE inverse problems, we may know that the function parameter lies within some fixed-radius ball. Without some restricted ball conditions, the stability (even the weak log-type stability) for solving the inverse problems of PDEs usually cannot be obtained, which leads to an inverse problem that may not be numerically solved in a stable manner. 

A limitation of the current setting concerns the employed Hilbert scale norm. In practical applications, the unknown function parameter may not be sufficiently smooth and may only lie in a ball defined by a weaker norm. Current state-of-the-art theoretical work \cite{Nickl2020JUQ} also requires higher regularity of the true parameter than some practical applications typically demand. Reducing the regularity requirements of the Hilbert scale norm presents an interesting problem worthy of investigation.
}

{\color{black}Under these assumptions, we have 
\begin{align}\label{FormPsiA}
	\begin{split}
		\Psi(S,S',u,\theta) = & \exp\bigg( 
		\langle \mathcal{C}_0(\theta_2)^{-1}(f_m(S;\theta_1) - f_m(S';\theta_1)), u - f_m(S';\theta_1) \rangle_{\mathcal{U}} \\
		& \quad\quad\quad
		- \frac{1}{2}\|\mathcal{C}_0(\theta_2)^{-1/2}(f_m(S;\theta_1) - f_m(S';\theta_1))\|_{\mathcal{U}}^2
		\bigg)
	\end{split}
\end{align}
based on Theorem 2.23 in \cite{Prato2014Book}. 
For a small real number  $0 < \epsilon < \min\left\{ \frac{1}{4}e^{-\theta_{\text{max}}}, \frac{1}{4\lambda_1} \right\}$, we have the following estimate
\begin{align*}
		\Psi(S,S',u,\theta)^2 \leq & T_1 \cdot T_2,
\end{align*}
where 
\begin{align*}
	T_1 = \exp\left( \frac{1}{\epsilon}\|\mathcal{C}_0(\theta_2)^{-1}(f_m(S;\theta_1) - f(S';\theta_1))\|_{\mathcal{U}}^2 \right), \quad
	T_2 = \exp\left( \epsilon\|u - f_m(S';\theta_1)\|_{\mathcal{U}}^2 \right).
\end{align*}
Obviously, we know that 
\begin{align*}
		\Psi_E \leq & \mathbb{E}_{S'\sim\mathbb{D}_{\mathcal{T}}}
		\mathbb{E}_{S\sim\mathbb{D}_{\mathcal{T}}}\mathbb{E}_{u\sim\mathbb{P}_{S'}^{\theta}} T_1 \cdot T_2 \\
		\leq &  \mathbb{E}_{S'\sim\mathbb{D}_{\mathcal{T}}}
		\mathbb{E}_{S\sim\mathbb{D}_{\mathcal{T}}}T_1 \cdot \exp\left(2\epsilon\|f(S';\theta_1) - f_m(S;\theta_1)\|_{\mathcal{U}}^2\right)\mathbb{E}_{u\sim\mathbb{P}_0^{\theta_2}}
		\exp\left( 2\epsilon\|u\|_{\mathcal{U}}^2 \right) \\
		\leq & e^{8\epsilon R_u^2}\left[ \det\left( \text{I} - 4\epsilon \mathcal{C}_0(\theta_2) \right) \right]^{-\frac{1}{2}}\mathbb{E}_{S'\sim\mathbb{D}_{\mathcal{T}}}
		\mathbb{E}_{S\sim\mathbb{D}_{\mathcal{T}}}T_1 \\
		= & e^{8\epsilon R_u^2}\prod_{k=1}^{N_c}(1-4\epsilon e^{\theta_{2k}})^{-\frac{1}{2}}\prod_{k=N_c+1}^{\infty}(1-4\epsilon \lambda_k)^{-\frac{1}{2}}
		\mathbb{E}_{S'\sim\mathbb{D}_{\mathcal{T}}}
		\mathbb{E}_{S\sim\mathbb{D}_{\mathcal{T}}}T_1, \\
		\leq & e^{8\epsilon R_u^2} \frac{N_c}{\sqrt{1-4\epsilon e^{\theta_{\text{max}}}}} \prod_{k=N_c+1}^{\infty}(1-4\epsilon \lambda_k)^{-\frac{1}{2}}
		\mathbb{E}_{S'\sim\mathbb{D}_{\mathcal{T}}}
		\mathbb{E}_{S\sim\mathbb{D}_{\mathcal{T}}}T_1,
\end{align*}
and where $\mathbb{P}_0^{\theta_2} := \mathcal{N}(0,\mathcal{C}_0(\theta_2))$.
Remember the bounded condition of $f_m(S;\theta_1)$ for any $S$ and $\theta_1$, we finally obtain 
\begin{align}\label{EstimatePsiEA}
	\Psi_E &\leq \exp\left(\frac{(4+8\epsilon^2)R_u^2}{\epsilon}\right)\frac{N_c}{\sqrt{1-4\epsilon e^{\theta_{\text{max}}}}}
	\prod_{k=N_c+1}^{\infty}(1-4\epsilon \lambda_k)^{-\frac{1}{2}} < +\infty, 
\end{align}
with $0 < \epsilon < \min\left\{ \frac{1}{4}e^{-\theta_{\text{max}}}, \frac{1}{4\lambda_1} \right\}$. }

\begin{remark}
	Since we have not assumed that the parameter $u$ is uniformly bounded, the term $\Psi(S,S',u,\theta)$ is not generally bounded. Therefore, we cannot verify the $\alpha$-differentially private condition under the current setting. The right-hand side of inequality (\ref{EstimatePsiEA}) contains a term $e^{\frac{(4+8\epsilon^2)R_u^2}{\epsilon}}$, which can become very large even for moderate values of $R_u$. However, it is only the term $\ln\Psi_E$ that appears in the generalization bound and scales like $R_u^2$.
\end{remark}

{\color{black}
\subsection{Discussion of Assumptions 3.1}\label{AppGasussianPrior} 
In Assumptions 3.1, 
we assumed the base prior measure to be a Gaussian measure $\mathbb{P}_{S}^{\theta} = \mathcal{N}(f_m(S;\theta_1),\mathcal{C}_0(\theta_2))$ defined on the separable Hilbert space $\mathcal{U}$. We notice that the study \citep{Mora2025CMAME} focuses on constructing a systematic method to learn an operator (typically a system of partial differential equations) by integrating Gaussian process regression with neural operator methods. In our view, the approach developed in \cite{Mora2025CMAME} can conceptually be adapted to our setting. Moreover, the techniques in \cite{Mora2025CMAME} for handling high-resolution data will be crucial for our purposes.  

Gaussian measures defined on separable Hilbert spaces and Gaussian process theory are deeply interrelated. Any Gaussian random variable in a Banach space $\mathcal{H}$ can be viewed as a Gaussian process indexed by $\mathcal{H}^*$ (the dual space of $\mathcal{H}$). Conversely, Gaussian processes can often be realized as $\mathcal{H}$-valued random variables for some appropriate $\mathcal{H}$. Thus, the theories of Gaussian measures and processes mutually inform each other, and tools from both fields prove useful when appropriately applied. For detailed illustrations, see the books \cite{Nickl2022LectureNotes,Nickl2016Book}. For a practical introduction to Gaussian processes, we refer to the excellent book \cite{Rasmussen2006Book}.

For the current work, we adopt the Gaussian measure perspective, which may align with the Bayesian approach for inverse problems of PDEs \citep{Dashti2017}. Building on the Gaussian measure perspective and finite element discretization, a systematic discretization framework has been developed in \cite{Tan2013SISC,Petra2014SISC}, and \cite{Thanh2016IPI}. By our understanding, this framework is deeply related to the Gaussian process perspective. Unlike Gaussian process methods, Gaussian measure methods discretize functions using numerical PDE techniques. Sampling from a Gaussian measure is transformed into solving an elliptic partial differential equation. In the Gaussian process setting, various methods have been developed to enhance scalability with respect to high-dimensional discrete points. For the numerical methods illustrated in \cite{Tan2013SISC,Petra2014SISC}, and \cite{Thanh2016IPI}, the scalability challenge (inversion of large matrices) is transformed into the scalability of PDE solvers, which have been widely investigated.  

To learn the covariance operator, we adopt the methods described in \cite{Pinski2015SIAMMA,Pinski2015SISC}, and \cite{Jia2023JMLR}, which are essentially based on the eigensystem of the covariance operator $\mathcal{C}_0$. The covariance operator is typically represented as an integral operator:  
\begin{align*}
	(\mathcal{C}_0u)(x) = \int_{\Omega}c(x,y)u(y)\,\mathrm{d}y.
\end{align*}  
In Gaussian process methods, researchers often focus on the integral kernel function $c(x, y)$, which becomes a large matrix when dealing with numerous discrete points. For example, on a $100 \times 100$ mesh, the matrix size is $10^4 \times 10^4$, which can be reduced by various methods \citep{Rasmussen2006Book}. In Gaussian measure methods, we compute the dominant eigenvalues and corresponding eigenvectors of the covariance operator using PDE techniques. Since the eigenvalues of $\mathcal{C}_0$ frequently decay rapidly, a few eigenvalues and eigenvectors suffice to provide a satisfactory approximation. See the work \cite{Tan2013SISC} for details.  

In summary, at an abstract level, the method developed in \cite{Mora2025CMAME} can be adapted for our work. In fact, the base prior measure can be constructed more cleverly for specific inverse problems. For instance, when considering inverse scattering problems, direct sampling methods could be combined with a Gaussian prior measure \citep{Li2020SIAMImaging}, which could enhance the method developed in this work. This represents one of our future research directions.
}
\subsection{Discussion of Bounded Loss}\label{AppA3}

Before further discussions, we note that the following corollary, which concerns the generalization bound under a bounded loss assumption and data-dependent prior, holds true.

\begin{corollary}\label{PACBoundsBoundedLoss2}
	Assume that all of the assumptions in Theorem 2.4 
    and Assumption 2.6 
	hold true. In addition, we assume that the loss function $\ell(u,\bm{z})$ is bounded in $[a,b]$ 
	with $-\infty < a < b < +\infty$. Let us denote $\bar{m} = \left( \frac{1}{n}\sum_{i=1}^{n}\frac{1}{m_i} \right)^{-1}.$
	For any confidence level $\delta\in(0,1]$, $\gamma\geq2\sqrt{n}$, and $\lambda\geq2\sqrt{n}$, we have 
	\begin{align*}
		\mathbb{P}\Bigg( \forall \mathscr{Q}\in\mathcal{P}(\Theta), \,\,
		\mathcal{L}(\mathscr{Q}, \mathcal{T}) \leq & \, \hat{\mathcal{L}}(\mathscr{Q},S_1,\ldots,S_n) + 
		\left( \frac{1}{\lambda} + \frac{1}{\gamma} \right)D_{\text{KL}}(\mathscr{Q} || \mathscr{P})  \\
		& \, 
		+ \frac{1}{\gamma}\sum_{i=1}^{n}\mathbb{E}_{\theta\sim\mathscr{Q}}D_{\text{KL}}
		\big(\mathbb{Q}(S_i,\mathbb{P}_{S_i}^{\theta}) || \mathbb{P}_{S_i}^{\theta}\big) \\
		& \, + \frac{n}{2\gamma}\ln\Psi_E
		+ \left( \frac{\gamma}{4n\bar{m}} + \frac{\lambda}{8n} \right)(b-a)^2 
		+ \frac{1}{\sqrt{n}}\ln\frac{1}{\delta} \Bigg) \geq 1-\delta.
	\end{align*}
\end{corollary}
\begin{proof}
	Since $\mathbb{E}\mathbb{E}_{\theta\sim\mathscr{P}}\cdot = \mathbb{E}_{\theta\sim\mathscr{P}}\mathbb{E}\cdot$ 
	(The expectation $\mathbb{E}$ is defined as in Theorem 2.4 
    of the main text), 
	the estimates of $e^{2\sqrt{n}\Pi^2(\lambda)}$ will be the same as for the data-independent case. 
	In the following, we concentrate on the estimates of $e^{2\sqrt{n}\Pi^1(\gamma)}$. 
	For $\gamma\geq2\sqrt{n}$, simple calculations yield 
	\begin{align}\label{DataDependentBoundedLoss1}
		\begin{split}
			\Big(\mathbb{E}e^{2\sqrt{n}\Pi^1(\gamma)}\Big)^{1/2} \leq & (\mathbb{E}T_1)^{\frac{\sqrt{n}}{\gamma}},
		\end{split}
	\end{align}
	where
	\begin{align*}
		T_1 = \mathbb{E}_{\theta\sim\mathscr{P}}\prod_{i=1}^{n}\mathbb{E}_{u_i\sim\mathbb{P}_{S_i}^{\theta}}
		\exp\left( 
		\frac{\gamma}{nm_i}\sum_{j=1}^{m_i}\left[ 
		\mathbb{E}_{z\sim\mathbb{D}_{i}}\ell(u_i, z) -  \ell(u_i, z_{ij})
		\right]\right).
	\end{align*}
	For the term $\mathbb{E}T_1$ in inequality (\ref{DataDependentBoundedLoss1}), we have 
	\begin{align*}
			& \mathbb{E}_{\theta\sim\mathscr{P}}\prod_{i=1}^{n}\mathbb{E}_{S_i\sim\mathbb{D}_{\mathcal{T}}}
			\mathbb{E}_{u_i\sim\mathbb{P}_{S_i}^{\theta}}
			\!\!\exp\!\!\left( \!
			\frac{\gamma}{nm_i} \!\sum_{j=1}^{m_i}\!\left[ 
			\mathbb{E}_{z\sim\mathbb{D}_{i}}\ell(u_i, z) \!- \! \ell(u_i, z_{ij})
			\right]\!\!\right)  \\
			= & \mathbb{E}_{\theta\sim\mathscr{P}}\prod_{i=1}^{n}\mathbb{E}_{S_i^{'}\sim\mathbb{D}_{\mathcal{T}}}
			\mathbb{E}_{S_i\sim\mathbb{D}_{\mathcal{T}}}
			\mathbb{E}_{u_i\sim\mathbb{P}_{S_i}^{\theta}}
			\!\!\exp\!\!\left( \!
			\frac{\gamma}{nm_i} \!\sum_{j=1}^{m_i}\!\left[ 
			\mathbb{E}_{z\sim\mathbb{D}_{i}}\ell(u_i, z) \!- \! \ell(u_i, z_{ij})
			\right]\!\!\right) \\
			\leq & \mathbb{E}_{\theta\sim\mathscr{P}}\!\prod_{i=1}^{n}\!
			\mathbb{E}_{S_i^{'}\sim\mathbb{D}_{\mathcal{T}}}\!
			\mathbb{E}_{u_i\sim\mathbb{P}_{S_i^{'}}^{\theta}}\!
			\mathbb{E}_{S_i\sim\mathbb{D}_{\mathcal{T}}}\!
			\Psi(S_i, S_i^{'},u_i,\theta)\!
			\exp\!\!\left( \!
			\frac{\gamma}{nm_i} \!\sum_{j=1}^{m_i}\!\left[ 
			\mathbb{E}_{z\sim\mathbb{D}_{i}}\ell(u_i, z) \!- \! \ell(u_i, z_{ij})
			\right]\!\!\right) \\
			\leq & \mathbb{E}_{\theta\sim\mathscr{P}}\prod_{i=1}^{n}
			\left\{\mathbb{E}_{S_i^{'}\sim\mathbb{D}_{\mathcal{T}}}
			\mathbb{E}_{u_i\sim\mathbb{P}_{S_i^{'}}^{\theta}}
			\mathbb{E}_{S_i\sim\mathbb{D}_{\mathcal{T}}}\Psi(S_i, S_i^{'},u_i,\theta)^2\right\}^{1/2} \\
			& \qquad\qquad
			\left\{
			\mathbb{E}_{u_i\sim\mathbb{P}_{S_i^{'}}^{\theta}}
			\mathbb{E}_{S_i\sim\mathbb{D}_{\mathcal{T}}}
			\exp\left( \!
			\frac{2\gamma}{nm_i} \!\sum_{j=1}^{m_i}\!\left[ 
			\mathbb{E}_{z\sim\mathbb{D}_{i}}\ell(u_i, z) \!- \! \ell(u_i, z_{ij})
			\right]\right)
			\right\}^{1/2},
	\end{align*}
	which implies
	\begin{align}\label{DataDependentBoundedLoss2}
		\begin{split}
			\mathbb{E}T_1 \leq & \Psi_E^{\frac{n}{2}}
			e^{\frac{\gamma^2}{4n \bar{m}}(b-a)^2} 
		\end{split}
	\end{align}
	with $\Psi_E$ defined as in Assumption 2.6 
    of the main text. 
	Plugging estimate (\ref{DataDependentBoundedLoss2}) into estimate (\ref{DataDependentBoundedLoss1}), we find that 
	\begin{align}\label{DataDependentBoundedLoss3}
		\begin{split}
			\Big(\mathbb{E}e^{2\sqrt{n}\Pi^1(\gamma)}\Big)^{1/2} \leq & \Psi_E^{\frac{n\sqrt{n}}{2\gamma}}
			e^{\frac{\gamma}{4\sqrt{n}\bar{m}}(b-a)^2}.
		\end{split}
	\end{align}
	Combining estimates on the term $\mathbb{E}e^{2\sqrt{n}\Pi^2(\lambda)}$ 
	(not given here since it is similar to the data-independent case), we obtain
	\begin{align}\label{DataDependentBoundedLoss4}
		\begin{split}
			\mathbb{E}\left[ e^{(\sqrt{n}\Pi^1(\gamma) + \sqrt{n}\Pi^2(\lambda))} \right] \leq 
			\Psi_E^{\frac{n\sqrt{n}}{2\gamma}}
			\exp\left( \left( \frac{\gamma}{4\sqrt{n}\bar{m}} + \frac{\lambda}{8\sqrt{n}} \right) (b-a)^2 \right).
		\end{split}
	\end{align}
	Combining estimate (\ref{DataDependentBoundedLoss4}) with the results given in Theorem 2.4 
    of the main text, 
	we complete the proof.
\end{proof}

Inspired by investigations on the generalized Bayes' method, where the likelihood function is not directly derived from some probability distributions \citep{Germain2016NIPS,Guedj2019ArXiv}, we introduce the following truncated loss function:
\begin{align}\label{TruncatedLossDef}
	\ell(u, z) := \left\{\begin{aligned} 
		&\frac{1}{2}\|\mathcal{L}_{x}\mathcal{G}(u) - y\|_{\mathcal{H}}^2,  \quad &
		\text{if  }\frac{1}{2}\|\mathcal{L}_{x}\mathcal{G}(u) - y\|_{\mathcal{H}}^2 \leq N, \\
		&N, & \text{if  }\frac{1}{2}\|\mathcal{L}_{x}\mathcal{G}(u) - y\|_{\mathcal{H}}^2 > N.
	\end{aligned}\right.
\end{align}
Here, $N$ is a prespecified positive constant. With this assumption, the loss function is contained within the interval $[0, N]$, which allows us to apply Theorems \ref{PACBoundsBoundedLoss1} and \ref{PACBoundsBoundedLoss2} by replacing $(b-a)^2$ with $N^2$.
Introducing such a truncation may seem unnatural at first glance, especially from the perspective of Bayesian inverse methods. However, in the following, we provide an explanation that illustrates the appropriateness of introducing such a truncated loss function.

As discussed in Subsection \ref{AppA2}, an admissible set—usually defined as a ball in some Banach space—will be introduced when we discuss the conditional stability estimates of inverse problems \citep{Jingzhi2009CPAA,Richter1981IP,Vollmer2013IP}.
Based on this consideration, we assume that there exists a constant $R_u$ such that
\begin{align*}
	\supp\mathscr{E}\subset B_{\mathcal{U}}(R_u)
\end{align*}
with $\mathscr{E}\in\mathcal{P}(\mathcal{U})$ and 
$B_{\mathcal{U}}(R_u) := \{u\in\mathcal{U} :\, \|u\|_{\mathcal{U}} \leq R_u\}$.
Assume the forward operator $\mathcal{G}$ is continuous that satisfies 
\begin{align*}
	\|\mathcal{L}_{x}\mathcal{G}(u)\|_{\mathcal{H}} \leq & M(\|u\|_{\mathcal{U}}),  \\
	\|\mathcal{L}_{x}\mathcal{G}(u) - \mathcal{L}_{x}\mathcal{G}(v)\|_{\mathcal{H}} \leq &
	\tilde{M}(\|u\|_{\mathcal{U}}, \|v\|_{\mathcal{U}})\|u-v\|_{\mathcal{U}},
\end{align*}
where $x\in\mathcal{X}$, $M(\cdot)$ is a monotonic non-decreasing function, and 
$\tilde{M}(\cdot,\cdot)$ is a function monotonic non-decreasing separately in each argument. 
By our assumptions on $\mathscr{E}$, the parameter $u\in B_{\mathcal{U}}(R_u)$ almost surely, which indicates
\begin{align*}
		\|\mathcal{L}_{x}\mathcal{G}(u)\|_{\mathcal{H}} \leq & M(R_u), \\
		\|\mathcal{L}_{x}\mathcal{G}(u) - \mathcal{L}_{x}\mathcal{G}(v)\|_{\mathcal{H}} \leq & 
		\tilde{M}(R_u, R_u)\|u-v\|_{\mathcal{U}}
\end{align*}
for all $u,v\in B_{\mathcal{U}}(R_u)$. 

In the study of regularization methods for inverse problems, the noise is assumed to be bounded by a constant, known as the ``noise level'' \citep{Engl1996Book}. Intuitively, the noise level cannot be too large, or it would obscure the true measured data. Based on this consideration, we assume that the noise $\eta\in\mathcal{H}$ is bounded by 
$\mathcal{L}_{x}\mathcal{G}(u)$, i.e.,
\begin{align*}
	\|\eta\|_{\mathcal{H}} \leq C_{\eta}\|\mathcal{L}_{x}\mathcal{G}(u)\|_{\mathcal{H}}
\end{align*}
with $C_{\eta}$ be a positive constant. Usually, the constant $C_{\eta}$ is assumed to be smaller than $1$.
Let us define the loss function as follow:
\begin{align*}
	\ell(u, z) = \frac{1}{2}\|\mathcal{L}_{x}\mathcal{G}(u) - y\|_{\mathcal{H}}^2. 
\end{align*}
Then, we easily know that 
\begin{align*}
	\ell(u,z) \leq & \|\mathcal{L}_{x}\mathcal{G}(u) - \mathcal{L}_{x}\mathcal{G}(u^{\dag})\|_{\mathcal{H}}^2 
	+ \|\eta\|_{\mathcal{H}}^2 \\
	\leq & \tilde{M}(R_u, R_u)^2 \|u-u^{\dag}\|_{\mathcal{U}}^2 + C_{\eta}^2M(R_u)^2\|u^{\dag}\|_{\mathcal{U}}^2 \\
	\leq & \left( 4\tilde{M}(R_u, R_u)^2 + C_{\eta}^2M(R_u)^2 \right)R_u^2. 
\end{align*}
Hence, we can take the constant $N$ in formula (\ref{TruncatedLossDef}) as follow:
$$N:=\left( 4\tilde{M}(R_u, R_u)^2 + C_{\eta}^2M(R_u)^2 \right)R_u^2.$$
Under the current setting, the loss function may not have a probabilistic interpretation. 
Hence we can only formulate a generalized Bayes' formula that is well accepted in the studies of PAC-Bayesian learning theory. 

\subsection{Analysis of the Backward Diffusion and Darcy Flow Problems}\label{AppA4}
In this section, let us provide detailed illustrations of the two concrete inverse problems of PDEs that were mentioned in Subsection 3.2 
of the main text.

\textbf{The backward diffusion problem. }
We apply the theory to the backward diffusion problem, which is one of the most investigated inverse problems of PDEs \citep{Jia2016IP,Jia2018IPI,Stuart2010AN}. Let $\Omega\subset\mathbb{R}^{d} (d\leq 3)$ be a bounded open set with smooth boundary $\partial\Omega$. We define the Hilbert space $\mathcal{H}$ and the operator $A$ as follows:
\begin{align*}
	& \mathcal{H} = (L^{2}(\Omega), \langle\cdot,\cdot\rangle, \|\cdot\|), \quad
	A = -\Delta, \quad \mathcal{D}(A) = H^2(\Omega)\cap H_0^1(\Omega).
\end{align*}
Then the forward diffusion equation is an ordinary differential equation in $\mathcal{H}$:
\begin{align*}
	\frac{dv}{dt} + Av = 0, \quad v(0) = u.
\end{align*} 
\textbf{Elements of the backward diffusion problem: }
\begin{itemize}
	\item Datasets: the solution of the diffusion equation at time $T>0$, i.e., $v(T)$ is the measured data;
	\item Interested parameter: the initial function $u$ of the diffusion equation. 
\end{itemize}

In this section, we consider the functional data setting to demonstrate the flexibility of the theoretical framework (the finite discrete data case has been illustrated for the Darcy flow problem). However, we could alternatively consider the setting of discrete measured data, which was used for the numerical illustration.
By the operator semigroup theory, we know that the solution at time $T$ is $v(T) = \exp(-AT)u$, i.e., the forward operator 
$\mathcal{G}(u) = \exp(-AT)u$ and $\mathcal{L}_{x}:=\text{Id}$ is the identity operator from $\mathcal{H}$ to $\mathcal{H}$.
Let $\mathcal{U} = \mathcal{H}$, i.e., the interested parameter $u$ is assumed to belong to $\mathcal{H}$.
For $j=1,\ldots,m$, we have 
\begin{align}\label{BackdiffusionDef1}
	y_j = \mathcal{G}(u,x_j) + \eta_j =  \exp(-AT)u + \eta_j,
\end{align}
where $\{x_j\}_{j=1}^{m}$ are dummy variables. The above problem (\ref{BackdiffusionDef1}) can be written compactly as 
\begin{align}\label{BackdiffusionDef2}
	\bm{y} = \mathcal{L}_{\bm{x}}\mathcal{G}(u) + \bm{\eta} =  \exp(-AT)\bm{u} + \bm{\eta},
\end{align}
where $\bm{u} = (u_1,\ldots,u_{m})^T$ with $u_j=u$ for $j=1,\ldots,m$.

For this example, we take $\Gamma:=\tau^2\text{Id}$ with $\tau\in\mathbb{R}^{+}$, $\mathcal{C}_1=A^{-2}$, $\mathcal{C}_0=\lambda A^{-2}$ ($\lambda > 0$)
and $\alpha=0$. Concerned with $\mathcal{C}_0$, we have the following eigen-system decomposition:
\begin{align*}
	\mathcal{C}_0e_k = \lambda_k e_k, \quad \forall \, k=1,2,\ldots,
\end{align*}
where $\{\lambda_k\}_{k=1}^{\infty}$ are arranged in a descending order and $\{e_k\}_{k=1}^{\infty}$ are normalized eigenfunctions. 
According to the properties of the operator $A$, we know that $\lambda_k\asymp k^{-\frac{4}{d}}$, 
which yields $s_0 = s_1 = \frac{d}{4}$. In fact, for every $s > s_0$, we have 
\begin{align*}
	\text{Tr}(\mathcal{C}_0^{s}) = \sum_{k=1}^{\infty}\lambda_k^{s} \asymp \sum_{k=1}^{\infty}k^{-\frac{4s}{d}} < +\infty. 
\end{align*}
With these discussions, we can take $\frac{d}{4} < s=\tilde{s} < 1$ for the backward diffusion problem discussed here. 
In addition, the parameter $\alpha=0$ and $\mathcal{Y} = \mathcal{H}^{-s}$. {\color{black}According to Subsection 2.3 
of the main text, we can introduce the learnable parameter $\theta_2\in [\theta_{\text{min}}, \theta_{\text{max}}]^{N_c}$ almost surely and define the learnable covariance operator as follows:
\begin{align*}
	\mathcal{C}_0(\theta_2) = \sum_{k=1}^{N_c}e^{\theta_{2k}}e_k\otimes e_k + \sum_{k=N_c+1}^{\infty}\lambda_k e_k\otimes e_k,
\end{align*} 
where $N_c$ is a positive integer. }

Now we give estimates required in Assumptions 3.2 
of the main text for the current problem as follows:
\begin{align}\label{BackDiffusionInequalities}
	\begin{split}
		\|\mathcal{C}_1^{\frac{s}{2}}\Gamma^{-\frac{1}{2}}u\|_{\mathcal{H}} = & \tau^{-1}\|\mathcal{C}_1^{\frac{s}{2}}u\|_{\mathcal{H}}, \\
		\|\mathcal{C}_1^{-\frac{\rho}{2}}\Gamma^{\frac{1}{2}}u\|_{\mathcal{H}} = & \tau \|\mathcal{C}_1^{-\frac{\rho}{2}}u\|_{\mathcal{H}}, \\
		\|\mathcal{C}_1^{-\frac{s}{2}}\Gamma^{-\frac{1}{2}}e^{-AT}u\|_{\mathcal{H}} \leq & 
		\tau^{-1}\sup_k\left[ \lambda^{\frac{s}{2}}\lambda_k^{\frac{1}{2}-s} e^{-T\sqrt{\frac{\lambda}{\lambda_k}}} \right]\|u\|_{\mathcal{U}^{1-\tilde{s}}},  \\
		\|\mathcal{C}_1^{-\frac{s}{2}}\Gamma^{-\frac{1}{2}}e^{-AT}(u_1-u_2)\|_{\mathcal{H}} \leq & 
		\tau^{-1}\sup_k\left[ \lambda^{\frac{s}{2}}\lambda_k^{\frac{1}{2}-s} e^{-T\sqrt{\frac{\lambda}{\lambda_k}}} \right]\|u_1-u_2\|_{\mathcal{U}^{1-\tilde{s}}},
	\end{split}
\end{align}
which indicate that $C_1 = \tau^{-1}$, $C_2 = \tau$, and 
\begin{align*}
	M_1= M_2 = \tau^{-1}\sup_k\left[ \lambda^{\frac{s}{2}}\lambda_k^{\frac{1}{2}-s} e^{-T\sqrt{\frac{\lambda}{\lambda_k}}} \right]. 
\end{align*}
Inequalities (\ref{BackDiffusionInequalities}) are easily verified by employing the standard estimates of parabolic equations 
(see for example the Section 1.2 of \cite{Dashti2017} or the book \cite{Pazy1983Book}). 
For reader's convenience, we provide the proof of (\ref{BackDiffusionInequalities}) in Subsection \ref{AppB2}.

{\color{black}In order to employ Lemma 3.6 
of the main text, 
we need to verify the following condition 
$$\max\{ 3\tilde{\gamma}\tilde{C}^2M_2^2, 3\tilde{\gamma}^2M_1^2\text{Tr}(\mathcal{C}_1^s) \} \leq \min\left\{e^{-\theta_{\text{max}}},  \lambda_1^{-1}\right\}.$$ 
After a simple calculation similar to the proof of (\ref{BackDiffusionInequalities}), we find that
\begin{align}\label{AssumpC0Example1}
	\max\{ \tilde{\gamma}\lambda_1^{1+s}, \tilde{\gamma}^2\lambda_1\text{Tr}(\mathcal{C}_0), \tilde{\gamma}e^{(1+s)\theta_{\text{max}}}, \tilde{\gamma}^2e^{\theta_{\text{max}}}\text{Tr}(\mathcal{C}_0) \}
	\sup_{k}\left[ 
	\lambda_{k}^{1-s}\left( \frac{\lambda}{\lambda_k} \right)^{s}\!\! e^{-2T\sqrt{\frac{\lambda}{\lambda_k}}}
	\right] \!\leq \!\frac{\lambda^2\tau^2}{3}
\end{align}
ensures the above condition. Obviously, the above inequality holds true when we choose a large enough parameter $\lambda$. 
By assuming (\ref{AssumpC0Example1}) and employing estimates (\ref{BackDiffusionInequalities}), we find out the parameters
$s_{\text{I}}^2$ and $s_{\text{II}}^2$ in Lemma \ref{LemmaLinear-sIsII}(Lemma 3.6 
of the main text) are as follows:
\begin{align*}
		s_{\text{I}}^2 = & 
		\frac{\lambda_1^{2s}M_s^2 R_u^4}{4\lambda^{2s}\tau^4} + 
		\frac{2M_s \text{Tr}(\mathcal{C}_0^s)R_u^2 }{\tau^2\lambda^2}  
		 + \max\left\{ 
		\frac{\tau^2\lambda^s\ln3}{3M_s\text{Tr}(\mathcal{C}_0^s)\min\{\lambda_{N_c+1}, e^{\theta_{\text{max}}}\}}, \frac{2M_s\text{Tr}(\mathcal{C}_0^s)}{\tau^2\lambda^2}
		\right\}\text{Tr}(\mathcal{C}_0),  \nonumber
\end{align*}
\begin{align*}
		s_{\text{II}}^2 = &
		\frac{2\max\{\lambda_1^{2s},e^{2s\theta_{\text{max}}}\}(R_u^4 + \left[  N_ce^{\theta_{\text{max}}} + \text{Tr}(\mathcal{C}_0) \right]^2)M_s^2}{4\lambda^{2s}\tau^4} + 144R_u^4 e^{2(\theta_{\text{max}} - \theta_{\text{min}})} \\
		& + \frac{9\max\{\lambda_1^{2s},e^{2s\theta_{\text{max}}}\}\beta^2\left[ N_c e^{(\theta_{\text{max}} - \theta_{\text{min}})s} + \text{Tr}\left( \mathcal{C}_0^{s} \right) \right]^2M_s^2}{\lambda^{2s}\tau^4}\mathbb{E}_{(\mathbb{D},m)\sim\mathcal{T}}\left[\frac{1}{m^2}\right].
\end{align*}
where $M_s = \sup_k \left[ \lambda^{s}\lambda_k^{1-2s} e^{-2T\sqrt{\frac{\lambda}{\lambda_k}}} \right].$}
To maintain a clear illustration, we provide the calculation details in Subsection \ref{AppB2}.
Notice that for the backward diffusion problem, we could illustrate the auxiliary boundedness conditions required in Lemma \ref{LemmaLinear-sIsII} under the assumption that the number of measurements is fixed. Hence, we obtain estimates independent of $\tilde{\gamma}$ and $\tilde{\lambda}$. 
Using the above formula, we can immediately derive the PAC-Bayesian bounds given in Corollary \ref{PACBoundsUnoundedLossThmSubGaussian} and Corollary 2.10 
of the main text for the data-independent prior and data-dependent prior cases when $\gamma=n\bar{m}$ and $\lambda=n$ or $\gamma=n\sqrt{\bar{m}}$ and $\lambda=2\sqrt{n}$.

\textbf{The Darcy flow problem. }
Regarding the nonlinear forward operator case, we apply the general theory to a steady-state Darcy flow problem, which is a well-known example in the field of statistical inverse problems \citep{Dashti2017,Jia2022SINUM}.
Let $\Omega \subset \mathbb{R}^2$ be a bounded open set with a smooth boundary $\partial\Omega$, and define $\mathcal{U}$ as the Hilbert space $(L^2(\Omega), \langle \cdot, \cdot \rangle, \| \cdot \|)$.
The Darcy flow equation has the following form: 
\begin{align}\label{DarcyEqApp}
	\begin{split}
		-\nabla\cdot(e^u\nabla w) & = f \quad\text{in }\Omega, \\
		w & = 0 \quad\text{on }\partial\Omega, 
	\end{split}
\end{align}
where $f$ denotes the sources, and $e^{u(x)}$ describes the permeability of the porous medium.

\textbf{Elements of the inverse problems for Darcy flow: }
\begin{itemize}
	\item Datasets: At some discrete points $\{x_j\}_{j=1}^{m}$, we measure the value of the solution $w$ through 
	a measurement operator defined as $\mathcal{L}_{x_j}(w) = w(x_j)$ with $x_j\in\Omega$ for $j=1,\ldots,m$.
	\item Interested parameter: The function $u(x)$.
\end{itemize}

For this problem, the operator $\mathcal{G}$ is the solution operator of the Darcy flow. Hence, we have 
\begin{align*}
	y_j = \mathcal{L}_{x_j}\mathcal{G}(u) + \eta_j = \mathcal{L}_{x_j}(w) + \eta_j,
\end{align*}
which can be written compactly as 
\begin{align*}
	\bm{y} = \mathcal{L}_{\bm{x}}\mathcal{G}(u) + \bm{\eta}.
\end{align*}
Obviously, the space $\mathcal{X} = \mathbb{R}^2$, $\mathcal{Y} = \mathbb{R}$, $\bm{x}\in\mathbb{R}^{2m}$, and $\bm{y}\in\mathbb{R}^m$ in this example. Let us set $\Gamma := \tau^2$($\tau\in\mathbb{R}^+$) and the covariance operator of the prior measure $\mathcal{C}_0:=A^{-3}$, where $A:=\alpha(\text{Id}-\Delta)$($\alpha>0$) with the domain of $\Delta$ given by $\mathcal{D}(\Delta):=\left\{u\in H^2(\Omega):\frac{\partial u}{\partial \bm{n}} = 0\right\}$. Under this setting, we know that $s_0=\frac{1}{3}$, i.e., we can take $\frac{1}{3} < \tilde{s} < \frac{2}{3}$. Since $\mathcal{H}$ is a finite-dimensional space, we need not introduce the operator $\mathcal{C}_1$. Considering the Lemma 20 illustrated in \cite{Nickl2020JUQ}, we can easily derive the estimates required in Assumptions 3.2 
of the main text for the Darcy flow problem as follows:
\begin{align}\label{ConditionForDarcyFlow1points}
	\begin{split}
		|\mathcal{L}_{x}\mathcal{G}(u)| \leq & M_3(\|u\|_{\mathcal{U}^{1-\tilde{s}}}), \\
		|\mathcal{L}_{x}\mathcal{G}(u_1) - \mathcal{L}_{x}\mathcal{G}(u_2)| \leq & M_4(\|u_1\|_{\mathcal{U}^{1-\tilde{s}}},\|u_2\|_{\mathcal{U}^{1-\tilde{s}}})\|u_1-u_2\|_{\mathcal{U}^{1-\tilde{s}}},
	\end{split}
\end{align}
where we employed $\|u\|_{L^{\infty}} \leq M\|u\|_{\mathcal{U}^{1-\tilde{s}}}$ for some general constant $M$ when $\tilde{s} < \frac{2}{3}$. With the above estimates (\ref{ConditionForDarcyFlow1points}), we immediately obtain the PAC-Bayesian bounds presented in Corollary \ref{PACBoundsUnoundedLossThmSubGaussian} 
and Corollary 2.10 
of the main text for both the data-independent prior and data-dependent prior cases. Due to the techniques used for proving the Lemma 20 illustrated in \cite{Nickl2020JUQ}, we may hardly find out the explicit forms of the functions $M_3$ and $M_4$.

Alternatively, we choose another type of measurement operator defined as follow:
\begin{align*}
	\mathcal{L}_{x_j}(w) = \int_{\Omega}\frac{1}{2\pi\delta^2}e^{-\frac{1}{2\delta^2}\|x-x_j\|^2}w(x)dx,
\end{align*}
with $\delta > 0$ being a sufficiently small number and $x_j\in\Omega$ for $j=1,\ldots,m$.
Using similar derivations as for Theorem 17 given in \cite{Jia2022SINUM} (a more careful calculation of the constants is needed), we give estimates required in Assumptions 3.2 
of the main text for the Darcy flow problem as follows:
\begin{align}\label{ConditionForDarcyFlow1}
	\begin{split}
		|\mathcal{L}_{x}\mathcal{G}(u)| \leq & M_3(\|u\|_{\mathcal{U}^{1-\tilde{s}}}), \\
		|\mathcal{L}_{x}\mathcal{G}(u_1) - \mathcal{L}_{x}\mathcal{G}(u_2)| \leq & M_4(\|u_1\|_{\mathcal{U}^{1-\tilde{s}}},\|u_2\|_{\mathcal{U}^{1-\tilde{s}}})\|u_1-u_2\|_{\mathcal{U}^{1-\tilde{s}}},
	\end{split}
\end{align}
where 
\begin{align*}
	M_3(\|u\|_{\mathcal{H}^{1-\tilde{s}}}) = & \frac{2d_{\Omega}^2}{\pi\delta}\exp\left(\|u\|_{\mathcal{U}^{1-\tilde{s}}}\right)\|f\|_{\mathcal{U}}, \\
	M_4(\|u_1\|_{\mathcal{H}^{1-\tilde{s}}},\|u_2\|_{\mathcal{U}^{1-\tilde{s}}}) = & 
	\frac{2d_{\Omega}^2}{\pi\delta}\exp\left( \|u_1\|_{\mathcal{U}^{1-\tilde{s}}} + \|u_2\|_{\mathcal{U}^{1-\tilde{s}}} \right)
	\|f\|_{\mathcal{U}}. 
\end{align*}
Here $d_{\Omega}$ denotes the diameter of the domain $\Omega$. 
With estimates (\ref{ConditionForDarcyFlow1}) at hand, we find out the parameters $s_{\text{I}}^2$ and $s_{\text{II}}^2$ 
in Lemma 3.7 
are as follows:
\begin{align}\label{example2est}
	\begin{split}
		s_{\text{I}}^2 = & \frac{100d_{\Omega}^8}{\pi^4\delta^4}e^{8R_u} + \frac{4d_{\Omega}^4}{\pi^2\delta^2}e^{2R_u}, \quad
		s_{\text{II}}^2 = \frac{400d_{\Omega}^8}{\pi^4\delta^4}e^{4R_u}.  
	\end{split}
\end{align}
Relying on the above formulas (\ref{example2est}), we immediately obtain the PAC-Bayesian bounds presented in Theorems \ref{PACBoundsUnoundedLossThmSubGaussian} and Corollary 2.10 
of the main text for both the data-independent prior and data-dependent prior cases. 

As a preliminary study, our main aim is to develop a general theoretical framework that can yield explicit bounds. For deriving sharper estimates, more detailed structural information of the PDEs is expected to be utilized. This suggests that different techniques for estimation may need to be developed for inverse problems associated with various PDEs.
Here, we provide two typical examples that satisfy our general formulation. It is worth mentioning that our general formulation can also be applied to study other problems, such as the inverse problems of fractional differential equations  \citep{Jia2017JDE,Jia2018IPI,Jin2015IP}, or the inverse medium scattering problems \citep{Bao2015IP,Jia2019IP,Jia2018JFA}. 

\subsection{Details of Constructing Learning Algorithm}\label{AppA5}
Here, let us provide more detailed discussions on constructing learning algorithms based on the PAC-Bayesian generalization bound derived in the main text. Restricted to finite-dimensional space, the following illustrations obviously hold true as illustrated in \cite{Rothfuss2021PMLR} for machine learning problems. For the inverse problems of PDEs, we need some auxiliary illustrations in infinite-dimensional space. In this part, we choose the parameter $\gamma = n\beta$, where $\beta$ is the parameter appeared in the generalized Bayes' formula given in Theorem 3.3 
of the main text. In the following, we denote $Z_m = Z_m(S, \mathbb{P}_{S}^{\theta})$, since $Z_m$ depends on the dataset and the prior measure. With this choice of the base posterior measure, the Kullback-Leibler (KL) divergence term of the base prior and posterior measures in PAC-Bayesian bounds (e.g., estimates in Corollary \ref{PACBoundsBoundedLoss2} and Corollary 2.10 
of the main text) can be further reduced as follows:
\begin{align}\label{AlgDLReducedA}
	\begin{split}
		D_{\text{KL}}(\mathbb{Q}(S_i,\mathbb{P}_{S_i}^{\theta}) || \mathbb{P}_{S_i}^{\theta}) = &
		\int_{\mathcal{U}}\left( \ln\frac{d\mathbb{Q}(S_i,\mathbb{P}_{S_i}^{\theta})}{d\mathbb{P}_{S_i}^{\theta}} \right)d\mathbb{Q}(S_i,\mathbb{P}_{S_i}^{\theta})  \\
		= & -\ln Z_{m_i}(S_i, \mathbb{P}_{S_i}^{\theta}) - \int_{\mathcal{U}} \frac{\beta}{m_i}\sum_{j=1}^{m_i}\Phi(u;z_{ij}) d\mathbb{Q}(S_i,\mathbb{P}_{S_i}^{\theta}) \\
		= & -\ln Z_{m_i}(S_i, \mathbb{P}_{S_i}^{\theta}) - \mathbb{E}_{u\sim\mathbb{Q}(S_i,\mathbb{P}_{S_i}^{\theta})}\bigg[\frac{\beta}{m_i}\sum_{j=1}^{m_i}\Phi(u;z_{ij})\bigg].
	\end{split}
\end{align}

Choosing the loss function $\ell(u;z)$ to be the potential function $\Phi(u;z)$ and using (\ref{AlgDLReducedA}), 
we immediately find 
\begin{align}\label{AlgDKSumA}
	\begin{split}
		\frac{1}{n\beta}\sum_{i=1}^{n}\mathbb{E}_{\theta\sim\mathscr{Q}}D_{\text{KL}}(\mathbb{Q}(S_i,\mathbb{P}_{S_i}^{\theta}) || \mathbb{P}_{S_i}^{\theta}) = & \,-\frac{1}{n\beta}\sum_{i=1}^{n}\mathbb{E}_{\theta\sim\mathscr{Q}}
		\ln Z_{m_i}(S_i,\mathbb{P}_{S_i}^{\theta})  \\
		&\, - \hat{\mathcal{L}}(\mathscr{Q},S_1,\ldots,S_n).
	\end{split}
\end{align}
Plugging the equality (\ref{AlgDKSumA}) into the PAC-Bayes bounds, e.g., the estimate given in Corollary 2.10 
of the main text, we obtain that
\begin{align}\label{AlgBoundA}
	\begin{split}
		\mathcal{L}(\mathscr{Q},\mathcal{T}) \leq & - \frac{1}{n\beta}\sum_{i=1}^{n}\mathbb{E}_{\theta\sim\mathscr{Q}}\ln 
		Z_{m_i}(S_i, \mathbb{P}_{S_i}^{\theta}) + \frac{\lambda + n\beta}{\lambda n \beta}D_{\text{KL}}(\mathscr{Q}||\mathscr{P}) \\
		& + \frac{1}{2\beta}\ln\Psi_{E} + \frac{\beta s_{\text{I}}^2}{\bar{m}} + \frac{\lambda s_{\text{II}}^2}{2n} + \frac{1}{\sqrt{n}}\ln\frac{1}{\delta}
	\end{split}
\end{align}
holds uniformly with probability $1-\delta$.
To minimize the right-hand side of (\ref{AlgBoundA}), we need to solve the following optimization problem
\begin{align}\label{AlgOptimalProblemA}
	\argmin_{\mathscr{Q}\in\mathcal{P}(\Theta)}\mathbb{E}_{\theta\sim\mathscr{Q}}\bigg[ 
	-\frac{\lambda}{\lambda + n\beta}\sum_{i=1}^{n}\ln Z_{m_i}(S_i, \mathbb{P}_{S_i}^{\theta}) 
	\bigg] + D_{\text{KL}}(\mathscr{Q}||\mathscr{P}).
\end{align} 
In addressing this optimization problem, we present Theorem 4.1 
in the main text, which provides an analytic formula for the optimal measure $\mathscr{Q}$. In the following remark, we offer a discussion on the utility of simplifying the problem.

\begin{remark}\label{TwoLevelRemarkA}
	Inspired by the work of \cite{Rothfuss2021PMLR} for finite-dimensional machine learning model, we choose the base posterior measure to be the posterior measure given by Bayes' formula. Relying on this, we obtain the hyper-posterior measure analytically. Generally, the bounds we derived in the present work (e.g., estimates in Corollaries \ref{PACBoundsBoundedLoss2} and 2.10 
    of the main text) yield an optimization problem as follow:
	\begin{align*}
		\argmin_{\mathscr{Q}\in\mathcal{P}(\Theta)} \,\hat{\mathcal{L}}(\mathscr{Q}, S_1,\ldots,S_n) \!+\! 
		\frac{\lambda + n\beta}{\lambda n \beta}D_{\text{KL}}(\mathscr{Q}||\mathscr{P}) \!+\! \frac{1}{n\beta}\sum_{i=1}^{n}\mathbb{E}_{\theta\sim\mathscr{Q}}
		D_{KL}(\mathbb{Q}(S_i,\mathbb{P}_{S_i}^{\theta})||\mathbb{P}_{S_i}^{\theta}),
	\end{align*}
	where the base posterior measure $\mathbb{Q}(S_i,\mathbb{P}_{S_i}^{\theta})$ is not necessarily provided by the Bayes' formula. 
	Solving an optimization problem like the above one turns into a difficult two-level optimization problem 
	\citep{Amit2018ICML,Pentina2014ICML} which can hardly be used for the inverse problems of PDEs due to the large computational complexity for solving PDEs. 
\end{remark}

With the formula (13) 
in the main text, we are ready to extract information about the parameter $\theta$ from the hyper-posterior measure $\mathscr{Q}$. There are three types of methods listed as follows:
\begin{itemize}
	\item Evaluate the maximum a posteriori estimator using gradient-based optimization algorithms \citep{Bottou2018SIAMReview,Kazufumi2015Book}; 
	\item Sampling from the hyper-posterior measure can be performed using Markov chain Monte Carlo (MCMC) or sequential Monte Carlo (SMC) algorithms. For the preconditioned Crank-Nicolson algorithm, we refer to \citep{Cotter2013SS,Dashti2017}, while implementations of SMC algorithms are discussed in \citep{Beskos2015SC,Lu2024arXiv}.
	\item Several approximate sampling algorithms are available, including variational inference methods. Notable examples include the mean-field variational inference approach \citep{Jia2021SISC}, Stein variational gradient descent \citep{Jin2010JCP,Liu2016NIPS}, and normalizing flow aided variational inference \citep{Premchandar2023NoticeAMS,Zhao2024FNF}.
\end{itemize}

Notice that the log-likelihood function of the hyper-poserior measure given in Theorem 4.1 
of the main text is defined as follows:
\begin{align}\label{logsumexp1A}
	\sum_{i=1}^{n}\ln Z_{m_i}(S_i,\mathbb{P}_{S_i}^{\theta}) = \sum_{i=1}^{n}\ln\int_{\mathcal{U}}\exp\bigg(
	-\frac{\beta}{m_i}\sum_{j=1}^{m_i}\Phi(u;z_{ij})
	\bigg) \mathbb{P}_{S_i}^{\theta}(du).
\end{align}

For the backward diffusion problem considered in Subsection \ref{AppA4}, it is possible to find the explicit formula of (\ref{logsumexp1A}). However, for general inverse problems of PDEs, such as the Darcy flow problem illustrated in Subsection \ref{AppA4}, it is generally not possible to calculate explicitly.
Generally, formula (\ref{logsumexp1A}) can be calculated by
\begin{align}\label{logsumexp2A}
	\begin{split}
		\sum_{i=1}^{n}\ln Z_{m_i}(S_i,\mathbb{P}_{S_i}^{\theta}) \approx \sum_{i=1}^{n}\ln\left[\frac{1}{L}\sum_{\ell=1}^{L}\exp\bigg(
		-\frac{\beta}{m_i}\sum_{j=1}^{m_i}\Phi(u_{\ell}^{i};z_{ij})
		\bigg)\right],
	\end{split}
\end{align}
where $\{u_{\ell}^{i}\}_{\ell=1}^{L}$ are samples from the base prior measure $\mathbb{P}_{S_i}^{\theta}$ for $i=1,\ldots,n$.
By applying a simple reduction, we obtain the following result:
\begin{align}\label{logsumexp3Asupp}
	\begin{split}
		\sum_{i=1}^{n}\ln Z_{m_i}(S_i,\mathbb{P}_{S_i}^{\theta}) \approx \sum_{i=1}^{n}\ln\left[\sum_{\ell=1}^{L}\exp\bigg(
		-\frac{\beta}{m_i}\sum_{j=1}^{m_i}\Phi(u_{\ell}^{i};z_{ij})
		\bigg)\right] - n\ln L.
	\end{split}
\end{align}
To evaluate the right-hand side of (\ref{logsumexp3Asupp}), we need to calculate $nL$ forward PDEs, which is a time-consuming procedure (it is computationally more feasible compared with the two-level optimization problem illustrated in Remark \ref{TwoLevelRemarkA}). Hence, we focus on evaluating the maximum a posteriori estimate in the present work.
For particular linear problems, e.g., backward diffusion, we can compute $\{Z_{m_i}\}_{i=1}^{n}$ much more efficiently, which allows us to utilize approximate sampling or Markov chain Monte Carlo type methods.

{\color{black} 
In Section 4 
of the main text, we asserted that the maximum a posteriori (MAP) estimator introduced in \cite{Agapiou2018IP} yields the optimization problem:
\begin{align}\label{MAPoptim4Appendix}
	\argmin_{\theta} - \frac{\lambda}{\lambda + n\beta}\sum_{i=1}^{n}\ln\left[\sum_{\ell=1}^{L}\exp\bigg(
	-\frac{\beta}{m_i}\sum_{j=1}^{m_i}\Phi(u_{\ell}^{i};z_{ij})
	\bigg)\right] + \frac{1}{2}\|\theta_1\|_{E_{\Theta_1}}^2 - \ln p_2(\theta_2).
\end{align}
We now provide the technical justification for this claim.

When the parameter $\theta$ belongs to $\Theta$ (a separable Hilbert space), the standard definition of the MAP estimator as the maximizer of a probability density function is unavailable, since Lebesgue measure does not exist in infinite-dimensional spaces. To address this, several generalized definitions of MAP estimators have been proposed \citep{Agapiou2018IP,Helin2015IP,Dashti2013IP}; these definitions coincide in finite-dimensional settings. Working within the framework of \cite{Agapiou2018IP}, we establish the following result.

\begin{theorem}\label{MAPtheorem}
Assume all of the Assumptions in Theorem 4.1 
hold true and, in addition, we assume 
$\|D_{\theta}\tilde{\Phi}(u;S_i,\theta)\|_{\Theta^{*}} \leq M_5(r,\theta_{\text{min}},\theta_{\text{max}}) \|u\|_{\mathcal{U}^{1-\tilde{s}}} + M_6(r,\theta_{\text{min}},\theta_{\text{max}})$, 
where $\theta_1\in B_{\Theta_1}(r):=\{\theta_1\in\Theta_1 : \|\theta_1\|_{\Theta_1} < r\}$, $M_5(r,\theta_{\text{min}},\theta_{\text{max}})$ and  $M_6(r,\theta_{\text{min}},\theta_{\text{max}})$ are constants depending on $r, \theta_{\text{min}}, \theta_{\text{max}}$, 
and $\tilde{\Phi}(u;S_i,\theta)$ is defined through $\tilde{\Phi}(u;S_i,\theta) = -\ln\frac{d\mathbb{P}_{S_i}^{\theta}}{d\mathbb{P}_0}(u)$
with $\mathbb{P}_0 := \mathcal{N}(0,\mathcal{C}_0)$ being a Gaussian measure. 
For the first component of the hyper-prior measure $\mathscr{P}_1$, we assume that it is a zero mean Gaussian measure defined on $\Theta_1$.
Denote $(E_{\Theta_1}, \|\cdot\|_{E_{\Theta_1}})$ as the Cameron-Martin space of $\mathscr{P}_1$, and $p_2(\cdot)$ as the density function of the second component of the hyper-prior measure $\mathscr{P}_2$. 
Then we have 
\begin{align}\label{MAPsupplement}
	\lim_{\delta\rightarrow 0}\frac{\mathscr{Q}(B_{\Theta}(\theta^{1}, \delta))}{\mathscr{Q}(B_{\Theta}(\theta^{2},\delta))} = 
	\exp(I(\theta^{2}) - I(\theta^{1})),
\end{align}
where 
$B_{\Theta}(\theta^i,\delta):=\{\theta\in\Theta : \|\theta - \theta^i\|_{\Theta} < \delta\}$ ($i=1,2$) and 
\begin{align*}
	I(\theta):=\left\{\begin{aligned}
		& \!-\sum_{i=1}^{n}\!\frac{\lambda\ln Z_{m_i}(S_i, \mathbb{P}_{S_i}^{\theta})}{\lambda + n\beta}\! + \! \frac{1}{2}\|\theta_1\|_{E_{\Theta_1}}^2 
		\!\!\! - \!\ln p_2(\theta_2) \quad\! \text{if }\theta_1\in E_{\Theta_1}, \theta_2\in \!(\theta_{\text{min}}, \theta_{\text{max}})^{N_c} \text{, and} \\
		& +\infty \qquad\qquad \qquad \qquad \qquad \qquad \quad\quad\quad\quad  \text{else.}
	\end{aligned}\right.
\end{align*}
\end{theorem}

\begin{remark}
We analyze formula (\ref{MAPsupplement}) in detail. When $\theta^1$ is the MAP estimator, the left-hand side of (\ref{MAPsupplement}) should exceed $1$. Similarly, the right-hand side must also exceed $1$, implying that $I(\theta^2) \geq I(\theta^1)$. This analysis confirms the natural interpretation that identifying the MAP estimator corresponds to solving the optimization problem (\ref{MAPoptim4Appendix}).
\end{remark}

\begin{proof}
Let us denotes
\begin{align}\label{mapthm1}
	\begin{split}
		\Phi^h(\theta) := & \frac{-\lambda}{\lambda + n\beta}\sum_{i=1}^{n}\ln Z_{m_i}(S_i,\mathbb{P}_{S_i}^{\theta})  \\
		= & \frac{-\lambda}{\lambda + n\beta}\sum_{i=1}^{n}\ln\int_{\mathcal{U}}\exp\bigg( 
		-\frac{\beta}{m_i}\sum_{j=1}^{m_i}\Phi(u;z_{ij}) - \tilde{\Phi}(u;S_i,\theta)
		\bigg) \mathbb{P}_0(du),
	\end{split}
\end{align}
where 
\begin{align*}
	\Phi(u;z_{ij}) := &   \frac{1}{2}\|\Gamma^{-1/2}\mathcal{L}_{x_{ij}}\mathcal{G}(u)\|_{\mathcal{H}}^2 - 
	\langle \Gamma^{-1/2}y_{ij}, \Gamma^{-1/2}\mathcal{L}_{x_{ij}}\mathcal{G}(u) \rangle_{\mathcal{H}}, \\
	\tilde{\Phi}(u;S_i,\theta) := & -\ln\frac{d\mathbb{P}_{S_i}^{\theta}}{d\mathbb{P}_0}(u). 
\end{align*}
Similar to the proof of Theorem 4.1 
of the main text, we will know that $\Phi^{h}$ is locally bounded from above and below with respect to the parameter $\theta$. So the main point is to verify the local Lipschitz continuity of $\Phi^h$ with respect to $\theta$. Obviously, the Fr\'{e}chet derivative of $\Phi^h$ is 
\begin{align*}
	D_{\theta}\Phi^h = \sum_{i=1}^{n}\frac{\int_{\mathcal{U}}
		\exp\left( -\frac{\beta}{m_i}\sum_{j=1}^{m_i}\Phi(u;z_{ij}) - \tilde{\Phi}(u;S_i,\theta) \right) D_{\theta}\tilde{\Phi}(u;S_i,\theta)\mathbb{P}_{0}(du)
	}{\left( 1+\frac{n\beta}{\lambda} \right)
		\int_{\mathcal{U}}\exp\left( -\frac{\beta}{m_i}\sum_{j=1}^{m_i}\Phi(u;z_{ij}) \right)\mathbb{P}_{S_i}^{\theta}(du)
	}.
\end{align*}
Recall the following assumption
\begin{align*}
	\|D_{\theta}\tilde{\Phi}(u;S_i,\theta)\|_{\Theta^{*}} \leq M_5(r,\theta_{\text{min}},\theta_{\text{max}}) \|u\|_{\mathcal{U}^{1-\tilde{s}}} + M_6(r,\theta_{\text{min}},\theta_{\text{max}}), 
\end{align*}
where $\theta_1\in B_{\Theta_1}(\delta):=\{\theta_1\in\Theta_1 : \|\theta_1\|_{\Theta_1} < r\}$, and $M_5(r,\theta_{\text{min}},\theta_{\text{max}}), M_6(r,\theta_{\text{min}},\theta_{\text{max}})$ are constants depend on $r$, $\theta_{\text{min}}$, and $\theta_{\text{max}}$. 
Then we have 
\begin{align}
		\|D_{\theta}& \Phi^h\|_{\Theta^*} \leq \sum_{i=1}^{n}\frac{\int_{\mathcal{U}}
			\exp\left( -\frac{\beta}{m_i}\sum_{j=1}^{m_i}\Phi(u;z_{ij}) - \tilde{\Phi}(u;S_i,\theta) \right) \|D_{\theta}\tilde{\Phi}(u;S_i,\theta)\|_{\Theta^*}\mathbb{P}_{0}(du)
		}{\left( 1+\frac{n\beta}{\lambda} \right)
			\int_{\mathcal{U}}\exp\left( -\frac{\beta}{m_i}\sum_{j=1}^{m_i}\Phi(u;z_{ij}) \right)\mathbb{P}_{S_i}^{\theta}(du)
		} \nonumber \\
		\leq & \sum_{i=1}^{n}\frac{ \int_{\mathcal{U}}
			\exp\left( -\frac{\beta}{m_i}\sum_{j=1}^{m_i}\Phi(u;z_{ij}) - \tilde{\Phi}(u;S_i,\theta) \right) M_5(r,\theta_{\text{min}},\theta_{\text{max}})\|u\|_{\mathcal{U}^{1-\tilde{s}}} \mathbb{P}_{0}(du)
		}{\left( 1+\frac{n\beta}{\lambda} \right)
			\int_{\mathcal{U}}\exp\left( -\frac{\beta}{m_i}\sum_{j=1}^{m_i}\Phi(u;z_{ij}) \right)\mathbb{P}_{S_i}^{\theta}(du)
		}  \label{mapthm2} \\ 
		& + \sum_{i=1}^{n}\frac{ \int_{\mathcal{U}}
			\exp\left( -\frac{\beta}{m_i}\sum_{j=1}^{m_i}\Phi(u;z_{ij}) - \tilde{\Phi}(u;S_i,\theta) \right)  M_6(r,\theta_{\text{min}},\theta_{\text{max}})\mathbb{P}_{0}(du)
		}{\left( 1+\frac{n\beta}{\lambda} \right)
			\int_{\mathcal{U}}\exp\left( -\frac{\beta}{m_i}\sum_{j=1}^{m_i}\Phi(u;z_{ij}) \right)\mathbb{P}_{S_i}^{\theta}(du).
		} \nonumber
\end{align}
For the numerator term on the right-hand side of (\ref{mapthm2}), it can be bounded from above by employing similar techniques 
as for deriving estimate (\ref{AlgBoundZpProof3}). The only difference is the extra term $\|u\|_{\mathcal{U}^{1-\tilde{s}}}$ that 
can be bounded by $\|u\|_{\mathcal{U}^{1-\tilde{s}}} \leq C\exp(\epsilon\|u\|_{\mathcal{U}^{1-\tilde{s}}})$ with arbitrarily small 
constant $\epsilon>0$. Hence, the upper bound is obtained by taking $\delta_2 + \epsilon$ 
($\delta_2$ is the same as in (\ref{AlgBoundZpProof3})) small enough. 
For the denominator term on the right-hand side of (\ref{mapthm2}), it is exactly the same as that of (\ref{AlgBoundZpProof5}) with 
a positive lower bound independent of $\theta$. Now we can conclude that 
\begin{align*}
	\|D_{\theta}\Phi^h\|_{\Theta^*} \leq C(r) < +\infty, 
\end{align*}
where $C(r)$ is some constant depends on $r$. This obviously indicates $\Phi^h$ is locally Lipschitz continuous with respect to 
the parameter $\theta$. Now, following the illustrations given in Section 3 of \cite{Dashti2013IP}, we will obtain the desired
results. Here, we have an extra finite-dimensional component $\theta_2$ with negative log-likelihood function $-\ln p_2(\theta_2)$. In comparison to the potentially infinite-dimensional component $\theta_1$, this finite-dimensional component $\theta_2$ can be more straightforwardly managed.  
\end{proof}
}

{\color{black}At the end of this section, we provide some clarifying remarks on the reparametrization trick employed in the computation of $u^{i}_{\ell}$ as defined in formula (\ref{logsumexp3Asupp}). As detailed in Subsection 3.4 of \cite{Thanh2016IPI}, the finite-dimensional approximation for sampling from a Gaussian measure takes the form
\begin{align*} 
	u_{\ell}^i = f_{m}(S_i;\theta_1) + V\Lambda a,
\end{align*}  
where $f_{m}(S;\theta_1)$ denotes the mean function, $V$ contains the eigenvectors of the covariance operator, $\Lambda$ is a diagonal matrix with eigenvalues as entries, and $a\sim\mathcal{N}(0,\text{Id})$ is a random vector. Through this parametrization, we can compute gradients with respect to the parameters in the mean vector and the eigenvalues. Similar reparametrization tricks are widely employed in machine learning; see, for example, \cite{Yue2019NIPS,Yue2024TPAMI}.

For non-Gaussian general prior measures, the simple reparametrization trick described above is inapplicable. Investigating non-Gaussian prior settings lies beyond the scope of this work, but we here outline potential directions. In a recent study \cite{Papamakarios2021JMLE}, the authors provided a comprehensive review of normalizing flows, which constitute a broad class of generative models. Briefly, a normalizing flow is a carefully designed neural network $F_{\theta}$ with trainable parameters $\theta$. For a random vector $a\sim\mathcal{N}(0,\text{Id})$, the transformed variable $F_{\theta}(a)$ follows a complex probability distribution $\mathbb{P}$. Evidently, $\mathbb{P}$ can serve as our prior. Generating a new random variable then involves sampling a standard Gaussian random vector $a$ and applying the transformation $F_{\theta}(a)$. The network parameters can be optimized via gradient descent algorithms.  

One obstacle to applying normalizing flows is that they are typically defined on discrete finite-dimensional spaces, which differs from the neural operator framework. Thus, classical normalizing flow models are incompatible with the current infinite-dimensional setting. However, a recent work \cite{Zhao2024FNF} established general conditions for the equivalence of the probability measures defined on infinite-dimensional space. Leveraging these conditions, infinite-dimensional normalizing flow models have been constructed that accommodate functions of arbitrary discrete dimension. How to integrate the methods developed in \cite{Zhao2024FNF} with the present study is an interesting future research direction.}

\subsection{Numerical Details and More Results}\label{AppA6}
We establish our theories within the framework of separable infinite-dimensional function spaces, which will be useful for constructing discretization-invariant algorithms. With this goal in mind, it is desirable that the employed neural network exhibits a discretization-invariant property, meaning that the network is capable of learning a nonlinear operator rather than merely a nonlinear function. Currently, there are several works focused on constructing neural networks for nonlinear operator learning \citep{anandkumar2020ICLR,Bhattacharya2021SMAI-JCM,Li2020ICLR,Nelsen2021SISC}. In our numerical studies, we employ the Fourier neural operator (FNO), as proposed in \cite{Li2020ICLR}, which consists of three Fourier layers. We choose the FNO due to its simplicity in implementation and its computational efficiency. For readers interested in more detailed theoretical studies on FNO, we refer to \cite{Kovachki2021JMLR}.

In our setting, the measurement data consist of discrete variables that are statistically equally spaced, as discussed in Subsection 2.1. 
However, these data cannot be directly utilized as inputs for the FNO framework, which is designed to operate on functions—rather than discrete vectors—as both inputs and outputs. To bridge this gap, we apply the adjoint of the measurement operator to the noisy data, thereby lifting the discrete observations into a function space. The resulting function is then slightly smoothed by applying the elliptic regularization operator $(\text{Id} - \alpha\Delta)^{-1}$, with $\alpha = 0.05$ in our numerical experiments. This smoothed function subsequently serves as the input to the standard FNO architecture. We now proceed to describe the details more precisely.

As in the main text, let $\{\phi_k\}_{k=1}^{N}$ be the basis functions of the finite element discretization. We define the finite-dimensional space $V_N = \text{span}\{\phi_1,\ldots,\phi_N\}$, where approximate functions $f = \sum_{k=1}^Nf_k\phi_k \in V_N\subset L^2(\Omega)$. Here, we use boldface letters to denote vectors and matrices.
Following \cite{Tan2013SISC}, we define the mass matrix $\bm{M} = (M_{k\ell})_{k,\ell=1,\ldots, N}$ by
\begin{align}
	M_{k\ell} = \int_{\Omega}\phi_k(\bm{x})\phi_{\ell}(\bm{x})d\bm{x}.
\end{align}
We define the weighted finite-dimensional space $\mathbb{R}_{\bm{M}}^N$ as the Euclidean space $\mathbb{R}^N$ equipped with a weighted inner product $(\bm{m}_1, \bm{m}_2)_{\bm{M}} := \bm{m}_1^T \bm{M}\bm{m}_2$, where $\bm{m}_1,\bm{m}_2 \in \mathbb{R}^N$.
For a function $u \in V_N$, we have $u = \sum_{k=1}^N u_k\phi_k$ and denote $\bm{u} := (u_1,\ldots,u_N)^T$.
Obviously, we have $(m_1, m_2)_{L^2(\Omega)} \approx (\bm{m}_1, \bm{m}_2)_{\bm{M}}$ for $m_1, m_2 \in V_N \subset L^2(\Omega)$.
We denote $\bm{S}$ as the discretized measurement operator that maps from $\mathbb{R}_{\bm{M}}^N$ to $\mathbb{R}^{N_d}$.
Here, $N_d$ is the number of measured points. We denote $\bm{S}^{\natural}$ as the adjoint of $\bm{S}$, which is an operator that maps from $\mathbb{R}^{N_d}$ to $\mathbb{R}_{\bm{M}}^N$.
By the definition of the adjoint operator, it follows that $(\bm{S}\bm{u}, \bm{y}) = (\bm{u}, \bm{S}^{\natural}\bm{y})_{\bm{M}}$ for all $\bm{u} \in \mathbb{R}_{\bm{M}}^N$ and $\bm{y} \in \mathbb{R}^{N_d}$, which implies that $\bm{u}^T\bm{S}^T\bm{y} = \bm{u}^T\bm{M}\bm{S}^{\natural}\bm{y}$. That is to say, we have $\bm{S}^{\natural} = \bm{M}^{-1}\bm{S}^T$.

{\color{black}
With the above setup in place, it becomes clear that the transformation of the discrete data vector $\bm{y} \in \mathbb{R}^{N_d}$ into a corresponding discretized function can be explicitly computed via the expression $\bm{M}^{-1}\bm{S}^T\bm{y}$. This discretized function is then mildly smoothed by applying an elliptic regularization step, resulting in the transformed input
$$
\bm{f} := (\bm{M} + 0.05\bm{B})^{-1}\bm{M}^{-1}\bm{S}^T\bm{y},
$$
where the matrix $\bm{B} = (B_{k\ell})_{k,\ell=1}^N$ is defined by the bilinear form
$$
B_{k\ell} = \int_{\Omega} \nabla \varphi_k(\bm{x}) \cdot \nabla \varphi_\ell(\bm{x}) \, d\bm{x}, \quad k, \ell = 1, \ldots, N.
$$
The resulting function $\bm{f}$ is then used as the input to the FNO architecture. Thanks to the intrinsic structure of Fourier Neural Operators—which operate directly on function spaces—the performance of the model is not expected to be significantly affected by the choice of discretization dimension.

To compute the gradient $\nabla_{\theta_k} L(\theta_k)$ within Algorithm 1 
of the main text, we employed PyTorch's automatic differentiation framework. However, since the loss functional $L(\theta_k)$ incorporates partial differential equations (PDEs) through the potential function $\Phi$, which are not natively differentiable in PyTorch, we implemented custom differentiation rules via extensions of ``torch.autograd.Function''. These extensions were derived using the adjoint method, tailored specifically to the backward diffusion and Darcy flow problems under consideration. For a detailed exposition of the adjoint method, we refer the reader to \cite{Ghattas2021ActaNum}.

It is important to note that all finite element computations are carried out using the open-source software FEniCS (version 2019.1.0) \citep{Logg2012Book}. The associated neural network architectures are implemented using PyTorch (version 2.5.1+cu121). The FNO is implemented with minor adaptations based on the publicly available code accompanying the original work \citep{Li2020ICLR}. All stiffness and mass matrices generated by FEniCS are exported as NumPy arrays and subsequently converted into PyTorch tensors. These are then incorporated into the neural network framework to enable the differentiable formulation of both the forward and adjoint PDE solvers within the learning pipeline. As stated in the main text, all programs ran on a system with an Intel(R) Xeon(R) Platinum 8180 CPU, 48 GB NVIDIA RTX A6000 GPU, and Ubuntu 22.04.5 LTS OS. 
}

\subsubsection*{Numerical results for the backward diffusion problem}
{\color{black}
The backward diffusion problem is a well-studied inverse problem in the field of PDEs \citep{Stuart2010AN}. We define the Hilbert space $\mathcal{H}$ to be $(L^{2}(\Omega), \langle\cdot,\cdot\rangle, \|\cdot\|)$, and the operator $A := -\Delta$ has the domain $\mathcal{D}(A) = H^2(\Omega)\cap H_0^1(\Omega)$. Here $H^2(\Omega)$ and $H_0^1(\Omega)$ denote the standard Sobolev space and Sobolev space with zero boundary trace functions. The forward diffusion equation is represented by the ordinary differential equation in $\mathcal{H}$:
\begin{align*}
	\frac{dv}{dt} + Av = 0, \qquad	v(0) = u.
\end{align*}
For the inverse problems, the dataset consists of the solution $v(T)$ of the diffusion equation at time $T>0$, and the parameter of interest is the initial condition $u$ of the equation.

In our study of the backward diffusion problem on $\Omega=(0,1)$ with $T=0.01$, we mitigate the inverse crime \citep{Kaipio2004Book} by using a 600-point mesh for dataset generation and a 70-point mesh to learn the base prior's mean function. For base prior learning, we set $L=10$ samples and $H=20$ mini-batches in the formula (16) 
of the main text.

To explore more possibilities, there are some different settings for the backward diffusion problem compared with the Darcy flow problem. Here, we also take $ m_i = m $ for each $i = 1, \ldots, n $, and set $ \beta = m $, $ \gamma = nm $, and $ \lambda = n.$ We choose $\mathcal{U} = \mathcal{H} = L^2(\Omega)$ for $\Omega=(0,1)$, and consider Gaussian noise $\eta\sim\mathcal{N}(0,\Gamma)$ with $\Gamma=\tau^2 \text{Id}$ ($\tau>0$). To define the base prior, we introduce the operator $\mathcal{C}_0:=A_h^{-2}$, where $A_h:=\text{Id}-0.01\Delta$ and $\Delta$ has domain $\mathcal{D}(\Delta):=\{u\in H^2(\Omega):\frac{\partial u}{\partial\bm{n}}=0\}$, with $\bm{n}$ as the outward normal vector. Recall Assumption 2.7 
stated in the main text, we take $N_c = 1$ and restate the assumption below for reader's convenience. 

\begin{assumption}
	Consider the parameter space $\Theta = \Theta_1 \times \Theta_2$, with $\Theta_1$ a separable Hilbert space, and $\Theta_2 = \mathbb{R}$. Assume the hyper-prior $\mathscr{P} = \mathscr{P}_1 \otimes \mathscr{P}_2$, where $\mathscr{P}_1$ is a probability measure on $\Theta_1$, and $\mathscr{P}_2$ is a probability measure on $\mathbb{R}$ with compact support $[\theta_{\text{min}}, \theta_{\text{max}}]$. Here, the parameters $\theta_{\text{min}},\theta_{\text{max}}\in\mathbb{R}$ ($\theta_{\text{min}} < \theta_{\text{max}}$) are fixed real numbers. 
\end{assumption}

We define the mapping  
$f(S; \theta) := \big( f_m(S; \theta_1), \theta_2 \big) \in \mathcal{V} = \mathcal{U} \times \mathbb{R},$  
where \( f_m(S; \theta_1) \) represents the data-dependent mean. The learned prior measure \(\mathbb{P}_S^\theta\) is then given by the Gaussian measure  
$\mathbb{P}_S^\theta := \mathcal{N}\big( f_m(S; \theta_1), \mathcal{C}_0(\theta_2) \big),$  
with covariance operator  
\begin{align}\label{SuppBackwardDiffusionPriorCov}
	\mathcal{C}_0(\theta_2) := \left( e^{\theta_2} \text{Id} - 0.01 \Delta \right)^{-2}. 
\end{align}  
This construction differs from the base prior measure used for the Darcy flow problem in the main text. For notational simplicity, we retain the same symbols, as the distinction will be clear from context.  

\begin{remark}
We define the covariance operator with learned parameter $\theta_2\in\mathbb{R}$ by (\ref{SuppBackwardDiffusionPriorCov}) motivated by considerations of Gaussian measure equivalence. A related framework was employed in \cite{Dunlop2017StatComput} for constructing hierarchical Bayesian models in infinite-dimensional spaces. By applying Theorem 2.9 of \cite{Prato2006IDAnalysis}, we can easily verify the equivalence of measures for distinct values of \(\theta_2\).  
\end{remark}

As the Darcy flow case, for the function $f_m(S;\theta_1)$ in the base prior, we consider two cases:   
\begin{itemize}
	\item If $f_m$ is data-independent, we use $f_m(\theta_1) = \sum_{k=1}^{N}\theta_{1k}\phi_k$ with $\{\phi_k\}_{k=1}^{N}$ as the basis of the finite approximate $V_N$.
	\item If $f_m$ depends on the dataset $S$, we employ a modified neural operator with $\theta_1$ as its parameters and $S$ as input.
\end{itemize} 
In our numerical experiments, the hyper-prior \(\mathscr{P}\) for the data-independent mean function of the base prior is defined as \(\mathcal{N}(0, A_p^{-4})\), where \(A_p := -0.005\Delta\) and \(\Delta\) is the Laplacian operator with domain $\mathcal{D}(\Delta) = \{ u \in H^2(\Omega) : u = 0 \text{ on } \partial\Omega \}.$  
For the data-dependent mean function \(f_m(S; \theta_1)\) of the base prior, specifying a hyper-prior is nontrivial due to the inherent complexity of the FNO's parameters \(\theta_1\). To address this, we exploit the a priori regularity of \(f_m(S; \theta_1)\) by randomly selecting \(n_S = 4\) datasets \(\{S_i\}_{i=1}^{n_S}\) and defining the auxiliary function $g(\theta_1)$ as the sum of $f_m(S_i;\theta_1)$ over these datasets. The hyper-prior \(\mathscr{P}\) for \(f_m(S; \theta_1)\) is then constructed as the pushforward measure \(g_{\#}\mathcal{N}(0, A_p^{-4})\).

Our numerical experiments consider three types of base prior measures:
\begin{itemize}
	\item Unlearned prior: $\mathbb{P} := \mathcal{N}(0,\mathcal{C}_0)$;
	\item Learned data-independent prior: $\mathbb{P}^{\theta} := \mathcal{N}(f_m(\theta_1),\mathcal{C}_0(\theta_2))$;
	\item Learned data-dependent prior: $\mathbb{P}_{S}^{\theta} := \mathcal{N}(f_m(S;\theta_1),\mathcal{C}_0(\theta_2))$.
\end{itemize}

\textbf{Simple and complex environments.} 
We evaluate our method in both simple and complex environments. In the simple environment, the measure \(\mathscr{E}\) generates ground-truth parameters via the following random function:
\begin{align}\label{RandFunBackDiff1}
	u(x) = \left(5\beta x + a \sin\left(2(5 x - b)\right) + c\right)e^{-20\left(x-\frac{1}{2}\right)^2},
\end{align}
where \(\beta \sim \mathcal{N}(0.5, 0.5)\), \(a \sim U(5, 15)\), \(b \sim U(0, 0.1)\), and \(c \sim \mathcal{N}(4, 1)\). Here, \(U(x_1, x_2)\) denotes the uniform distribution with lower and upper bounds \(x_1\) and \(x_2\), respectively. Additionally, let \(\text{Bern}(0.5)\) denotes the Bernoulli distribution with parameter \(0.5\). For the complex environment, we employ the following random function:
\begin{align}\label{RandFunBackDiff2}
	u(x) = (2\alpha - 1)\left(5\beta x + a \sin\left(2(5 x - b)\right) + c\right)e^{-20\left(x-\frac{1}{2}\right)^2},
\end{align}
where \(\alpha \sim \text{Bern}(0.5)\), and the remaining parameters are identical to those in the simple environment setting. For both settings, we generate 2000 random functions for training and 50 random functions for testing. It is important to emphasize that the learning algorithm does not directly access the random functions themselves. Instead, it extracts information from noisy datasets. These noisy datasets are generated by solving the forward diffusion equations using the random functions, and additive Gaussian noise is introduced at the measurement points. For this example, the number of the measurement points is $20$ ($m_i = m = 20$). Specifically, the noise level is set to \(0.1\) (the relative errors under $L^2$-norm approximately equal to $8\%$), meaning that \(\eta \sim \mathcal{N}(0, 0.1^2\text{Id})\), where \(\text{Id}\) denotes the identity matrix. 

Before presenting quantitative comparisons, we provide a visual comparison in Figure \ref{FigEx1_1} for the simple environment. In panel (a) of Figure \ref{FigEx1_1}, we display five random functions generated from formula (\ref{RandFunBackDiff1}), illustrating the model parameters. In panels (b) and (c) of Figure \ref{FigEx1_1}, we compare the true functions (solid black lines) with the estimated mean functions of the data-independent prior \(\mathbb{P}^{\theta}\) (dashed blue lines) and the data-dependent prior \(\mathbb{P}_S^{\theta}\) (dash-dotted red lines). The estimated means of both priors capture the main characteristics of the true function, though the data-dependent prior demonstrates closer visual similarity to the true function. This observation holds for both background source functions presented in panels (b) and (c).

\begin{figure}[ht!]
	\centering
	\includegraphics[width=1.0\textwidth]{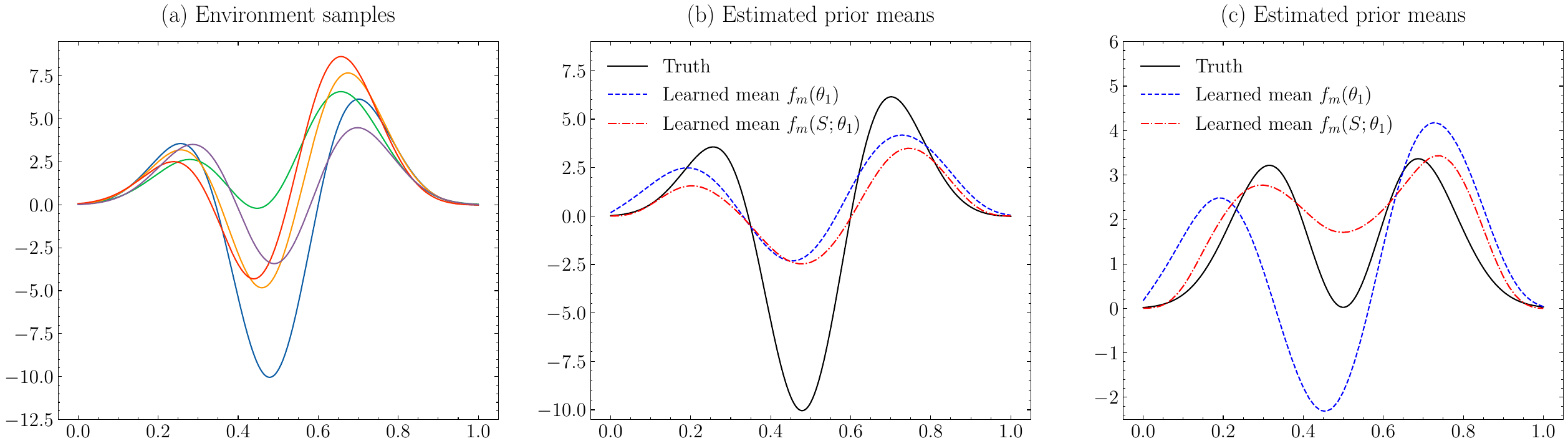}\\
	\vspace*{0.0em}
	\caption{Random functions generated according to formula (\ref{RandFunBackDiff1}).
		(a) Ground truth functions used for constructing the datasets.
		(b)(c) Learned mean functions: the data-independent assumption $f_m(\theta_1)$ (dashed blue line) and the data-dependent assumption $f_m(S;\theta_1)$ (dash-dotted red line).
	}\label{FigEx1_1}
\end{figure}

For the complex environment, we provide a visual comparison in Figure \ref{FigEx1_2}. The parameter \(\alpha\) in formula (\ref{RandFunBackDiff2}) introduces two distinct branches in the model parameters, as illustrated in panels (a) and (d) of Figure \ref{FigEx1_2}, each displaying five samples. In panels (b), (c), (e), and (f) of Figure \ref{FigEx1_2}, we compare the true functions (solid black lines) with the estimated mean functions of the data-independent prior \(\mathbb{P}^{\theta}\) (dashed blue lines) and the data-dependent prior \(\mathbb{P}_{S}^{\theta}\) (dash-dotted red lines). The estimated mean of \(\mathbb{P}^{\theta}\) fails to capture the main features of the branches, whereas the estimated mean of \(\mathbb{P}_{S}^{\theta}\) effectively reflects the characteristics of the true source function.

\begin{figure}[ht!]
	\centering
	\includegraphics[width=1.0\textwidth]{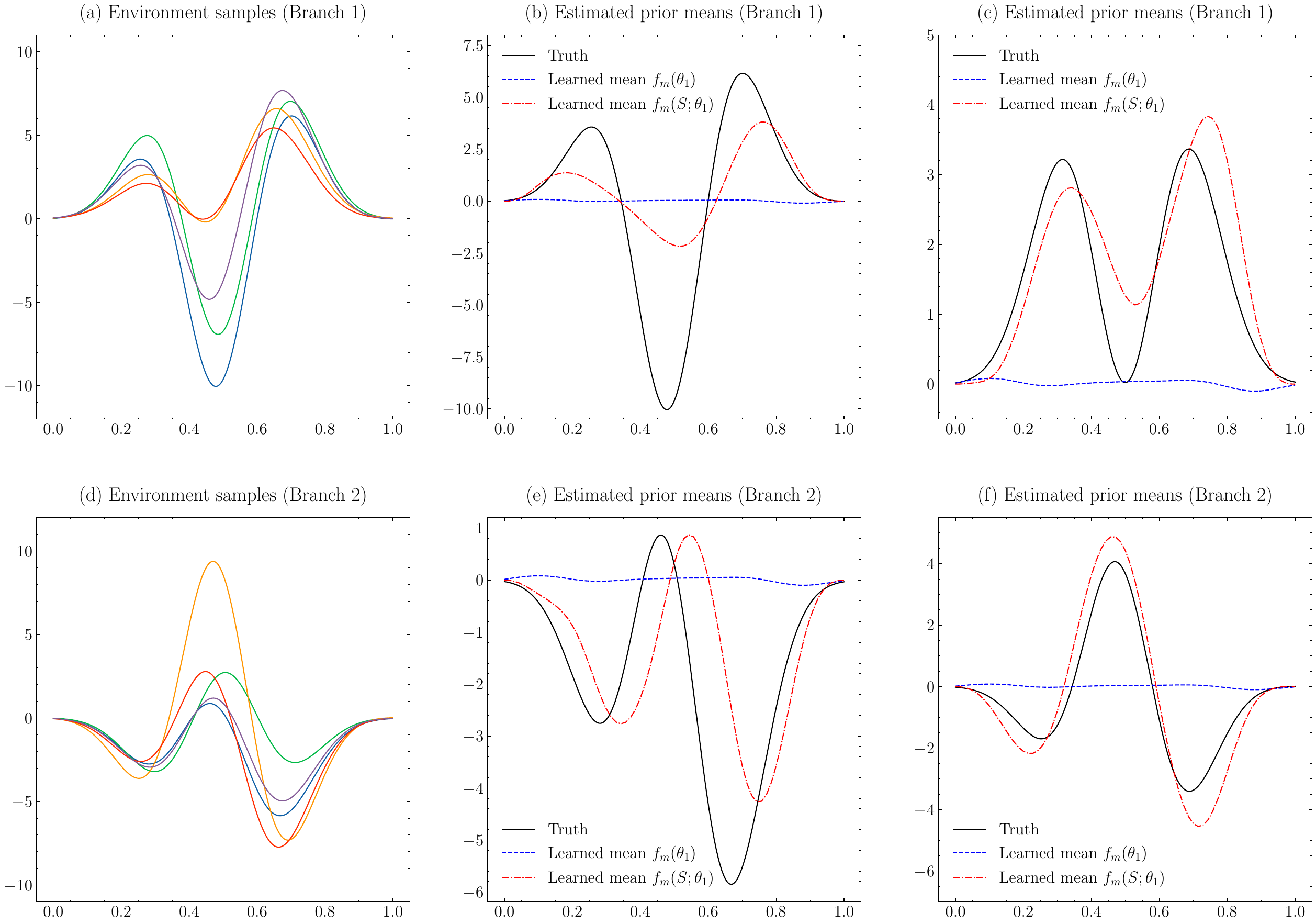}\\
	\vspace*{0.0em}
	\caption{Constructed datasets from equation (\ref{RandFunBackDiff2}) yield the following random and estimated mean functions: (a) shows one set of random functions from Branch 1; (b) and (c) display the ground true functions of Branch 1 alongside the estimated mean functions from data-independent and data-dependent priors, respectively; (d) presents another set of random functions from Branch 2; and (e) and (f) compare the estimated mean functions to the ground true functions of Branch 2.
	}\label{FigEx1_2}
\end{figure}

\begin{table}[ht!]
	\caption{For the 50 test datasets, we list the average relative errors of the MAP estimates from base posteriors under priors $\mathbb{P}$, $\mathbb{P}^{\theta}$, and $\mathbb{P}_{S}^{\theta}$, respectively. Here, ItN denotes the iteration count of the inexact matrix-free Newton-conjugate gradient method. }
	\begin{center}
		\begingroup
		\setlength{\tabcolsep}{1.0pt}
		\renewcommand{\arraystretch}{1.0}
		\begin{spacing}{1.1}
			\begin{tabular}{c|ccc|ccc}
				\Xhline{1.1pt}
				& \multicolumn{3}{c|}{Simple Environment} & \multicolumn{3}{c}{Complex Environment}  \\
				\hline
				$\quad$ItN $\quad$& $\qquad$$\mathbb{P}$$\qquad$ & $\qquad$$\mathbb{P}^{\theta}$$\qquad$ & $\qquad$$\mathbb{P}_S^{\theta}$$\qquad$ &  $\qquad$$\mathbb{P}$$\quad$ & $\qquad$$\mathbb{P}^{\theta}$$\qquad$ & $\qquad$$\mathbb{P}_S^{\theta}$$\qquad$ \\
				\hline
				$1$  & .6717 & .2148 & \textbf{.1623} & .6737 & .6758 & \textbf{.1966} \\
				$2$  & .5733 & .2010 & \textbf{.1620} & .5792 & .5889 & \textbf{.1951}	\\
				$3$  & .5711 & .2001 & \textbf{.1615} & .5773 & .5885 & \textbf{.1944} 	\\ 
				$4$  & .5738 & .1989 & \textbf{.1612} & .5816 & .5947 & \textbf{.1939}	\\               
				$5$  & .5692 & .1988 & \textbf{.1611} & .5776 & .5900 & \textbf{.1939}	\\
				$6$  & .5686 & .1988 & \textbf{.1612} & .5773 & .5899 & \textbf{.1939}	\\
				$7$  & .5685 & .1988 & \textbf{.1611} & .5770 & .5900 & \textbf{.1939}	\\
				\Xhline{1.1pt}
			\end{tabular}
		\end{spacing}
		\endgroup
	\end{center}\label{TableExamp1_1}
\end{table}

Following the visual comparison, we now present quantitative comparisons in Table \ref{TableExamp1_1}. For the backward diffusion problem - a linear inverse problem - we assume Gaussian priors and noise. Consequently, the posterior mean equals the maximum a posteriori (MAP) estimate. Similarly, for the Darcy flow problem discussed in the main text, we use the matrix-free Newton conjugate gradient algorithm \citep{Ghattas2021ActaNum} to compute the MAP estimate. The terms ``ItN'' (iteration number) and ``relative error'' maintain their definitions as provided in the main text.

In the left portion of Table \ref{TableExamp1_1}, we present the average relative errors obtained using the unlearned, learned data-independent, and learned data-dependent base Bayesian models under the simple environment setting. In this simple environment, the learned data-independent model provides more accurate estimates than the unlearned model. The learned data-dependent model's posterior mean estimate exhibits the smallest average relative error, indicating that the proposed method is superior to the data-independent model even in a simple environment setting. In the right portion of Table \ref{TableExamp1_1}, we present the average relative errors obtained using the unlearned, learned data-independent, and learned data-dependent base Bayesian models under the complex environment setting. Similar to the simple environment case, the learned data-dependent model provides the most accurate inversion estimates. 

\begin{figure}[ht!]
	\centering
	\includegraphics[width=1\textwidth]{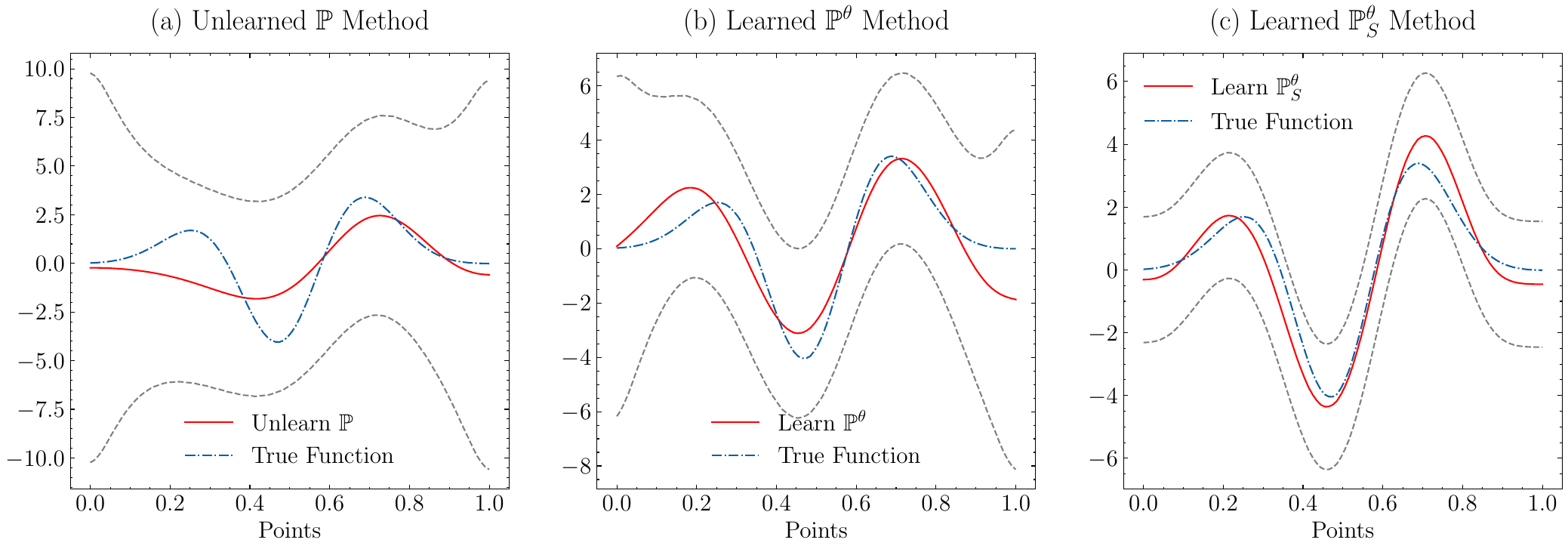}\\
	\vspace*{-3pt}
	\caption{The figure shows one of the posterior mean estimates and $95\%$ credible regions in the simple environment. In all three panels: the blue dash-dotted line represents the true function, the solid red line shows the base posterior mean estimate, and the dashed gray line indicates the $95\%$ credible region. Results from the unlearned, learned data-independent, and learned data-dependent methods are displayed in panels (a), (b), and (c), respectively.
	}\label{FigSimpleDiffusionStd}
\end{figure}

\begin{figure}[ht!]
	\centering
	\includegraphics[width=1\textwidth]{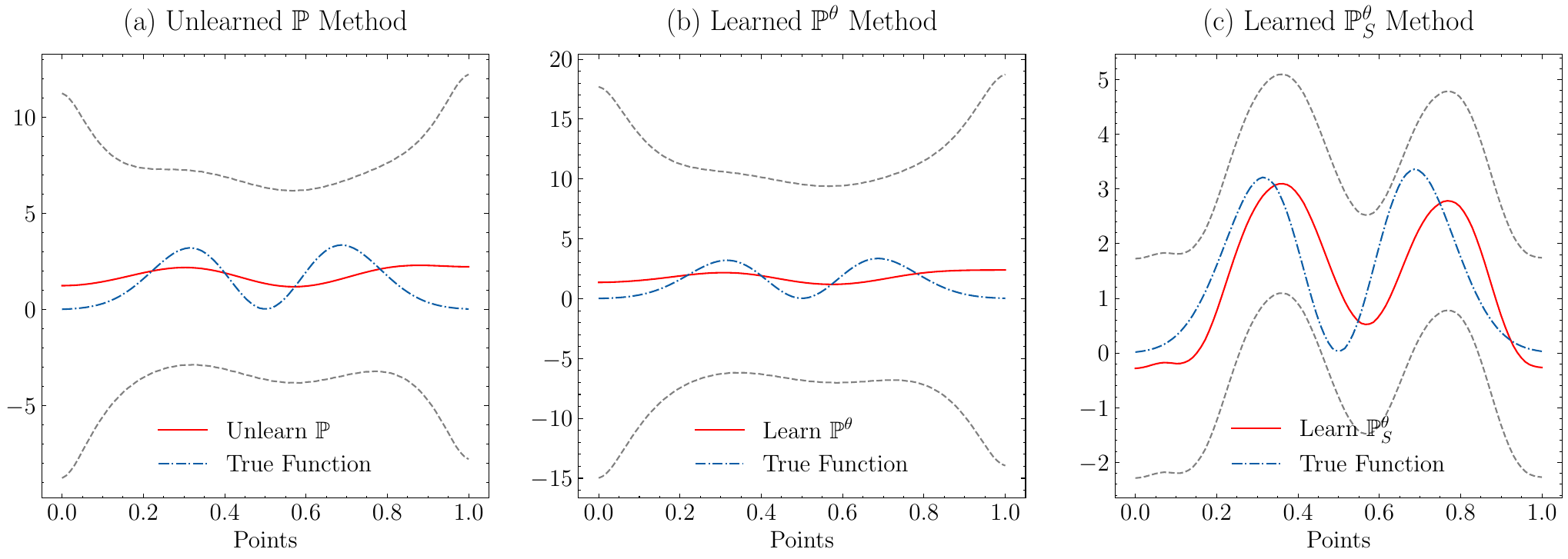}\\
	\vspace*{-3pt}
	\caption{The figure shows one of the posterior mean estimates and $95\%$ credible regions in the branch with $\alpha=1$ of the complex environment. In all three panels: the blue dash-dotted line represents the true function, the solid red line shows the base posterior mean estimate, and the dashed gray line indicates the $95\%$ credible region. Results from the unlearned, learned data-independent, and learned data-dependent methods are displayed in panels (a), (b), and (c), respectively.
	}\label{FigComplexDiffusionStd1}
\end{figure}

\begin{figure}[ht!]
	\centering
	\includegraphics[width=1\textwidth]{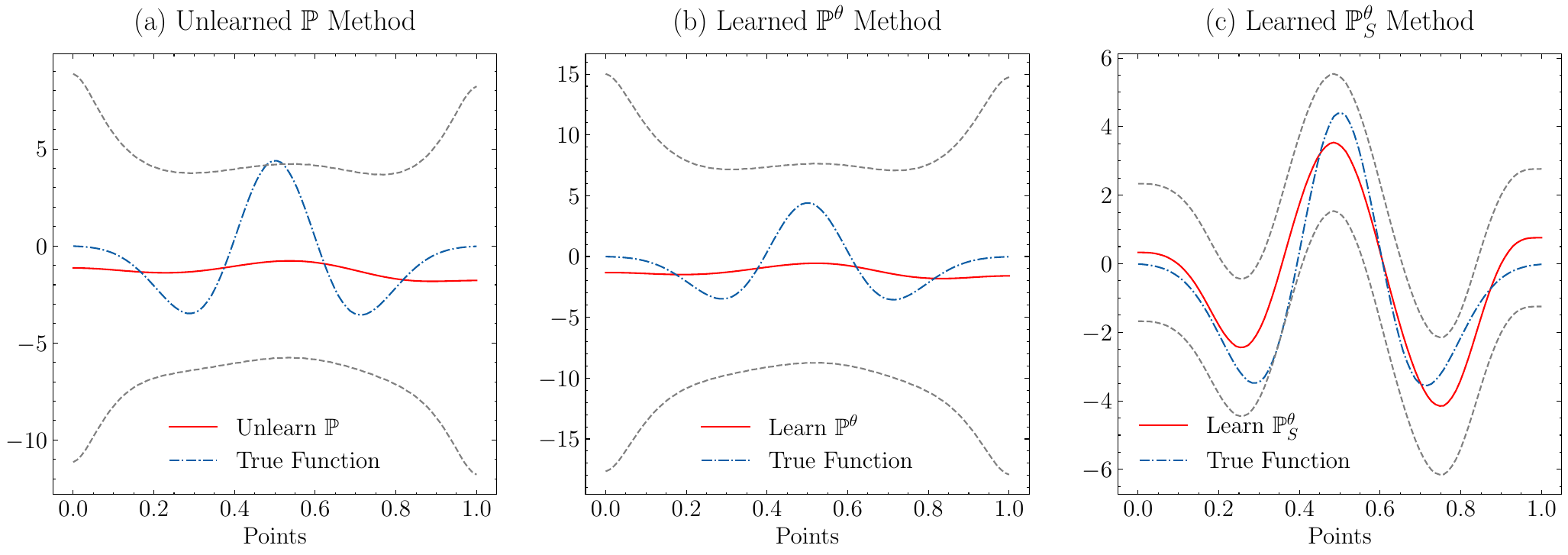}\\
	\vspace*{-3pt}
	\caption{The figure shows one of the posterior mean estimates and $95\%$ credible regions in the branch with $\alpha=0$ of the complex environment. In all three panels: the blue dash-dotted line represents the true function, the solid red line shows the base posterior mean estimate, and the dashed gray line indicates the $95\%$ credible region. Results from the unlearned, learned data-independent, and learned data-dependent methods are displayed in panels (a), (b), and (c), respectively.
	}\label{FigComplexDiffusionStd2}
\end{figure}

Under the Bayesian framework, it is important to also present uncertainty quantification results. In Figure \ref{FigSimpleDiffusionStd}, we display one of the true functions selected from the testing dataset. In all three panels: the blue dash-dotted line represents the true function, the solid red line shows the base posterior mean estimate, and the dashed gray line indicates the $95\%$ credible region. Results from the unlearned, learned data-independent, and learned data-dependent methods are shown in panels (a), (b), and (c), respectively. The unlearned Bayesian model provides a wide credible region, which appears uninformative. The learned data-independent Bayesian model yields a more accurate posterior mean with a narrower credible region. For the learned data-dependent Bayesian model, we obtain the most accurate posterior mean estimates along with a tight credible region. These results demonstrate that the learned data-dependent base prior measure effectively encodes informative features, thereby enhancing performance when using a fixed number of measurement points.

In Figures \ref{FigComplexDiffusionStd1} and \ref{FigComplexDiffusionStd2}, we present uncertainty quantification results for the complex environment. Figure \ref{FigComplexDiffusionStd1} displays one estimate selected from the branch with $\alpha=1$, while Figure \ref{FigComplexDiffusionStd2} shows one estimate from the branch with $\alpha=0$. All results demonstrate that the learned data-dependent Bayesian model provides the most accurate posterior mean estimates along with tight credible regions.
Under the complex environment setting, the learned data-independent model adapts its prior covariance to enlarge the $95\%$ credible region. From formula (\ref{RandFunBackDiff2}), we observe that functions in both branches have values contained within $[-15, 15]$. With this information, we note that it is reasonable for the learned data-independent model to enlarge its credible region, as the posterior mean estimate cannot provide useful information.

\textbf{Effectiveness of the current theory.}
In Sections 2 and 3 
of the main text, we establish several PAC-Bayes bounds. The tightness of these bounds remains theoretically unclear. While tighter bounds could potentially improve algorithm design, we demonstrate that the prior measure learning method derived from these PAC-Bayes bounds achieves a certain degree of tightness. 

To achieve this objective, we generate only 250 and 500 training datasets for the simple and complex environment cases, respectively. Furthermore, we set the number of measurement points to $m=5$. Table \ref{TableExamp1_1_small_sample_size} presents the average relative errors of the base posterior mean estimates obtained using unlearned, learned data-independent, and learned data-dependent Bayesian models trained with this limited number of historical datasets. Comparing these results with those in Table \ref{TableExamp1_1} reveals that the average relative errors are only slightly higher than the results obtained with 2000 training datasets and $m=20$ measurement points, indicating that the bound achieves a certain degree of tightness.

\begin{table}[ht!]
	\caption{For the 50 test datasets, we list the average relative errors of the MAP estimates from base posteriors under priors $\mathbb{P}$, $\mathbb{P}^{\theta}$, and $\mathbb{P}_{S}^{\theta}$, respectively. Here, ItN denotes the iteration count of the inexact matrix-free Newton-conjugate gradient method. }
	\begin{center}
		\begingroup
		\setlength{\tabcolsep}{1.0pt}
		\renewcommand{\arraystretch}{1.0}
		\begin{spacing}{1.1}
			\begin{tabular}{c|ccc|ccc}
				\Xhline{1.1pt}
				& \multicolumn{3}{c|}{Simple Environment} & \multicolumn{3}{c}{Complex Environment}  \\
				\hline
				$\quad$ItN $\quad$& $\qquad$$\mathbb{P}$$\qquad$ & $\qquad$$\mathbb{P}^{\theta}$$\qquad$ & $\qquad$$\mathbb{P}_S^{\theta}$$\qquad$ &  $\qquad$$\mathbb{P}$$\quad$ & $\qquad$$\mathbb{P}^{\theta}$$\qquad$ & $\qquad$$\mathbb{P}_S^{\theta}$$\qquad$ \\
				\hline
				$1$  & .6717 & .3450 & \textbf{.1826} & .6737 & .6762 & \textbf{.2386} \\
				$2$  & .5733 & .3338 & \textbf{.1812} & .5792 & .5921 & \textbf{.2371}	\\
				$3$  & .5711 & .3324 & \textbf{.1809} & .5773 & .5935 & \textbf{.2361} 	\\ 
				$4$  & .5738 & .3305 & \textbf{.1807} & .5816 & .5997 & \textbf{.2356}	\\                 
				$5$  & .5692 & .3302 & \textbf{.1806} & .5776 & .5955 & \textbf{.2355}	\\
				$6$  & .5686 & .3302 & \textbf{.1806} & .5773 & .5957 & \textbf{.2355}	\\
				$7$  & .5685 & .3302 & \textbf{.1806} & .5770 & .5957 & \textbf{.2355}	\\
				\Xhline{1.1pt}
			\end{tabular}
		\end{spacing}
		\endgroup
	\end{center}\label{TableExamp1_1_small_sample_size}
\end{table}

\textbf{Analysis of parameter $L$.}  
The parameter $L$ appears for approximating the following integral (see formula (14) 
of the main text) shown as follows:
\begin{align*}
	\int_{\mathcal{U}}\exp\left( -\frac{\beta}{m_i}\sum_{j=1}^{m_i}\Phi(u;z_{ij}) \right)\mathbb{P}_{S_i}^{\theta}(du) \approx 
	\frac{1}{L}\sum_{\ell=1}^{L}\exp\left( -\frac{\beta}{m_i}\sum_{j=1}^{m_i}\Phi(u_{\ell}^i;z_{ij}) \right).
\end{align*}
How to choose $L$ depends on the measure $\mathbb{P}_{S_i}^{\theta}$ and the potential function $\Phi$, which may be theoretically analyzed by using tools from the field of functional data analysis. However, it is not the main focus of the current work and it seems not a trivial task. 
Here, if we use $L$ samples $u_{\ell}^i$ to approximate the integral, we need to calculate $L$ forward problems, i.e., solving $L$ PDEs during the computations. 

At the beginning of training, the prior measure \(\mathbb{P}_{S_i}^{\theta}\) has large support, which suggests that a large \(L\) is needed. For a learned measure \(\mathbb{P}_{S_i}^{\theta}\), the support shrinks into smaller areas. Hence, we may employ a small \(L\). According to some related investigations \citep{Yue2024TPAMI,Jia2023JMLR}, we recognize that \(L\) may not need to be so large. If \(L\) is small, the approximation errors may be compensated through more loops of training. Choosing a very large \(L\) may waste a lot of computational resources at the beginning of training, since a comparably accurate approximation is not needed at the beginning stage. Since the problems are too complicated to provide a theoretical illustration, we present some numerical discussions below. This problem seems not highly relevant to the environment setting, so we only discuss the backward diffusion problem under the simple environment setting.

To provide intuition regarding the parameter \( L \) and the maximum number of iterations denoted as \( N_{\text{max}} \), we test the learning algorithm with  
\begin{align*}
	L & = 1, 5, 10, 15, 20,  \\ 
	N_{\text{max}} & = 500, 1000, 1500, 2000, 2500, 3000, 3500, 4000, 
\end{align*}  
for the backward diffusion equations under the simple setting with the data-dependent prior $\mathbb{P}_{S}^{\theta}=\mathcal{N}(f_m(S;\theta_1),\mathcal{C}_0(\theta_2))$. We learn the prior mean functions under the above settings and compare the estimated mean function with those from 50 test samples. Smaller relative errors clearly indicate better performance. In Table \ref{TableAnalysisL}, we present the relative errors averaged over 50 test samples. As the maximum iteration number \( N_{\text{max}} \) increases, the performance improves even when \( L = 1 \). As the sampling number \( L \) increases (especially when \( L \geq 5 \)), the improvement in the obtained estimates is minimal. At least in this experiment, increasing the maximum iteration number appears more efficient than increasing the parameter \( L \). In our understanding, similar scenarios are commonly encountered in the field of machine learning; see, for example, \cite{Xie2019AAAI} (where only 4 steps of MCMC sampling are used).

\begin{table}[ht!]
	\caption{Average relative errors of the learned mean function for the base prior measure $\mathbb{P}^{\theta}=\mathcal{N}(f_m(S;\theta_1),\mathcal{C}_0(\theta_2))$, computed over 50 test samples, for the backward diffusion problem under a simple environmental setting.}
	\begin{center}
		\vskip 0.2 cm
		\begin{tabular}{c|ccccc}
			\Xhline{1.1pt}
			$\qquad$$\qquad$$\quad$$\qquad$ & $\,\,\,\,$$L=1$$\,\,\,\,$ & $\,\,\,\,$$L=5$$\,\,\,\,$ & $\,\,\,$$L=10$$\,\,\,$ & $\,\,\,$$L=15$$\,\,\,$ & $\,\,\,$$L=20$$\,\,\,$ \\
			\hline
        	$N_{\text{max}}=500$  & 0.64381 & 0.65140 & 0.66780 & 0.63853 & 0.79176 \\
        	$N_{\text{max}}=1000$ & 0.60163 & 0.61871 & 0.67920 & 0.62381 & 0.65089 \\
        	$N_{\text{max}}=1500$ & 0.57476 & 0.58199 & 0.59338 & 0.57781 & 0.63865 \\
        	$N_{\text{max}}=2000$ & 0.46756 & 0.54002 & 0.57640 & 0.51437 & 0.50309 \\
        	$N_{\text{max}}=2500$ & 0.51841 & 0.49031 & 0.50162 & 0.43469 & 0.61830 \\
        	$N_{\text{max}}=3000$ & 0.42851 & 0.34553 & 0.30276 & 0.30557 & 0.40400 \\
        	$N_{\text{max}}=3500$ & 0.38195 & 0.23546 & 0.23547 & 0.23434 & 0.28699 \\
        	$N_{\text{max}}=4000$ & 0.34303 & 0.20873 & 0.20528 & 0.21346 & 0.20832 \\
			\Xhline{1.1pt}
		\end{tabular}
	\end{center}\label{TableAnalysisL}
\end{table}

\textbf{Discussion of the hidden features in FNO.}  
In this work, we used hidden channels with dimensions 15 for the backward diffusion. The number of hidden channels used here is much smaller than in the original paper \cite{Li2020ICLR}. Here, we provide some explanations. The objective of the current work is to learn the base prior mean function, which can be viewed as a rough guess providing an inaccurate estimate. This inaccurate estimate is refined by computing the MAP or posterior mean of the base posterior. Since we have a learned prior encoding useful (though insufficiently accurate) information, the computation of the MAP and posterior mean is faster compared to methods relying solely on an unlearned prior. These settings resemble meta-learning, which accelerates the retraining process and reduces the required number of training data pairs. In our case, we reduce computational time and enhance estimation accuracy for a new inverse task by encoding information from related historical inverse tasks into the learned prior measure.

Unlike typical applications of FNOs, we lack supervised learning data pairs (i.e., dataset \( S \) and true function parameter \( u \)). Instead, we have a series of datasets \( \{S_i\}_{i=1}^{n} \) generated by  
\[
\bm{y}_i = \mathcal{L}_{\bm{x}_i}\mathcal{G}(u_i) + \bm{\eta}_i, \quad i = 1,\ldots,n,
\]  
where \( \bm{x}_{i} = (x_{i1},\ldots,x_{im_i})^T \), \( \bm{y}_i = (y_{i1},\dots,y_{im_i})^T \), and \( S_i = \{(x_{ij}, y_{ij})\}_{j=1}^{m_i} \). The true parameters \( \{u_i\}_{i=1}^{n} \) are unknown, as inverse problems for PDEs rarely provide access to true values. Given this setting, FNOs cannot be expected to generate a mean function that fully captures the true function's detailed structure. Instead, FNOs capture the general trend and large-scale characteristics, which (compared to typical FNO applications) may not require complex architectures. Indeed, in meta-learning, smaller-scale neural networks are preferred for constructing meta-learners (see works \cite{Shu2021JMLR,ShuYuan2023TPAMI}, and \cite{ShuZhu2023TPAMI}). We build upon this intuition to adopt a simpler neural network architecture. 

In the following, we present numerical illustrations for the backward diffusion problem under complex environmental settings with varying hidden channel dimensions (denoted as $D_h$). Specifically, we set $D_h = 5, 10, 15, 20, 25, 30$ to analyze the effect of the hidden dimension. Using 50 test sample datasets, we compute the average relative errors of estimates predicted by the learned mean function $f_m(S;\theta_1)$ of the base prior, where smaller average relative errors indicate better performance. For other parameters, we fix $L=10$ and the maximum iteration number $N_{\text{max}}=4000$ across all $D_h$ values. Under these settings, the average relative errors are found to be  
$$0.46115, \,\, 0.52680, \,\, 0.28499, \,\, 0.19798, \,\, 0.18275, \,\, 0.17392$$  
for hidden channel dimensions $D_h = 5, 10, 15, 20, 25$ and $30$, respectively. From these results, we observe that the algorithm's performance does not improve significantly when the hidden dimension $D_h \geq 15$. In our numerical experiment, using more hidden channels seems to enhance the performance a little bit, but a larger neural operator requires more computational resources.
}

\subsubsection*{Numerical results for the Darcy flow problem}
{\color{black}
In this section, we present additional numerical results and provide more detailed interpretations of certain subtle phenomena observed in the experiments. For methodological consistency and to avoid committing the inverse crime \citep{Kaipio2004Book}—an issue not discussed in the main text due to page limitations—we adopt the following experimental setup: the dataset is generated using a fine spatial discretization on a $200 \times 200$ uniform grid, while all computations involving prior learning, as well as the estimation of the MAP point, posterior mean, and credible intervals, are performed on a coarser $50 \times 50$ grid.

\textbf{Sequential Monte Carlo algorithm.} In the main text, we mentioned that the sequential Monte Carlo (SMC) algorithm has been utilized to generate samples from the base posterior and we used the total number of temperatures $K_{\text{total}}$ to measure the sampling speed. Hence, for the reader's convenience, we would better give a short introduction. The basic ideas of SMC type algorithms are dividing the sampling procedure into a series of easier sampling problems. Let us introduce a series of positive numbers $\{h_k\}_{j=k}^{K_{\text{total}}}$ usually called temperatures such that 
\begin{align*} 
	\sum_{k=1}^{K_{\text{total}}}h_k = 1. 
\end{align*} 
Then the original problem 
\begin{align*}
	\frac{d\mathbb{Q}}{d\mathbb{P}}(u) \propto \exp\left( -\frac{\beta}{m}\sum_{j=1}^{m}\Phi(u;x_j,y_j) \right) 
\end{align*}
decomposed into a series of problems
\begin{align*}
	\frac{d\mathbb{Q}^{K}}{d\mathbb{P}}(u) \propto \exp\left( -\sum_{k=1}^{K}h_k\frac{\beta}{m}\sum_{j=1}^{m}\Phi(u;x_j,y_j) \right) 
\end{align*}
with $K = 1,\ldots,K_{\text{total}}$. Obviously, we have 
\begin{align}
	\frac{d\mathbb{Q}^{K+1}}{d\mathbb{Q}^K} \propto \exp\left( -h_{K+1}\frac{\beta}{m}\sum_{j=1}^{m}\Phi(u;x_j,y_j) \right).
\end{align}
Since the parameter $h_K$ is typically small, one expects that the measure $\mathbb{Q}^{K+1}$ is close to $\mathbb{Q}^K$ in an appropriate sense. Motivated by this observation, we draw $N_s$ independent samples $\{u_k^0\}_{k=1}^{N_s}$, referred to as particles, from the prior distribution $\mathbb{P}$. Leveraging the proximity between consecutive measures in the sequence $\{\mathbb{Q}^K\}_{K=1}^{K_{\text{total}}}$, we can efficiently transform particles sampled from $\mathbb{Q}^K$ into approximate samples from $\mathbb{Q}^{K+1}$. By iterating this transformation over $K_{\text{total}}$ steps, we ultimately obtain a set of samples from the base posterior measure $\mathbb{Q} = \mathbb{Q}^{K_{\text{total}}}$. The transition between these measures is typically carried out through the following three-step procedure:
\begin{itemize}
	\item Re-weighting:  Importance sampling is employed to transition particles from the measure $\mathbb{Q}^K$ to $\mathbb{Q}^{K+1}$. This step involves assigning each particle an importance weight based on its likelihood under the updated measure, thereby adjusting the contribution of each particle in representing the new distribution.
	\item Resampling: Following the re-weighting stage, resampling is performed to eliminate particles with low importance weights and to replicate those with high weights. This helps maintain numerical stability and prevents degeneracy in the particle approximation. However, repeated resampling may reduce sample diversity over successive iterations.
	\item Mutation: To recover lost diversity and ensure better exploration of the support of $\mathbb{Q}^{K+1}$, a mutation step is introduced. This involves applying a $\mathbb{Q}^{K+1}$-invariant Markov transition kernel to each particle, typically implemented via a Markov chain Monte Carlo (MCMC) method, allowing the particles to explore the posterior distribution more effectively.
	
\end{itemize}
Choosing the temperatures is extremely important for SMC algorithm. If these temperatures are too dense, the SMC will run slowly; if they are too sparse, the adjacent posteriors \(\mathbb{Q}^{K}\) and \(\mathbb{Q}^{K+1}\) will differ significantly, which is not conducive to importance sampling. Following the works \cite{Beskos2015SC} and \cite{Lu2024arXiv}, we use the effective sample size (ESS) to select appropriate temperatures:
\begin{align*}
	\text{ESS} := \left(\sum_{i=1}^{N_s} w_i^2\right)^{-1},
\end{align*}
where \(w_i\) is the weight of the \(i\)-th particle. A small \(h_k\) can result in a large ESS. An excessively large value of \(h_k\) can cause the loss of a large number of samples during the resampling step, thereby leading to a small ESS. Thus, we can use a simple bisection method to find an \(h_k\) such that \(\text{ESS} > N_{\text{thresh}} := 0.6 N_s\). From the above brief introduction of SMC, we know that the computational cost rough corresponding to the number $K_{\text{total}}$ when the particle numbers are fixed. 

In the main text, we primarily present results obtained using the mixture Gaussian based sequential Monte Carlo (MGSMC) algorithm \citep{Lu2024arXiv}, motivated by considerations of computational efficiency. As the name suggests, the MGSMC algorithm replaces the mutation step with sampling from a mixture Gaussian distribution that is adaptively estimated from the current set of particles. This approximation significantly reduces the computational burden compared to full MCMC-based mutation step.

\textbf{Posterior mean. }In the main text, we focused on reporting the average relative errors of the MAP estimates obtained from three different approaches: the unlearned method, the learned data-independent method, and the learned data-dependent method. It is also informative to examine the corresponding average relative errors of the posterior mean estimates, which provide another characterization of the inferred solutions.

In Table~\ref{TableSMCDarcyFlowMain}, we present the average relative errors in the posterior mean functions computed via the MGSMC algorithm \citep{Lu2024arXiv}. Due to the computational demands of the sampling procedure, these averages are computed over 10 test datasets rather than 50. In both the simple and complex environmental settings, the learned data-independent method performs slightly worse than the unlearned method. This degradation is attributed to the fact that the learned mean function introduces biased information for solving inverse problems. In line with the results observed for the MAP estimates, the learned data-dependent method achieves the lowest error, further substantiating its effectiveness in capturing the true underlying characteristics of the solution space.

\begin{table}[ht!]
\caption{For the 10 test datasets, we list the average relative errors of the posterior mean (base posteriors) under priors $\mathbb{P}$, $\mathbb{P}^{\theta}$, and $\mathbb{P}_{S}^{\theta}$, respectively. }
\begin{center}
\begin{tabular}{c|ccc}
	\Xhline{1.1pt}
	& $\,\,\mathcal{N}(0,\mathcal{C}_0)\,\,\,$ & $\,\,\,\mathcal{N}(f_m(\theta_1),\mathcal{C}_0(\theta_2))\,\,\,$ & $\,\,\,\mathcal{N}(f_m(S;\theta_1),\mathcal{C}_0(\theta_2))\,\,\,$ \\
	\hline
	Simple environment  $\,\,$& .1588 & .1634 & .0945 \\
	Complex environment $\,\,$& .1371 & .1953 & .0885 \\
	\Xhline{1.1pt}
\end{tabular}
\end{center}\label{TableSMCDarcyFlowMain}
\end{table}

\textbf{Analysis of the MAP. }In the main text, the two branches are generated by the following randomized functions: 
\begin{align*}
	u(x_1,x_2) = (2\alpha-1)(u_1(x_1,x_2) + u_2(x_1,x_2) + u_3(x_1,x_2)),
\end{align*}
where $\alpha\sim\text{Bern}(0.5)$ and 
$$u_i(x_1,x_2) = a_{i3}(1-x_1^2)^{a_{i1}}(1-x_2^2)^{a_{i2}}e^{-a_{i4}(x_1-a_{i5})^2 - a_{i6}(x_2-a_{i7})^2}$$ 
with $i=1,2,3$.
For the parameters in the above formula, we let 
\begin{align*}
	& a_{i1}\sim U(0.1,0.5), \quad a_{i2}\sim U(0.1,0.5), \quad a_{i3}\sim U(3,4), \quad a_{i4}\sim U(30, 35), \\
	& a_{i6}\sim U(30, 35), \quad a_{15}\sim U(0.15,0.25), \quad a_{17}\sim U(0.65,0.75), \quad a_{25}\sim U(0.45,0.55), \\ 
	& a_{27}\sim U(0.45,0.55), \quad a_{35}\sim U(0.65,0.75), \quad a_{37}\sim U(0.15,0.25),
\end{align*} 
where $i=1,2,3$. One branch of functions takes positive values, and the other takes negative values.
To see clearly, we present the relative errors of the estimations for each branch separately in Table~\ref{TableDarcyFlowPositiveBranch} (positive-valued solutions) and Table~\ref{TableDarcyFlowNegativeBranch} (negative-valued solutions).

\begin{table}[ht!]
	\caption{Consider the branches with positive function values in the complex environment, we assess the average relative errors of maximum a posteriori estimates using unlearned, learned data-independent, and learned data-dependent prior measures. Here, ItN denotes the iteration count of the applied inexact matrix-free Newton-conjugate gradient method.}
	\begin{center}
		\vskip 0.2 cm
		\begin{tabular}{c|ccccc}
			\Xhline{1.1pt}
			& $\mathcal{N}(0,\mathcal{C}_0)$ & $\mathcal{N}(f_m(\theta_1),\mathcal{C}_0(\theta_2))$ & $\mathcal{N}(f_m(S;\theta_1),\mathcal{C}_0(\theta_2))$ \\
			\hline
			$\text{ItN}=1$  & .7308 & .3823 & .1727  \\
			$\text{ItN}=5$  & .2747 & .1873 & .1502  \\
			$\text{ItN}=10$ & .1574 & .1498 & .1424  \\
			$\text{ItN}=15$ & .1393 & .1489 & .1386  \\
			$\text{ItN}=20$ & .1332 & .1489 & .1377  \\
			$\text{ItN}=25$ & .1283 & .1489 & .1376  \\
			$\text{ItN}=30$ & .1258 & .1489 & .1375  \\
			\Xhline{1.1pt}
		\end{tabular}
	\end{center}\label{TableDarcyFlowPositiveBranch}
\end{table}

\begin{table}[ht!]
	\caption{Consider the branches with negative function values in the complex environment, we assess the average relative errors of maximum a posteriori estimates using unlearned, learned data-independent, and learned data-dependent prior measures. Here, ItN denotes the iteration count of the applied inexact matrix-free Newton-conjugate gradient method.}
	\begin{center}
		\vskip 0.2 cm
		\begin{tabular}{c|ccccc}
			\Xhline{1.1pt}
			& $\mathcal{N}(0,\mathcal{C}_0)$ & $\mathcal{N}(f_m(\theta_1),\mathcal{C}_0(\theta_2))$ & $\mathcal{N}(f_m(S;\theta_1),\mathcal{C}_0(\theta_2))$ \\
			\hline
			$\text{ItN}=1$  & .6952 & 1.4003 & .0425  \\
			$\text{ItN}=5$  & .4283 & 1.0054 & .0358  \\
			$\text{ItN}=10$ & .4023 & .9697  & .0358  \\
			$\text{ItN}=15$ & .4018 & .9675  & .0358  \\
			$\text{ItN}=20$ & .4013 & .9647  & .0357  \\
			$\text{ItN}=25$ & .4008 & .9624  & .0357  \\
			$\text{ItN}=30$ & .3806 & .9607  & .0357  \\
			\Xhline{1.1pt}
		\end{tabular}
	\end{center}\label{TableDarcyFlowNegativeBranch}
\end{table}

For the branch with positive-valued functions, all three approaches yield similar results as measured by the $L^2$-norm based average relative error. For the branch with negative-valued functions, the learned data-dependent Bayesian model provides superior estimates compared to the unlearned and learned data-independent Bayesian model. We examined the employed matrix-free Newton-conjugate gradient approach, which can hardly find appropriate descending directions for the unlearned and learned data-independent Bayesian model. The matrix-free Newton-conjugate gradient approach frequently meets the negative curvature condition, which makes the algorithm struggle to find an appropriate descending direction (see Page 465 of \cite{Ghattas2021ActaNum} for more details). The unlearned method performs better because it uses an initial value of 0, whereas the learned data-independent Bayesian model starts with the learned mean function $f_m(\theta_1)$, which takes positive values (see Figure 5 
in the main text), leading to slower convergence when approximating negative-valued solutions.

To validate our analysis, we executed the optimization algorithm on a test dataset sampled from the branch of negative-valued functions. The maximum iteration number was set to $10^4$ to ensure the algorithm converges before reaching the iteration limit. In Figure \ref{FigRelativeErrorMAP} below, we display the relative errors across iterations. Notably, the convergence rates in both scenarios are exceedingly slow; it is only in the final few steps - when negative curvature does not occur - that the optimization process accelerates. \textbf{The finally obtained relative errors are 0.1057 and 0.0607 for the learned data-independent and unlearned Bayesian models, respectively.} The results mean that: if we ignore the computational cost, we expect to obtain more accurate estimates for the negative-valued function branch compared with the positive-valued function branch. The proposed data-dependent learned Bayesian model achieves the MAP estimates with a much smaller computational cost. Since the estimates of the three methods are similar for the positive-valued function branch (and the proposed learned data-dependent Bayesian model consumes much less time), the final mixed relative errors are approximately determined by the negative-valued function branch.  

\begin{figure}[ht!]
	\centering
	\includegraphics[width=0.85\textwidth]{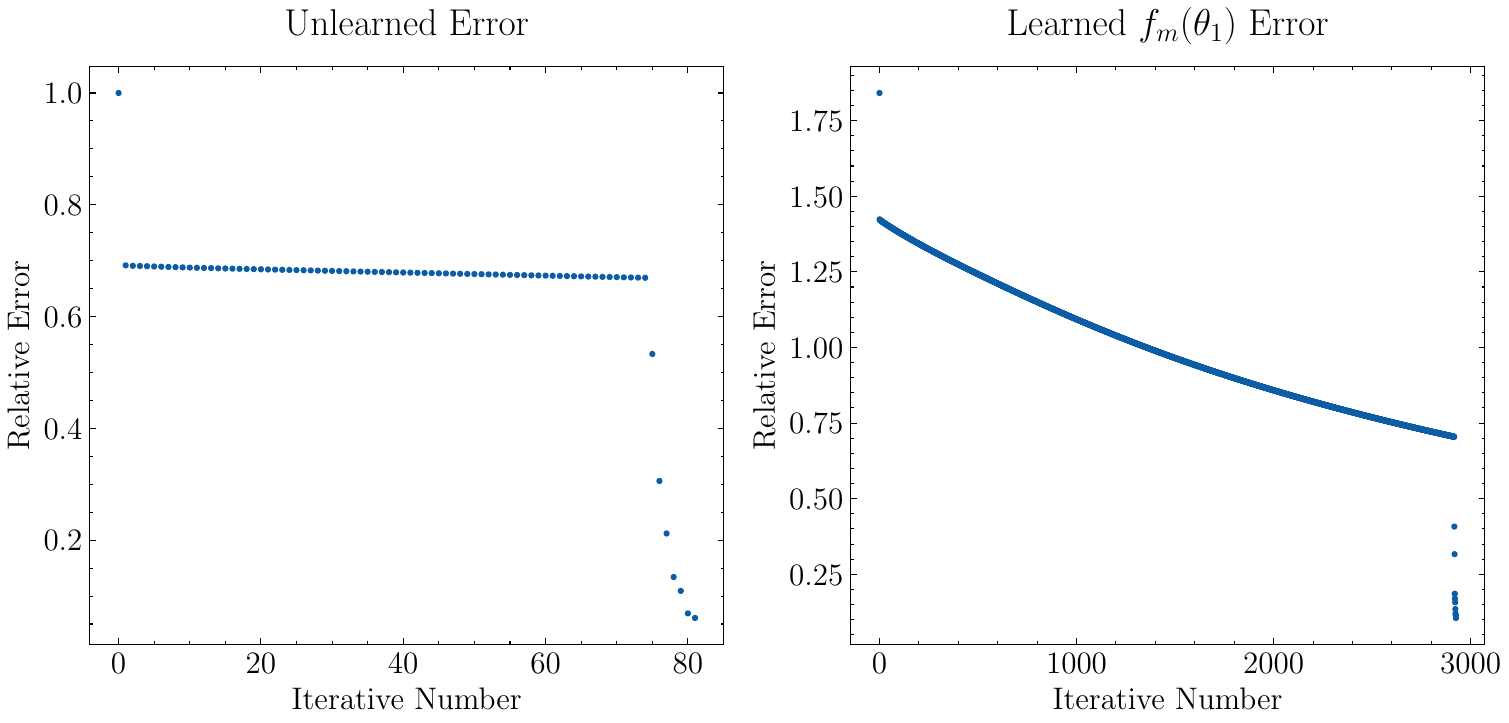}\\
	\vspace*{0.0em}
	\caption{\emph{\small Comparison of relative errors between unlearned and learned data-independent methods for a dataset generated from a negative-valued function branch.}
	}\label{FigRelativeErrorMAP}
\end{figure}

In summary, the learned mean function $f_m(\theta_1)$ provides a poor initial guess for the branch of functions assuming negative values. When the initial guess is insufficiently good, the optimization algorithm converges slowly for datasets generated from negative-valued functions. More advanced optimization algorithms may circumvent these difficulties. However, the current results just verified that our proposed learned data-dependent prior method captures the essential information from the historical inverse tasks under the complex environment situation. 

We examine the results presented in Table~\ref{TableDarcyFlowPositiveBranch} (positive-valued solutions) and Table~\ref{TableDarcyFlowNegativeBranch} (negative-valued solutions). The learned data-dependent method yields average relative errors of $0.1375$ and $0.0357$ for the positive-valued and negative-valued branches, respectively. For comparison, the learned data-independent and unlearned methods achieve relative errors of $0.1057$ and $0.0607$ upon convergence of the optimization (computed using a single randomly selected testing dataset but averaged over $50$ test datasets). Notably, the optimization algorithm provides more accurate estimates for the negative-valued branch once convergence is attained. We now provide an explanation for this observation.

A key aspect of our approach is the transformation of the randomized function $u$ via the exponential mapping $u \mapsto e^u$, where the corresponding data are obtained by solving the elliptic partial differential equation  
$$
-\nabla \cdot (e^u \nabla w) = f \quad \text{in } \Omega, \qquad w = 0 \quad \text{on } \partial\Omega.
$$
We notice that the parameter $u$ has been transformed by an exponential function to $e^u$, which rapidly stretches the positive-valued functions. Intuitively, the solution 
\begin{align*} 
w = (-\nabla\cdot e^u\nabla)^{-1}f,
\end{align*}
where $e^u$ can be thought of as occupying the denominator position. For positive-valued parameters $u$, the stretching by $e^u$ makes the data $\{w(x_j)\}_{j=1}^{m}$ exhibit only small differences for different positive-valued $u$. Hence, it is harder to obtain accurate estimates for positive-valued functions in the complex environment that contains one positive-valued branch and one negative-valued branch. Thus, the nonlinear transformation and the structure of the PDE yield different difficulties for estimating the two branches under the complex environment setting.

\begin{table}[ht!]
	\caption{Consider the complex environments, we assess the average relative errors of the posterior mean estimates using unlearned, learned data-independent, and learned data-dependent prior measures.}
	\begin{center}
		\vskip 0.2 cm
		\begin{tabular}{c|ccccc}
			\Xhline{1.1pt}
			& $\mathcal{N}(0,\mathcal{C}_0)$ & $\mathcal{N}(f_m(\theta_1),\mathcal{C}_0(\theta_2))$ & $\mathcal{N}(f_m(S;\theta_1),\mathcal{C}_0(\theta_2))$ \\
			\hline
			All of the 10 test datasets   & .13707 & .19528 & .08851  \\
			Branch of positive functions  & .17889 & .21597 & .15981  \\
			Branch of negative functions  & .06866 & .15360 & .01721  \\
			\Xhline{1.1pt}
		\end{tabular}
	\end{center}\label{TableDarcyFlowMGSMCComplexBranches}
\end{table}

To further validate our results, we compute posterior mean estimates of the base posterior measure using the MGSMC algorithm. Since sampling algorithms are computationally more expensive than optimization methods, we report relative errors averaged over 10 datasets, compared to the 50 test datasets used for the optimization approach. Table \ref{TableDarcyFlowMGSMCComplexBranches} shows the relative errors of the base posterior mean function. These results align with the conclusions from the MAP estimates: the learned data-dependent model consistently yields reliable estimates, while the negative-valued branch remains easier to estimate than the positive-valued branch. However, due to the nonlinear transformation $e^u$, the learned data-independent model generates inaccurate mean function estimates, resulting in the largest relative errors.
}

\section{Proof Details}\label{AllProof}
\subsection{Proof Details of the Main Text}\label{ProofMainApp}
In this supplementary section, we will offer detailed proofs for all the theorems and lemmas that were mentioned but not elaborated upon in the main text. These proofs are essential for a rigorous understanding of the mathematical framework and the theoretical underpinnings of the concepts discussed.
\subsubsection*{\textbf{Proof of Theorem 2.4 
		(The General PAC-Bayesian Bound)}}
\begin{proof}
	For proving Theorem 2.4, 
	we need to introduce an intermediate quantity, 
	the expected multi-task error:
	\begin{align*}
		\tilde{\mathcal{L}}(\mathscr{Q}, \mathbb{D}_1,\ldots,\mathbb{D}_n) = \mathbb{E}_{\theta\sim\mathscr{Q}}\left[ 
		\frac{1}{n}\sum_{i=1}^n \mathcal{L}(\mathbb{Q}(S_i, \mathbb{P}_{S_i}^{\theta}), \mathbb{D}_i)
		\right]. 
	\end{align*}
	Let us recall the Donsker-Varadhan variational formula (formula (A.1) in \cite{Pinski2015SIAMMA})
	\begin{align}\label{DV_VariationalFormula}
		D_{\text{KL}}(\mathsf{Q}||\mathsf{P}) = \sup_{g} \left( 
		\mathbb{E}_{\mathsf{Q}}[g] - \ln\mathbb{E}_{\mathsf{P}}[e^g]
		\right),
	\end{align}
	where the measures $\mathsf{Q}$ and $\mathsf{P}$ are defined on some separable Banach space $\mathcal{W}$, and 
	the supremum are taken all bounded measurable functions $g:\mathcal{W}\rightarrow\mathbb{R}$. 
	From the above Donsker-Varadhan variational formula (\ref{DV_VariationalFormula}), we easily obtain the following inequality
	\begin{align}\label{MeasureInequality}
		\mathbb{E}_{\mathsf{Q}}[g] \leq \frac{1}{\zeta} \Big(
		D_{\text{KL}}(\mathsf{Q}||\mathsf{P}) + \ln\mathbb{E}_{\mathsf{P}}[e^{\zeta g}]
		\Big)
	\end{align}
	holds for any positive real number $\zeta>0$ and any integrable measurable function $g$.
	
	In the following, we assume that the hyper-posterior measure $\mathscr{Q}$ is absolutely continuous with respect 
	to the hyper-prior measure $\mathscr{P}$. For every $i=1,\ldots,n$, we assume that 
	the posterior measure $\mathbb{Q}(S_i,\mathbb{P}_{S_i}^{\theta})$ 
	is absolutely continuous with respect to the prior measure $\mathbb{P}_{S_i}^{\theta}$ for all 
	$S_i\in\mathcal{Z}^{m_i}$ and $\theta\in\Theta$. If any of these absolute continuity assumptions are violated, 
	the right-hand side of the estimate (3) 
    of the main text will be infinite. Hence the estimate (3) 
    holds true. The following proof is inspired by the finite-dimensional case presented in \cite{Rothfuss2021PMLR}.
 
	\textbf{Step 1 (Base-posterior learning generalization)}:  
	The posterior measure $\mathbb{Q}$ depends on the datasets and the prior measure. 
	Given a dataset $S_i\sim\mathbb{D}_i^{m_i}$ of size $m_i$ and a prior probability measure $\mathbb{P}_{S_i}^{\theta}$, 
	we denote $\mathbb{Q}_i = \mathbb{Q}(S_i, \mathbb{P}_{S_i}^{\theta})$ be the posterior measure of 
	the parameter $u_i\in\mathcal{U}$ for $i=1,\ldots,n$. Let $S'= \cup_{i=1}^{n}S_i$ be the union of all the datasets. 
	Let $\mathcal{W} = \Theta\times\prod_{i=1}^n\mathcal{U}$, 
	$\mathsf{P}(d\theta, du_1,\ldots,du_n):=\left(\otimes_{i=1}^n\mathbb{P}_{S_i}^{\theta}(du_i)\right)\mathscr{P}(d\theta)$,  
	$\mathsf{Q}(d\theta,du_1,\ldots,du_n):=\left(\otimes_{i=1}^n\mathbb{Q}_i(du_i)\right)\mathscr{Q}(d\theta)$, and 
	\begin{align*} 
		g(\theta, u_1,\ldots,u_n) := \sum_{i=1}^{n}\sum_{j=1}^{m_i}\left( 
		\mathbb{E}_{z\sim\mathbb{D}_i}\left[ \frac{1}{nm_i}\ell(u_i, z) \right] - 
		\frac{1}{nm_i}\ell(u_i, z_{ij})
		\right).
	\end{align*} 
	Now, we can use inequality (\ref{MeasureInequality}) directly to derive
	\begin{align}\label{Estimate1Step1}
		\mathbb{E}_{\theta\sim\mathscr{Q}}\mathbb{E}_{u_1\sim\mathbb{Q}_1}\cdots\mathbb{E}_{u_n\sim\mathbb{Q}_n}[g] \leq 
		\frac{1}{\gamma}D_{\text{KL}}\left( \mathsf{Q} || \mathsf{P} \right) + \Pi^1(\gamma),
	\end{align}
	where
	\begin{align*}
		\Pi^{1}(\gamma) \!= \!\frac{1}{\gamma}\!\ln\mathbb{E}_{\theta\sim\mathscr{P}}\mathbb{E}_{u_1\sim\mathbb{P}_{S_1}^{\theta}}
		\!\!\cdots\mathbb{E}_{u_n\sim\mathbb{P}_{S_n}^{\theta}}\!\!\exp\!\!\left( \!
		\frac{\gamma}{n} \!\sum_{i=1}^{n}\!\frac{1}{m_i}\sum_{j=1}^{m_i}\!\left[ 
		\mathbb{E}_{z\sim\mathbb{D}_{i}}\ell(u_i, z) \!- \! \ell(u_i, z_{ij})
		\right]\!\!\right). 
	\end{align*}
	For the left-hand side of (\ref{Estimate1Step1}), we have 
	\begin{align}\label{Estimate2Step1}
		\mathbb{E}_{\theta\sim\mathscr{Q}}\mathbb{E}_{u_1\sim\mathbb{Q}_1}\cdots\mathbb{E}_{u_n\sim\mathbb{Q}_n}[g] = &
		\frac{1}{n}\sum_{i=1}^{n}\mathbb{E}_{\theta\sim\mathscr{Q}}\mathbb{E}_{u_i\sim\mathbb{Q}_i}\mathcal{L}(u_i,\mathbb{D}_i)
		- \frac{1}{n}\sum_{i=1}^{n}\mathbb{E}_{\theta\sim\mathscr{Q}}\hat{\mathcal{L}}(\mathbb{Q}_i, S_i)  \nonumber \\
		= & \frac{1}{n}\sum_{i=1}^{n}\mathbb{E}_{\theta\sim\mathscr{Q}}\mathcal{L}(\mathbb{Q}_i,\mathbb{D}_i)
		- \frac{1}{n}\sum_{i=1}^{n}\mathbb{E}_{\theta\sim\mathscr{Q}}\hat{\mathcal{L}}(\mathbb{Q}_i,S_i). 
	\end{align}
	For the first term on the right-hand side of (\ref{Estimate1Step1}), we have  
	\begin{align}\label{Estimate3Step1}
		\begin{split}
			D_{\text{KL}}(\mathsf{Q} || \mathsf{P}) = & \mathbb{E}_{\theta\sim\mathscr{Q}}\mathbb{E}_{u_1\sim\mathbb{Q}_1}\cdots\mathbb{E}_{u_n\sim\mathbb{Q}_n}
			\ln\left( 
			\frac{d\mathscr{Q}}{d\mathscr{P}}(\theta)\prod_{i=1}^{n}\frac{d\mathbb{Q}_i}{d\mathbb{P}_{S_i}^\theta}(u_i) 
			\right) \\
			= & \mathbb{E}_{\theta\sim\mathscr{Q}}\ln\frac{d\mathscr{Q}}{d\mathscr{P}} + 
			\sum_{i=1}^{n} \mathbb{E}_{\theta\sim\mathscr{Q}}\mathbb{E}_{u_i\sim\mathbb{Q}_i}\ln\frac{d\mathbb{Q}_i}{d\mathbb{P}_{S_i}^{\theta}} \\
			= & D_{\text{KL}}(\mathscr{Q} || \mathscr{P}) + 
			\sum_{i=1}^{n}\mathbb{E}_{\theta\sim\mathscr{Q}}D_{\text{KL}}(\mathbb{Q}_i || \mathbb{P}_{S_i}^{\theta}). 
		\end{split}
	\end{align}
	Plugging formulas (\ref{Estimate2Step1}) and (\ref{Estimate3Step1}) into inequality (\ref{Estimate1Step1}), we finally arrive at 
	\begin{align}\label{Estimate4Step1}
		\begin{split}
			\frac{1}{n}\sum_{i=1}^{n}\mathbb{E}_{\theta\sim\mathscr{Q}}\mathcal{L}(\mathbb{Q}_i, \mathbb{D}_i) \leq & \,
			\frac{1}{n}\sum_{i=1}^{n}\mathbb{E}_{\theta\sim\mathscr{Q}}\hat{\mathcal{L}}(\mathbb{Q}_i,S_i) + 
			\frac{1}{\gamma}D_{\text{KL}}(\mathscr{Q}||\mathscr{P}) \\
			& \, + \frac{1}{\gamma}\sum_{i=1}^{n}\mathbb{E}_{\theta\sim\mathscr{Q}}D_{\text{KL}}(\mathbb{Q}_i||\mathbb{P}_{S_i}^{\theta})
			+ \Pi^1(\gamma).  
		\end{split}
	\end{align}
	
	\textbf{Step 2 (Hyper-posterior learning generalization)}:  
	Now, we choose the space $\mathcal{W} = \Theta$, $\mathsf{Q} = \mathscr{Q}$, $\mathsf{P} = \mathscr{P}$, and 
	\begin{align*}
		g(\theta) = \sum_{i=1}^{n}\left( 
		\mathbb{E}_{(\mathbb{D},m)\sim\mathcal{T}}\mathbb{E}_{S\sim\mathbb{D}^{m}}\left[ \frac{1}{n}\mathcal{L}(\mathbb{Q}(S, \mathbb{P}_{S}^{\theta}), \mathbb{D}) \right] - 
		\frac{1}{n}\mathcal{L}(\mathbb{Q}(S_i,\mathbb{P}_{S_i}^{\theta}),\mathbb{D}_i) 
		\right)
	\end{align*}
	in inequality (\ref{MeasureInequality}). Then, we obtain 
	\begin{align}\label{Estimate1Step2} 
		\mathbb{E}_{\theta\sim\mathscr{Q}}[g(\theta)] \leq \frac{1}{\lambda}D_{\text{KL}}(\mathscr{Q}||\mathscr{P}) + \Pi^2(\lambda),
	\end{align}
	where 
	\begin{align*}
		\Pi^2(\lambda) \! = \! \frac{1}{\lambda} \! \ln\mathbb{E}_{\theta\sim\mathscr{P}}\exp\!\left(\!
		\frac{\lambda}{n} \sum_{i=1}^{n}\left[ 
		\mathbb{E}_{(\mathbb{D},m)\sim\mathcal{T}}\mathbb{E}_{S\sim\mathbb{D}^{m}}\!\left[ \mathcal{L}(\mathbb{Q}(S, \mathbb{P}_{S}^{\theta}), \mathbb{D}) \right] \! - \! 
		\mathcal{L}(\mathbb{Q}(S_i,\mathbb{P}_{S_i}^{\theta}),\mathbb{D}_i) 
		\right]
		\!\!\right).
	\end{align*}
	By a simple reduction for the left-hand side of (\ref{Estimate1Step2}), we obtain
	\begin{align}\label{Estimate2Step2} 
		\mathcal{L}(\mathscr{Q},\mathcal{T}) \leq \tilde{L}(\mathscr{Q}, \mathbb{D}_1, \ldots, \mathbb{D}_n)
		+ \frac{1}{\lambda}D_{\text{KL}}(\mathscr{Q}||\mathscr{P}) + \Pi^2(\lambda). 
	\end{align}
	
	\textbf{Step 3 (Combination)}:
	Combining estimates (\ref{Estimate4Step1}) and (\ref{Estimate2Step2}), we obtain 
	\begin{align}\label{Estimate1Both1}
		\begin{split}
			\mathcal{L}(\mathscr{Q}, \mathcal{T}) \leq & \, \hat{\mathcal{L}}(\mathscr{Q},S_1,\ldots,S_n) + 
			\left( \frac{1}{\lambda} + \frac{1}{\gamma} \right)D_{\text{KL}}(\mathscr{Q} || \mathscr{P})  \\
			& \, + \frac{1}{\gamma}\sum_{i=1}^{n}\mathbb{E}_{\theta\sim\mathscr{Q}}D_{\text{KL}}(\mathbb{Q}_i || \mathbb{P}_{S_i}^{\theta})
			+ \Pi^1(\gamma) + \Pi^2(\lambda). 
		\end{split}
	\end{align}
	Taking $\epsilon > 0$ be a positive small number, we have 
	\begin{align}\label{Estimate1Both3} 
		\begin{split}
			\mathbb{P}\left[ \Pi^1(\gamma) + \Pi^2(\lambda) > \ln\epsilon \right] = & 
			\mathbb{P}\left[ \exp\left( \Pi^1(\gamma) + \Pi^2(\lambda) \right) > \epsilon \right] \\
			\leq & \mathbb{P}\left[ \exp\left( \sqrt{n}\Pi^1(\gamma) + \sqrt{n}\Pi^2(\lambda) \right) > \epsilon^{\sqrt{n}} \right] \\
			\leq & \frac{\mathbb{E}\left(\exp(\sqrt{n}\Pi^1(\gamma) + \sqrt{n}\Pi^2(\lambda))\right)}{\epsilon^{\sqrt{n}}},
		\end{split}
	\end{align}
	where we employed the Markov's inequality to derive the second inequality. Let
	\begin{align}\label{Estimate1Both4}
		\delta = \frac{\mathbb{E}\left(\exp(\sqrt{n}\Pi^1(\gamma) + \sqrt{n}\Pi^2(\lambda))\right)}{\epsilon^{\sqrt{n}}},
	\end{align}
	then we find that
	\begin{align}\label{Estimate1Both5}
		\ln\epsilon = \frac{1}{\sqrt{n}}\ln\mathbb{E}\left(\exp(\sqrt{n}\Pi^1(\gamma) + \sqrt{n}\Pi^2(\lambda))\right) 
		+ \frac{1}{\sqrt{n}}\ln\frac{1}{\delta}. 
	\end{align}
	Combining the estimates (\ref{Estimate1Both3}), (\ref{Estimate1Both4}), and (\ref{Estimate1Both5}), we proved that 
	\begin{align*}
		\Pi^1(\gamma) + \Pi^2(\lambda) \leq 
		\frac{1}{\sqrt{n}}\ln\mathbb{E}\left(\exp(\sqrt{n}\Pi^1(\gamma) + \sqrt{n}\Pi^2(\lambda))\right) 
		+ \frac{1}{\sqrt{n}}\ln\frac{1}{\delta}
	\end{align*}
	holds true with probability at least $1-\delta$, which implies the desired conclusion. 
\end{proof}

\subsubsection*{\textbf{Proof of Corollary 2.10
		(Bound for the Sub-Gaussian Loss [Data-Dependent Prior])}}
\begin{proof}
	For the term $V_i^1$, we have 
	\begin{align}\label{PACDependentUnboundedLoss21}
		\begin{split}
			V_i^1 = & \mathbb{E}_{S_i\sim\mathbb{D}_{\mathcal{T}}}
			\mathbb{E}_{\theta\sim\mathscr{P}}\mathbb{E}_{u_i\sim\mathbb{P}_{S_i}^{\theta}}\exp\left( 
			\tilde{\gamma}\left[ 
			\mathcal{L}(u_i,\mathbb{D}_i) - \ell(u_i, z_{i})
			\right]\right) \\
			\leq & \mathbb{E}_{S_i^{'}\sim\mathbb{D}_{\mathcal{T}}}\mathbb{E}_{S_i\sim\mathbb{D}_{\mathcal{T}}}
			\mathbb{E}_{\theta\sim\mathscr{P}}\mathbb{E}_{u_i\sim\mathbb{P}_{S_i^{'}}^{\theta}} 
			\Psi(S_i,S_i^{'},u_i,\theta)\exp\left( \tilde{\gamma}[\mathcal{L}(u_i,\mathbb{D}_i) - \ell(u_i,z_i)] \right) \\
			\leq & \Psi_E^{1/2}
			\bigg(\mathbb{E}_{S_i^{'}\sim\mathbb{D}_{\mathcal{T}}}\mathbb{E}_{S_i\sim\mathbb{D}_{\mathcal{T}}}
			\mathbb{E}_{\theta\sim\mathscr{P}}\mathbb{E}_{u_i\sim\mathbb{P}_{S_i^{'}}^{\theta}} 
			\exp\left( 2\tilde{\gamma}[\mathcal{L}(u_i,\mathbb{D}_i) - \ell(u_i,z_i)] \right)\bigg)^{1/2} \\
			\leq & \Psi_E^{1/2}
			\exp\left( \tilde{\gamma}^2s_{\text{I}}^2\right).
		\end{split}
	\end{align}
	Through some simple calculations as for the bounded loss case shown in Corollary \ref{PACBoundsBoundedLoss2}, we find that 
	\begin{align}\label{GeneralSubGammaLossS22}
		\begin{split}
			\Big(\mathbb{E}e^{2\sqrt{n}\Pi^1(\gamma)}\Big)^{1/2} \leq & \Psi_E^{\frac{n\sqrt{n}}{2\gamma}}
			\exp\left( \frac{\gamma s_{\text{I}}^2}{\sqrt{n}\bar{m}} \right).
		\end{split}
	\end{align}
	Combining estimates on the term $\mathbb{E}e^{2\sqrt{n}\Pi^2(\lambda)}$ (not given here since it is 
	similar to the data-independent case shown in Corollary \ref{PACBoundsUnoundedLossThmSubGaussian}), we obtain 
	\begin{align}\label{GeneralSubGammaLossS23}
		\begin{split}
			\mathbb{E}\left[ e^{\sqrt{n}\Pi^1(\gamma) + \sqrt{n}\Pi^2(\lambda)} \right] \leq \Psi_E^{\frac{n\sqrt{n}}{2\gamma}}
			\exp\left( \frac{\gamma s_{\text{I}}^2}{\sqrt{n}\bar{m}} + 
			\frac{\lambda s_{\text{II}}^2}{2\sqrt{n}} \right).
		\end{split}
	\end{align}
	Combining estimate (\ref{GeneralSubGammaLossS23}) with the results given in Theorem 2.4 
    of the main text, we obtain the desired result.
\end{proof}

\subsubsection*{\textbf{Proof of Theorem 3.3
		(Bayesian Well-Posedness of Inverse Problems)}} 
\begin{proof}In the following, we only present the proofs when $\mathcal{H}$ is assumed to be an infinite-dimensional separable Hilbert space. When $\mathcal{H}=\mathbb{R}^{N_d}$, the proof is similar.
	
	\textbf{Step 1 (Well defined Bayes' formula).}
	Let us define $\mathbb{P}_0:=\mathcal{N}(0,\mathcal{C}_0)$.
	Since $f(S;\theta)\in\mathcal{C}_0^{\frac{1}{2}}\mathcal{U}$, 
	the measure $\mathbb{P}_S^{\theta}$ is equivalent to the measure $\mathbb{P}_0$. 
	Then we transform the Bayes' formulation as follow
	\begin{align}\label{ProofBayesFormula1}
		\frac{d\mathbb{Q}(S,\mathbb{P}_S^{\theta})}{d\mathbb{P}_0} = \frac{1}{Z_m}\exp\left( 
		-\frac{\beta}{m}\sum_{j=1}^{m}\Phi(u;x_j,y_j)
		\right)\frac{d\mathbb{P}_S^{\theta}}{d\mathbb{P}_0}.
	\end{align}
	If formula (\ref{ProofBayesFormula1}) is well defined, we know that formula (6) 
    of the main text is well defined. {\color{black}Let $u_k := \langle u, e_k \rangle$, where $\{e_k\}_{k=1}^{\infty}$ denotes the complete orthonormal system of eigenfunctions associated with the covariance operator $\mathcal{C}_0$.
	Based on Theorem 2.23 in \cite{Prato2014Book}, we find that 
	\begin{align*}
		\frac{d\mathbb{P}_{S}^{\theta}}{d\mathbb{P}_0}(u) = \exp\left( -\tilde{\Phi}(u; S) \right),
	\end{align*}
	where 
	\begin{align*}
		\tilde{\Phi}(u;S) := & -\left\langle \mathcal{C}_0(\theta_2)^{-\frac{1}{2}}f_m(S;\theta_1),\mathcal{C}_0(\theta_2)^{-\frac{1}{2}}(u-f_m(S;\theta_1)) \right\rangle_{\mathcal{U}}
		+\frac{1}{2}\left\| \mathcal{C}_0(\theta_2)^{-\frac{1}{2}}f_m(S;\theta_1) \right\|_{\mathcal{U}}^2 \\
		& + \frac{1}{2}\sum_{k=1}^{N_c}(e^{-\theta_{2k}}-\lambda_{k}^{-1})u_k^2
		+ \frac{1}{2}\sum_{k=1}^{N_c}(\theta_{2k} - \ln\lambda_k).
	\end{align*}
	Then, formula (\ref{ProofBayesFormula1}) can be written as 
	\begin{align}\label{ProofBayesFormula2}
		\frac{d\mathbb{Q}(S,\mathbb{P}_S^{\theta})}{d\mathbb{P}_0} = \frac{1}{Z_m}\exp\left( 
		-\frac{\beta}{m}\sum_{j=1}^{m}\Phi(u;x_j,y_j) - \tilde{\Phi}(u;S)
		\right).
	\end{align}}
	Considering the Assumptions 3.1 
    of the main text and Lemma 3.3 in \cite{Agapiou2013SPA}, 
	we notice that $u\in \mathcal{U}^{1-\tilde{s}}$ $\mathbb{P}_S^{\theta}$-a.s. when $\tilde{s}\in(s_0,1]$.
	According to the theory presented in Section 4.1 of \cite{Dashti2017}, we need to verify
	\begin{itemize}
		\item (\textbf{continuity condition}) the potential function 
		$$\frac{\beta}{m}\sum_{j=1}^m\Phi(u;x_j,y_j) + \tilde{\Phi}(u;S)\in C(\mathcal{U}^{1-\tilde{s}}\times(\mathcal{H}^{\alpha-s})^{m};\mathbb{R}),$$ 
		\item (\textbf{lower bound condition}) for all $u\in\mathcal{U}^{1-\tilde{s}}$, 
		$y_j\in B_{\mathcal{H}^{\alpha-s}}(r):=\{y\in\mathcal{H}:\|y\|_{\mathcal{H}^{\alpha-s}}< r\}$
		$$\frac{\beta}{m}\sum_{j=1}^m\Phi(u;x_j,y_j) + \tilde{\Phi}(u;S) \geq -N(r,\|u\|_{\mathcal{U}^{1-\tilde{s}}}),$$
		where $N:\mathbb{R}^{+}\times\mathbb{R}^{+}\rightarrow\mathbb{R}^{+}$ is monotonic non-decreasing separately in each argument.
		In addition, $\exp\left( N(r,\|u\|_{\mathcal{U}^{1-\tilde{s}}}) \right) \in L_{\mathbb{P}_0}^1(\mathcal{U}^{1-\tilde{s}};\mathbb{R})$. 
	\end{itemize}
	
	\textbf{Continuity of $\frac{\beta}{m}\sum_{j=1}^m\Phi(u;x_j,y_j)$.}
	For the summation $\frac{\beta}{m}\sum_{j=1}^m\Phi(u;x_j,y_j)$, it is enough to consider one of its term $\Phi(u;x_j,y_j)$. 
	For the continuity of $u$, we have 
	\begin{align}\label{ProofBayesFormula3}
		\begin{split}
			|\Phi(u_1;x_j,y_j) - \Phi(u_2;x_j,y_j)| = & \frac{1}{2}\left| \|\Gamma^{-\frac{1}{2}}\mathcal{L}_{x_j}\mathcal{G}(u_1)\|_{\mathcal{H}}^2 - \|\Gamma^{-\frac{1}{2}}\mathcal{L}_{x_j}\mathcal{G}(u_2)\|_{\mathcal{H}}^2\right| \\
			& + \left| \left\langle \Gamma^{-\frac{1}{2}}y_j, \Gamma^{-\frac{1}{2}}\mathcal{L}_{x_j}\mathcal{G}(u_1)- \Gamma^{-\frac{1}{2}}\mathcal{L}_{x_j}\mathcal{G}(u_2) \right\rangle_{\mathcal{H}} \right| \\
			= & \text{I}_1 + \text{I}_2 
		\end{split}
	\end{align}
	For the first term on the right-hand side of (\ref{ProofBayesFormula3}), we have 
	\begin{align}\label{ProofBayesFormula4}
		\begin{split}
			\text{I}_1 \leq & \frac{\tilde{C}^2}{2}(M_1(\|u_1\|_{\mathcal{U}^{1-\tilde{s}}})+M_1(\|u_2\|_{\mathcal{U}^{1-\tilde{s}}}))\| \mathcal{C}_1^{-\frac{s}{2}}\Gamma^{-\frac{1}{2}}\mathcal{L}_{x_j}(\mathcal{G}(u_1)- 
			\mathcal{G}(u_2))\|_{\mathcal{H}} \\
			\leq & \frac{\tilde{C}^2}{2}(M_1(\|u_1\|_{\mathcal{U}^{1-\tilde{s}}})+M_1(\|u_2\|_{\mathcal{U}^{1-\tilde{s}}}))
			M_2(\|u_1\|_{\mathcal{U}^{1-\tilde{s}}},\|u_2\|_{\mathcal{U}^{1-\tilde{s}}})\|u_1-u_2\|_{\mathcal{U}^{1-\tilde{s}}},
		\end{split}
	\end{align}
	where $\tilde{C} := \|\mathcal{C}_1^{\frac{s}{2}}\|_{\mathcal{B}(\mathcal{H})}$ is the operator norm of $\mathcal{C}_1^{\frac{s}{2}}$. 
	For the second term on the right-hand side of (\ref{ProofBayesFormula3}), we have 
	\begin{align}\label{ProofBayesFormula5}
		\begin{split}
			\text{I}_2 \leq & \|y_j\|_{\mathcal{H}^{\alpha-s}}\| \mathcal{C}_1^{-\frac{s}{2}}\Gamma^{-\frac{1}{2}}(\mathcal{L}_{x_j}\mathcal{G}(u_1)- \mathcal{L}_{x_j}\mathcal{G}(u_2))\|_{\mathcal{H}} \\
			\leq & \|y_j\|_{\mathcal{H}^{\alpha-s}} M_2(\|u_1\|_{\mathcal{U}^{1-\tilde{s}}},\|u_2\|_{\mathcal{U}^{1-\tilde{s}}})\|u_1-u_2\|_{\mathcal{U}^{1-\tilde{s}}}.
		\end{split}
	\end{align}
	From the estimates (\ref{ProofBayesFormula3}), (\ref{ProofBayesFormula4}), and (\ref{ProofBayesFormula5}), we obtain the continuity
	of $\Phi$ for the parameter $u$. {\color{black}For the variable $y_j$, we have 
	\begin{align*}
			|\Phi(u;x_j,y_j) - \Phi(u;x_j,y'_j)| & = \left| \left\langle 
			\Gamma^{-\frac{1}{2}}(y_j-y'_j), \Gamma^{-\frac{1}{2}}\mathcal{L}_{x_j}\mathcal{G}(u)
			\right\rangle_{\mathcal{H}} \right| + 
			\frac{1}{2}\sum_{k=1}^{N_c}|e^{-\theta_{2k}} - \lambda_k^{-1}||u_{1k}^2 - u_{2k}^2| \\
			\leq & \|y_j-y'_j\|_{\mathcal{H}^{\alpha-s}}M_1(\|u\|_{\mathcal{U}^{1-\tilde{s}}})+ \frac{1}{2}N_c(e^{-\theta_{\text{min}}} + \lambda_{N_c}^{-1})\max_{1\leq k\leq N_c}|u_{1k}^2 - u_{2k}^2|,
	\end{align*}
	which implies the continuity of $\Phi$ for the variable $y_j$.} Hence, we proved 
	\begin{align}\label{ProofBayesFormula6}
		\frac{\beta}{m}\sum_{j=1}^m\Phi(u;x_j,y_j) \in C(\mathcal{U}^{1-\tilde{s}}\times(\mathcal{H}^{\alpha-s})^{m};\mathbb{R}).
	\end{align}
	
	\textbf{Continuity of $\tilde{\Phi}(u;S)$.} For the continuity of $u$, we have 
	\begin{align}\label{ProofBayesFormula7}
		\begin{split}
			|\tilde{\Phi}(u_1, S)-\tilde{\Phi}(u_2, S)| \leq & \left| \left\langle 
			\mathcal{C}_0(\theta_2)^{-\frac{1}{2}}f_m(S;\theta_1), \mathcal{C}_0(\theta_2)^{-\frac{1}{2}}(u_1-u_2)
			\right\rangle_{\mathcal{U}} \right| \\
			\leq & \lambda_1 e^{-\theta_{\text{min}}} \|f_m(S;\theta_1)\|_{\mathcal{U}^{1+\tilde{s}}}\|u_1-u_2\|_{\mathcal{U}^{1-\tilde{s}}}.
		\end{split}
	\end{align}
	{\color{black}Let $S=\{(x_j,y_j)\}_{j=1}^{m}$ and $S'=\{(x_j,y'_j)\}_{j=1}^{m}$, we have 
	\begin{align}\label{ProofBayesFormula8}
		\begin{split}
			|\tilde{\Phi}(u,S) -  \tilde{\Phi}& (u,S')| = \left| 
			\left\langle \mathcal{C}_0(\theta_2)^{-\frac{1}{2}}(f_m(S;\theta_1)-f(S';\theta_1)), \mathcal{C}_0(\theta_2)^{-\frac{1}{2}}u \right\rangle_{\mathcal{U}}\right| \\
			& + \frac{3}{2}\left|\|\mathcal{C}_0(\theta_2)^{-\frac{1}{2}}f(S';\theta_1)\|_{\mathcal{U}}^2
			- \|\mathcal{C}_0(\theta_2)^{-\frac{1}{2}}f_m(S;\theta_1)\|_{\mathcal{U}}^2
			\right|  \\
			\leq & \lambda_1 e^{-\theta_{\text{min}}}\|f_m(S;\theta_1) - f(S';\theta_1)\|_{\mathcal{U}^{1+\tilde{s}}}\|u\|_{\mathcal{U}^{1-\tilde{s}}} \\
			& + \frac{3\lambda_1 e^{-\theta_{\text{min}}}}{2}\left( \|f_m(S;\theta_1)\|_{\mathcal{U}^{1}} + \|f(S';\theta_1)\|_{\mathcal{U}^{1}} \right)
			\|f_m(S;\theta_1) - f(S';\theta_1)\|_{\mathcal{U}^{1}}. 
		\end{split}
	\end{align}
	}Combining the estimates (\ref{ProofBayesFormula7}) and (\ref{ProofBayesFormula8}), 
	and the following assumptions
        \begin{align}\label{ContinuityOfFS}
		\|f(S;\theta) - f(S';\theta)\|_{\mathcal{U}^{1+\tilde{s}}}\rightarrow 0, \quad 
		\text{as }x_j\rightarrow x'_j \text{ in }\mathcal{X}\text{ and }
		y_j\rightarrow y'_j\text{ in }\mathcal{\mathcal{H}}^{\alpha-s},
	\end{align}
	we obtain $\tilde{\Phi}(u;S)\in C(\mathcal{U}^{1-\tilde{s}}\times(\mathcal{H}^{\alpha-s})^{m};\mathbb{R})$.
	
	{\color{black}To verify the lower bound condition, we have 
	\begin{align}\label{ProofBayesFormula9}
		\begin{split}
			\frac{\beta}{m}\sum_{j=1}^{m}\Phi(u;x_j,y_j) & \geq -\frac{\beta}{m}\sum_{j=1}^{m} \|\mathcal{C}_1^{\frac{s}{2}}\Gamma^{-\frac{1}{2}}y_j\|_{\mathcal{H}} 
			\|\mathcal{C}_1^{-\frac{s}{2}}\Gamma^{-\frac{1}{2}}\mathcal{L}_{x_j}\mathcal{G}(u)\|_{\mathcal{H}} \\
			& \geq - \frac{\beta}{m}\sum_{j=1}^{m}\|y_j\|_{\mathcal{H}^{\alpha-s}}M_1(\|u\|_{\mathcal{U}^{1-\tilde{s}}}) \\
			& \geq -\frac{1}{4\delta_1}\left(\frac{\beta}{m}\sum_{j=1}^{m}\|y_j\|_{\mathcal{H}^{\alpha-s}}\right)^2 - \delta_1 M_1(\|u\|_{\mathcal{U}^{1-\tilde{s}}})^2,
		\end{split}
	\end{align}
	and 
	\begin{align}\label{ProofBayesFormula10}
		\begin{split}
			\tilde{\Phi}(u,S) = & -\left\langle \mathcal{C}_0(\theta_2)^{-\frac{1}{2}}f_m(S;\theta_1),\mathcal{C}_0(\theta_2)^{-\frac{1}{2}}u \right\rangle_{\mathcal{U}} + \frac{3}{2}\|\mathcal{C}_0(\theta_2)^{-\frac{1}{2}}f_m(S;\theta_1)\|_{\mathcal{U}}^2 \\
			& +\frac{1}{2}\sum_{k=1}^{N_c}(e^{-\theta_{2k}}-\lambda_{k}^{-1})u_k^2 
			+ \frac{1}{2}\sum_{k=1}^{N_c}(\theta_{2k} - \ln\lambda_{k})  \\
			\geq & -\lambda_1e^{-\theta_{\text{min}}}\|f_m(S;\theta_1)\|_{\mathcal{U}^{1+\tilde{s}}}\|u\|_{\mathcal{U}^{1-\tilde{s}}} + 
			\frac{3\lambda_1}{2e^{\theta_{\text{min}}}}\|\mathcal{C}_0^{-\frac{1}{2}}f_m(S;\theta_1)\|_{\mathcal{U}}^2 \\ 
			& +\frac{1}{2}\sum_{k=1}^{N_c}(e^{-\theta_{2k}}-\lambda_{k}^{-1})u_k^2 
			+ \frac{1}{2}\sum_{k=1}^{N_c}(\theta_{2k} - \ln\lambda_{k})     \\ 
			\geq & -\delta_2 \|u\|_{\mathcal{U}^{1-\tilde{s}}}^2 - 
			\left( \frac{\lambda_1^2}{4\delta_2e^{2\theta_{\text{min}}}} - \frac{3\lambda_1}{2e^{\theta_{\text{min}}}} \right)\|f(S;\theta)\|_{\mathcal{U}^{1+\tilde{s}}}^2 \\
			& +\frac{1}{2}\sum_{k=1}^{N_c}(e^{-\theta_{2k}}-\lambda_{k}^{-1})u_k^2 
			+ \frac{1}{2}\sum_{k=1}^{N_c}(\theta_{2k} - \ln\lambda_{k}).
		\end{split}
	\end{align}
	According to the assumptions on $M_1(\cdot)$ and Proposition 1.13 in \cite{Prato2006IDAnalysis}, we know that 
	\begin{align}\label{AppliLowerBound11}
		\mathbb{E}_{u\sim\mathbb{P}_{S}^{\theta}}\exp\left( \delta_1  M_1(\|u\|_{\mathcal{U}^{1-\tilde{s}}})^2  + 
		\delta_2 \|u\|_{\mathcal{U}^{1-\tilde{s}}}^2 \right) < +\infty,
	\end{align}
	when $\delta_1$ and $\delta_2$ are chosen to be small enough positive numbers.}
	Considering the above inequality (\ref{AppliLowerBound11}) and $\|f(S;\theta)\|_{\mathcal{U}^{1+\tilde{s}}}\leq M(r)$ if $\|y_j\|_{\mathcal{H}^{\beta-s}} \leq r$ ($j=1,\ldots,m$), we easily find that the lower bound condition is fulfilled. 
	
	\textbf{Step 2 (Stability with respect to datasets).}
	From the definition of Hellinger distance, we have 
	\begin{align*}
		d_{\text{Hell}}(\mathbb{Q}(S,\mathbb{P}_{S}^{\theta}), \mathbb{Q}(S',\mathbb{P}_{S'}^{\theta}))^2 \leq I_1 + I_2,
	\end{align*}
	where
	\begin{align*}
			I_1 = & \frac{1}{Z_m} \int_{\mathcal{U}^{1-s}}\Bigg[ 
			\exp\left( -\frac{\beta}{2m}\sum_{j=1}^{m}\Phi(u;x_j,y_j) - \frac{1}{2}\tilde{\Phi}(u;S) \right) \\
			& \qquad\qquad\quad\,\,
			- \exp\left( -\frac{\beta}{2m}\sum_{j=1}^{m}\Phi(u;x'_j,y'_j) - \frac{1}{2}\tilde{\Phi}(u;S') \right)
			\Bigg]^2 \mathbb{P}_0(du),
	\end{align*}
	and 
	\begin{align*}
			I_2 = & \left| Z_m^{-\frac{1}{2}} - (Z'_m)^{-\frac{1}{2}} \right|^2 \int_{\mathcal{U}^{1-\tilde{s}}} 
			\exp\left( -\frac{\beta}{m}\sum_{j=1}^{m}\Phi(u;x'_j,y'_j) - \tilde{\Phi}(u;S') \right) \mathbb{P}_0(du),
	\end{align*}
	where $Z'_m$ is the normalization constant concerned with the datasets $S'$.
	For the term $I_1$, we have 
	\begin{align*}
			I_1 \leq \frac{1}{2Z_m}\int_{\mathcal{U}^{1-\tilde{s}}}e^{N(r,\|u\|_{\mathcal{U}^{1-\tilde{s}}})} (I_{11} + I_{12}) \mathbb{P}_0(du),
	\end{align*}
	where
	\begin{align*}
		I_{11} & = \left| \frac{\beta}{m}\sum_{j=1}^{m}\left[ 
		\Phi(u;x_j,y_j) - \Phi(u;x'_j,y'_j)
		\right] \right|^2, \\
		I_{12} & = \left| 
		\tilde{\Phi}(u;S) - \tilde{\Phi}(u;S')
		\right|^2.
	\end{align*}
	For the term $I_{11}$, we firstly consider the following calculations
	\begin{align*}
			\Phi(u;x_j,y_j) - & \Phi(u;x'_j,y'_j) = \frac{1}{2}\|\Gamma^{-\frac{1}{2}}\mathcal{L}_{x_j}\mathcal{G}(u)\|_{\mathcal{H}}^2 
			- \frac{1}{2}\|\Gamma^{-\frac{1}{2}}\mathcal{L}_{x'_j}\mathcal{G}(u)\|_{\mathcal{H}}^2 \\
			& + \left\langle \Gamma^{-\frac{1}{2}}y'_j, \Gamma^{-\frac{1}{2}}\mathcal{L}_{x'_j}\mathcal{G}(u)\right\rangle_{\mathcal{H}}
			- \left\langle \Gamma^{-\frac{1}{2}}y'_j, \Gamma^{-\frac{1}{2}}\mathcal{L}_{x_j}\mathcal{G}(u)\right\rangle_{\mathcal{H}}\\
			& + \left\langle \Gamma^{-\frac{1}{2}}y'_j, \Gamma^{-\frac{1}{2}}\mathcal{L}_{x_j}\mathcal{G}(u)\right\rangle_{\mathcal{H}}
			- \left\langle \Gamma^{-\frac{1}{2}}y_j, \Gamma^{-\frac{1}{2}}\mathcal{L}_{x_j}\mathcal{G}(u)\right\rangle_{\mathcal{H}},
	\end{align*}
	which implies that 
	\begin{align}\label{ApplicationEstimateI11}
		\begin{split}
			& |\Phi(u;x_j,y_j) - \Phi(u;x'_j,y'_j)| \leq \tilde{C}^2 M_1(\|u\|_{\mathcal{U}^{1-\tilde{s}}})
			\|\mathcal{C}_1^{-\frac{s}{2}}\Gamma^{-\frac{1}{2}}(\mathcal{L}_{x_j}\mathcal{G}(u) -  \mathcal{L}_{x_j'}\mathcal{G}(u))\|_{\mathcal{H}} \\
			& \quad 
			+ M_1(\|u\|_{\mathcal{U}^{1-\tilde{s}}})\|y_j-y'_j\|_{\mathcal{H}^{\alpha-s}}
			+ \|y'_j\|_{\mathcal{H}^{\alpha-s}}
			\|\mathcal{C}_1^{-\frac{s}{2}}\Gamma^{-\frac{1}{2}}(\mathcal{L}_{x_j}\mathcal{G}(u) -  \mathcal{L}_{x_j'}\mathcal{G}(u))\|_{\mathcal{H}}.
		\end{split}
	\end{align}
	The above estimate (\ref{ApplicationEstimateI11}) yields the following result
	\begin{align}\label{ApplicationEstimateI11Final}
		\text{I}_{11} \rightarrow 0, \quad\text{as }x_j\rightarrow x'_j \text{ in }\mathcal{X}\text{ and }
		y_j\rightarrow y'_j\text{ in }\mathcal{\mathcal{H}}^{\alpha-s}.
	\end{align}
	Concerned with the term $\text{I}_{12}$, estimate (\ref{ProofBayesFormula8}) already implies
	\begin{align}\label{ApplicationEstimateI12Final}
		\text{I}_{12} \rightarrow 0, \quad\text{as }x_j\rightarrow x'_j \text{ in }\mathcal{X}\text{ and }
		y_j\rightarrow y'_j\text{ in }\mathcal{\mathcal{H}}^{\alpha-s}.
	\end{align}
	Combining the above statements (\ref{ApplicationEstimateI11Final}) and (\ref{ApplicationEstimateI12Final}), we finally obtain
	\begin{align}\label{ApplicationEstimateI1Final}
		\text{I}_{1} \rightarrow 0, \quad\text{as }x_j\rightarrow x'_j \text{ in }\mathcal{X}\text{ and }
		y_j\rightarrow y'_j\text{ in }\mathcal{\mathcal{H}}^{\alpha-s}
	\end{align}
	by applying the Lebesgue's dominated convergence theorem.
	For the term $\text{I}_2$, we have 
	\begin{align*}
			\text{I}_2 \leq C\max\left(Z_m^{-3}, (Z'_m)^{-3}\right)|Z_m-Z'_m|^2.
	\end{align*}
	Similar to the derivation of (\ref{ApplicationEstimateI1Final}), we can prove $Z_m\rightarrow Z'_m$ as 
	$x_j\rightarrow x'_j \text{ in }\mathcal{X}\text{ and }y_j\rightarrow y'_j\text{ in }\mathcal{\mathcal{H}}^{\alpha-s}$. 
	Hence, we easily obtain 
	\begin{align}\label{ApplicationEstimateI2Final}
		\text{I}_{2} \rightarrow 0, \quad\text{as }x_j\rightarrow x'_j \text{ in }\mathcal{X}\text{ and }
		y_j\rightarrow y'_j\text{ in }\mathcal{\mathcal{H}}^{\alpha-s}.
	\end{align}
\end{proof}

\subsubsection*{\textbf{Proof of Lemma 3.6
		(Estimate Sub-Gaussian Parameters for the General Linear Inverse Problem)}} 

Before delving into the details of the proof, let us restate Lemma 3.6, 
incorporating the explicit estimations of $s_I$ and $s_{II}$ that were omitted in the main text.

\vskip 0.2 cm 
{\color{black}
\begin{lemma}\label{LemmaLinear-sIsII}
	Assume that Assumptions 3.1 and 3.2 
    in the main text are satisfied, and the conditions for the linear forward operator are reformulated as follows:
	\begin{align}\label{LinearForwardConditions}
		\begin{split}
			\|\mathcal{C}_1^{-\frac{s}{2}}\Gamma^{-\frac{1}{2}}\mathcal{L}_{x}\mathcal{G}u\|_{\mathcal{H}} \leq M_1 \|u\|_{\mathcal{U}^{1-\tilde{s}}} \text{ and }
			\|\mathcal{C}_1^{-\frac{s}{2}}\Gamma^{-\frac{1}{2}}\mathcal{L}_{x}\mathcal{G}(u_1 - u_2)\|_{\mathcal{H}} \leq
			M_2 \|u_1 - u_2\|_{\mathcal{U}^{1-\tilde{s}}},
		\end{split}
	\end{align}
	where $M_1$ and $M_2$ are two positive constants. Also, assume that the probability measure $\mathscr{E}$ has compact support, satisfying $\supp\mathscr{E}\subset B_{\tilde{s}}(R_u):=\{ u\in\mathcal{U} \,:\, \|u\|_{\mathcal{U}^{1-\tilde{s}}} \leq R_u \}$ with $R_u\in\mathbb{R}^{+}$. Let $\tilde{C}=\|\mathcal{C}_1^{\frac{s}{2}}\|_{\mathcal{B}(\mathcal{H})}$
	and assume that $\max\left\{3\tilde{\gamma}\tilde{C}^2M_2^2, 3\tilde{\gamma}^2M_1^2\text{Tr}(\mathcal{C}_1^s) \right\}\leq \min\left\{e^{-\theta_{\text{max}}},\lambda_1^{-1}\right\}$, where $\theta_{\text{max}}$ and $\lambda_1$ are as defined in Assumptions 3.1. 
    Furthermore, the base posterior measure is derived from formula (8). 
	Let us define the following two operators: 
	\begin{align*}
		\mathcal{C}_0(\theta_{\text{max}}) & = \sum_{k=1}^{N_c} e^{\theta_{\text{max}}}e_k\otimes e_k + \sum_{k=N_c+1}^{\infty} \lambda_{k}e_k \otimes e_k,  \\
		\mathcal{C}_0(\theta_{\text{min}}) & = \sum_{k=1}^{N_c} e^{\theta_{\text{min}}}e_k\otimes e_k + \sum_{k=N_c+1}^{\infty} \lambda_{k}e_k \otimes e_k,
	\end{align*}
	where $\theta_{\text{max}}$ and $\theta_{\text{min}}$ are two prespecified constants as in Subsection 2.3 
    of the main text. Consequently, for $s_{\text{I}}^2$ and $s_{\text{II}}^2$ as defined in Assumptions 2.8 and 2.9, 
    we have 
	\begin{align*}
		s_{\text{I}}^2 = &  
		\frac{8\tilde{C}^2M_2^2R_u^2}{\tilde{\gamma}} - \frac{\ln\det(\text{Id}-2\tilde{\gamma}\tilde{C}^2M_2^2\mathcal{C}_0(\theta_{\text{max}}))}{2\tilde{\gamma}^2}
		+ \frac{\tilde{C}^4M_1^4R_u^4}{4} 
		+ 2M_1^2\text{Tr}(\mathcal{C}_1^s)R_u^2 \\ 
		& + \max\left\{ \frac{\ln3}{3M_1^2\text{Tr}(\mathcal{C}_1^s)\lambda_{N_c+1}}, \frac{\ln3}{3M_1^2\text{Tr}(\mathcal{C}_1^s)e^{\theta_{\text{max}}}}, 2M_1^2\text{Tr}(\mathcal{C}_1^s) \right\}\text{Tr}(\mathcal{C}_0(\theta_{\text{max}})).
	\end{align*}
	\vskip -0.4 cm
	\begin{align*}
		s_{\text{II}}^2 = & \frac{\tilde{C}^4M_1^4R_u^4}{4} + \frac{12R_u^2\tilde{C}^2M_2^2}{\tilde{\lambda}} + 
		\frac{\tilde{C}^2M_2^2}{\tilde{\lambda}}\mathbb{E}_{(\mathbb{D},m)\sim\mathcal{T}}\left[ \frac{3\beta^2}{m^2} \mathbb{E}_{\bm{x}\sim\mathbb{D}_1^m}\text{Tr}(K_{\bm{x}}\Gamma_m K_{\bm{x}}^*) + 
		\frac{m}{\beta}\mathbb{E}_{\bm{x}\sim\mathbb{D}_{1}^m}\text{Tr}(\tilde{\mathcal{C}}_p)
		\right] \\
		&\, + \frac{12R_u^2}{\tilde{\lambda}} \mathbb{E}_{(\mathbb{D},m)\sim\mathcal{T}}\left[  \frac{\beta^2}{m^2}\mathbb{E}_{\bm{x}\sim\mathbb{D}_1^{m}}\|\mathcal{C}_0(\theta_{\text{max}})(\mathcal{L}_{\bm{x}}\mathcal{G})^{*}
		(\Gamma_m + \frac{\beta}{m}\mathcal{L}_{\bm{x}}\mathcal{G}\mathcal{C}_0(\theta_{\text{min}})(\mathcal{L}_{\bm{x}}\mathcal{G})^{*})^{-1}\mathcal{L}_{\bm{x}}\mathcal{G}
		\|_{\mathcal{B}(\mathcal{U}^{1-\tilde{s}})}^2 \right],
	\end{align*}
	where 
	\begin{align*}
	& \qquad\quad  \tilde{\mathcal{C}}_p^{-1} = (\mathcal{L}_{x}\mathcal{G})^*\Gamma_m^{-1}\mathcal{L}_{x}\mathcal{G} + 
	\frac{m}{\beta}\mathcal{C}_0(\theta_{\text{max}})^{-1}, \\
	& K_{\bm{x}} = \mathcal{C}_0(\theta_{\text{max}})^{\frac{1+\tilde{s}}{2}}(\mathcal{L}_{\bm{x}}\mathcal{G})^{*}(\Gamma_m + \frac{\beta}{m}\mathcal{L}_{\bm{x}}\mathcal{G}\mathcal{C}_0(\theta_{\text{min}})(\mathcal{L}_{\bm{x}}\mathcal{G})^{*})^{-1}.
	\end{align*}
	In addition, if the measurement variables $\bm{x}$ are dummy variables and the following boundedness assumptions hold true
	\begin{align*}
		& \quad \mathbb{E}_{(\mathbb{D},m)\sim\mathcal{T}}\left[\frac{\beta^2}{m^2} \mathbb{E}_{\bm{x}\sim\mathbb{D}_1^m}\text{Tr}(K_{\bm{x}}\Gamma_mK_{\bm{x}}^*)\right] < \infty, \\
		\mathbb{E}_{(\mathbb{D},m)\sim\mathcal{T}}\mathbb{E}_{S\sim\mathbb{D}^m}\| 
		& \mathcal{C}_0(\theta_{\text{max}})(\mathcal{L}_{\bm{x}}\mathcal{G})^{*}\big(\text{\small$\frac{m}{\beta}$}\Gamma_m  + \mathcal{L}_{\bm{x}}\mathcal{G}\mathcal{C}_0(\theta_{\text{min}})(\mathcal{L}_{\bm{x}}\mathcal{G})^{*}\big)^{-1}
		\mathcal{L}_{\bm{x}}\mathcal{G}
		\|_{\mathcal{B}(\mathcal{U}^{1-\tilde{s}})}^2 < \infty,
	\end{align*}
	we have 
	\begin{align*}
		s_{\text{I}}^2 = &  
		\max\left\{ \frac{\ln3}{3M_1^2\text{Tr}(\mathcal{C}_1^s)\lambda_{N_c+1}}, \frac{\ln3}{3M_1^2\text{Tr}(\mathcal{C}_1^s)e^{\theta_{\text{max}}}}, 2M_1^2\text{Tr}(\mathcal{C}_1^s) \right\}\text{Tr}(\mathcal{C}_0(\theta_{\text{max}})) \\
		& + \frac{\tilde{C}^4M_1^4R_u^4}{4} 
		+ 2M_1^2\text{Tr}(\mathcal{C}_1^s)R_u^2.
	\end{align*}
	and 
	\begin{align*}
		s_{\text{II}}^2 &  =\frac{\tilde{C}^4(M_1^4 + M_2^4)(R_u^4 + \text{Tr}(\mathcal{C}_p)^2)}{4} + 
		9\tilde{C}^4M_2^4 \left(\mathbb{E}_{(\mathbb{D},m)\sim\mathcal{T}}\left[ 
		\frac{\beta^2}{m^2} \left[\mathbb{E}_{\bm{x}\sim\mathbb{D}_1^m}\text{Tr}(K_{\bm{x}}\Gamma_mK_{\bm{x}}^*) \right]
		\right]\right)^2 \\
		& + 144R_u^4
		\left(\mathbb{E}_{(\mathbb{D},m)\sim\mathcal{T}}\left[
		\frac{\beta^2}{m^2}
		\mathbb{E}_{\bm{x}\sim\mathbb{D}_1^{m}}\|\mathcal{C}_0(\theta_{\text{max}})(\mathcal{L}_{\bm{x}}\mathcal{G})^{*}
		(\Gamma_m + \frac{\beta}{m}\mathcal{L}_{\bm{x}}\mathcal{G}\mathcal{C}_0(\theta_{\text{min}})(\mathcal{L}_{\bm{x}}\mathcal{G})^{*})^{-1}\mathcal{L}_{\bm{x}}\mathcal{G}
		\|_{\mathcal{B}(\mathcal{U}^{1-\tilde{s}})}^2
		\right]\right)^2,
	\end{align*}
	which are independent of $\tilde{\gamma}$ and $\tilde{\lambda}$.
\end{lemma}
}
\begin{remark}
	Before delving into the details of the proof, let us provide some discussions concerned with the above estimates of $s_{\text{I}}^2$ and $s_{\text{II}}^2$. The general estimates of $s_{\text{I}}^2$ and $s_{\text{II}}^2$ presented in the Lemma \ref{LemmaLinear-sIsII} differ slightly from those in the conventional sub-Gaussian case, as they depend on $\tilde{\gamma}$ and $\tilde{\lambda}$. 
	To construct the general estimates, we shall set the parameters to $\tilde{\gamma}=\frac{\gamma}{n\bar{m}}$ and $\tilde{\lambda} = \frac{\lambda}{n}$. This implies that we must select $\gamma$ and $\lambda$ to scale as $n\bar{m}$ and $n$, respectively, to ensure that $s_{\text{I}}$ and $s_{\text{II}}$ are finite. However, by incorporating some addition assumptions, we obtain the estimates that are independent of $\tilde{\gamma}$ and $\tilde{\lambda}$, which implies that the parameters $\gamma$ and $\lambda$ can be selected to scale as $n\sqrt{\bar{m}}$ and $\sqrt{n}$. Under this selection, the gap of between the transfer-error and the multi-task error would vanish as $n,\bar{m}\rightarrow\infty$. Finally, we should mention that the addition assumptions can be illustrated for the backward diffusion problem.
\end{remark}

\begin{proof}
	\textbf{Estimates of $V_i^1$}: 
	For the measurement data $y$, we assume that $y = \mathcal{L}_{x}\mathcal{G}u^{\dag} + \eta$ with 
	$u^{\dag}$ be the background true parameter sampled from $\mathscr{E}$. 
	Then the loss function can be reformulated as follow:
	\begin{align}\label{liziwujiexianxing1}
		\begin{split}
			\ell(u, z) = &\, \frac{1}{2}\|\Gamma^{-\frac{1}{2}}\mathcal{L}_{x}\mathcal{G}(u)\|_{\mathcal{H}}^2 - 
			\langle \Gamma^{-\frac{1}{2}}y, \Gamma^{-\frac{1}{2}}\mathcal{L}_{x}\mathcal{G}(u)\rangle_{\mathcal{H}} \\
			= &\, \frac{1}{2}\|\Gamma^{-\frac{1}{2}}\mathcal{L}_{x}(\mathcal{G}(u) - \mathcal{G}(u^{\dag}))\|_{\mathcal{H}}^2 - 
			\frac{1}{2}\|\Gamma^{-\frac{1}{2}}\mathcal{L}_{x}\mathcal{G}(u^{\dag})\|_{\mathcal{H}}^2 \\
			& \,- \langle\Gamma^{-\frac{1}{2}}\mathcal{L}_{x}\mathcal{G}(u), \Gamma^{-\frac{1}{2}}\eta \rangle_{\mathcal{H}}. 
		\end{split}
	\end{align}
	Relying on the above reformulation (\ref{liziwujiexianxing1}), for $i=1,\ldots,n$, we have 
	\begin{align}
		\begin{split}
			\mathcal{L}(u_i,\mathbb{D}_i) - \ell(u_i,z_{i}) = &\, \mathbb{E}_{z\sim\mathbb{D}_i}\ell(u_i,z) - \ell(u_i,z_i) \\ 
			= &\, I_1 + I_2 + I_3,
		\end{split}
	\end{align}
	where
	\begin{align*}
		I_1 = & \, \frac{1}{2}\mathbb{E}_{x\sim\mathbb{D}_1}
		\|\Gamma^{-\frac{1}{2}}\mathcal{L}_{x}(\mathcal{G}(u_i) - \mathcal{G}(u^{\dag}))\|_{\mathcal{H}}^2
		- \frac{1}{2}\|\Gamma^{-\frac{1}{2}}\mathcal{L}_{x_i}(\mathcal{G}(u_i) - \mathcal{G}(u^{\dag}))\|_{\mathcal{H}}^2, \\
		I_2 = &\, \frac{1}{2}\|\Gamma^{-\frac{1}{2}}\mathcal{L}_{x_i}\mathcal{G}(u_i^{\dag})\|_{\mathcal{H}}^2 - 
		\frac{1}{2}\mathbb{E}_{x\sim\mathbb{D}_1}\|\Gamma^{-\frac{1}{2}}\mathcal{L}_{x}\mathcal{G}(u_i^{\dag})\|_{\mathcal{H}}^2, \\
		I_3 = &\, \langle\Gamma^{-\frac{1}{2}}\mathcal{L}_{x_i}\mathcal{G}(u_i), \Gamma^{-\frac{1}{2}}\eta_i\rangle_{\mathcal{H}}. 
	\end{align*}
	Let us denote the constant $\tilde{C}:=\|\mathcal{C}_1^{\frac{s}{2}}\|_{\mathcal{B}(\mathcal{H})}$.
	For the term $I_1$, we have 
	\begin{align*}
		I_1 \leq & \frac{\tilde{C}^2}{2}\mathbb{E}_{x\sim\mathbb{D}_1}\| \mathcal{C}_1^{-\frac{s}{2}}\Gamma^{-\frac{1}{2}}
		\mathcal{L}_{x}(\mathcal{G}u_i - \mathcal{G}u_{i}^{\dag}) \|_{\mathcal{H}}^2 
		\leq \frac{\tilde{C}^2M_2^2}{2}\| u_i - u_i^{\dag} \|_{\mathcal{U}^{1-\tilde{s}}}^2. 
	\end{align*}
	{\color{black}Hence, we obtain
	\begin{align}\label{I1Est1}
		\begin{split} 
			& \mathbb{E}_{(\mathbb{D}_i,m_i)\sim\mathcal{T}}\mathbb{E}_{\theta\sim\mathscr{P}}\mathbb{E}_{u_i\sim\mathbb{P}_{S'}^{\theta}}
			\exp\left( \tilde{\gamma}I_1 \right)  \\
			\leq & \mathbb{E}_{(\mathbb{D}_i,m_i)\sim\mathcal{T}}\mathbb{E}_{\theta\sim\mathscr{P}}\mathbb{E}_{u_i\sim\mathbb{P}_{S'}^{\theta}}
			\!\exp\left( \tilde{\gamma}\tilde{C}^2M_2^2(\|u_i - f(S';\theta)\|_{\mathcal{U}^{1-\tilde{s}}}^2 
			\!+ \!\|f(S';\theta) - u_{i}^{\dag}\|_{\mathcal{U}^{1-\tilde{s}}}^2)\right) \\
			\leq & \exp\left( \tilde{\gamma}4\tilde{C}^2M_2^2R_u^2 \right)\mathbb{E}_{\theta\sim\mathscr{P}}\mathbb{E}_{u_i\sim\mathbb{P}_{S'}^{\theta}}
			\exp\left( \tilde{\gamma}\tilde{C}^2M_2^2(\|u_i - f(S';\theta)\|_{\mathcal{U}^{1-\tilde{s}}}^2 \right) \\
			\leq & \exp\left( \tilde{\gamma}4\tilde{C}^2M_2^2R_u^2 \right)
			\mathbb{E}_{\theta\sim\mathscr{P}}\mathbb{E}_{u_i\sim\mathbb{P}_0^{\theta_2}}\exp\left( \tilde{\gamma}\tilde{C}^2M_2^2\|u_i\|_{\mathcal{U}^{1-\tilde{s}}}^2 \right),
		\end{split}
	\end{align}
	where $\mathbb{P}_0^{\theta_2} := \mathcal{N}(0,\mathcal{C}_0(\theta_2))$. For the last term of the above inequality (\ref{I1Est1}), we have 
	\begin{align}\label{I1Est2}
		\begin{split}
			\mathbb{E}_{u_i\sim\mathbb{P}_0^{\theta_2}}\exp\left( \tilde{\gamma}\tilde{C}^2M_2^2\|u_i\|_{\mathcal{U}^{1-\tilde{s}}}^2 \right)
			\leq \det(\text{Id}-2\tilde{\gamma}\tilde{C}^2M_2^2\mathcal{C}_0(\theta_{\text{max}}))^{-1/2},
		\end{split}
	\end{align} 
	where the assumption $3\tilde{\gamma}\tilde{C}^2M_2^2 \leq \min\left\{e^{-\theta_{\text{max}}},\lambda_1^{-1}\right\}$ and Theorem 2.17 in \cite{Prato2014Book} are employed. 
	Combining (\ref{I1Est1}) and (\ref{I1Est2}), we find that 
	\begin{align}\label{I1FinalEst}
		\mathbb{E}_{\mathbb{D}_i\sim\mathcal{T}}\mathbb{E}_{\theta\sim\mathscr{P}}\mathbb{E}_{u_i\sim\mathbb{P}_{S'}^{\theta}}
		e^{\tilde{\gamma}I_1} \leq 
		\exp\left( \frac{\tilde{\gamma}^2}{2}\left[  
		\frac{8\tilde{C}^2M_2^2R_u^2}{\tilde{\gamma}} - \frac{\ln\det(\text{Id}-2\tilde{\gamma}\tilde{C}^2M_2^2\mathcal{C}_0(\theta_{\text{max}}))}{2\tilde{\gamma}^2}
		\right]\right).
	\end{align}
	}For the term $I_2$, we obviously have $\mathbb{E}_{z_i\sim\mathbb{D}_i}I_2 = 0$. 
	According to Assumptions 3.2 
	and the compactly supported condition illustrated in Lemma 3.6 
    of $\mathscr{E}$ in the main text, we find that
	\begin{align}\label{boundedEst2}
		\begin{split}
			\|\Gamma^{-\frac{1}{2}}\mathcal{L}_{x}\mathcal{G}(u_i^{\dag})\|_{\mathcal{H}}^2 \leq &\, 
			\tilde{C}^2 \|\mathcal{C}_1^{-\frac{s}{2}}\Gamma^{-\frac{1}{2}}\mathcal{L}_{x}\mathcal{G}(u_i^{\dag})\|_{\mathcal{H}}^2
			\leq \tilde{C}^2 M_1^2 R_u^2. 
		\end{split}
	\end{align}
	From the above estimate (\ref{boundedEst2}), we arrive at 
	\begin{align}\label{boundedEst3}
		\begin{split}
			-\frac{1}{2}\tilde{C}^2M_1^2 R_u^2 & \leq \text{I}_2 \leq \frac{1}{2}\tilde{C}^2M_1^2R_u^2. 
		\end{split}
	\end{align}
	Now, using the Hoeffding's lemma and estimate (\ref{boundedEst3}), we find that 
	\begin{align}\label{boundedEst4}
		\mathbb{E}_{(\mathbb{D}_i,m_i)\sim\mathcal{T}}\mathbb{E}_{x_i\sim\mathbb{D}_{i1}}e^{\tilde{\gamma}\text{I}_2}
		\leq \exp\left( 
		\frac{\tilde{\gamma}^2\tilde{C}^4M_1^4R_u^4}{8}
		\right). 
	\end{align}
	Let us denote $\{\zeta_k, \varphi_k\}_{k=1}^{\infty}$ be an eigen-system of the operator $\mathcal{C}_1$. 
	Since $\mathcal{C}_1\Gamma = \Gamma\mathcal{C}_1$, we know that the operator $\Gamma$ has an eigen-system $\{\gamma_k, \varphi_k\}_{k=1}^{\infty}$, which has the same eigenbasis as the operator $\mathcal{C}_1$ (Subsection 2.4 in \cite{Prato2006IDAnalysis}). 
	Without loss of generality, 
	the eigenvalues $\{\zeta_k\}_{k=1}^{\infty}$ and $\{\gamma_k\}_{k=1}^{\infty}$ are all rearranged in a descending order. 
	With these notations, we have 
	\begin{align}\label{boundedEst5-0}
		\begin{split}
			\mathbb{E}_{\eta_i\sim\mathbb{D}_3}e^{\tilde{\gamma}I_3} = &\, 
			\mathbb{E}_{\eta_i\sim\mathbb{D}_3}\exp\left( \tilde{\gamma}\langle\mathcal{C}_1^{-\frac{s}{2}}
			\Gamma^{-\frac{1}{2}}\mathcal{L}_{x_i}\mathcal{G}(u_i), \mathcal{C}_1^{\frac{s}{2}}\Gamma^{-\frac{1}{2}}\eta_i\rangle_{\mathcal{H}}
			\right) \\
			= &\, \mathbb{E}_{\eta_i\sim\mathbb{D}_3}\exp\left( \tilde{\gamma}\sum_{k=1}^{\infty}
			\langle\mathcal{C}_1^{-\frac{s}{2}}
			\Gamma^{-\frac{1}{2}}\mathcal{L}_{x_i}\mathcal{G}(u_i), \varphi_k\rangle_{\mathcal{H}}
			\langle\mathcal{C}_1^{\frac{s}{2}}\Gamma^{-\frac{1}{2}}\eta_i, \varphi_k\rangle_{\mathcal{H}}
			\right) \\
			\leq &\, \mathbb{E}_{\eta_i\sim\mathbb{D}_3}\prod_{k=1}^{\infty}\exp\left( 
			\tilde{\gamma}M_1 \|u_i\|_{\mathcal{U}^{1-\tilde{s}}} \zeta_k^{\frac{s}{2}}\gamma_k^{-\frac{1}{2}}\langle\eta_i, \varphi_k\rangle_{\mathcal{H}}
			\right)  \\
			= &\, \prod_{k=1}^{\infty}\frac{1}{\sqrt{2\pi}\gamma_k}\int_{\mathbb{R}}
			\exp\left(\tilde{\gamma}M_1 \|u_i\|_{\mathcal{U}^{1-\tilde{s}}} \zeta_k^{\frac{s}{2}}\gamma_k^{-\frac{1}{2}}\eta_{ik}
			- \frac{1}{2\gamma_k}\eta_{ik}^2
			\right)d\eta_{ik} \\
			= &\, \prod_{k=1}^{\infty}\exp\left( 
			\frac{\tilde{\gamma}^2M_1^2\|u_i\|_{\mathcal{U}^{1-\tilde{s}}}^2}{2}\zeta_k^{s}
			\right) \\
			= &\, \exp\left( \frac{\tilde{\gamma}^2}{2}M_1^2\text{Tr}(\mathcal{C}_1^s) \|u_i\|_{\mathcal{U}^{1-\tilde{s}}}^2 \right),
		\end{split}
	\end{align}
	where $\eta_{ik} = \langle\eta_i,\varphi_k\rangle_{\mathcal{H}}$. 
	{\color{black}Relying on the above estimate (\ref{boundedEst5-0}), we find that 
	\begin{align}\label{boundedEst5-1}
		\begin{split} 
			\mathbb{E}_{\eta\sim\mathbb{D}_3}\mathbb{E}_{\theta\sim\mathscr{P}}\mathbb{E}_{u_i\sim\mathbb{P}_{S'}^{\theta}}& e^{\tilde{\gamma}I_3} \leq \mathbb{E}_{\theta\sim\mathscr{P}}\mathbb{E}_{u_i\sim\mathbb{P}_{S'}^{\theta}}\exp\left( 
			\tilde{\gamma}^2M_1^2\text{Tr}(\mathcal{C}_1^s)\left[\|u_i - f(S';\theta)\|_{\mathcal{U}^{1-\tilde{s}}}^2 + R_u^2\right]\right) \\
			\leq & \exp\left( \tilde{\gamma}^2M_1^2\text{Tr}(\mathcal{C}_1^s)R_u^2 \right) 
			\mathbb{E}_{u\sim \mathbb{P}_0}\exp\left( \tilde{\gamma}^2M_1^2\text{Tr}(\mathcal{C}_1^2)\|u\|_{\mathcal{U}^{1-\tilde{s}}}^2 \right) \\
			\leq & \exp\left( \tilde{\gamma}^2M_1^2\text{Tr}(\mathcal{C}_1^s)R_u^2 \right) \det\left( \text{Id} - 2\tilde{\gamma}^2M_1^2\text{Tr}(\mathcal{C}_1^s)\mathcal{C}_0(\theta_{\text{max}}) \right)^{-1/2} \\
			= & \exp\left( 
			\tilde{\gamma}^2M_1^2\text{Tr}(\mathcal{C}_1^s)R_u^2 + \tilde{\gamma}^2 
			\frac{-\ln\det\left( \text{Id} - 2\tilde{\gamma}^2M_1^2\text{Tr}(\mathcal{C}_1^s)\mathcal{C}_0(\theta_{\text{max}}) \right)}{2\tilde{\gamma}^2}
			\right).
		\end{split}
	\end{align}
	Now we need to provide a detailed analysis of the term 
	$$M_d := \frac{-\ln\det\left( \text{Id} - 2\tilde{\gamma}^2M_1^2\text{Tr}(\mathcal{C}_1^s)\mathcal{C}_0(\theta_{\text{max}}) \right)}{2\tilde{\gamma}^2}$$ 
	appeared in the above estimate (\ref{boundedEst5-1}), which is important for our applications. 
	According to the definition of the determinant of the operator $\mathcal{C}_0(\theta_{\text{max}})$, we have 
	\begin{align}\label{estimateMd} 
		\begin{split}
			M_d & = \sum_{k=1}^{N_c} \frac{-\ln\left(1-2\tilde{\gamma}^2M_1^2\text{Tr}(\mathcal{C}_1^s)e^{\theta_{\text{max}}}\right)}{2\tilde{\gamma}^2} + \sum_{k=N_c+1}^{\infty}\frac{-\ln\left(1-2\tilde{\gamma}^2M_1^2\text{Tr}(\mathcal{C}_1^s)\lambda_k\right)}{2\tilde{\gamma}^2} \\
			& = 
			\sum_{k=1}^{N_c}e^{\theta_{\text{max}}} \frac{-\ln\left(1-2\tilde{\gamma}^2e^{\theta_{\text{max}}}M_1^2\text{Tr}(\mathcal{C}_1^s)\right)}{2\tilde{\gamma}^2e^{\theta_{\text{max}}}} + 
			\sum_{k=N_c+1}^{\infty}\lambda_k \frac{-\ln\left(1-2\tilde{\gamma}^2\lambda_kM_1^2\text{Tr}(\mathcal{C}_1^s)\right)}{2\tilde{\gamma}^2\lambda_k}.
		\end{split}
	\end{align}
	Since $\tilde{\gamma}$ is bounded, $\lambda_k\rightarrow 0$ as $k\rightarrow\infty$,  
	\begin{align*}
			\lim_{t\rightarrow 0}\frac{-\ln\left( 1-2tM_1^2\text{Tr}(\mathcal{C}_1^2) \right)}{2t} = M_1^2\text{Tr}(\mathcal{C}_1^s),
	\end{align*}
	\begin{align*}
			\frac{-\ln\left(1-2\tilde{\gamma}^2\lambda_kM_1^2\text{Tr}(\mathcal{C}_1^s)\right)}{2\tilde{\gamma}^2\lambda_k} \leq \frac{\ln3}{6M_1^2\text{Tr}(\mathcal{C}_1^s)\lambda_{N_c+1}}, \quad \forall k\geq N_c+1, 
	\end{align*}
	and 
	\begin{align*}
			\frac{-\ln\left(1-2\tilde{\gamma}^2e^{\theta_{\text{max}}}M_1^2\text{Tr}(\mathcal{C}_1^s)\right)}{2\tilde{\gamma}^2e^{\theta_{\text{max}}}} \leq \frac{\ln3}{6M_1^2\text{Tr}(\mathcal{C}_1^s)e^{\theta_{\text{max}}}},
	\end{align*}
	we know that 
	\begin{align}\label{estimateMdfinal}
		M_d \leq \max\left\{ \frac{\ln3}{6M_1^2\text{Tr}(\mathcal{C}_1^s)\lambda_1}, \frac{\ln3}{6M_1^2\text{Tr}(\mathcal{C}_1^s)e^{\theta_{\text{max}}}}, M_1^2\text{Tr}(\mathcal{C}_1^s) \right\}\text{Tr}(\mathcal{C}_0(\theta_{\text{max}})),
	\end{align}
	which is bounded when $\tilde{\gamma}\rightarrow 0$. Plugging the estimate (\ref{estimateMdfinal}) into the estimate (\ref{boundedEst5-1}), we finally arrive at 
	\begin{align*}
			\mathbb{E}_{\eta\sim\mathbb{D}_3}\mathbb{E}_{\theta\sim\mathscr{P}}\mathbb{E}_{u_i\sim\mathbb{P}_{S'}^{\theta}}e^{\tilde{\gamma}I_3} & \leq 
			\exp\Bigg( 
			\tilde{\gamma}^2 \Bigg[ M_1^2\text{Tr}(\mathcal{C}_1^s)R_u^2  \\
			& + 
			\max\left\{ \frac{\ln3}{6M_1^2\text{Tr}(\mathcal{C}_1^s)\lambda_1}, \frac{\ln3}{6M_1^2\text{Tr}(\mathcal{C}_1^s)e^{\theta_{\text{max}}}}, M_1^2\text{Tr}(\mathcal{C}_1^s) \right\}\text{Tr}(\mathcal{C}_0(\theta_{\text{max}}))
			\Bigg] \! \Bigg).
	\end{align*}
	Combining the above estimate of $I_3$ and the estimates (\ref{I1FinalEst}) and (\ref{boundedEst4}), we finally obtain the desired inequality
	\begin{align*}
		V_i^1 \leq \exp\Bigg( \frac{\tilde{\gamma}^2}{2}\Bigg[ &  
		\frac{8\tilde{C}^2M_2^2R_u^2}{\tilde{\gamma}} - \frac{\ln\det(\text{Id}-2\tilde{\gamma}\tilde{C}^2M_2^2\mathcal{C}_0(\theta_{\text{max}}))}{2\tilde{\gamma}^2}
		+ \frac{\tilde{\gamma}^2\tilde{C}^4M_1^4R_u^4}{4} + 2M_1^2\text{Tr}(\mathcal{C}_1^s)R_u^2   \\
		&  + 
		\max\left\{ \frac{\ln3}{6M_1^2\text{Tr}(\mathcal{C}_1^s)\lambda_1}, \frac{\ln3}{6M_1^2\text{Tr}(\mathcal{C}_1^s)e^{\theta_{\text{max}}}}, M_1^2\text{Tr}(\mathcal{C}_1^s) \right\}\text{Tr}(\mathcal{C}_0(\theta_{\text{max}}))
		\Bigg]\Bigg).
	\end{align*}
	}\textbf{Estimates of $V_i^2$}: 
	As in the estimates of $V_i^1$, we assume that the data $y_i$ is generated from some 
	true parameter $u_i^{\dag}$ sampled from the measure $\mathscr{E}$. Hence, we have 
	\begin{align}
			\mathcal{L}(\mathbb{Q}(S_i, \mathbb{P}_{S_i}^{\theta}), \mathbb{D}_i) = &\,
			\mathbb{E}_{u\sim\mathbb{Q}(S_i,\mathbb{P}_{S_i}^{\theta})}\mathbb{E}_{z\sim\mathbb{D}_i}\left[ 
			\frac{1}{2}\|\Gamma^{-\frac{1}{2}}\mathcal{L}_{x}\mathcal{G}u\|_{\mathcal{H}}^2 - 
			\langle \Gamma^{-\frac{1}{2}}\mathcal{L}_{x}\mathcal{G}u, \Gamma^{-\frac{1}{2}}y_i\rangle_{\mathcal{H}}
			\right] \nonumber \\
			= &\, \mathbb{E}_{u\sim\mathbb{Q}(S_i,\mathbb{P}_{S_i}^{\theta})}\mathbb{E}_{z\sim\mathbb{D}_i}\bigg[ 
			\frac{1}{2}\|\Gamma^{-\frac{1}{2}}\mathcal{L}_{x}\mathcal{G}u\|_{\mathcal{H}}^2 - 
			\langle \Gamma^{-\frac{1}{2}}\mathcal{L}_{x}\mathcal{G}u, \Gamma^{-\frac{1}{2}}\mathcal{L}_{x}\mathcal{G}u_i^{\dag}\rangle_{\mathcal{H}} \nonumber\\
			&\, -\langle \Gamma^{-\frac{1}{2}}\mathcal{L}_{x}\mathcal{G}u, \Gamma^{-\frac{1}{2}}\bm{\eta}\rangle_{\mathcal{H}}
			\bigg]  \label{LEst1} \\
			= &\, \mathbb{E}_{u\sim\mathbb{Q}(S_i,\mathbb{P}_{S_i}^{\theta})}\mathbb{E}_{x\sim\mathbb{D}_1}\left[ 
			\frac{1}{2}\|\Gamma^{-\frac{1}{2}}\mathcal{L}_{x}\mathcal{G}(u-u_i^{\dag})\|_{\mathcal{H}}^{2}\right] \nonumber\\
			& \, - \mathbb{E}_{x\sim\mathbb{D}_1}
			\left[ \frac{1}{2}\|\Gamma^{-\frac{1}{2}}\mathcal{L}_{x}\mathcal{G}u_i^{\dag}\|_{\mathcal{H}}^2 \right],\nonumber
	\end{align}
	which yields 
	\begin{align}
		\begin{split}
			V_i^2 \leq &\, \mathbb{E}_{(\mathbb{D}_i,m_i)\sim\mathcal{T}}\mathbb{E}_{S_i\sim\mathbb{D}_i^{m_i}}\mathbb{E}_{\theta\sim\mathscr{P}}
			\exp\left( \tilde{\lambda}\text{I}_1 + \tilde{\lambda}\text{I}_2 \right),
		\end{split}
	\end{align}
	where 
	\begin{align*}
		\text{I}_1 = &\, \mathbb{E}_{(\mathbb{D},m)\sim\mathcal{T}}\mathbb{E}_{S\sim\mathbb{D}^{m}}\mathbb{E}_{u\sim\mathbb{Q}(S,\mathbb{P}_{S}^{\theta})}
		\mathbb{E}_{x\sim\mathbb{D}_1}\left[ 
		\frac{1}{2}\|\Gamma^{-\frac{1}{2}}\mathcal{L}_{x}\mathcal{G}(u-u^{\dag})\|_{\mathcal{H}}^2
		\right], \\
		\text{I}_2 = &\, \mathbb{E}_{x\sim\mathbb{D}_1}\left[ \frac{1}{2}\|\Gamma^{-\frac{1}{2}}\mathcal{L}_{x}\mathcal{G}u_i^{\dag}\|_{\mathcal{H}}^2 \right] - 
		\mathbb{E}_{(\mathbb{D},m)\sim\mathcal{T}}\mathbb{E}_{S\sim\mathbb{D}^m}\mathbb{E}_{x\sim\mathbb{D}_1}\left[ 
		\frac{1}{2}\|\Gamma^{-\frac{1}{2}}\mathcal{L}_{x}\mathcal{G}u^{\dag}\|_{\mathcal{H}}^2 \right].
	\end{align*}
	For the term $\text{I}_2$, we easily obtain
	\begin{align*}
		-\frac{1}{2}\tilde{C}^2M_1^2R_u^2 \leq \text{I}_2 \leq \frac{1}{2}\tilde{C}^2M_1^2R_u^2. 
	\end{align*}
	The above boundedness estimate combined with the Hoeffding's lemma yields
	\begin{align}\label{linearEstI2} 
		\mathbb{E}_{\mathbb{D}_i\sim\mathcal{T}}\mathbb{E}_{S_i\sim\mathbb{D}_i^m}e^{\tilde{\lambda}\text{I}_2} \leq 
		\exp\left( \frac{\tilde{\lambda}^2 \tilde{C}^4M_1^4R_u^4}{8} \right).
	\end{align} 
	For the term $\text{I}_1$, we have 
	\begin{align}\label{linearEstI3}
		\begin{split}
			\text{I}_1 \leq &\, \frac{1}{2}\tilde{C}^2M_2^2\mathbb{E}_{(\mathbb{D},m)\sim\mathcal{T}}\mathbb{E}_{S\sim\mathbb{D}^m}
			\mathbb{E}_{u\sim\mathbb{Q}(S,\mathbb{P}_{S}^{\theta})}\|u-u^{\dag}\|_{\mathcal{U}^{1-\tilde{s}}}^2 \\
			= &\, \frac{1}{2}\tilde{C}^2M_2^2\bigg( 
			\mathbb{E}_{(\mathbb{D},m)\sim\mathcal{T}}\mathbb{E}_{S\sim\mathbb{D}^m}\mathbb{E}_{u\sim\mathbb{Q}(S,\mathbb{P}_{S}^{\theta})}\|u-u_p\|_{\mathcal{U}^{1-\tilde{s}}}^2  \\
			&\, \qquad\qquad + \mathbb{E}_{(\mathbb{D},m)\sim\mathcal{T}}\mathbb{E}_{S\sim\mathbb{D}^m}\mathbb{E}_{u\sim\mathbb{Q}(S,\mathbb{P}_{S}^{\theta})}\|u_p - u^{\dag}\|_{\mathcal{U}^{1-\tilde{s}}}^2
			\bigg) \\
			= &\, \frac{1}{2}\tilde{C}^2M_2^2\left( 
			\mathbb{E}_{(\mathbb{D},m)\sim\mathcal{T}}\mathbb{E}_{S\sim\mathbb{D}^m}\text{Tr}(\mathcal{C}_p) + 
			\mathbb{E}_{(\mathbb{D},m)\sim\mathcal{T}}\mathbb{E}_{S\sim\mathbb{D}^m}\mathbb{E}_{u\sim\mathbb{Q}(S,\mathbb{P}_{S}^{\theta})}\|u_p - u^{\dag}\|_{\mathcal{U}^{1-\tilde{s}}}^2
			\right).
		\end{split}
	\end{align}
	For the posterior covariance operator $\mathcal{C}_p$, we should notice that $\text{Tr}(\mathcal{C}_p)$ is bounded since the operator $\mathcal{C}_0$ is of trace class. 
	To obtain our final estimate, we need to analyze the term 
	$\mathbb{E}_{(\mathbb{D},m)\sim\mathcal{T}}\mathbb{E}_{S\sim\mathbb{D}^m}\mathbb{E}_{u\sim\mathbb{Q}(S,\mathbb{P}_{S}^{\theta})}\|u_p - u^{\dag}\|_{\mathcal{U}^{1-\tilde{s}}}^2$ in detail. Relying on the explicit form of $u_p$, we have 
	\begin{align}\label{linearEstI4}
		\begin{split}
			\mathbb{E}_{(\mathbb{D},m)\sim\mathcal{T}}\mathbb{E}_{S\sim\mathbb{D}^m}\|u_p - u^{\dag}\|_{\mathcal{U}^{1-\tilde{s}}}^2 
			\leq \text{I}_{11} + \text{I}_{12} + \text{I}_{13},
		\end{split}
	\end{align}
	{\color{black}where 
	\begin{align*}
		\text{I}_{11} = & 3\mathbb{E}_{(\mathbb{D},m)\sim\mathcal{T}}\mathbb{E}_{S\sim\mathbb{D}^m}\|f_m(S;\theta_1) - u^{\dag}\|_{\mathcal{U}^{1-\tilde{s}}}^2, \\
		\text{I}_{12} = & 3\mathbb{E}_{(\mathbb{D},m)\sim\mathcal{T}}\mathbb{E}_{S\sim\mathbb{D}^m}\|
		\mathcal{C}_0(\theta_{\text{max}})(\mathcal{L}_{\bm{x}}\mathcal{G})^{*}
		\big(\text{\small$\frac{m}{\beta}$}\Gamma_m + \mathcal{L}_{\bm{x}}\mathcal{G}\mathcal{C}_0(\theta_{\text{min}})(\mathcal{L}_{\bm{x}}\mathcal{G})^{*}\big)^{-1}\bm{\eta}
		\|_{\mathcal{U}^{1-\tilde{s}}}^2, \\
		\text{I}_{13} = & 3\mathbb{E}_{(\mathbb{D},m)\sim\mathcal{T}}\mathbb{E}_{S\sim\mathbb{D}^m}\|
		\mathcal{C}_0(\theta_{\text{max}}\!)(\mathcal{L}_{\bm{x}}\mathcal{G})^{*}\big(\text{\small$\frac{m}{\beta}$}\Gamma_m \! + \! \mathcal{L}_{\bm{x}}\mathcal{G}\mathcal{C}_0(\theta_{\text{min}})(\mathcal{L}_{\bm{x}}\mathcal{G})^{*}\big)^{-1}(
		\mathcal{L}_{\bm{x}}\mathcal{G}(u^{\dag} \! - \! f_m(S;\theta_1)))
		\|_{\mathcal{U}^{1-\tilde{s}}}^2.
	\end{align*}
	Relying on the assumptions of $f$ and $\mathscr{E}$, we have the following estimate for term $\text{I}_{11}$ and $\text{I}_{13}$
	\begin{align*}
			\text{I}_{11} \leq & 12 R_u^2, \\
			\text{I}_{13} \leq & 12 R_u^2 \mathbb{E}_{(\mathbb{D},m)\sim\mathcal{T}}\left[ 
			\frac{\beta^2}{m^2}\mathbb{E}_{\bm{x}\sim\mathbb{D}_1^{m}}\|\mathcal{C}_0(\theta_{\text{max}}\!)(\mathcal{L}_{\bm{x}}\mathcal{G})^{*}
			\big(\Gamma_m + \text{\small$\frac{\beta}{m}$}\mathcal{L}_{\bm{x}}\mathcal{G}\mathcal{C}_0(\theta_{\text{min}}\!)(\mathcal{L}_{\bm{x}}\mathcal{G})^{*}\big)^{-1}\mathcal{L}_{\bm{x}}\mathcal{G}
			\|_{\mathcal{B}(\mathcal{U}^{1-\tilde{s}})}^2
			\right].
	\end{align*}
	Remembering our assumptions on $\bm{\eta}$, we obtain 
	\begin{align}\label{linearEstI6}
		\text{I}_{12} \leq 3 \mathbb{E}_{(\mathbb{D},m)\sim\mathcal{T}}\left[ \frac{\beta^2}{m^2} \mathbb{E}_{\bm{x}\sim\mathbb{D}_1^m}\text{Tr}(K_{\bm{x}}\Gamma_mK_{\bm{x}}^*) \right],
	\end{align}
	where $K_{\bm{x}} = \mathcal{C}_0(\theta_{\text{max}})^{\frac{1+\tilde{s}}{2}}(\mathcal{L}_{\bm{x}}\mathcal{G})^{*}(\Gamma_m + \text{\small$\frac{\beta}{m}$}\mathcal{L}_{\bm{x}}\mathcal{G}\mathcal{C}_0(\theta_{\text{min}})(\mathcal{L}_{\bm{x}}\mathcal{G})^{*})^{-1}$. 
	Combining estimates from (\ref{linearEstI2}) to (\ref{linearEstI6}), we obtain the desired result. 
	In addition, if the following assumptions hold true
	\begin{align*}
		\mathbb{E}_{(\mathbb{D},m)\sim\mathcal{T}}\mathbb{E}_{S\sim\mathbb{D}^m}\| &
		\mathcal{C}_0(\theta_{\text{max}})(\mathcal{L}_{\bm{x}}\mathcal{G})^{*}\big(\text{\small$\frac{m}{\beta}$}\Gamma_m + \mathcal{L}_{\bm{x}}\mathcal{G}\mathcal{C}_0(\theta_{\text{min}})(\mathcal{L}_{\bm{x}}\mathcal{G})^{*}\big)^{-1}
		\mathcal{L}_{\bm{x}}\mathcal{G}
		\|_{\mathcal{B}(\mathcal{U}^{1-\tilde{s}})}^2 < \infty, 
	\end{align*}
	and 
	\begin{align*}
		\mathbb{E}_{(\mathbb{D},m)\sim\mathcal{T}}\left[\frac{\beta^2}{m^2} \mathbb{E}_{\bm{x}\sim\mathbb{D}_1^m}\text{Tr}(K_{\bm{x}}\Gamma_mK_{\bm{x}}^*)\right] < \infty,
	\end{align*}
	we can apply the Hoeffding's lemma to obtain an estimate of $s_{\text{II}}^2$ independent of $\tilde{\lambda}$.}
\end{proof}

\subsubsection*{\textbf{Proof of Lemma 3.7
		(Estimate Sub-Gaussian Parameters for the General Nonlinear Inverse Problem)}} 
\begin{proof}
	\textbf{Calculate} $s_{\text{I}}^2$. 
	The calculation of $s_{\text{I}}^2$ is similar to the first part proof of Lemma 3.6. 
	The only difference is the estimate of $I_1$ defined in the proof of Lemma 3.6. 
	Hence we will only provide the different parts in the following. The term $I_1$ is defined as follow:
	\begin{align*}
		I_1 = & \, \frac{1}{2}\mathbb{E}_{x\sim\mathbb{D}_1}
		\|\Gamma^{-\frac{1}{2}}\mathcal{L}_{x}(\mathcal{G}(u) - \mathcal{G}(u^{\dag}))\|_{\mathcal{H}}^2
		- \frac{1}{2}\|\Gamma^{-\frac{1}{2}}\mathcal{L}_{x_i}(\mathcal{G}(u_i) - \mathcal{G}(u^{\dag}))\|_{\mathcal{H}}^2.
	\end{align*}
	Obviously, we have $\mathbb{E}_{x_i\sim\mathbb{D}}I_1 = 0$. Let us denote $\tilde{C}:=\|\mathcal{C}_1^{\frac{s}{2}}\|_{\mathcal{B}(\mathcal{H})}$. 
	According to the truncated Gaussian prior measure assumption and the compactly supported condition of $\mathcal{E}$, i.e., $\supp\mathcal{E}\subset B_{\tilde{s}}(R_u)=\{ u\in\mathcal{U} \,:\, \|u\|_{\mathcal{U}^{1-\tilde{s}}} \leq R_u \}$, we find that 
	\begin{align*}
		\|\Gamma^{-\frac{1}{2}}\mathcal{L}_{x}(\mathcal{G}(u_i) - \mathcal{G}(u_i^{\dag}))\|_{\mathcal{H}}^2 \leq &\, 
		\tilde{C}^2 \|\mathcal{C}_1^{-\frac{s}{2}}\Gamma^{-\frac{1}{2}}\mathcal{L}_{x}(\mathcal{G}(u_i) - \mathcal{G}(u_i^{\dag}))\|_{\mathcal{H}}^2  \\
		\leq &\, \tilde{C}^2 M_2(\|u_i\|_{\mathcal{U}^{1-\tilde{s}}}, \|u_i^{\dag}\|_{\mathcal{U}^{1-\tilde{s}}})^2 
		\|u_i - u_i^{\dag}\|_{\mathcal{U}^{1-\tilde{s}}}^2 \\
		\leq &\, 4\tilde{C}^2 M_2(R_u,R_u)^2 R_u^2,
	\end{align*}
	which yields 
	\begin{align*}
		-2\tilde{C}^2M_2(R_u,R_u)^2R_u^2 & \leq \text{I}_1 \leq 2\tilde{C}^2M_2(R_u,R_u)^2R_u^2.
	\end{align*}
	{\color{black}Now, relying on the Hoeffding's lemma, we find that 
	\begin{align}\label{Example2EstimationI1buchong}
		\mathbb{E}_{(\mathbb{D}_i,m_i)\sim\mathcal{T}}\mathbb{E}_{x_i\sim\mathbb{D}_{i1}}e^{\tilde{\gamma}\text{I}_1}
		\leq \exp\left( 
		\frac{\tilde{\gamma}^2 4\tilde{C}^4M_2(R_u,R_u)^4R_u^4}{2}
		\right). 
	\end{align}
	Combining estimates (\ref{Example2EstimationI1buchong}) with the estimates (\ref{boundedEst4}) and (\ref{boundedEst5-0}) given in the proof of Lemma 3.6: 
	\begin{align*}
		\mathbb{E}_{(\mathbb{D}_i,m_i)\sim\mathcal{T}}\mathbb{E}_{x_i\sim\mathbb{D}_{i1}}e^{\tilde{\gamma}I_2} & \leq \exp\left( \frac{\tilde{\gamma}^2}{2}\frac{1}{4}\tilde{C}^4M_1(R_u)^4R_u^4 \right),  \\
		\mathbb{E}_{\eta\sim\mathbb{D}_3}\mathbb{E}_{\theta\sim\mathscr{P}}\mathbb{E}_{u_i\sim\mathbb{P}_{S'}^{\theta}}e^{\tilde{\gamma}I_3} & \leq \exp\left( \frac{\tilde{\gamma}^2}{2}M_1(R_u)^2\text{Tr}(\mathcal{C}_1^s)R_u^2 \right), 
	\end{align*} 
	we obtain the required estimate.}
	
	\textbf{Calculate} $s_{\text{II}}^2$. 
	Similar to the derivation of (\ref{LEst1}), we have 
	\begin{align*}
		\mathcal{L}(\mathbb{Q}(S_i, \mathbb{P}_{R_u}^{S_i,\theta}), \mathbb{D}_i) = &\, 
		\mathbb{E}_{u\sim\mathbb{Q}(S_i,\mathbb{P}_{R_u}^{S_i,\theta})}\mathbb{E}_{x\sim\mathbb{D}_1}
		\left[ \frac{1}{2}\|\Gamma^{-\frac{1}{2}}\mathcal{L}_{x}(\mathcal{G}(u) - \mathcal{G}(u_i^{\dag}))\|_{\mathcal{H}}^2 
		\right]  \\
		&\, - \mathbb{E}_{x\sim\mathbb{D}_1}\left[ \frac{1}{2}\|\Gamma^{-\frac{1}{2}}\mathcal{L}_{x}\mathcal{G}(u_i^{\dag})\|_{\mathcal{H}}^2 \right],
	\end{align*}
	where $\mathbb{P}_{R_u}^{S,\theta}$ is the truncated Gaussian prior measure defined in Remark 3.5. 
	Relying on the above formulation, we obtain
	\begin{align}\label{EstVi2}
		V_i^2 \leq \mathbb{E}_{(\mathbb{D}_i,m_i)\sim\mathcal{T}}\mathbb{E}_{S_i\sim\mathbb{D}_i^m}\mathbb{E}_{\theta\sim\mathscr{P}}
		\exp\left( \tilde{\lambda}\text{I}_1 + \tilde{\lambda}\text{I}_2 \right),
	\end{align}
	where 
	\begin{align*}
		\text{I}_1 = &\, \mathbb{E}_{(\mathbb{D},m)\sim\mathcal{T}}\mathbb{D}_{S\sim\mathbb{D}^m}
		\mathbb{E}_{u\sim\mathbb{Q}(S,\mathbb{P}_{R_u}^{S,\theta})}\mathbb{E}_{x\sim\mathbb{D}_1}
		\left[ \frac{1}{2}\|\Gamma^{-\frac{1}{2}}\mathcal{L}_{x}(\mathcal{G}(u) - \mathcal{G}(u^{\dag}))\|_{\mathcal{H}}^2 
		\right] \\
		&\, - \mathbb{E}_{u\sim\mathbb{Q}(S_i,\mathbb{P}_{R_u}^{S_i,\theta})}\mathbb{E}_{x\sim\mathbb{D}_1}
		\left[ \frac{1}{2}\|\Gamma^{-\frac{1}{2}}\mathcal{L}_{x}(\mathcal{G}(u) - \mathcal{G}(u_i^{\dag}))\|_{\mathcal{H}}^2 
		\right], \\
		\text{I}_2 = &\, \mathbb{E}_{x\sim\mathbb{D}_1}\left[ \frac{1}{2}\|\Gamma^{-\frac{1}{2}}\mathcal{L}_{x}\mathcal{G}(u_i^{\dag})\|_{\mathcal{H}}^2\right] -
		\mathbb{E}_{(\mathbb{D},m)\sim\mathcal{T}}\mathbb{D}_{S\sim\mathbb{D}^m}\mathbb{E}_{x\sim\mathbb{D}_1}\left[ \frac{1}{2}\|\Gamma^{-\frac{1}{2}}\mathcal{L}_{x}\mathcal{G}(u^{\dag})\|_{\mathcal{H}}^2\right].
	\end{align*}
	At this point, we should notice that the posterior measure $\mathbb{Q}(S,\mathbb{P}_{R_u}^{S,\theta})$ obtained 
	from the Bayes' formula with the truncated Gaussian measure has compact support which means 
	$\supp\mathbb{Q}(S, \mathbb{P}_{R_u}^{S,\theta})\subset B_{\tilde{s}}(R_u)$. Then we have 
	\begin{align}\label{InequalityEst}
		\begin{split}
			-2\tilde{C}^2M_1(R_u)^2 & \leq \text{I}_1 \leq 2\tilde{C}^2M_1(R_u)^2, \\
			-\frac{1}{2}\tilde{C}^2M_1(R_u)^2 &\leq \text{I}_2 \leq \frac{1}{2}\tilde{C}^2M_1(R_u)^2.
		\end{split}
	\end{align}
	Combining the above estimates (\ref{InequalityEst}) and estimate (\ref{EstVi2}), we could apply the Hoeffding's lemma to obtain
	\begin{align*}
		V_i^2 \leq \exp\left( \frac{\tilde{\lambda}^2 25\tilde{C}^4M_1(R_u)^4}{2} \right),
	\end{align*}
	which is just the desired result. 
\end{proof}

\subsubsection*{\textbf{Proof of Theorem 4.1
		(Bayes' Formula for the Hyper-Posterior)}} 
\begin{proof}
	Here we only provide details for the case of the unbounded loss function since the bounded loss function case is actually easier to be proved.
	{\color{black}In the main text, we consider the parameter space $\Theta = \Theta_1 \times \Theta_2$, where $\Theta_1$ is a separable Hilbert space and $\Theta_2 = \mathbb{R}^{N_c}$. Specifically, we assume the hyper-prior measure $\mathscr{P} = \mathscr{P}_1 \otimes \mathscr{P}_2$, where $\mathscr{P}_1$ denotes a probability measure defined on $\Theta_1$, and $\mathscr{P}_2$ denotes a probability measure defined on $\mathbb{R}^{N_c}$ with compact support $[\theta_{\text{min}},\theta_{\text{max}}]^{N_c}$. Here, the parameters $\theta_{\text{min}},\theta_{\text{max}}\in\mathbb{R}$ ($\theta_{\text{min}} < \theta_{\text{max}}$) are two fixed real numbers.} 
	Recall the proofs of Theorem 15 in \cite{Dashti2017}, we know that the key point is to prove $0 < Z_p < \infty$. 
	Through a trivial calculation, we have 
	\begin{align}
		Z_p = \int_{\Theta} \prod_{i=1}^{n}\bigg[ \int_{\mathcal{U}} 
		\exp\bigg( -\frac{\beta}{m_i}\sum_{j=1}^{m_i}\Phi(u;z_{ij}) \bigg) \mathbb{P}_{S_i}^{\theta}(du)
		\bigg]^{\frac{\lambda}{\lambda + n\beta}} \mathscr{P}(d\theta). 
	\end{align} 
	In the following, we assume the potential function is defined by 
	\begin{align*}
		\Phi(u;z_{ij}) = \frac{1}{2}\|\Gamma^{-\frac{1}{2}}\mathcal{L}_{x_{ij}}\mathcal{G}(u)\|_{\mathcal{H}}^2 
		- \langle\Gamma^{-\frac{1}{2}}y_{ij}, \Gamma^{-\frac{1}{2}}\mathcal{L}_{x_{ij}}\mathcal{G}(u)\rangle_{\mathcal{H}}.
	\end{align*}
	For the potential function defined as in Remark 3.4, 
	similar techniques can be used to derive the 
	desired results. As in the proof of Theorem 3.3, 
	we find that for a small enough constant $\delta_1$ such that
	\begin{align}\label{AlgBoundZpProof1}
		\frac{\beta}{m_i}\sum_{j=1}^{m_i}\Phi(u;z_{ij}) \geq -\frac{1}{4\delta_1}\left(\frac{\beta}{m_i}\sum_{j=1}^{m_i}\|y_{ij}\|_{\mathcal{H}^{\alpha-s}}\right)^2
		- \delta_1 M_1(\|u\|_{\mathcal{U}^{1-\tilde{s}}})^2. 
	\end{align}
	Obviously, we have 
	\begin{align*}
		\frac{d\mathbb{P}_{S_i}^{\theta}}{d\mathbb{P}_0}(u) = \exp\bigg( 
		-\tilde{\Phi}(u;S_i)	\bigg),
	\end{align*}
	{\color{black}where 
	\begin{align*}
	\tilde{\Phi}(u;S_i)= & -\langle \mathcal{C}_0(\theta_2)^{-\frac{1}{2}}f(S_i;\theta_1), \mathcal{C}_0(\theta_2)^{-\frac{1}{2}}(u-f(S_i;\theta_1))\rangle_{\mathcal{U}} 
	+ \frac{1}{2}\|\mathcal{C}_0(\theta_2)^{-\frac{1}{2}}f(S_i;\theta_1)\|_{\mathcal{U}}^2  \\
	& + \frac{1}{2}\sum_{k=1}^{N_c}(e^{\theta_{2k}} - \lambda_k^{-1})u_k^2 + \sum_{k=1}^{N_c}(\theta_{2k} - \ln\lambda_k).
	\end{align*} 
	Exactly the same as for the proof of Theorem 3.3, 
	we derive 
	\begin{align}\label{AlgBoundZpProof2}
		-\tilde{\Phi}(u;S_i) \leq & \delta_2 \|u\|_{\mathcal{U}^{1-\tilde{s}}}^2 
		+ \bigg( \frac{\lambda_1^2}{4\delta_2e^{2\theta_{\text{min}}}} - \frac{3\lambda_1}{2e^{\theta_{\text{min}}}} \bigg) R_u^2 \\
		& + \frac{1}{2}\sum_{k=1}^{N_c}(\lambda_k^{-1}-e^{\theta_{2k}})u_k^2 + \sum_{k=1}^{N_c}(\ln\lambda_k-\theta_{2k}).
	\end{align}
	Combining the estimates (\ref{AlgBoundZpProof1}) and (\ref{AlgBoundZpProof2}), for sufficiently small positive real numbers $\delta_1$ and $\delta_2$, we find 
	\begin{align}\label{AlgBoundZpProof3}
		\begin{split}
		\int_{\mathcal{U}} &
		\exp\bigg( -\frac{\beta}{m_i}\sum_{j=1}^{m_i}\Phi(u;z_{ij}) \bigg) \mathbb{P}_{S_i}^{\theta}(du) \\
		\leq & 
		\exp\bigg( \frac{1}{4\delta_1}\left(\frac{\beta}{m_i}\sum_{j=1}^{m_i}\|y_{ij}\|_{\mathcal{H}^{\alpha-s}}\!\!\right)^2 \!\!+ \bigg( \frac{\lambda_1^2}{4\delta_2e^{2\theta_{\text{min}}}} - \frac{3\lambda_1}{2e^{\theta_{\text{min}}}} \bigg)R_u^2
		+ \frac{N_c(\ln\lambda_1 + \theta_{\text{max}})}{2}
		\bigg)F,
		\end{split} 
	\end{align}
	where 
	\begin{align*} 
		F = \int_{\mathcal{U}} \exp\bigg( 
		\delta_1 M_1(\|u\|_{\mathcal{U}^{1-\tilde{s}}})^2 + \delta_2\|u\|_{\mathcal{U}^{1-\tilde{s}}}^2 + 
		\frac{1}{2}\sum_{k=1}^{N_c}(\lambda_k^{-1} - e^{-\theta_{2k}})u_k^2
		\bigg)\mathbb{P}_0(du) < +\infty.
	\end{align*}
	Plugging estimate (\ref{AlgBoundZpProof3}) into the formula of $Z_p$, we derive $Z_p < \infty$. 
	Using the assumptions on the forward operator and the function $f(S;\theta)$, we have 
	\begin{align}\label{AlgBoundZpProof4}
		\begin{split}
			\sum_{j=1}^{m_i}\Phi(u;z_{ij}) \leq & \frac{1}{2}\|\mathcal{C}_0^{\frac{s}{2}}\|_{\mathcal{B}(\mathcal{U})}^2
			M_1(\|u\|_{\mathcal{U}^{1-\tilde{s}}})^2 + M_1(\|u\|_{\mathcal{U}^{1-\tilde{s}}})\sum_{j=1}^{m_i}\|y_{ij}\|_{\mathcal{H}^{\alpha-s}} \\
			\tilde{\Phi}(u;S_i) \leq & \, \frac{3R_u^2}{2} + R_u \|u\|_{\mathcal{U}^{1-\tilde{s}}} + \frac{N_c}{2e^{\theta_{\text{min}}}}\|u\|_{\mathcal{U}^{1-\tilde{s}}}^2 + N_c(\theta_{\text{max}} + \ln\lambda_1). 
		\end{split}
	\end{align}
	Using estimates in (\ref{AlgBoundZpProof4}), we derive 
	\begin{align}\label{AlgBoundZpProof5}
		\int_{\mathcal{U}}\exp\bigg( -\frac{\beta}{m_i}\sum_{j=1}^{m_i}\Phi(u;z_{ij}) \bigg) \mathbb{P}_{S_i}^{\theta}(du) > 0 
	\end{align}
	independent of the parameter $\theta$, which yields $Z_p > 0$.}
\end{proof}

\subsection{Proof Details of Subsection \ref{AppA4}}\label{AppB2}
\subsubsection*{\textbf{Verify Condition (\ref{BackDiffusionInequalities}) for the Backward Diffusion Problem}}
\begin{proof}
	Since the operator $\mathcal{C}_0$ has the following eigen-system decomposition: 
	\begin{align}\label{Example11}
		\mathcal{C}_0e_k = \lambda_ke_k, \quad \text{for all } k=1,2,\ldots,
	\end{align}
	we easily find that 
	\begin{align}\label{Example12}
		Ae_k = \sqrt{\frac{\lambda}{\lambda_k}}e_k, \quad \mathcal{C}_1e_k = \frac{\lambda_k}{\lambda}e_k, \quad \text{for all} k=1,2,\ldots. 
	\end{align}
	Employing (\ref{Example11}) and (\ref{Example12}), we obtain 
	\begin{align*}
		\|\mathcal{C}_1^{-\frac{s}{2}}\Gamma^{-\frac{1}{2}}e^{-AT}u\|_{\mathcal{H}}^2 = & \tau^{-2}
		\left\| \sum_{k=1}^{\infty}u_k \left( \frac{\lambda}{\lambda_k} \right)^{\frac{s}{2}} e^{-T\sqrt{\frac{\lambda}{\lambda_k}}}e_k \right\|_{\mathcal{H}}^2 \\
		= & \tau^{-2}\sum_{k=1}^{\infty}u_k^2 \frac{\lambda^s}{\lambda_k^s}e^{-2T\sqrt{\frac{\lambda}{\lambda_k}}} \\
		= & \tau^{-2}\sum_{k=1}^{\infty}\lambda_k^{-(1-\tilde{s})}u_k^2 \lambda_k^{1-s}\frac{\lambda^s}{\lambda_k^s}e^{-2T\sqrt{\frac{\lambda}{\lambda_k}}} \\
		\leq & \tau^{-2}\left( 
		\sup_{k}\lambda^s\lambda_k^{1-2s}e^{-2T\sqrt{\frac{\lambda}{\lambda_k}}}
		\right)\|u\|_{\mathcal{U}^{1-\tilde{s}}}^2,
	\end{align*}
	which is the required estimate. 
\end{proof}

\subsubsection*{\textbf{Calculate $s_{\text{I}}$ and $s_{\text{II}}$ for the Backward Diffusion Problem}}
\begin{proof}
	For every $u\in\mathcal{H}$, denote $u_k = \langle u, e_k\rangle_{\mathcal{H}}$, we have 
	\begin{align*}
		\|\mathcal{C}_0^{\frac{s}{2}}u\|_{\mathcal{H}}^2 = & \sum_{k=1}^{\infty}\lambda_k^{s} u_k^2 \leq \lambda_1^{s}\|u\|_{\mathcal{H}}^2.
	\end{align*}
	Taking $u = u_1e_1$, we obtain $\|\mathcal{C}_0^{\frac{s}{2}}u\|_{\mathcal{H}}^2\geq \lambda_1^{s}\|u\|_{\mathcal{H}}^2$, 
	which yields $\tilde{C} = \lambda^{-s/2}\|\mathcal{C}_0^{\frac{s}{2}}\|_{\mathcal{B}(\mathcal{H})}= \left( \frac{\lambda_1}{\lambda} \right)^{s/2}$.
	Plugging $\tilde{C}$, $M_1$, and $M_2$ into the equation of $s_{\text{I}}^2$, we have 
	\begin{align*}
		s_{\text{I}}^2 = & 
		\frac{\lambda_1^{2s}M_s^2 R_u^4}{4\lambda^{2s}\tau^4} + 
		\frac{2M_s \text{Tr}(\mathcal{C}_0^s)R_u^2 }{\tau^2\lambda^2} 
		 + \max\left\{ 
		\frac{\tau^2\lambda^s\ln3}{3M_s\text{Tr}(\mathcal{C}_0^s)\min\{\lambda_{N_c+1},e^{\theta_{\text{max}}}\}}, \frac{2M_s\text{Tr}(\mathcal{C}_0^s)}{\tau^2\lambda^2}
		\right\}\text{Tr}(\mathcal{C}_0),
	\end{align*}
	where $M_s = \sup_k \left[ \lambda^{s}\lambda_k^{1-2s} e^{-2T\sqrt{\frac{\lambda}{\lambda_k}}} \right].$
	{\color{black}Concerned with $s_{\text{II}}^2$, we firstly give the following two estimates: 
	\begin{align}\label{liziapp1}
		\begin{split}
			\frac{m}{\beta}\text{Tr}(\tilde{\mathcal{C}}_p) = & 
			\frac{m}{\beta}\sum_{k=1}^{\infty}\langle\tilde{\mathcal{C}}_pe_k, e_k\rangle_{\mathcal{H}} \\
			= & \frac{m}{\beta}\sum_{k=1}^{N_c}\frac{1}{\tau^{-2}me^{-2\lambda_k^{-1/2}T} + \frac{m}{\beta}e^{-\theta_{\text{max}}}} 
			+ \frac{m}{\beta}\sum_{k=N_c+1}^{\infty}\frac{1}{\tau^{-2}me^{-2\lambda_k^{-1/2}T} + \frac{m}{\beta}\lambda_k^{-1}} \\
			= & \sum_{k=1}^{N_c}\frac{1}{\tau^{-2}\beta e^{-2\lambda_k^{-1/2}T} + e^{-\theta_{\text{max}}}} 
			+\sum_{k=N_c+1}^{\infty} \frac{1}{\tau^{-2}\beta e^{-2\lambda_k^{-1/2}T} + \lambda_k^{-1}}  \\
			\leq &  N_ce^{\theta_{\text{max}}} + \sum_{k=N_+1}^{\infty}\lambda_k \\
			\leq & N_ce^{\theta_{\text{max}}} + \text{Tr}(\mathcal{C}_0),
		\end{split}
	\end{align}
	and
	\begin{align}\label{liziapp2}
		\begin{split}
			\frac{3\beta^2}{m^2} \mathbb{E}_{\bm{x}\sim\mathbb{D}_1^m}\text{Tr}(K_{\bm{x}}\Gamma_m K_{\bm{x}}^*)= &  
			\tau^2\frac{3\beta^2}{m^2} \sum_{k=1}^{\infty}\langle K_{\bm{x}}K_{\bm{x}}^* e_k, e_k\rangle_{\mathcal{H}} \\
			= & \tau^2\frac{3\beta^2}{m^2} \Bigg(\sum_{k=1}^{N_c} e^{\theta_{\text{max}}(1+s)}e^{-2\sqrt{\frac{\lambda}{\lambda_k}}T}
			\left( \tau^2 + \frac{\beta}{m}e^{-2\sqrt{\frac{\lambda}{\lambda_k}}T}e^{\theta_{\text{min}}} \right)^{-2} \\
			& + \sum_{k=N_c}^{\infty}
			\lambda_k^{1+s}e^{-2\sqrt{\frac{\lambda}{\lambda_k}}T}
			\left( \tau^2 + \frac{\beta}{m}e^{-2\sqrt{\frac{\lambda}{\lambda_k}}T}\lambda_k \right)^{-2} \Bigg) \\
			\leq & \frac{3\beta}{m} N_c e^{(\theta_{\text{max}} - \theta_{\text{min}})s} + \frac{3\beta}{m} \sum_{k=N_c+1}^{\infty}\lambda_k^{s}  \\
			= & \frac{3\beta}{m}\Big(N_c e^{(\theta_{\text{max}} - \theta_{\text{min}})s} + \text{Tr}\left( \mathcal{C}_0^{s} \right) \Big).
		\end{split}
	\end{align}
	From the following estimate
	\begin{align*}
		& \|\mathcal{C}_0(\theta_{\text{max}})(\mathcal{L}_{\bm{x}}\mathcal{G})^{*}
		(\Gamma_m + \frac{\beta}{m}\mathcal{L}_{\bm{x}}\mathcal{G}\mathcal{C}_0(\theta_{\text{min}})(\mathcal{L}_{\bm{x}}\mathcal{G})^{*})^{-1}\mathcal{L}_{\bm{x}}\mathcal{G}u\|_{\mathcal{H}^{1-s}}^2  \\
		= & \sum_{k=1}^{N_c}\left[ 
		e^{\theta_{\text{max}}}e^{-2\sqrt{\frac{\lambda}{\lambda_k}}T}\left( \tau^2 + \frac{\beta}{m}e^{-2\sqrt{\frac{\lambda}{\lambda_k}}T}e^{\theta_{\text{min}}} \right)^{-1}
		\right]^2\lambda_k^{s-1}u_k^2  \\
		& + 
		\sum_{k=N_c}^{\infty}\left[ 
		\lambda_ke^{-2\sqrt{\frac{\lambda}{\lambda_k}}T}\left( \tau^2 + \frac{\beta}{m}e^{-2\sqrt{\frac{\lambda}{\lambda_k}}T}\lambda_k \right)^{-1}
		\right]^2\lambda_k^{s-1}u_k^2 \\
		\leq & \frac{m^2}{\beta^2}e^{2(\theta_{\text{max}} - \theta_{\text{min}})}\|u\|_{\mathcal{H}^{1-s}}^2, 
	\end{align*}
	we derive 
	\begin{align*}
			\frac{12\beta^2}{m^2} R_u^2 \mathbb{E}_{\bm{x}\sim\mathbb{D}_1^{m}}\|\mathcal{C}_0(\mathcal{L}_{\bm{x}}\mathcal{G})^{*}
			(\Gamma_m + \frac{\beta}{m}\mathcal{L}_{\bm{x}}\mathcal{G}\mathcal{C}_0(\mathcal{L}_{\bm{x}}\mathcal{G})^{*})^{-1}\mathcal{L}_{\bm{x}}\mathcal{G}
			\|_{\mathcal{B}(\mathcal{U}^{1-\tilde{s}})}^2 \leq 12 R_u^2 e^{2(\theta_{\text{max}} - \theta_{\text{min}})}. 
	\end{align*}
	}Combining the above estimate with estimates (\ref{liziapp1}) and (\ref{liziapp2}), we obtain the desired formula of $s_{\text{II}}^2$, 
	which completes the proof. 
\end{proof}


\bibliographystyle{agsm}
\bibliography{reference}
\end{appendix}